\documentclass[11pt, twoside, openright]{book} 

\usepackage[utf8x]{inputenc}
\usepackage[italian, english]{babel} 
\usepackage[a4paper]{geometry}
\usepackage{graphicx}
\usepackage[all,cmtip]{xy}
\usepackage{amsmath}
\usepackage{amssymb}
\usepackage{amsfonts}
\usepackage{latexsym}
\usepackage{amsthm}
\usepackage{enumerate}
\usepackage{multirow}
\usepackage{hyperref}  \hypersetup{pdfborder={0 0 0}, colorlinks=false, pdfauthor={Giovanni Staglian{\`o}}, pdftitle={On special quadratic birational transformations of a projective space} }
\usepackage{tikz,pgf}
\usepackage{color}
\usepackage{verbatim} 
\usepackage{fancyhdr}
\usepackage{mathptmx}
\usepackage{appendix}

\theoremstyle{plain}
\newtheorem{proposition}{Proposition}[chapter]
\newtheorem{theorem}[proposition]{Theorem}
\newtheorem{lemma}[proposition]{Lemma}
\newtheorem{corollary}[proposition]{Corollary}
\newtheorem{fact}[proposition]{Fact}
\theoremstyle{definition}
\newtheorem{definition}[proposition]{Definition}
\newtheorem{assumption}[proposition]{Assumption}
\newtheorem{notation}[proposition]{Assumption}
\newtheorem{conjecture}[proposition]{Conjecture}
\newtheorem{example}[proposition]{Example}
\theoremstyle{remark}
\newtheorem{remark}[proposition]{Remark}
\newtheorem{claim}{Claim}[proposition]
\newtheorem{case}{Case}[proposition]
\newtheorem{subcase}{Subcase}[case]

\newcommand{\sS}{{\mathbf{S}}}    
\newcommand{\Q}{{\mathbf{Q}}}     
\newcommand{\B}{{\mathfrak{B}}}   
\newcommand{\I}{{\mathcal{I}}}    
\renewcommand{\L}{{\mathcal{L}}}  
\renewcommand{\O}{{\mathcal{O}}}  
\newcommand{\T}{{\mathcal{T}}}    
\newcommand{\E}{{\mathcal{E}}}    
\newcommand{\N}{{\mathcal{N}}}    
\newcommand{\CC}{{\mathbb{C}}}    
\newcommand{\FF}{\mathbb{F}}      
\newcommand{\GG}{{\mathbb{G}}}    
\newcommand{\ZZ}{{\mathbb{Z}}}    
\newcommand{\PP}{{\mathbb{P}}}    
\newcommand{\Sec}{{\mathrm{Sec}}} 
\newcommand{\Bl}{{\mathrm{Bl}}}   
\newcommand{\reg}{{\mathrm{reg}}}     
\newcommand{\sing}{{\mathrm{sing}}}   
\newcommand{\Pic}{{\mathrm{Pic}}}     
\newcommand{\Hilb}{{\mathrm{Hilb}}}   
\newcommand{\G}{{\mathcal{G}}}        
\newcommand{\Graph}{{\mathrm{Graph}}} 
\newcommand{\rk}{{\mathrm{rank}}}       

\makeatletter

\newcommand{\Rmnum}[1]{\expandafter\@slowromancap\romannumeral #1@}
\makeatother

\fancypagestyle{plain}{
\fancyhf{}

}
\makeatletter
\def\cleardoublepage{\clearpage\if@twoside \ifodd\c@page\else
    \hbox{}
    \thispagestyle{plain}
    \newpage
    \if@twocolumn\hbox{}\newpage\fi\fi\fi}
\makeatother \clearpage{\pagestyle{plain}\cleardoublepage}
\renewcommand{\chaptermark}[1]{\markboth{#1}{}}

\begin{document}

\frontmatter 
\begin{titlepage}
\ifpdf
\begin{tikzpicture}[remember picture,overlay]
\node [opacity=0.05,scale=2.0] at (current page.center) {\includegraphics{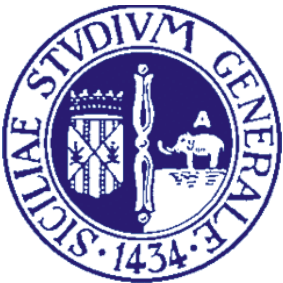}};
\end{tikzpicture} 
\else
\fi 
\begin{center}
\line(1,0){418}
\begin{huge}
\textbf{UNIVERSIT\`A DEGLI STUDI DI CATANIA\\}
\end{huge}
\begin{LARGE}
\textsl{Facolt\`a di Scienze Matematiche, Fisiche e Naturali\\}
\end{LARGE}
\begin{LARGE}
\textsl{Dipartimento di Matematica e Informatica\\}
\end{LARGE}
\line(1,0){418}
\end{center}
\begin{center}
\vspace{2.0cm}
\begin{LARGE}
\begin{center}
\textbf{Giovanni Staglian\`o}
\end{center}
\end{LARGE}
\vspace{1.0cm}
\begin{Huge} 
\begin{center}
\textbf{ON SPECIAL QUADRATIC BIRATIONAL TRANSFORMATIONS OF A PROJECTIVE SPACE} 
\end{center}
\end{Huge} 
\vspace{0.5cm}
\begin{LARGE}
\textsc{Tesi di Dottorato di Ricerca in Matematica (\Rmnum{23} Ciclo)\\}
\end{LARGE}
\vspace{2.5cm}
\begin{LARGE}
\begin{center} 
\begin{tabular}{@{}l} 
\textsc{Coordinatore:}\\ 
\textbf{Prof. Alfonso Villani}
\end{tabular}
\hfill
\begin{tabular}{l@{}} 
\textsc{Supervisore:}\\ 
\textbf{Prof. Francesco Russo}
\end{tabular}
\end{center}
\end{LARGE}
\vspace{0.7cm}
\line(1,0){418}
\vspace{0.3cm}
\begin{LARGE}
\textsc{Dicembre 2012}
\end{LARGE}
\line(1,0){418}
\end{center}
\end{titlepage}

\thispagestyle{empty}
    \null\vspace{\stretch {1}}
        \begin{flushright}
              {\itshape  To Maria Luisa}
        \end{flushright}
\vspace{\stretch{2}}\null

\cleardoublepage\thispagestyle{empty}
\begin{center}\textsc{Ph.D. Thesis in Mathematics (\Rmnum{23} Cycle) \\} \end{center}
\begin{center}\textbf{\Large ON SPECIAL QUADRATIC BIRATIONAL TRANSFORMATIONS OF A PROJECTIVE SPACE \\} \end{center}
\begin{center}\textbf{ Giovanni Staglian\`o \\} \end{center}

\begin{enumerate}
\item[] \textsc{Supervisor:}  
\begin{enumerate}
\item[] \textbf{Prof. Francesco Russo} \dotfill 
\end{enumerate}
\item[] \textsc{Coordinator:}
\begin{enumerate}
\item[] \textbf{Prof. Alfonso Villani} \dotfill 
\end{enumerate}
\end{enumerate}

\newenvironment{abstract} 
    {\null\vfill
\begin{center}\scshape\abstractname\end{center}
    }{\vfill\null}
\begin{abstract}
A birational map from a projective space onto a 
not too much singular projective variety
with a single irreducible non-singular base locus scheme 
(\emph{special birational transformation}) is a rare enough phenomenon 
to allow meaningful and concise classification results.

We shall concentrate on transformations defined by quadratic equations onto some 
varieties (especially 
projective hypersurfaces of small degree), where quite surprisingly the base 
loci are interesting projective manifolds appearing in other contexts;
for example, 
exceptions for adjunction theory, small degree or small codimensional manifolds, 
Severi or more generally homogeneous varieties.

In particular, we shall classify:
\begin{itemize} 
\item quadro-quadric transformations into a quadric hypersurface;
\item quadro-cubic transformations into a del Pezzo variety;
\item transformations whose base locus (scheme) has dimension at most three.
\end{itemize}
\end{abstract}

\newenvironment{acknowledgements} 
    {\clearpage\thispagestyle{empty}\null\vfill\begin{center} 
    \itshape Acknowledgements\end{center}} 
    {\vfill\null}

\begin{acknowledgements}
I would like to thank my supervisor, Prof. Francesco Russo,
for his continued guidance and support throughout the duration of my Ph.D.
He taught me a lot with his mathematical knowledge and skills, 
and gave me many valuable suggestions and 
kind help in the elucidation of difficulties.
\end{acknowledgements}

\tableofcontents
\chapter*{Introduction}
\addcontentsline{toc}{chapter}{Introduction} 
\chaptermark{Introduction}

Consider, on a complex projective space $\PP^n$, a
fixed component free 
 sublinear system 
$\sigma\subset|\O_{\PP^n}(d_0)|$, 
of dimension $N\geq n$, such that 
the associated rational map $\varphi=\varphi_{\sigma}:\PP^n\dashrightarrow\PP^N$ is birational onto its image
and moreover such that the image is not too much singular. 
Understanding all such linear systems 
(or the corresponding birational transformations) is 
clearly a too ambitious goal. Already for $N=n$ the Cremona group of all these 
transformations is a very complicated object. 
From now on, assume $d_0=2$ and 
denote by $d$ the degree of the linear system giving the inverse to $\varphi$.
We then say that the transformation $\varphi$ is of type $(2,d)$.
Moreover, assume that $\varphi$ is \emph{special},
i.e. its base locus scheme (also called center) 
is smooth and connected.
The first interesting case is when $N=n$,
i.e. that of special quadratic Cremona transformations. 
\section*{Special Cremona transformations}
The first general results were obtained by B. Crauder and S. Katz  
 in \cite{crauder-katz-1989} 
(see also \cite{crauder-katz-1991} and \cite{katz-cubo-cubic}), by classifying
all special Cremona transformations 
whose base locus has dimension at most two. 
In particular, they obtained that the base locus of a
quadratic transformation of this kind (which is nondegenerate in $\PP^n$)
is one of the following: 
\begin{itemize}\renewcommand{\labelitemi}{\checkmark}
\item a quintic elliptic curve in $\PP^4$;
\item the Veronese surface $\nu_2(\PP^2)$ in $\PP^5$;
\item a septic elliptic scroll in lines embedded in $\PP^6$;
\item the plane blown-up at eight points and embedded in $\PP^6$ 
       as an octic surface.
\end{itemize}
The second general result was obtained
by L. Ein and N. Shepherd-Barron in \cite{ein-shepherdbarron}:
the base locus of a special Cremona transformation of type $(2,2)$ 
is a Severi variety. Moreover F. L. Zak (see \cite{lazarsfeld-vandeven}) 
had shown that there are just four Severi varieties:
\begin{itemize} 
\item the Veronese surface $\nu_2(\PP^2)$ in $\PP^5$; 
\item the Segre embedding of $\PP^2\times\PP^2$ in $\PP^8$; 
\item the Pl{\"{u}}cker embedding of  $\GG(1,5)$ in $\PP^{14}$;
\item the $16$-dimensional Cartan variety $E_6$ in $\PP^{26}$. 
\end{itemize}
The next step was taken by F. Russo in \cite{russo-qel1}, 
who observed that base loci of 
special Cremona transformations of type $(2,d)$ are of a very peculiar type, the 
so called \emph{quadratic entry locus} varieties. 
He classified some cases (for example with $n$ odd, or of type $(2,3)$) and 
suggested that such base loci are subject to 
very strong restrictions.

On special quadratic Cremona transformations whose base locus has dimension three, 
K. Hulek, S. Katz and F. O. Schreyer in \cite{hulek-katz-schreyer} 
provided an example, the only example so far known.
Later, M. Mella and F. Russo in the unpublished paper \cite{mella-russo-baselocusleq3}, 
collected a series of ideas and remarks
on the study of these transformations. 
One of these ideas was to apply the Castelnuovo theory to the general
 zero-dimensional linear section of the base locus.
This idea is central to the present thesis.

\section*{The next case}
The first main goal  
of the thesis is to deal with the ``next case'', $N=n+1$, 
under the assumption that the image of the special transformation of type $(2,d)$ is a 
sufficiently regular hypersurface.
 Under such hypotheses, 
we are still able to show that the base locus is a quadratic entry locus variety; 
moreover, we compute the dimension and the secant defect of the base locus
in terms of the other numerical invariants: $n$, $d$ and the degree of 
the hypersurface image.

The first and easiest  example of such a  transformation is 
the inverse of a stereographic projection of a quadric.
This example can be characterized in various ways. 
For example,  
it is the only case in which 
the base locus is degenerate.

In this direction, the main results of the thesis are three:
\begin{itemize}\renewcommand{\labelitemi}{\maltese}
 \item Complete classification when the type is $(2,2)$ and the image is a quadric: 
the base locus is a hyperplane section of a Severi variety.
 \item  Complete classification when the type is $(2,3)$ and the image is a cubic: 
the base locus is  a 
three-dimensional quadric blown-up at five points 
and embedded in $\PP^8$.
 \item (Almost) complete classification when the dimension of the base locus 
       is at most three.
\end{itemize}
On the first item, the existence of examples is clear and 
the difficulty is to prove that they are the only ones.
This is done in several steps, 
in order to prove that the inverse transformation is still special.
In the third item, we wrote ``(Almost) complete classification''
because in one case we do not know if it really exists.
However, we are able to say that
the case exists if one proves that 
a linearly normal scroll in lines over the Hirzebruch surface $\FF_1$, embedded in $\PP^8$ 
as a variety of degree $11$ and sectional genus $5$ (whose 
existence has been established by A. Alzati and G. M. Besana in \cite{alzati-fania-ruled}) 
is also cut out by quadrics.
On the second item, by using some results of P. Ionescu and F. Russo in 
\cite{russo-qel1} and \cite{ionescu-russo-conicconnected},
we deduce that  the base locus is three-dimensional and so we apply the 
previously obtained classification.
Next, we also compute the possible numerical invariants 
for transformations of type $(2,2)$ into 
a cubic and a quartic hypersurface. 
\section*{Towards the general case}
The second main goal of the thesis is to deal with a  more general case, that is 
when there are no restrictions on $N$, 
but however the image $\overline{\varphi(\PP^n)}\subseteq\PP^N$
of the transformation  is a sufficiently regular variety.
We study these transformations either
keeping the dimension of the base locus small 
or   fixing $d$ and
another numerical invariant,  the \emph{coindex} of the image.

Recall that in \cite{semple}, J. G. Semple 
 constructed transformations 
$\PP^{2m-2}\dashrightarrow\PP^{\binom{m+1}{2}-1}$ of type $(2,2)$ 
(resp. type $(2,1)$) having 
as image the Grassmannian $\GG(1,m)\subset \PP^{\binom{m+1}{2}-1}$ 
and having as base locus a nondegenerate (resp. degenerate) rational 
normal scroll. 
Further, F. Russo and A. Simis in  \cite{russo-simis},  
have characterized these examples 
as the only special birational transformations 
of type $(2,2)$ (resp. type $(2,1)$) into 
the Grassmannian of lines in projective space.
 
Note that in the Semple's examples, the image of the transformation is smooth. 
However, the smoothness of the image 
is a very rare phenomenon and
therefore, in order not to exclude relevant cases, 
we only require that the image is ``sufficiently regular''.
On the other hand, this assumption on the image is reasonable to restrict 
the classification in a confined meaningful list.

By applying techniques and results 
obtained for the case $N=n+1$, 
we extend some of our results.
More precisely we obtain:
\begin{itemize}\renewcommand{\labelitemi}{\maltese}
 \item Classification when the type is $(2,3)$ and the image is a ``del Pezzo variety'': the image  
is a cubic hypersurface or the base locus is either
 a scroll in lines over a quadric surface
or a quadric surface fibration over a line.   
 \item Classification of all transformations when the dimension of the base locus 
       is at most three: there are (at most) $33$ types of such transformations.
\end{itemize}
Here, we determine all possible cases mainly by
 applying the M. Mella and F. Russo's idea aforementioned 
and the classification of smooth varieties of low degree, 
but the main difficulty is to exhibit examples.
Even in the simplest cases, 
the number of calculations is so huge that 
 the use of a computer algebra system
is indispensable.
In some cases we are able to say that there is a transformation
just as we wish, 
except for the fact that we do not know
 if the image satisfies all our assumptions. 

We point out that, as a consequence,  
one can  
deduce that 
there are (at most) two types of special quadratic Cremona transformations 
having three-dimensional base locus: we have $n=N=8$ and the base locus is either
\begin{itemize}\renewcommand{\labelitemi}{\checkmark}
 \item the projection from a point of a Fano variety (here 
we have the K. Hulek, S. Katz and F. O. Schreyer's example) or
 \item a scroll in lines over a surface (here we do not know examples).
\end{itemize}
From F. Russo's results in \cite{russo-qel1} it follows that 
these transformations are the only special Cremona transformations 
of type $(2,5)$. 

In an appendix,  we show how the techniques used  may be adapted to the case 
 in which the dimension of the base locus is greater than three,
focusing to the case of dimension four. We also classify these transformations, 
but with $N$ large.
It does appear clear that the complexity of the objects increases 
with the decreasing of $N-n$.  
\section*{Two-sides principle}
From the thesis it follows a two-sides principle:
on one hand, base loci of such special birational transformations are very particular 
(they are manifolds satisfying some strong numerical and geometric conditions)
and further, 
at least where we were able to construct them, 
their properties may be described quite precisely. Conversely, 
all these special manifolds seem to 
appear as base loci of conveniently chosen transformations.

\section*{Organization of the thesis} 
In the first chapter we recall basic facts on geometric objects as 
the secant variety and the tangential projections of a projective variety.

The second chapter is dedicated to presenting 
some well-known classes of projective varieties.
In particular, we treat the class of quadratic entry locus varieties, 
largely studied by P. Ionescu and F. Russo.  

In the third chapter, we outline some well-known extensions 
of Castelnuovo's results. These are applied to study 
 transformations
having base locus of small dimension.

The fourth chapter is
an introduction to the subject of the thesis. We reinterpret 
the well-known example of the stereographic projection and 
treat the case of transformations of $\PP^n$ into a quadric, with 
$n\leq4$ and having reduced base loci.

 The fifth and sixth chapter are the main part of the thesis. 
There we prove the main results 
 on special quadratic birational transformations aforementioned.
Much of these two chapters is also contained in the papers \cite{note} and \cite{note2}.

\mainmatter
\chapter{Basic tools}\label{sec: basic tools}
For much of the material in this chapter, 
the main references are \cite{zak-tangent} and \cite{russo-specialvarieties}.
Unless otherwise specified, 
we shall keep the following: 
\begin{notation}\label{notation: X irred nondegen}
$X\subset\PP^n$ is an irreducible nondegenerate (i.e. not contained in a hyperplane)
$r$-dimensional projective complex algebraic variety.
\end{notation}
\section{Secant and tangent loci}
\subsection{Secant variety}
Let $X\subset\PP^n$ be as in Assumption \ref{notation: X irred nondegen} and  let
$Y\subset\PP^n$ be an irreducible projective variety.
We put 
 $$
 \Delta_{Y}=\{(y,x)\in Y\times X : y=x\},
 $$
 $$
 S^0_{Y,X}=\{(y,x,z)\in (Y\times X\setminus \Delta_Y)\times\PP^n: z\in\langle y,x\rangle\},
 $$ 
and denote by $S_{Y,X}$ the closure of $S^0_{Y,X}$ in $Y\times X\times\PP^n$. 
So we obtain the diagram:
\begin{equation}\label{eq: diagram secant variety}
\xymatrix{& &          S_{Y,X} \ar[dl]_{\pi_{Y,X}} \ar[dr]^{\pi_{\PP^n}}\\ 
          & Y\times X \ar[dl]_{p_{Y}} \ar[dr]^{p_{X}} &        & \PP^n \\
           Y  && X}
\end{equation}
\begin{definition}
The \emph{join} of $Y$ and $X$, $S(Y,X)$, is defined as the 
scheme-theoretic image of $\pi_{\PP^n}$, i.e.
$$
S(Y,X)=\pi_{\PP^n}(S_{Y,X})=\overline{\bigcup_{(y,x)\in Y\times X\setminus \Delta_Y}\langle y, x \rangle}.
$$
\end{definition}
Note that $S_{Y,X}$ is an irreducible variety of dimension 
 $\dim(Y)+\dim(X)+1$, hence $S(Y,X)$ is also irreducible of dimension 
\begin{equation}
 \dim(S(Y,X))\leq \dim(Y)+\dim(X)+1.
\end{equation}
An essential tool for studying joins of varieties is the 
so-called Terracini Lemma 
(see also \cite[Theorem~1.3.1]{russo-specialvarieties}):
\begin{theorem}[Terracini Lemma]\label{prop: terracini lemma}
Let $Y,X\subset\PP^n$ be as above. 
There exists a  nonempty open  subset  $U$ of $S(Y,X)$ such that 
for any  $p\in U$, $(y,x)\in Y\times X\setminus \Delta_Y$, 
$p\in\langle y,x\rangle$, we have
$$
T_p(S(Y,X))=\langle T_y(Y), T_x(X) \rangle.
$$
\end{theorem}
\begin{definition}
In the case when $Y=X$, we put $\Sec(X):=S(X,X)$ 
and call it  \emph{secant variety} of $X$.
A line  $l$  is said to be a
 \emph{secant} of $X$ if the length of the scheme $X\cap l$
is at least $2$; in these terms, 
$\Sec(X)$ is the union of all secant lines of $X$.
The \emph{secant defect} of $X$ is defined as the nonnegative integer
\begin{equation}
\delta(X)=2\dim(X)+1-\dim(\Sec(X))
\end{equation}
and $X$ is called \emph{secant defective} if $\delta(X)>0$.
\end{definition}
\begin{remark}
The definition of secant variety can be given in a more general context. Indeed
for any locally closed subscheme $B\subset\PP^n$ we can consider 
the Hilbert scheme 
 $\Hilb_2(B)$ of its $0$-dimensional subschemes of length $2$.
For each $Z\in \Hilb_2(B)$ the evaluation map  
$H^0(B,\mathcal{O}_B(1))\rightarrow H^0(Z, \mathcal{O}_Z(1))$
is obviously surjective, therefore its kernel 
 defines a secant line $\langle Z\rangle\subset\PP^n$.
This defines a $\PP^1$-bundle over $\Hilb_2(B)$
$$
\mathcal{S}(B):=\{(Z,p)\in \Hilb_2(B)\times \PP^n : p\in \langle Z\rangle\}\longrightarrow \Hilb_2(B)
$$
and the \emph{secant scheme} of $B$, $\Sec(B)$, is defined as 
the scheme-theoretic image of the projection 
 $\mathcal{S}(B)\rightarrow\PP^n$.
\end{remark}
\subsection{Contact loci} 
Consider the diagram  (\ref{eq: diagram secant variety}) with  $Y=X$, defining 
the secant variety   $\Sec(X)$,
and let $p\in\Sec(X)\setminus X$ be a point.   Put 
\begin{eqnarray*}
L_p(X)&:=&\pi_{\PP^n}(\pi_{X,X}^{-1}(\pi_{X,X}(\pi_{\PP^n}^{-1}(p)))) \\
      & =&\bigcup\{l: l \mbox{ secant line of } X \mbox{ through }p\},\\
\end{eqnarray*}
\begin{eqnarray*}
\Sigma_p(X)&:=& \pi_X(\pi_{\PP^n}^{-1}(p))=L_p(X)\cap X \\
          &=& \overline{\{x\in X: \exists x'\in X \mbox{ with } x\neq x' \mbox{ and } p\in \langle x,x' \rangle \}},
\end{eqnarray*}
where $\pi_X=p_X\circ\pi_{X,X}$.
\begin{definition}
$\Sigma_p(X)$ is called the \emph{entry locus} 
of $X$ with respect to  $p$.
\end{definition}
\begin{figure}[htbp]
\begin{center}
\includegraphics[width=0.55\textwidth]{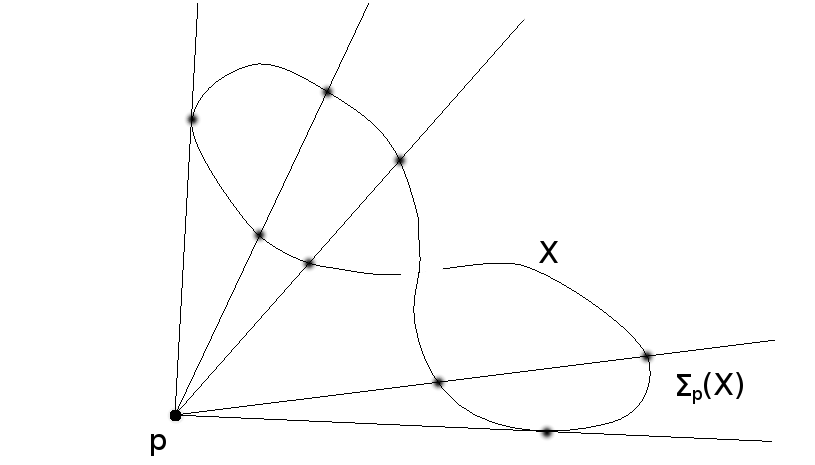}
\end{center}
\caption{Entry locus of $X$.}
\label{fig: entry locus}   
\end{figure}
Note that  $L_p(X)$ is a cone over
 $\Sigma_p(X)$ of vertex $p$, hence  
 $\dim(L_p(X))=\dim(\Sigma_p(X))+1$. 
If $p\in\Sec(X)\setminus X$ is a general point, then 
$\Sigma_p(X)$ is reduced and equidimensional of dimension  $\delta(X)$.
\begin{definition}\label{def tangential contact locus}
The \emph{tangential contact locus} of  $X$
with respect to the general point $p$ is defined as 
$$
\Gamma_{p}(X):=\overline{\{x\in \reg(X): T_x(X)\subseteq T_p(\Sec(X)) \} }.
$$
\end{definition}
By Terracini Lemma it follows that 
$
\Sigma_{p}(X)\subseteq \Gamma_p(X)
$
and it is possible to prove that for  $x,x'\in X$ 
general points and  $p\in\langle x,x' \rangle$ general point, the 
irreducible components of  
 $\Gamma_p(X)$ through $x$ and $x'$ are uniquely determined and 
have the same dimension, independent of $p$. 
Denoting this dimension by  $\gamma=\gamma(X)$, we have
$\delta(X)\leq\gamma(X)$. See also \cite[Definition~2.3.3]{russo-specialvarieties} 
and \cite[Definition~3.4]{chiantini-ciliberto}.
\subsection{\texorpdfstring{$J$}{J}-tangency}
Let $Y\subseteq X$ be a subvariety ($X$ as in Assumption \ref{notation: X irred nondegen}) and consider 
the diagram  
(\ref{eq: diagram secant variety}).
\begin{definition}
For a point $y\in Y$, the  \emph{(relative) tangent star} to $X$ 
with respect to $Y$ at $y$,   is defined as
$$
T^{\ast}_y(Y,X):=\pi_{\PP^n}(\pi_{Y,X}^{-1}((y,y))),
$$
and put
$$
T^{\ast}(Y,X):=\bigcup_{y\in Y} T^{\ast}_y(Y,X).
$$
\end{definition}
We have $T^{\ast}_y(Y,X)\subseteq T_y(X)$ and, if $y\in Y\setminus \sing(X)$, 
then $T^{\ast}_y(Y,X)= T_y(X)$;
in particular, if $Y\subseteq X\setminus\sing(X)$, then
$$
T^{\ast}(Y,X)=T(Y,X):=\bigcup_{y\in Y}T_y(X).
$$
\begin{theorem}[\cite{zak-tangent}]\label{prop: zak theorem}
  It holds one of the following two conditions:
\begin{enumerate}
 \item $\dim(T^{\ast}(Y,X))=\dim(S(Y,X))-1=\dim(Y)+\dim(X)$;
 \item $T^{\ast}(Y,X)=S(Y,X)$.
\end{enumerate}
\end{theorem}
\begin{definition}\label{def: J-tangency} 
 If $L\subset\PP^n$ is a linear space, we say that 
$L$ is tangent to $X$ along $Y$ if $L\supseteq T(Y,X)$;
we say that $L$ is $J$-tangent to $X$ along $Y$ if $L\supseteq T^{\ast}(Y,X)$.
\end{definition}
\begin{corollary}[Zak's Theorem on tangencies, \cite{zak-tangent}]\label{prop: theorem on tangencies}
 Notation as above. If $L\subset\PP^n$ is a linear space
 $J$-tangent to $X$ along $Y$, then
$$
\dim(Y)+\dim(X)\leq \dim(L).
$$
\end{corollary}
\begin{proof} See also \cite[\Rmnum{1} Corollary~1.8]{zak-tangent} 
and \cite[Theorem~2.2.3]{russo-specialvarieties}.
 Without loss of generality we can suppose that $Y$ is irreducible.  
By hypothesis $T^{\ast}(Y,X)\subseteq L$, but 
$S(Y,X)\nsubseteq L$ (because $X\subseteq S(Y,X)$ and $X$ is nondegenerate). 
Thus, by Theorem \ref{prop: zak theorem}, we have
$\dim(Y)+\dim(X)=\dim(T^{\ast}(Y,X))\leq \dim(L)$.
\end{proof}
\subsection{Gauss map}
Recall that the Gauss map of $X$ is the rational map  
$\G_{X} :X\dashrightarrow \GG(r,n)$,
which sends the point  $x\in\reg(X)$  
to the tangent space  $T_x(X)\in \GG(r,n)$.
For a general point  $x\in\reg(X)$, 
the closure of the fiber 
of $\G_X$ at $\G_X(x)$ is
$$
\overline{\G_X^{-1}(\G_X(x))}=\overline{\{y\in \reg(X): T_y(X)=T_x(X)\}}.
$$
\begin{theorem}\label{prop: linear fiber gauss}
 For $x\in X$ general point, 
 $\overline{\G_X^{-1}(\G_X(x))}$ is a 
 linear space. 
\end{theorem}
\begin{corollary}\label{prop: birational gauss}
 If $X$ is smooth, then  $\G_X:X \rightarrow \G_X(X)$ is a birational
morphism.
\end{corollary}
For details and proof of Theorem \ref{prop: linear fiber gauss},
we refer the reader to \cite[\Rmnum{1} \S  2]{zak-tangent}.
Corollary \ref{prop: birational gauss} 
is instead a straightforward application of
 Corollary \ref{prop: theorem on tangencies} 
and Theorem \ref{prop: linear fiber gauss}.
\section{Tangential projections}
For a general point 
 $x\in X$  consider the projection of  
 $X$ from the tangent space  $T_x(X)\simeq \PP^r$ onto a linear space   
$\PP^{n-r-1}$ skew to $T_x(X)$,
$$
\tau_{x,X}:X\subset\PP^n\dashrightarrow \overline{\tau_{x,X}(X)}=W_{x,X}\subseteq\PP^{n-r-1}.
$$
From Terracini Lemma it follows (see also \cite[Proposition~1.3.8]{russo-specialvarieties}):
\begin{proposition}\label{prop: nondegenerate tang projection}
$W_{x,X}\subseteq\PP^{n-r-1}$ is an irreducible nondegenerate variety 
 of dimension $r-\delta(X)$. 
\end{proposition}
In \S  \ref{sec: second fund form} 
we introduce
a useful concept for studying
tangential projections of 
secant defective varieties having secant variety that doesn't fill up 
the whole ambient space.
\subsection{Second fundamental form}\label{sec: second fund form} 
Consider the blow-up
 of $X$ at the general point  $x$,
$
\pi_x:\Bl_x(X)\rightarrow X
$,
and denote by 
 $E=E_X=E_{x,X}=\PP^{r-1}$ its exceptional divisor and by 
 $H$ a divisor of the linear system $|\pi_{x}^{\ast}(\O_X(1))|$.
Since $X$  is not linear, it is defined the rational map
$$
\phi: E{\dashrightarrow} W_{x,X}\subseteq\PP^{n-r-1}.
$$
\begin{definition} The linear system associated to 
 $\phi$, i.e. $|H-2E|_{|E}\subseteq |-2E|_{|E}=|\O_{\PP^{r-1}}(2)|$, is called 
\emph{second fundamental form}  of $X$ (at the point $x$) 
and it is denoted with $|II_{x,X}|$.
\end{definition}
Of course, $\dim(|II_{x,X}|)\leq n-r-1$ and 
the image of $\phi=\phi_{|II_{x,X}|}$ is contained in $W_{x,X}$.
\begin{proposition}\label{prop: surjective second fundamental form}
If $X$ is smooth,
secant defective  
and $\Sec(X)\subsetneq\PP^n$, then 
$\overline{\phi_{|II_{x,X}|}(E)}=W_{x,X}$ and
 in particular $\dim(|II_{x,X}|)=n-r-1$.
\end{proposition}
See \cite[Proposition~2.3.2]{russo-specialvarieties} for a proof 
of Proposition \ref{prop: surjective second fundamental form}. 
In the following we shall denote by 
 $B_{x,X}=\mathrm{Bs}(|II_{x,X}|)$ the base locus of the second fundamental form.
\subsection{Hilbert scheme of lines passing through a point}
Let $X\subseteq\PP^n$ be a closed subscheme, 
$x\in X$.
\begin{definition}\label{def: hilbert scheme} 
The \emph{Hilbert scheme of lines contained in $X$ and 
passing through $x$} is a scheme (over $\CC$ as always), $\L_{x,X}$, 
together with a closed subscheme $\chi_{x,X}\hookrightarrow X\times \L_{x,X}$ 
(called \emph{universal family}) such that the fibers of the projection
$\pi_{x,X}:\chi_{x,X}\rightarrow\L_{x,X} $ 
(which are naturally immersed in $X\subseteq\PP^n$) 
are lines passing through $x$. 
Moreover $(\L_{x,X},\chi_{x,X})$ 
must satisfy the following universal property:
for each scheme $\mathcal{Y}$ and for each closed subscheme 
$Z\hookrightarrow X\times \mathcal{Y}$, having the property that 
the fibers of the projection 
 $Z\rightarrow \mathcal{Y}$ 
are lines passing through $x$, there exists a unique morphism  
 $f=f_{(\mathcal{Y},Z)}:\mathcal{Y}\rightarrow \L_{x,X}$ 
such that we have  the diagram of fibred products:
\begin{equation}\label{diagram hilbert scheme of lines}
\xymatrix{
Z \ar@{^{(}->}[r] \ar[d] & X\times \mathcal{Y} \ar[d] \ar[r] &\mathcal{Y} \ar[d]^{f} \\
\chi_{x,X} \ar@{^{(}->}[r] & X\times \L_{x,X} \ar[r] & \L_{x,X}
}
\end{equation}
\end{definition}
\begin{figure}[htbp]
\begin{center}
\includegraphics[width=0.55\textwidth]{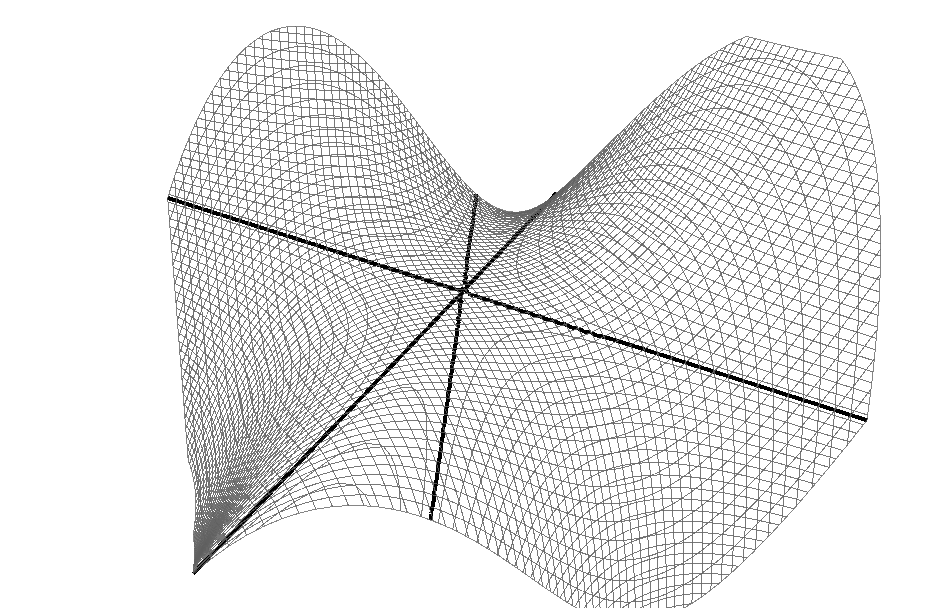} 
\end{center}
\caption{Hilbert scheme of lines through a point.}
\label{fig: LxX}   
\end{figure}
As a consequence of the general result in \cite[\Rmnum{1} Theorem~1.4]{Kollar}
we obtain that   $\L_{x,X}$ exists and it is a projective scheme.
Moreover, if $x\in\reg(X)$, there exists a natural 
closed embedding of $\L_{x,X}$ into the exceptional divisor 
 $E_{x,X}$ (see \cite{russo-linesonvarieties}).

If $l$ is a line contained in $X$ and 
passing through $x$, we denote by 
 $[l]$ the corresponding element of  $\L_{x,X}$ (i.e. $l=\pi_{x,X}^{-1}([l])$).
Applying \cite[Proposition~4.14]{debarre} and 
\cite[\Rmnum{1} Definition~2.6, Theorem~2.8]{Kollar} 
and arguing as in \cite[Proposition~1.1(1)]{russo-linesonvarieties}, 
one can show the following:
\begin{proposition}\label{prop: sing LxX}
Let $X\subseteq\PP^n$ be a generically reduced closed subscheme,
$x\in X$ a general point.  
Then 
$$
\sing(\L_{x,X})\subseteq S_{x,X}:=\{[l]\in \L_{x,X}: l\cap\sing(X)\neq\emptyset \}.
$$
\end{proposition}
In the following, let $X$ be as in Assumption \ref{notation: X irred nondegen}.
Note that, in general $\L_{x,X}$ is neither irreducible nor reduced.
However,
 from Proposition \ref{prop: sing LxX}, it follows that if 
$X$ is smooth and  $x\in X$ is a general point, then $\L_{x,X}$ is smooth; 
moreover, for $[l]\in\L_{x,X}$, we have 
\begin{equation}\label{eq: dim LxX} 
 \dim_{[l]}(\L_{x,X})=H^0(\PP^1,\N_{l,X}(-1))=H^0(\PP^1,{\T_X}|_l(-1))-2= -K_X\cdot l -2.
\end{equation}
The relation between $\L_{x,X}$ and $B_{x,X}$ is given by the following:
\begin{proposition}\label{prop: relation between hilbert scheme and second funfamental form}
Let $X\subset\PP^n$ be as in Assumption \ref{notation: X irred nondegen} and let $x\in\reg(X)$. 
Then there exists 
a natural closed embedding of $\L_{x,X}$ into $B_{x,X}$, which is an isomorphism if 
$X$ is defined by  quadratic forms. 
\end{proposition}
For the proof of Proposition \ref{prop: relation between hilbert scheme and second funfamental form}, 
we refer the reader to  \cite[Corollary~1.6]{russo-linesonvarieties}.

\chapter{Some special varieties}\label{sec: some special varieties}
In this chapter 
 we introduce some well-known classes of varieties 
which play an important role in the rest of the thesis.
Unless otherwise specified,
we shall keep Assumption \ref{notation: X irred nondegen}.
\section{Fano varieties}
\begin{definition}
The variety  $X$ is called a \emph{Fano variety} if 
it is smooth and its anticanonical bundle
$-K_X$ is ample; in such case 
the \emph{index} of $X$ is
$$
i(X):=\mathrm{sup}\{m\in\ZZ: -K_X\sim mL \mbox{ for some ample divisor }L \},
$$ 
while the \emph{coindex} of $X$ is
$$
c(X):=\dim(X)+1-i(X).
$$
\end{definition}
\begin{remark}
If $X$ is a Fano variety, since $X$ is smooth, 
the group $\Pic(X)$ is torsion free and therefore there exists
a  unique (ample) divisor $L$ such that $-K_X\sim i(X)L$.
\end{remark}
We are interested in the case in which $-K_X\sim i(X)H_X$, 
where $H_X$ is a hyperplane section of $X\subset\PP^n$.
We also say that a Fano variety $X\subset\PP^n$ is 
 \emph{of the first species} if $\Pic(X)\simeq \ZZ\langle \O_X(1) \rangle$.
\begin{remark}\label{remark: second betti number fano}
If $X\subset\PP^n$ is a Fano variety, from 
Kodaira Vanishing Theorem and Serre Duality
(see e.g. \cite[\Rmnum{3} Corollary~7.7, Remark~7.15]{hartshorne-ag}, 
\cite[Theorem~0.4.11, Corollary~0.4.13]{fujita-polarizedvarieties}) it follows 
$H^i(X,\O_X)=0$ for $i\geq 1$. Hence
from the cohomology sequence 
$$
\cdots \rightarrow H^1(X,\mathcal{O}_{X})\rightarrow \underbrace{H^1(X,\mathcal{O}_{X}^{\ast})}_{\simeq\Pic(X)}\rightarrow H^2(X,\ZZ)\rightarrow H^2(X,\mathcal{O}_{X})\rightarrow\cdots
$$
of the exponential sheaf sequence
(see \cite[p.~37]{griffiths-harris} or \cite[p.~446]{hartshorne-ag})
 we see that the rank of the Picard group
  $\Pic(X)$ equals the second Betti number $b_2(X)$.
Moreover note that, from Barth-Larsen Theorem \cite{larsen-coomology},
 each Fano variety is of the first species if $2r-n\geq 2$. 
\end{remark}
\begin{example} 
Simple examples of Fano varieties are: projective spaces;
 smooth complete intersections in $\PP^n$ defined by equations of degrees  
 $d_1,\ldots,d_s$ with $d_1+\cdots + d_s\leq n$; 
finite products of Fano varieties.
\end{example}
\begin{theorem}[Kobayashi-Ochiai \cite{Kobayashi-Ochiai}]
If $X$ is a Fano variety, then $c(X)\geq0$ and
\begin{enumerate}
 \item if $c(X)=0$, then $X$ is a projective space;
 \item if $c(X)=1$, then $X$ is a quadric hypersurface.
\end{enumerate}
\end{theorem}
Below we treat Fano varieties of coindex $2$ and $3$.
\subsection{Del Pezzo varieties}
Let $X\subset\PP^n$ be as in Assumption \ref{notation: X irred nondegen} 
and further suppose that it is smooth and linearly normal.
Denote by $r$, $\lambda$, $g$ and $\Delta_X$, 
respectively, the dimension, the degree, the sectional genus and the 
$\Delta$-genus of $X$ (the latter defined as $\Delta_X=r+\lambda-n-1$).
\begin{proposition}[\cite{fujita-polarizedvarieties}]\label{prop: definition del pezzo varieties}
The following are equivalent:
\begin{itemize}
 \item $\Delta_X=g=1$;
 \item $K_X\in|\O_X(1-r)|$.
\end{itemize}
\end{proposition}
If one of the equivalent conditions of Proposition 
\ref{prop: definition del pezzo varieties} holds, then $X$ is called a \emph{del Pezzo variety}.
A del Pezzo surface, different from  $\PP^1\times\PP^1$, 
can be obtained from $\PP^2$ by blowing up several points successively;
these points are not infinitely near and no three of them are collinear
and no six of them lie on a conic (see \cite[\Rmnum{1} \S  8]{fujita-polarizedvarieties}).
Del Pezzo varieties of higher dimension are classified by the following:
\begin{theorem}[\cite{fujita-polarizedvarieties}]\label{prop: classification del pezzo varieties} Let $X\subset\PP^n$ be 
a del Pezzo variety  as above,
with $r\geq 3$. Then $X$ is
projectively equivalent to one of the following:
\begin{description}
 \item[($\lambda=3$)] a cubic hypersurface;
 \item[($\lambda=4$)] a complete intersection of two quadric hypersurfaces;
 \item[($\lambda=5$)] a linear section of the Grassmannian $\GG(1,4)\subset\PP^9$;
 \item[($\lambda=6$)] $\PP^1\times\PP^1\times\PP^1$ or $\PP^2\times\PP^2$ 
or $\PP(\T_{\PP^2})$ for the tangent bundle $\T_{\PP^2}$ of $\PP^2$; 
 \item[($\lambda=7$)] the blowing-up of $\PP^3$ at a point;
 \item[($\lambda=8$)] the Veronese immersion $\nu_2(\PP^3)\subset\PP^9$.
\end{description}
\end{theorem}
\subsection{Mukai varieties}
Let $X\subset\PP^n$ be as in Assumption \ref{notation: X irred nondegen} 
and further suppose that it is
 smooth and linearly normal.
As above, denote by $r$, $\lambda$ and $g$, respectively, dimension, 
degree and sectional genus of $X$. 
$X$ is called a \emph{Mukai variety} 
if one of the following equivalent conditions holds:
\begin{itemize}
 \item $X$ has a smooth curve section $C$ 
embedded by the canonical linear system $|K_C|$;
 \item $K_X\in|\O_X(2-r)|$.
 \end{itemize}
Note that for a Mukai variety $X$, by sectional genus formula we get the relation:  
\begin{equation}
 \lambda=2\,\left(g-1 \right).
\end{equation}
\begin{theorem}[\cite{mukai-biregularclassification}]\label{prop: classification first species mukai varieties}
The following varieties 
(denoted by $X$ and embedded in $\langle X\rangle\simeq\PP^n$) 
are examples of
 Mukai varieties of the first species (where $g,r,n$ are indicated by side):
\begin{description}
 \item[($g=6$, $r=6$, $n=10$)] A smooth intersection $X$ of 
a quadric hypersurface in $\PP^{10}$ with the 
cone over $\GG(1,4)\subset\PP^9$.
 \item[($g=7$, $r=10$, $n=15$)] The spinorial variety $X=S^{10}\subset\PP^{15}$.
\item[($g=8$, $r=8$, $n=14$)] The Grassmannian $X=\GG(1,5)\subset\PP^{14}$.
\item[($g=9$, $r=6$, $n=13$)] Let $U$ be a $6$-dimensional 
vector space (over $\CC$) and $F$ be a nondegenerate skew-symmetric bilinear form 
 on $U$. All $3$-dimensional subspaces $W$ of $U$
with $F(W,W)=0$ form a smooth variety $X\subset\GG(2,5)\subset\PP^{19}$.
\item[($g=10$, $r=5$, $n=13$)] Let $U$ be a $7$-dimensional vector space and 
$F$ be a nondegenerate skew-symmetric $4$-linear form on $U$. 
All $5$-dimensional subspaces 
$W$ of $U$ with $F(W,W,W,W)=0$ form  
a smooth variety $X\subset\GG(4,6)\subset\PP^{20}$.
\item[($g=12$, $r=3$, $n=13$)] Let $U$ be a $7$-dimensional vector space and 
$F_1, F_2, F_3$ be three linearly independent skew-symmetric bilinear forms on $U$.
All $3$-dimensional subspaces $W$ of $U$ with 
$F_1(W,W)=F_2(W,W)=F_3(W,W)=0$ form
a subvariety $X$ of 
$\GG(2,6)\simeq G(3,U)\subset\PP^{\ast}({\bigwedge}^3 U)\simeq\PP^{34}$;
if the subspace 
$F_1\wedge U^{\vee}+F_2\wedge U^{\vee}+F_3\wedge U^{\vee}$ of $\bigwedge^3 U^{\vee}$ 
contains no non-zero vectors of the form 
$f_1\wedge f_2\wedge f_3$ with $f_1,f_2,f_3\in U^{\vee}$,
then 
$X\subset\PP^{34}$ is a $3$-dimensional smooth variety.
\end{description}
Moreover, each Mukai variety of the first species 
 of dimension $\geq 3$ and sectional genus  $\geq 6$,
 is projectively equivalent to a linear section 
of a variety $X$ constructed as above.
\end{theorem}
\begin{remark}
A Mukai surface is of type $K3$. 
A Mukai variety of sectional genus $<6$ is a complete intersection 
of hypersurfaces.
\end{remark}
We refer the reader to \cite{mukai-biregularclassification} and \cite{mori-mukai} 
for the classification of Mukai varieties with $b_2\geq2$.
Finally we observe that in  \cite{mukai-biregularclassification} the classification is obtained
under the assumption of the existence of a smooth divisor in $|\O_{X}(1)|$,
a condition which is clearly satisfied in our case; in any event, this restriction 
has been removed in \cite{mella-mukai}.
\section{\texorpdfstring{$QEL$}{QEL}/\texorpdfstring{$LQEL$}{LQEL}/\texorpdfstring{$CC$}{CC}-varieties}
In the following definition, we consider varieties having the simplest entry locus.
\begin{definition}[\cite{russo-qel1,ionescu-russo-conicconnected,zak-tangent}]\label{def: LQEL QEL CC} 
Let $X\subset\PP^n$ be as in Assumption \ref{notation: X irred nondegen}.
\begin{enumerate}
\item $X$ is called a \emph{quadratic entry locus variety of type} $\delta=\delta(X)$,
briefly a 
$QEL$-variety, if the general entry locus  
 $\Sigma_p(X)$  is a quadric hypersurface (of dimension $\delta$).
\item $X$ is called a \emph{locally quadratic entry locus variety of type} $\delta$,
briefly an  $LQEL$-variety, if each irreducible component of the 
general entry locus $\Sigma_p(X)$ is a quadric hypersurface; 
equivalently, $X$ is an $LQEL$-variety of type $\delta$ if through two general points 
there passes a quadric hypersurface of dimension $\delta$ contained in $X$.
\item $X$ is called a \emph{conic-connected variety}, briefly a $CC$-variety, 
if through two general points of $X$ there passes an irreducible conic contained in $X$.
\end{enumerate}
\end{definition}
Of course, we have
$$
\mbox{$X$ is a $QEL$-variety}\ \Longrightarrow \ \mbox{$X$ is an $LQEL$-variety} \ \stackrel{\mathrm{if}\ \delta>0}{\Longrightarrow} \mbox{$X$ is a $CC$-variety} 
$$
but the inverse implications are not true in general.
\begin{proposition}\label{prop: first properties LQEL}
 Let $X$ be smooth. Then
\begin{enumerate}
 \item if $X'\subset\PP^{m}$ (with $m\leq n$) is an isomorphic
projection of  $X\subset\PP^n$, then
$\delta(X')=\delta(X)$ and $X$ is an $LQEL$-variety (resp. $CC$-variety)
if and only if  $X'$ is an $LQEL$-variety (resp. $CC$-variety);
 \item if $X$ is a $QEL$-variety (resp. $LQEL$-variety, $CC$-variety) of type 
 $\delta(X)\geq1$ and if 
$\widetilde{X}\subset\PP^{n-1}$ is a general hyperplane section  
of $X\subset\PP^n$, 
then $\widetilde{X}$ is a $QEL$-variety (resp. $LQEL$-variety, $CC$-variety)  
of type
 $\delta(\widetilde{X})=\delta(X)-1$.
\end{enumerate}
\end{proposition}
Now, for $X\subset\PP^n$ as in Assumption \ref{notation: X irred nondegen} and with $\Sec(X)\subsetneq\PP^n$, 
we consider the general tangential projection 
$\tau_{x,X}:X \dashrightarrow W_{x,X}\subset\PP^{n-r-1}$  and  
we define the non-negative integer
 $\widetilde{\gamma}(X)$
as the dimension of the general fiber of the Gauss map 
$$
\G_{W_{x,X}}:W_{x,X}\dashrightarrow\GG(r-\delta(X), n-r-1).
$$
From \cite[Lemma~2.3.4]{russo-specialvarieties} we obtain the relation
\begin{equation}\label{eq: gamma-delta}
\widetilde{\gamma}(X)=\gamma(X)-\delta(X).
\end{equation}
Theorem \ref{prop: scorza lemma} 
provides a sufficient condition for $X$ to be
 an $LQEL$-variety.
For its proof, see \cite[Theorem~2.3.5]{russo-specialvarieties}.
\begin{theorem}[Scorza Lemma]\label{prop: scorza lemma} 
Let $X\subset\PP^n$ be as in Assumption \ref{notation: X irred nondegen}, with $\Sec(X)\subsetneq\PP^n$ 
and $\delta(X)\geq 1$ and let $x,y\in X$ ($x\neq y$) be general points. 
If $\widetilde{\gamma}(X)=0$, then
\begin{enumerate}
 \item The irreducible component of the closure of fiber of the rational map 
 $\tau_{x,X}:X\dashrightarrow W_{x,X}$ passing through
 $y$ is either an irreducible quadric hypersurface of dimension $\delta(X)$ or a linear 
space of dimension  $\delta(X)$, the last case occurring only for singular varieties.
\item There exists on  $X\subset\PP^n$
a $2(r-\delta(X))$-dimensional family $\mathcal{Q}$ of 
quadric hypersurfaces of dimension  $\delta(X)$ such that 
there exists a unique member  $Q_{x,y}\in\mathcal{Q}$ passing through the general 
points $x,y$.  Furthermore,   $Q_{x,y}$ is smooth at the points  
$x,y$ and it consists of the irreducible components of $\Sigma_{p}(X)$ 
passing through $x,y$, where $p\in\langle x,y\rangle$ is general. 
\item If $X$ is smooth, then a general member of $\mathcal{Q}$ is smooth.
\end{enumerate}
\end{theorem}
Theorems \ref{prop: main qel1} and \ref{prop: divisibility theorem} 
provide strong 
numerical and geometric constraints
to  smooth 
$LQEL$-variety of type $\delta\geq 3$. 
\begin{theorem}[\cite{russo-qel1}]\label{prop: main qel1}
 Let $X\subset\PP^n$ be a smooth $LQEL$-variety of type $\delta(X)\geq3$. Then
\begin{enumerate}
 \item $X$ is a Fano variety of the first species of index  
$i(X)=\frac{\dim(X)+\delta(X)}{2}$.
 \item If $x\in X$ is a general point, 
then $\L_{x,X}\subset\PP^{\dim(X)-1}$ is a smooth irreducible $QEL$-variety 
of dimension 
$\dim(\L_{x,X})= i(X)-2$ and type $\delta(\L_{x,X})=\delta(X)-2$.
\end{enumerate}
\end{theorem}
As a direct consequence of Theorem \ref{prop: main qel1}, one obtains the following:
\begin{theorem}[Divisibility Theorem, \cite{russo-qel1}]\label{prop: divisibility theorem}
If $X\subset\PP^n$ is a smooth $LQEL$-variety of type $\delta(X)>0$, then
$$
\dim(X)\equiv \delta(X)\ \mathrm{mod}\ 2^{\lfloor (\delta(X)-1)/2 \rfloor}.
$$ 
\end{theorem}
Note that for a variety with secant defect  $\delta=0$, being an 
 $LQEL$-variety imposes no restriction;
differently, being 
a $QEL$-variety is fairly unique. 
Smooth $QEL$-varieties $X\subset\PP^n$ of type $\delta=0$ and with $\Sec(X)=\PP^n$ are 
also called \emph{varieties with one apparent double point}, briefly $OADP$-varieties 
(see \cite{ciliberto-mella-russo}).
We have the following:
\begin{theorem}[\cite{ionescu-russo-qel2,ciliberto-mella-russo}]\label{prop: birational tangential projection when delta is 0} If 
$X\subset\PP^n$ is a smooth $QEL$-variety of type $\delta=0$, then
the general tangential projection
 $\tau_{x,X}:X\dashrightarrow W_{x,X}$ is birational.
\end{theorem}
\subsection{\texorpdfstring{$LQEL$}{LQEL}-varieties of higher type}
The first part of Proposition \ref{prop: LQEL with delta=r and delta=r-1} 
is an application of
the definition of  $LQEL$-variety (see \cite[Proposition~1.3]{russo-qel1}), 
while the second part follows from 
 Theorem \ref{prop: main qel1} (see \cite[Proposition~3.4]{russo-qel1}).
\begin{proposition}\label{prop: LQEL with delta=r and delta=r-1} 
Let $X\subset\PP^n$ be an $LQEL$-variety of type $\delta$ and  dimension $r\geq1$.
 \begin{enumerate}
  \item $\delta=r$ if and only if 
$X$ is a quadric hypersurface of $\PP^n$.
  \item If $X$ is smooth, then
$\delta=r-1$ if and only if $X$ is one of the following: 
           \begin{enumerate}
            \item $\PP^1\times\PP^2\subset\PP^5$ or one of its hyperplane sections;
            \item the Veronese surface $\nu_2(\PP^2)\subset\PP^5$ or one of its isomorphic projections in $\PP^4$.
           \end{enumerate}
 \end{enumerate}
\end{proposition}
Theorem \ref{prop: LQEL of higher type} is an application of Theorems  
\ref{prop: divisibility theorem}, \ref{prop: classification del pezzo varieties} and \ref{prop: classification first species mukai varieties}.
\begin{theorem}[\cite{russo-qel1}]\label{prop: LQEL of higher type}
 Let $X\subset\PP^n$ be a smooth 
  $r$-dimensional 
 $LQEL$-variety  of type $\delta$.
If $\frac{r}{2}<\delta< r$, then $X$ is projectively 
equivalent to one of the following:
\begin{enumerate}
 \item the Segre $3$-fold $\PP^1\times\PP^2\subset\PP^5$;
 \item a hyperplane section of the Grassmannian $\GG(1,4)\subset\PP^9$;
 \item the Grassmannian $\GG(1,4)\subset\PP^9$;
 \item a hyperplane section of the spinorial variety $S^{10}\subset\PP^{15}$;
 \item the spinorial variety $S^{10}\subset\PP^{15}$.
\end{enumerate}
\end{theorem}
\begin{proof} 
The proof is located in 
 \cite[page~609]{russo-qel1}, but we sketch it here.
 We have
$$
0< \frac{r-\delta}{2^{\lfloor (\delta-1)/2 \rfloor}}< \frac{2\delta-\delta}{2^{\lfloor (\delta-1)/2 \rfloor}}< 1,\quad \mbox{if } \delta\geq 9.
$$
Thus, by Theorem \ref{prop: divisibility theorem}, 
we obtain $\delta\leq 8$ and hence $r\leq 15$.
Moreover, by Theorems \ref{prop: divisibility theorem} and
\ref{prop: main qel1} we obtain either that
$(r,\delta)=(3,2)$  
or $X$ is a Fano variety of the first species of coindex $c$ such that
$(r,\delta,c)\in\{ (5,3,2),\ (6,4,2),\ (9,5,3),\ (10,6,3) \}$.
In the first case we conclude from Proposition \ref{prop: LQEL with delta=r and delta=r-1};
in the second case we conclude from
 Theorems \ref{prop: classification del pezzo varieties} and \ref{prop: classification first species mukai varieties}.
\end{proof}
Theorem \ref{prop: LQEL of type delta=r/2} 
is contained in \cite[Corollary~3.2]{russo-qel1}.
It also classifies the so-called Severi varieties 
 that we will treat in more detail in \S  \ref{sec: severi varieties}.
\begin{theorem}[\cite{russo-qel1}]\label{prop: LQEL of type delta=r/2}
 Let $X\subset\PP^n$ be a smooth linearly normal $r$-dimensional 
 $LQEL$-variety  of type $\delta=r/2$.
Then $X$ is projectively 
equivalent to one of the following:
\begin{enumerate}
\item the cubic scroll $\PP_{\PP^1}(\O(1)\oplus\O(2))\subset\PP^4$;
\item the Veronese surface $\nu_2(\PP^2)\subset\PP^5$;
\item the Segre $4$-fold $\PP^1\times\PP^3\subset\PP^7$;
\item a $4$-dimensional linear section of $\GG(1,4)\subset\PP^9$;
\item the Segre $4$-fold $\PP^2\times\PP^2\subset\PP^8$;
\item an $8$-dimensional linear section of $S^{10}\subset\PP^{15}$;
\item the Grassmannian $\GG(1,5)\subset\PP^{14}$;
\item the Cartan variety $E_6\subset\PP^{26}$;
\item a $16$-dimensional Fano variety $X\subset\PP^{25}$ of coindex $5$,
 with $\Sec(X)=\PP^{25}$ and such that $B_{x,X}\subset\PP^{15}$ 
is the union of a spinorial variety $S^{10}\subset\PP^{15}$ with 
 $L_p(X)\simeq \PP^7$, $p\in\PP^{15}\setminus S^{10}$.
\end{enumerate}
\end{theorem}
\subsection{\texorpdfstring{$CC$}{CC}-varieties}
From Proposition \ref{prop: first properties LQEL},
the classification of the smooth $CC$-varieties is reduced to the case of
linearly normal varieties. 
In this case, Theorem \ref{prop: classification CC-varieties} 
reduces the classification to the study of Fano 
varieties having large index and Picard group $\ZZ$. Note that, conversely, 
such Fano varieties are $CC$-varieties (see \cite{ionescu-russo-qel2}).
\begin{theorem}[\cite{ionescu-russo-conicconnected}]\label{prop: classification CC-varieties} Let $X\subset\PP^n$ be 
a  smooth linearly normal $CC$-variety 
of dimension $r$. Then either $X\subset\PP^n$
is a Fano variety of the first species of index  
$i(X)\geq\frac{r+1}{2}$, or $X$ is projectively equivalent to one
of the following:
\begin{enumerate}
 \item The Veronese embedding $\nu_2(\PP^r)\subset\PP^{\frac{r(r+3)}{2}}$.
 \item The projection of $\nu_2(\PP^r)$ from the linear space 
 $\langle \nu_2(\PP^s)\rangle$,
where $\PP^s\subset\PP^r$ is a linear subspace; 
equivalently $X\simeq\Bl_{\PP^s}(\PP^r)$
embedded in $\PP^n$ 
by the linear system of quadric hypersurfaces of
 $\PP^r$ passing through $\PP^s$;
alternatively $X\simeq\PP_{\PP^t}(\E)$ with 
$\E\simeq\O_{\PP^t}(1)^{\oplus r-t}\oplus\O_{\PP^t}(2)$, $t=1,2,\ldots,r-1$,
embedded by $|\O_{\PP(\E)}(1)|$. Here 
$n=\frac{r(r+3)}{2}  -\binom{s+2}{2}
$ and $s$ is an integer such that $0\leq s \leq r-2$.
\item A hyperplane section of the Segre embedding
$\PP^a\times\PP^b\subset\PP^{n+1}$. Here 
$r\geq 3$ and $n=ab+a+b-1$, where $a\geq 2$ and $b\geq 2$ are 
such that $a+b=r+1$.
\item The Segre embedding $\PP^a\times\PP^b\subset\PP^{ab+a+b}$, 
where $a,b$ are positive integers such that $a+b=r$.
\end{enumerate}
\end{theorem}
From Theorem \ref{prop: classification CC-varieties} it follows that 
the smooth  $CC$-varieties are Fano varieties 
and have second Betti number  $b_2\leq 2$; moreover 
those with $b_2=2$ 
are $LQEL$-varieties (see \cite[Corollary~2.3]{ionescu-russo-conicconnected}).
\section{Severi varieties}\label{sec: severi varieties}
Let $X\subset\PP^n$ be as in Assumption \ref{notation: X irred nondegen} and further suppose 
that   $\Sec(X)\subsetneq\PP^n$.
From \cite[Theorem~3.1.6]{russo-specialvarieties} (see also 
\cite[\Rmnum{2} Corollary~2.11]{zak-tangent}) it follows that
\begin{equation}
 n-1\geq \dim(\Sec(X))\geq \frac{3}{2}r+\frac{1-\dim(\sing(X))}{2}+\frac{\widetilde{\gamma}(X)}{2}  \geq \frac{3}{2}r+\frac{1-\dim(\sing(X))}{2}  ;
\end{equation}
in particular, if $X$ is also smooth, then
\begin{equation}\label{eq: inequality zak}
n \geq \frac{3}{2}r+2.
\end{equation}
\begin{definition}
A smooth variety $X$ as above  is called a \emph{Severi variety} 
if the inequality (\ref{eq: inequality zak}) 
is actually an equality. 
\end{definition}
Proposition \ref{prop: properties severi varieties} is contained in
\cite[\Rmnum{4} Theorem~2.4]{zak-tangent}.  
\begin{proposition}[\cite{zak-tangent}]\label{prop: properties severi varieties}
Let $X\subset \PP^n$ be a  Severi variety.
\begin{enumerate}
 \item $\Sec(X)$ is a normal hypersurface and $\mathrm{sing}(\Sec(X))=X$. 
 \item For each $p\in \Sec(X)\setminus X$, $\Sigma_p(X)$ 
is a smooth quadric hypersurface of dimension $r/2$.
 \item For each $p\in \Sec(X)\setminus X$, the linear projection  
$\pi_{L}:X\dashrightarrow\PP^r$
from $L=\langle \Sigma_p(X) \rangle$ is birational and it is 
an isomorphism outside  $T_p(\Sec(X))\cap X$.
\end{enumerate}
\end{proposition}
\subsection{Classification}
Severi varieties are  completely classified by
Theorem \ref{prop: classification severi varieties};
for its proof see
\cite{lazarsfeld-vandeven} or \cite[\Rmnum{4} Theorem~4.7]{zak-tangent} 
(see also \cite[\S 3.3]{russo-specialvarieties}).
\begin{theorem}\label{prop: classification severi varieties} 
Each Severi variety $X$ is
projectively equivalent to one of the following four varieties:
\begin{enumerate}
 \item the Veronese surface $\nu_2(\PP^2)\subset\PP^5$;
 \item the Segre variety $\PP^2\times\PP^2\subset\PP^8$;
 \item the Grassmannian $\GG(1,5)\subset\PP^{14}$;
 \item the Cartan variety $E_6\subset\PP^{26}$.
\end{enumerate}
All these varieties are homogeneous, rational and are defined by 
quadratic equations.
\end{theorem}
There exists another interesting
description of Severi varieties, given in 
Theorem \ref{prop: classification quadro-quadric special cremona}.
First recall that a quadro-quadric Cremona transformations 
is a  birational transformations of $\PP^n$ defined by quadratic forms
and having inverse of the same kind; we also say that a Cremona transformation 
is special if its base locus is smooth and connected. 
Then we have the following:
\begin{theorem}[\cite{ein-shepherdbarron}, see also \cite{pirio-russo}]\label{prop: classification quadro-quadric special cremona}
Let $\varphi:\PP^n\dashrightarrow\PP^n$ be a special Cremona transformation
and let $\B$ be its base locus.
$\varphi$ is a quadro-quadric Cremona transformation if and only if 
 $\B$ is a Severi variety.
Moreover, in this case
 $\varphi$ is an involution (i.e. $\varphi^{-1}=\varphi$).
\end{theorem}
\subsection{\texorpdfstring{$R_1$}{R1}-varieties}
Let $X\subset\PP^n$ be as in Assumption \ref{notation: X irred nondegen} and further suppose that 
 $\Sec(X)\subsetneq\PP^n$.
\begin{definition}
We say that  $X$ enjoys the \emph{$R_1$-property} 
(or briefly that $X$ is a \emph{$R_1$-variety})
if for the general point   $p\in\Sec(X)$ and the general 
hyperplane  $H\in(\PP^n)^{\ast}$ containing $T_p(\Sec(X))$
   it  holds that 
 $H$ is $J$-tangent to $X$ along $\Gamma_p(X)$ (see Definitions \ref{def tangential contact locus} and \ref{def: J-tangency}).
\end{definition}
Proposition \ref{prop: criterion R1 property} provides a simple criterion 
for the  $R_1$-property; for its proof see \cite[Proposition~5.13]{chiantini-ciliberto}.
\begin{proposition}\label{prop: criterion R1 property}
 If $\widetilde{\gamma}(X)=0$ 
and $X$ is different from a cone,\footnote{Note that $X$ is different from a cone 
if and only if $\bigcap_{x\in X}T_x(X)=\emptyset$, 
see  \cite[Proposition~1.2.6]{russo-specialvarieties}.} 
then $X$ is a $R_1$-variety.
\end{proposition}
The importance of the $R_1$-property appears clear in the following:
\begin{theorem}[\cite{chiantini-ciliberto}]\label{prop: R1 varieties and severi varieties}
 If $X\subset\PP^n$ is a $R_1$-variety with $n=\frac{3}{2}r+2$, 
then $X$ is smooth (i.e. $X$ is a Severi variety).
\end{theorem}
\subsection{Equations of Severi varieties}\label{sec: equations severi varieties}
Here we obtain the equations of the Severi varieties 
using a known construction (see \cite[\Rmnum{4} Theorem~4.5]{zak-tangent}).
As usual, we denote by $x_0,\ldots,x_n$ 
homogeneous coordinates on $\PP^n$. 
\subsubsection{\texorpdfstring{$\PP^2\hookrightarrow\PP^5$}{}} 
The Veronese surface
 $\PP^2\hookrightarrow\PP^5$ can be 
defined by the  $2\times2$ minors of the matrix (see \cite{harris-firstcourse}):
$$\begin{pmatrix} x_0 & x_3 & x_4\\  x_3 & x_1 & x_5\\  x_4 & x_5 & x_2 \end{pmatrix}$$
Its dimension is $2$ and its degree is $4$.
 \subsubsection{\texorpdfstring{$\PP^2\times\PP^2\hookrightarrow\PP^8$}{}}  
The Segre embedding $\PP^2\times\PP^2\hookrightarrow\PP^8$ can be
defined by the  $2\times2$ minors of the matrix:
$$
\begin{pmatrix} x_0 & x_1 & x_2\\  x_3 & x_4 & x_5\\  x_6 & x_7 & x_8 \end{pmatrix}
$$
i.e. by the quadrics:
\begin{equation}\label{eq: P2xP2inP8}
\begin{array}{ccc}
 x_5x_7-x_4x_8, & 
x_2x_7-x_1x_8, &
x_5x_6-x_3x_8, \\
x_4x_6-x_3x_7, &
x_2x_6-x_0x_8, & 
x_1x_6-x_0x_7, \\
x_2x_4-x_1x_5, & 
x_2x_3-x_0x_5, &
x_1x_3-x_0x_4. 
\end{array}
\end{equation}
Its dimension is $4$ and its degree is  $6$.
 \subsubsection{\texorpdfstring{$\mathbb{G}(1,5)\hookrightarrow\PP^{14}$}{}} 
Consider the Segre embedding
 $\PP^1\times\PP^3\hookrightarrow\PP^7$. 
This variety is defined by the 
 $2\times2$ minors of the matrix:
\begin{equation}\label{eq: equations seg1_3}
\begin{pmatrix} {x}_{0} & {x}_{1} & {x}_{2} & {x}_{3}\\  {x}_{4} & {x}_{5} & {x}_{6} & {x}_{7} \end{pmatrix}
\end{equation}
Take the hyperplane $\PP^7=V(x_8)$ of $\PP^8$ and consider 
the composition 
$
\PP^1\times\PP^3\hookrightarrow\PP^7\hookrightarrow\PP^8 
$.
The image is defined by the $6$ minors of the matrix (\ref{eq: equations seg1_3})
plus the $9$ monomials
 $x_ix_8$,   for $i=0,\ldots,8$, and  these $15$ quadrics define 
a rational map 
 $t_{\PP^{1}\times\PP^{3}}:\PP^{8}\dashrightarrow\PP^{14}$. 
The image of $t_{\PP^{1}\times\PP^{3}}$
is the Grassmannian $\mathbb{G}(1,5)$ and it is defined by the $15$ quadrics:
\begin{equation}
\begin{array}{cc} x_{4} x_{6} - x_{2} x_{7} + x_{0} x_{9},& x_{4} x_{10} - x_{2} x_{11} + x_{0} x_{13},\\ x_{8} x_{11} - x_{7} x_{12} + x_{3} x_{14},& x_{8} x_{10} - x_{6} x_{12} + x_{1} x_{14},\\ - x_{3} x_{6} + x_{1} x_{7} - x_{0} x_{8},& - x_{3} x_{10} + x_{1} x_{11} - x_{0} x_{12},\\ - x_{2} x_{3} + x_{1} x_{4} - x_{0} x_{5},& - x_{5} x_{6} + x_{2} x_{8} - x_{1} x_{9},\\ - x_{7} x_{10} + x_{6} x_{11} - x_{0} x_{14},& x_{5} x_{7} - x_{4} x_{8} + x_{3} x_{9},\\ x_{5} x_{10} - x_{2} x_{12} + x_{1} x_{13},& x_{9} x_{11} - x_{7} x_{13} + x_{4} x_{14},\\ x_{5} x_{11} - x_{4} x_{12} + x_{3} x_{13},& - x_{9} x_{12} + x_{8} x_{13} - x_{5} x_{14},\\ x_{9} x_{10} - x_{6} x_{13} + x_{2} x_{14}. 
\end{array}
\end{equation}
The dimension and the degree of  $\mathbb{G}(1,5)$
respectively are $8$ and $14$.
 \subsubsection{\texorpdfstring{$E_6\hookrightarrow\PP^{26}$}{}}  
Consider the Segre embedding
$\PP^1\times\PP^2\hookrightarrow\PP^5$. 
This variety is defined by the 
 $3$ quadrics:
\begin{equation}\label{eq: equations seg1_2}
\begin{array}{ccc}- x_{1} x_{3} + x_{0} x_{4},& - x_{2} x_{3} + x_{0} x_{5}, &- x_{2} x_{4} + x_{1} x_{5}. \end{array}
\end{equation}
Take the hyperplane
 $\PP^5=V(x_6)$ of $\PP^6$ and consider the composition 
$
\PP^1\times\PP^2\hookrightarrow\PP^5\hookrightarrow\PP^6 
$.
The image is defined by the $3$ quadrics  (\ref{eq: equations seg1_2}) plus 
the $7$ monomials
 $x_ix_6$,   for $i=0,\ldots,6$, and these $10$ quadrics
  define a rational map $t_{\PP^1\times\PP^2}:\PP^6\dashrightarrow\PP^9$.     
The image of  $t_{\PP^1\times\PP^2}$ is the Grassmannian $\mathbb{G}(1,4)$ and 
it is defined by the $5$ quadrics:
\begin{equation}\label{eq: equations G14}
\begin{array}{ccc} x_{4} x_{6} - x_{3} x_{7} + x_{0} x_{9},& - x_{5} x_{7} + x_{4} x_{8} - x_{2} x_{9},& x_{2} x_{6} - x_{1} x_{7} + x_{0} x_{8},\\ x_{2} x_{3} - x_{1} x_{4} + x_{0} x_{5},& x_{5} x_{6} - x_{3} x_{8} + x_{1} x_{9}.\end{array}
\end{equation}
The dimension and the degree of
 $\mathbb{G}(1,4)$ respectively are $6$ and $5$.
Now take the hyperplane $\PP^9=V(x_{10})$ of $\PP^{10}$ and consider the composition 
$
\mathbb{G}(1,4)\hookrightarrow\PP^{9}\hookrightarrow\PP^{10} 
$.
The image is defined by the $5$ quadrics (\ref{eq: equations G14}) plus the $11$ monomials
 $x_ix_{10}$,  for $i=0,\ldots,10$, 
and these $16$ quadrics define a rational map 
 $t_{\mathbb{G}(1,4)}:\PP^{10}\dashrightarrow\PP^{15}$. 
The image of $t_{\mathbb{G}(1,4)}$ is the spinorial variety
 $S^{10}$ and it is defined by the 
 $10$ quadrics:
\begin{equation}\label{eq: equations S10}
\begin{array}{cc} x_{7} x_{8} - x_{6} x_{9} + x_{5} x_{10} - x_{3} x_{15}, &x_{1} x_{6} + x_{4} x_{7} - x_{2} x_{10} + x_{3} x_{13},\\ - x_{4} x_{5} + x_{0} x_{6} - x_{2} x_{8} + x_{3} x_{11},& x_{9} x_{11} - x_{8} x_{12} + x_{5} x_{14} - x_{0} x_{15}, \\x_{10} x_{12} - x_{9} x_{13} + x_{7} x_{14} + x_{1} x_{15},& - x_{1} x_{11} - x_{4} x_{12} + x_{0} x_{13} - x_{2} x_{14},\\ x_{10} x_{11} - x_{8} x_{13} + x_{6} x_{14} - x_{4} x_{15},& x_{1} x_{5} + x_{0} x_{7} - x_{2} x_{9} + x_{3} x_{12}, \\- x_{7} x_{11} + x_{6} x_{12} - x_{5} x_{13} + x_{2} x_{15}, & x_{1} x_{8} + x_{4} x_{9} - x_{0} x_{10} + x_{3} x_{14}.\end{array}
\end{equation}
The dimension and the degree of  $S^{10}$
respectively are $10$ and $12$. Finally,
 take the hyperplane 
$\PP^{15}=V(x_{16})$ of $\PP^{16}$ and consider the composition
$
S^{10}\hookrightarrow\PP^{15}\hookrightarrow\PP^{16} 
$.
The image is defined by the $10$ quadrics  (\ref{eq: equations S10}) plus 
the $17$ monomials
 $x_ix_{16}$,  for $i=0,\ldots,16$, and these $27$ quadrics 
define a rational map  $t_{S^{10}}:\PP^{16}\dashrightarrow\PP^{26}$. 
The image of  $t_{S^{10}}$ is the Cartan variety $E_6$ and it is 
defined by the $27$ quadrics: 
\begin{equation}
\begin{array}{cc} x_{3} x_{10} - x_{15} x_{17} + x_{14} x_{18} - x_{12} x_{21} + x_{11} x_{26},& x_{7} x_{22} + x_{5} x_{23} + x_{9} x_{24} + x_{2} x_{25} - x_{1} x_{26},\\
 x_{2} x_{16} - x_{1} x_{18} + x_{9} x_{20} + x_{0} x_{22} - x_{6} x_{23},& x_{2} x_{10} + x_{21} x_{22} - x_{15} x_{23} + x_{18} x_{24} - x_{20} x_{26},\\ 
- x_{2} x_{13} + x_{1} x_{15} + x_{5} x_{20} + x_{8} x_{22} + x_{6} x_{24},& x_{1} x_{10} - x_{19} x_{22} - x_{13} x_{23} + x_{16} x_{24} - x_{20} x_{25},\\
 - x_{7} x_{16} + x_{1} x_{17} - x_{9} x_{19} + x_{4} x_{23} + x_{0} x_{25},& x_{2} x_{17} - x_{7} x_{18} + x_{9} x_{21} - x_{3} x_{23} + x_{0} x_{26},\\
 x_{1} x_{11} + x_{0} x_{13} + x_{8} x_{16} + x_{6} x_{19} - x_{4} x_{20},& x_{2} x_{12} - x_{9} x_{15} - x_{5} x_{18} + x_{3} x_{22} - x_{6} x_{26},\\
 - x_{7} x_{12} + x_{9} x_{14} + x_{5} x_{17} + x_{3} x_{25} + x_{4} x_{26},& - x_{1} x_{12} + x_{9} x_{13} + x_{5} x_{16} + x_{4} x_{22} + x_{6} x_{25},\\
 - x_{5} x_{11} - x_{8} x_{12} - x_{3} x_{13} + x_{6} x_{14} - x_{4} x_{15},& x_{2} x_{19} - x_{7} x_{20} + x_{1} x_{21} + x_{8} x_{23} + x_{0} x_{24},\\
 x_{7} x_{10} - x_{14} x_{23} + x_{17} x_{24} - x_{21} x_{25} - x_{19} x_{26},& x_{0} x_{10} - x_{18} x_{19} + x_{17} x_{20} - x_{16} x_{21} + x_{11} x_{23},\\
 x_{8} x_{10} + x_{15} x_{19} - x_{14} x_{20} + x_{13} x_{21} - x_{11} x_{24},& x_{1} x_{3} + x_{2} x_{4} - x_{0} x_{5} - x_{6} x_{7} + x_{8} x_{9},\\
 - x_{7} x_{13} + x_{1} x_{14} + x_{5} x_{19} + x_{4} x_{24} - x_{8} x_{25},& x_{2} x_{11} + x_{0} x_{15} + x_{8} x_{18} + x_{3} x_{20} - x_{6} x_{21},\\
 x_{5} x_{10} + x_{14} x_{22} + x_{12} x_{24} + x_{15} x_{25} - x_{13} x_{26},& x_{6} x_{10} - x_{15} x_{16} + x_{13} x_{18} - x_{12} x_{20} + x_{11} x_{22},\\
 x_{9} x_{11} + x_{0} x_{12} - x_{3} x_{16} + x_{6} x_{17} - x_{4} x_{18},& x_{7} x_{11} + x_{0} x_{14} + x_{8} x_{17} + x_{3} x_{19} - x_{4} x_{21},\\
 x_{9} x_{10} - x_{17} x_{22} - x_{12} x_{23} - x_{18} x_{25} + x_{16} x_{26},& - x_{2} x_{14} + x_{7} x_{15} + x_{5} x_{21} + x_{3} x_{24} + x_{8} x_{26},\\
 - x_{4} x_{10} + x_{14} x_{16} - x_{13} x_{17} + x_{12} x_{19} + x_{11} x_{25} . \end{array}
\end{equation}
The dimension and the degree of  $E_6$ respectively are $16$ and $78$.
\subsection{\texorpdfstring{$\L_{x,X}$}{LxX} of Severi varieties}
\begin{proposition}
Table \ref{tab: hilbert scheme severi varieties} describes
 the Hilbert scheme $\L_{x,X}$ of lines passing through a 
point  $x$ of a Severi variety $X$. 
\begin{table}[htbp]
\begin{center}
\begin{tabular}{|c|c|}
\hline
$X$ & $\L_{x,X}$\\
\hline
\hline
$\nu_2(\PP^2)\subset\PP^5$ & $\emptyset$\\
 \hline
$\PP^2\times\PP^2\subset\PP^8$ & $\PP^1\sqcup\PP^1\subset \PP^3$\\
 \hline
$\mathbb{G}(1,5)\subset\PP^{14}$ & $\PP^1\times\PP^3\subset\PP^7$\\
 \hline
$E_6\subset\PP^{26}$ & $S^{10}\subset\PP^{15}$\\
\hline
\end{tabular}
\end{center}
\caption{$\L_{x,X}$ of Severi varieties.}
\label{tab: hilbert scheme severi varieties}
\end{table}
\end{proposition}
\begin{proof}
Consider the rational maps above defined 
\begin{equation}\label{eq: birational maps above defined}
 t_{\PP^{1}\times\PP^{3}}:\PP^{8}\dashrightarrow\PP^{14},\quad
t_{S^{10}}:\PP^{16}\dashrightarrow\PP^{26}
\end{equation}
and similarly define
\begin{equation}\label{eq: birational maps similarly defined}
t_{\emptyset}:\PP^2\dashrightarrow\PP^{5},\quad t_{\PP^1\sqcup\PP^1}:\PP^4\dashrightarrow\PP^{8}.
\end{equation}
The rational maps in (\ref{eq: birational maps above defined}) and (\ref{eq: birational maps similarly defined})
are obviously birational onto them images and 
they have inverses defined by linear forms.
The rest of the proof follows from 
Lemma \ref{prop: hilbert scheme image birational map of type 2 1} and
the homogeneity of Severi varieties.
\end{proof}
\begin{lemma}\label{prop: hilbert scheme image birational map of type 2 1}
 Let $Y\subset H=\PP^{r-1}\subset\PP^r$ be the 
base locus of a quadratic birational map 
$t:\PP^r\dashrightarrow \overline{t(\PP^r)}=X\subset\PP^n$ whose inverse 
is defined by linear forms. Let $x\in X$ be a general point 
and assume that  $Y$ is reduced.  Then
$$
\L_{x,X}\simeq Y \subset \PP^{r-1}.
$$
\end{lemma}
\begin{proof}
 Let $p\in \PP^r\setminus H$ be in the open set 
 where $t$ is an isomorphism and put   $x=t(p)$. 
Consider the isomorphism 
 $H\stackrel{\simeq}{\longrightarrow}\L_{p,\PP^r}$ 
 which sends the
 point $q\in H$ to the line $[\langle q,p\rangle]\in\L_{p,\PP^r}$
and denote by $\mathcal{Y}$ the image of $Y$ in $\L_{p,\PP^r}$. 
$\mathcal{Y}$ consists of all lines of
$\PP^r$ passing through $p$ and intersecting $Y$, so 
 the images of these lines via
$t$ are lines contained in $X$. 
Consider the closed subscheme 
$$
Z=\{(z,[l]): z\in t(l)\} \hookrightarrow X\times \mathcal{Y}.
$$
By universal property of
 $\L_{x,X}$, 
there exists a unique  morphism 
 $f:\mathcal{Y}\rightarrow\L_{x,X}$ such that we have 
the diagram (\ref{diagram hilbert scheme of lines}) of fibred products.     
By the hypothesis on $t^{-1}$ and  choice of the point $p$,  
we have a well-defined morphism 
$$
[l]\in\L_{x,X}\longmapsto t^{-1}(l)\in \mathcal{Y} ,
$$
which is obviously the inverse of $f$.     
\end{proof}

\chapter{Outline on Castelnuovo theory}\label{cap: castelnuovo theory}
This chapter is devoted to outline some well-known extensions of Castelnuovo's results 
and conjectures about them.
These facts will be useful when we study the 
special quadratic birational transformations
having base locus of small dimension, 
see Chap. \ref{sec: transformations into a hypersurface} \S \ref{sec: small dimension of B}, 
Chap. \ref{chapter: transformations whose base locus has dimension at most three}
and App. \ref{app: towards the case of dimension four}.
The main references are:
\cite{harris-curves}, 
\cite{ciliberto-hilbertfunctions},
\cite{eisenbud-green-harris} and
\cite{petrakiev}
(see also the classical work \cite{castelnuovo-curve}).
\paragraph{Notation and conventions}
As always, we shall work over $\CC$. Recall that,
for any projective $\CC$-scheme $\Lambda\subset\PP^N$ the \emph{Hilbert function} 
of $\Lambda$ is
$$
h_{\Lambda}(t)=\dim_{\CC}\left( \mathrm{Im}\left\{ \rho_{t}(\Lambda): H^0(\O_{\PP^N}(t))\rightarrow H^0(\O_{\Lambda}(t)) \right\}  \right) , 
$$
where $t\in\mathbb{N}$ and $\rho_{t}(\Lambda)$ is the natural restriction map. 
So that for example $h_{\Lambda}(2)$ is the number of conditions imposed by $\Lambda$ on quadrics.
A finite set of points $\Lambda\subset\PP^N$ is in \emph{uniform position} if 
for any subset $\Lambda'\subseteq \Lambda$ of order $\lambda'$ and for any $t\in \mathbb{N}$, 
it is $h_{\Lambda'}(t)=\min\{\lambda',h_{\Lambda}(t)\}$.
If a set of points $\Lambda\subset\PP^N$ is in uniform position, 
then it is also in \emph{general position}, which means that any set of at most $N+1$
points in $\Lambda$ is linearly independent.
If  $\Lambda\subset\PP^{N}$ 
is obtained by taking a general hyperplane section 
of a nondegenerate irreducible curve $C\subset\PP^{N+1}$,
then $\Lambda$ satisfies a stronger property than being in uniform position, 
namely $\Lambda$ is in \emph{symmetric position} (see \cite{petrakiev}).
\section{Castelnuovo's bound}
Let $C\subset\PP^{N+1}$, $N\geq 2$, be an irreducible nondegenerate projective curve of 
degree $\lambda$ and arithmetic genus $g$ and let $\Lambda\subset\PP^{N}$ be a general 
hyperplane section of $C$. 
As it is well-known, there is an upper bound, the so-called \emph{Castelnuovo's bound},
for $g$  as a function of $\lambda$ and $N$. Precisely one has 
the following:
\begin{proposition}[Castelnuovo's bound]\label{prop: castelnuovo bound} Let notation be 
as above. Then
\begin{equation}\label{eq: castelnuovo bound}
 g\leq \pi_{0}(\lambda,N)=\binom{q_{0}}{2}N+q_{0} r_{0} ,
\end{equation}
where
$$
\lambda-1=q_0 N+r_0, \quad 0\leq r_0<N .
$$
\end{proposition}
\begin{proof} We briefly review the idea of the proof.
It is a standard fact (see e.g. \cite[Lemma~1.1]{ciliberto-hilbertfunctions}) 
that for 
$t\in\mathbb{N}$ one has:
\begin{equation}\label{eq: ciliberto lemma 1.1}
 h_{\Lambda}(t)\leq h_{C}(t)-h_{C}(t-1).
\end{equation}
Now, the Hilbert polynomial of $C$ is $P_C(t)=\lambda t -g+1$. So, by 
summing (\ref{eq: ciliberto lemma 1.1}) over all $t\geq1$, 
we get (see e.g. \cite[Lemma~1.4]{ciliberto-hilbertfunctions})
\begin{equation}\label{eq: ciliberto lemma 1.4}
 g\leq \sum_{t=1}^{\infty}\left( \lambda-h_{\Lambda}(t) \right) .
\end{equation}
Moreover, an easy consequence of the uniform position property of $\Lambda$ 
is that 
\begin{equation}\label{eq: ciliberto lemma 1.5}
 h_{\Lambda}(t+s)\geq \min\left\{ \lambda, h_{\Lambda}(t)+h_{\Lambda}(s)-1  \right\},\quad \forall  t,s\in\mathbb{N} ,
\end{equation}
from which it follows that 
\begin{equation}\label{eq: ciliberto 1.6}
 h_{\Lambda}(t)\geq \min\left\{ \lambda, t N+1 \right\}, \quad \forall t\in\mathbb{N}.
\end{equation}
Then it is
\begin{equation}\label{eq: ciliberto 1.7}
\begin{array}{ll}
 h_{\Lambda}(t) \geq  t N+1, &   1\leq t\leq  q_0  \\
 h_{\Lambda}(t) =\lambda , &   t\geq q_0+1 
\end{array}
\end{equation}
and  (\ref{eq: castelnuovo bound}) follows by using (\ref{eq: ciliberto lemma 1.4}).
\end{proof}
More generally, (\ref{eq: ciliberto 1.6}) 
holds for a set of points $\Lambda$ in general position.
For $t=2$, this follows from the elementary Lemma \ref{prop: castelnuovo argument}; 
the general case follows from a similar argument.
\begin{lemma}[Usual Castelnuovo's argument]\label{prop: castelnuovo argument}  
 If $\Lambda\subset\PP^{N}$ is a set 
 of $\lambda\leq 2N+1$ points in general position, then
$\Lambda$ imposes independent conditions to the quadrics of $\PP^N$, i.e. 
$h_{\Lambda}(2)=\lambda$.
\end{lemma}
\begin{proof}
We can assume $\lambda=2N+1$.     
Put $\Lambda=\{p_0,\ldots,p_{2N}\}$ 
and consider the hyperplanes  $H_1=\langle p_1,\ldots,p_N\rangle$ and
 $H_2=\langle p_{N+1},\ldots,p_{2N}\rangle$.     
Since the points are in general position,
the quadric  
 $H_1\cup H_2$ contains the points  $p_1,\ldots,p_{2N}$,  but not  $p_0$.     
This proves the exactness of the sequence 
$0\rightarrow H^0(\PP^N,\I_{\Lambda,\PP^N}(2))\rightarrow 
H^0(\PP^N, \O_{\PP^N}(2))\rightarrow
 H^0(\Lambda, \O_{\Lambda}(2))=\bigoplus_{i=0}^{2N}\CC\rightarrow 0 $,
from which the assertion follows.
\end{proof} 
The following well-known result  describes $\Lambda$  
in the case when $h_{\Lambda}(2)$
is minimal, assuming that $\lambda$ is not too small. 
\begin{proposition}[Castelnuovo Lemma]\label{prop: castelnuovo lemma} Let $\Lambda\subset\PP^N$ be 
a set of $\lambda\geq 2N+3$ points in  general position. 
 If  $h_{\Lambda}(2)=2N+1$, then $\Lambda$ lies on 
a rational normal curve of degree $N$, cut out by all quadrics containing $\Lambda$.
\end{proposition}
By Proposition \ref{prop: castelnuovo lemma} is obtained the description
of curves $C\subset\PP^{N+1}$ that attain the maximal genus, the 
so-called \emph{Castelnuovo curves}.
\section{Refinements of Castelnuovo's bound}
One question raised by Fano in \cite{fano} is the following:
\emph{to get, under suitable conditions on $C$, better estimates for $h_{\Lambda}$ that 
(\ref{eq: ciliberto 1.7}), in order to obtain, using (\ref{eq: ciliberto lemma 1.4}), 
better bounds for $g$.} 
If $\lambda\geq 2N+3$ and $h_{\Lambda}(2)=2N+1$ then, 
by Proposition \ref{prop: castelnuovo lemma}, no better estimates 
for $h_{\Lambda}$ can be found.
So Fano's idea was to assume
\begin{equation}\label{eq: ciliberto 2.1}
 h_{\Lambda}(2)\geq 2N+1+ \vartheta,\quad 0<\vartheta\leq \lambda-2N-1
\end{equation}
and to estimate $h_{\Lambda}$, and then $g$, under this hypothesis.
The first result in this direction is 
the following: 
\begin{proposition}[Fano]\label{prop: refinement castelnuovo bound by fano}
 If (\ref{eq: ciliberto 2.1}) holds, then $g\leq \pi_{0}(\lambda-\vartheta,N)+\vartheta$ .
\end{proposition}
\begin{proof}
See also \cite[Theorem~2.3]{ciliberto-hilbertfunctions}. Let $t\in\mathbb{N}$ and 
let $h=\lfloor n/2\rfloor$. By using (\ref{eq: ciliberto lemma 1.5}) and 
the hypothesis (\ref{eq: ciliberto 2.1}) we get 
$$
h_{\Lambda}(t)\geq \min \left\{ \lambda, tN+1+h\vartheta \right\}
\geq \min \left\{\lambda, tN+1+\vartheta \right\}.
$$
So, by (\ref{eq: ciliberto lemma 1.4}) we deduce the assertion.
\end{proof}
\begin{figure}[htb]   
\centering
\includegraphics[width=0.7\textwidth]{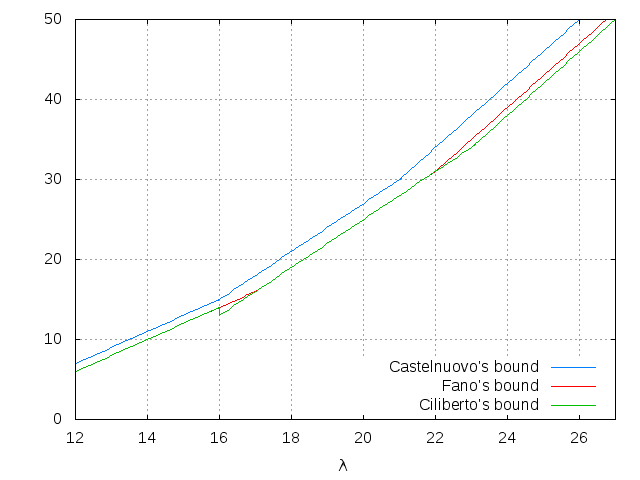} 
\caption{Comparison of upper bounds for $g$, when $N=5$ and $\vartheta=1$.}
\label{fig: comparisonupperbounds}
\end{figure}
Fano's bound of Proposition \ref{prop: refinement castelnuovo bound by fano} 
is too rough; 
in fact it
can be attained only for $\lambda\leq 4N+1+\vartheta$.
A more accurate result is the following 
(see \cite[Theorem~2.5]{ciliberto-hilbertfunctions}):
\begin{theorem}[Ciliberto]\label{prop: refinement castelnuovo bound by ciliberto}
 If (\ref{eq: ciliberto 2.1}) holds, then
$$
g\leq \theta(\lambda, N)= \left\{ 
\begin{array}{ll} 
\mu_{\vartheta}^2 (2N+\vartheta)-\mu_{\vartheta}N+2\mu_{\vartheta}\eta_{\vartheta},  & \mbox{if }   \eta_{\vartheta} < N, \\
\mu_{\vartheta}^2 (2N+\vartheta)-\mu_{\vartheta}N+2\mu_{\vartheta}\eta_{\vartheta} + (\eta_{\vartheta}-N),  & \mbox{if }   \eta_{\vartheta}\geq N ,
\end{array}    
\right.
$$
where 
$$
\lambda-1=\mu_{\vartheta}(2N+\vartheta)+\eta_{\vartheta}, \quad 0\leq \eta_{\vartheta}< 2N+\vartheta .
$$
\end{theorem}
\section{Eisenbud-Harris Conjecture}    
By extending                             
 Castelnuovo's lemma, one is naturally 
lead to make the following:
\begin{conjecture}[Eisenbud, Harris]\label{conjecture: eisenbud harris}
 Let $\Lambda\subset\PP^N$ be a finite set of $\lambda$ points in uniform position and 
let $\vartheta$ such that $0\leq \vartheta\leq N-3$.
If $\lambda\geq 2N+3+2\vartheta$, 
 and $h_{\Lambda}(2)=2N+1+\vartheta$, then 
$\Lambda$ lies on a curve $D\subset\PP^N$ of degree at most $N+\vartheta$.
\end{conjecture}
\begin{remark}
If the curve $D$ exists, then 
 $D$ is unique  and  it is 
a component of $\mathrm{Bs}(|\I_{\Lambda}(2)|)$ (see e.g. \cite{petrakiev}).
\end{remark}
\begin{remark}
 The assumption $\vartheta\leq N-3$ is necessary. For example, if $\vartheta=N-2$,
then one could take $\Lambda$ to be a complete intersection of a del Pezzo surface 
of degree $N+1$ and two general quadrics to produce a counterexample.  
\end{remark}
Of course, the case $\vartheta=0$ of Conjecture \ref{conjecture: eisenbud harris} 
is  
Castelnuovo's lemma;
the case $\vartheta=1$ is 
the following theorem due to Harris, see \cite{harris-curves}
(the result has been already know to Fano in \cite{fano}).
\begin{theorem}[Fano, Harris]\label{prop: extension castelnuovo lemma by fano harris}
 Let $\Lambda\subset\PP^{N}$, $N\geq3$, be a finite set of $\lambda\geq 2N+5$ points in 
uniform position such that $h_{\Lambda}(2)=2N+2$. 
Then $\Lambda$ lies on an elliptic normal 
curve $D$ of degree $N+1$ in $\PP^{N}$, cut out by all quadrics containing $\Lambda$.
\end{theorem}
Cases $\vartheta=2$ and $\vartheta=3$ 
of Conjecture \ref{conjecture: eisenbud harris} have been 
proved by Petrakiev in \cite{petrakiev},
under the stronger assumption that
  $\Lambda$ is a set of points 
in symmetric position (and hence also under the assumption 
that $\Lambda\subset\PP^N$ is a general hyperplane 
section of a nondegenerate irreducible curve $C\subset\PP^{N+1}$).
\begin{theorem}[Petrakiev]\label{prop: petrakiev}
 Let $\Lambda\subset \PP^N$ be a finite set of $\lambda$ points in symmetric position.
\begin{enumerate}[(a)]
 \item\label{part: petrakiev vartheta=2} If $N\geq 4$, $h_{\Lambda}(2)=2N+3$ and $\lambda\geq 2N+7$,
       then $\Lambda$ lies on a curve $D$ of degree $\leq N+2$.
 \item\label{part: petrakiev vartheta=3} If $N\geq 6$, $h_{\Lambda}(2)=2N+4$ and $\lambda\geq 2N+9$,
       then $\Lambda$ lies on a curve $D$ of degree $\leq N+3$.
\end{enumerate}
\end{theorem}
A results of same type as Theorem \ref{prop: petrakiev}(\ref{part: petrakiev vartheta=2})
was already known and due to 
Ciliberto (see \cite[Theorem~3.8]{ciliberto-hilbertfunctions}),
but it is in general weaker than 
Theorem \ref{prop: petrakiev}(\ref{part: petrakiev vartheta=2}).
The precise statement is as follows:
\begin{theorem}[Ciliberto]\label{prop: extension castelnuovo lemma by ciliberto}  Let 
$\Lambda\subset\PP^{N}$, $N\geq5$, be 
a general hyperplane section of a nondegenerate irreducible 
curve $C\subset\PP^{N+1}$.
 If $h_{\Lambda}(2)=2N+3$ and $\lambda\geq \frac{8}{3}(N+1)$, then 
$\Lambda$ lies on an irreducible curve of degree $\leq N+2$, 
cut out by all quadrics containing $\Lambda$.
\end{theorem}
\section{Bound on the degree of zero-dimensional quadratic schemes} 
Let $\Lambda\subset\PP^N$ be a finite set of $\lambda$ points in uniform position
and furthermore 
we hypothesize that $\Lambda$ is cut out by quadrics. 
Following \cite{eisenbud-green-harris}, we can ask: 
given $h_{\Lambda}(2)$, what is the largest possible $\lambda$?
In other words, 
\emph{
What is the largest number $\lambda(h)$ of points of intersection of a linear 
system of quadrics of codimension $h$ in the space of all quadrics in $\PP^N$,
given that the intersection of those quadrics is zero-dimensional?}
In these terms, we may summarize the state    of 
our knowledge in  Table \ref{table: our knowledge on upper bound of the degree} and,
\begin{table}[htbp]
\begin{center}
\begin{tabular}{|c|c||c|}
\hline
 $h_{\Lambda}(2)$ & $\lambda\leq$ & Proof \\
\hline
\hline
 $N+1$ & $N+1$ & elementary \\
 \hline
 $N+2$ & $N+2$ & elementary \\
 \hline
 $\vdots$ &$\vdots$ &$\vdots$ \\
 \hline
 $2N-1$ & $2N-1$ & elementary \\
\hline
 $2N$ & $2N$ & elementary \\
\hline
 $2N+1$ & $2N+2$ & Proposition \ref{prop: castelnuovo lemma} \\
\hline
 $2N+2$ & $2N+4$ & Theorem \ref{prop: extension castelnuovo lemma by fano harris}  \\
\hline
 $2N+3$ & $2N+6$ & \begin{tabular}{c} Theorem \ref{prop: petrakiev}(\ref{part: petrakiev vartheta=2}) (if $N\geq4$ and \\ the points are in symmetric position) \end{tabular} \\
\hline
 $2N+4$ & $2N+8$ & \begin{tabular}{c}  Theorem \ref{prop: petrakiev}(\ref{part: petrakiev vartheta=3}) (if $N\geq6$ and \\ the points are in symmetric position) \end{tabular} \\
 \hline
\end{tabular}
\end{center}
\caption{State of our knowledge.}
\label{table: our knowledge on upper bound of the degree}
\end{table}
by Conjecture \ref{conjecture: eisenbud harris}, we can also extend the pattern, 
conjecturing an upper bound on the largest $\lambda$,
for $2N+5\leq h_{\Lambda}(2)\leq 3N-2$. 
From the reasons given in \cite{eisenbud-green-harris}, this conjectured upper bound 
 can be further extended. In order to present this, 
let $m=\binom{N+2}{2}-h_{\Lambda}(2)$ be the number of independent 
quadrics containing $\Lambda\subset\PP^N$ and observe that,  
given $N$, any number $m\geq N+1$ can be uniquely written in the form
$$
m=(N+1)+\binom{b}{2}+c,\quad b>c\geq 0 .
$$
With this notation, 
we make the following:
\begin{conjecture}[Eisenbud, Green, Harris]\label{conjecture: on upper bound of the degree}
 If $\Lambda$ is any nondegenerate collection of $\lambda$ points in uniform position 
in $\PP^N$ lying on $m$ independent quadrics whose intersection is zero-dimensional,
then
$$
\lambda\leq (2b-c+1)2^{N-b-1}.
$$
\end{conjecture}
\begin{remark}
The bound of Conjecture \ref{conjecture: on upper bound of the degree}, if indeed it holds, 
is sharp: for $m$ quadrics, we can take $\Lambda$ the intersection of 
$N-b-1$ quadrics with a linearly normal variety of degree $2b-c+1$ and dimension 
$N-b-1$ in $\PP^N$ (for example, the divisor residual to 
$c+1$ planes in the intersection of a rational normal $(N-b)$-fold scroll in $\PP^N$ 
with a quadric).
\end{remark}
Conjecture \ref{conjecture: on upper bound of the degree} is 
 explained in Table \ref{table: conjecture explained}, where
 $h_{\Lambda}(2)=2N+1+\vartheta$, i.e.  
$m=N(N-1)/2-\vartheta$. 
\begin{table}[htbp]
\begin{center}
\begin{tabular}{|c|c|c|c|}
\hline
 $\vartheta$ & $b$ & $c$ & $\lambda\leq$ \\
\hline
\hline
$-2< \vartheta \leq N-4$ & $N-2$ & $N-\vartheta-4$ & $2N+2\vartheta+2$ \\
 \hline
$ N-4< \vartheta \leq 2N-7$ & $N-3$ & $2N-\vartheta-7$ & $4\vartheta+8$ \\
 \hline
$ 2N-7< \vartheta \leq 3N-11$ & $N-4$ & $3N-\vartheta-11$ & $-8N+8\vartheta+32$ \\
 \hline
$ 3N-11< \vartheta \leq 4N-16$ & $N-5$ & $4N-\vartheta-16$ & $-32N+16\vartheta+112$ \\
 \hline
 $\vdots$ &$\vdots$ &$\vdots$ &$\vdots$\\
 \hline
\end{tabular}
\end{center}
\caption{Conjectured upper bound on $\lambda$.}
\label{table: conjecture explained}
\end{table}
Note that the first row of this table 
is a consequence of Conjecture \ref{conjecture: eisenbud harris}.
So, for $\vartheta=0$ and $\vartheta=1$, 
Conjecture \ref{conjecture: on upper bound of the degree} follows, respectively, 
from Proposition \ref{prop: castelnuovo lemma} and 
Theorem \ref{prop: extension castelnuovo lemma by fano harris}. 
For $\vartheta=2$ 
and $\vartheta=3$, 
it remains open in its generality, but
it follows from Theorem \ref{prop: petrakiev}
 under the hypothesis that the points are in symmetric position.
In \cite{eisenbud-green-harris} another special case  
has been proved,
precisely the case in which $m=N+1$ and $N\leq6$.

\chapter{Introduction to quadratic birational transformations of a projective space into a quadric}
In this chapter we begin the study of quadratic birational transformations 
$\varphi:\PP^n\dashrightarrow\overline{\varphi(\PP^n)} =\sS\subset\PP^{n+1}$,
with particular regard to the case in which
$\sS$ is a quadric hypersurface.
\section{Transformations of type \texorpdfstring{$(2,1)$}{(2,1)}}\label{sec: type 2-1}
\begin{definition}
A birational transformation 
 $\varphi:\PP^n\dashrightarrow\overline{\varphi(\PP^n)} =\sS\subset\PP^{n+1}$ 
is said to be of type  $(2,d)$ if it is quadratic 
(i.e. defined by a linear system of quadrics without fixed component) 
and $\varphi^{-1}$ can be defined by a linear system contained in $|\O_{\sS}(d)|$,
with $d$ minimal with this property.
More generally, $\varphi$ is said to be of type 
$(d_1,d_2)$ if $\varphi$ (resp. $\varphi^{-1}$) 
is defined by a linear system contained in $|\O_{\PP^n}(d_1)|$ 
(resp. $|\O_{\sS}(d_2)|$), with $d_1$ and $d_2$ minimal.
$\varphi$ is said to be \emph{special} if its base locus  is smooth and connected.
\end{definition}
\begin{lemma}\label{prop: cohomology I2B} 
Let $\varphi:\PP^n\dashrightarrow\sS\subset\PP^{n+1}$ be 
a quadratic birational transformation, 
with base locus $\B$ and image a normal nonlinear hypersurface $\sS$. 
Then  $h^0(\PP^{n}, \I_{\B}(2))=n+2$. 
\end{lemma}
\begin{proof} 
 Resolve the indeterminacies of $\varphi$ with the diagram 
\begin{equation}\label{eq: diagram resolving map} 
\xymatrix{ & X \ar[dl]_{\pi} \ar[dr]^{\pi'}\\ \PP^n\ar@{-->}[rr]^{\varphi}& & \sS } 
\end{equation}
where $\pi:X=\Bl_{\B}(\PP^n)\rightarrow\PP^n$ is the blow-up of
 $\PP^n$ along $\B$,   $E$ the exceptional divisor,
 $\pi'=\varphi\circ\pi$.     
By  Zariski's Main Theorem (\cite[\Rmnum{3} Corollary~11.4]{hartshorne-ag} or \cite[\Rmnum{3} \S  9]{mumford})
we have ${\pi'}_{\ast}(\O_{X})=\O_{\sS}$ and by  projection formula 
 \cite[\Rmnum{2} Exercise~5.1]{hartshorne-ag} it follows
${\pi'}_{\ast}({\pi'}^{\ast}(\O_{\sS}(1)))=\O_{\sS}(1)$.
Now, putting $V\subseteq H^0(\PP^n,\O_{\PP^n}(2))$ 
the linear vector space associated to the linear system $\sigma$ defining $\varphi$,
 we have the natural inclusions
\begin{eqnarray*}
 V & \hookrightarrow & H^0(\PP^n,\I_{\B}(2)) 
 \hookrightarrow   H^0(X,\pi^{\ast}(\O_{\PP^n}(2))\otimes\pi^{-1}\I_{\B}\cdot \O_{X}) \\
&\stackrel{\simeq}{\rightarrow} & H^0(X,\pi^{\ast}(\O_{\PP^n}(2))\otimes\O_{X}(-E))
 \stackrel{\simeq}{\rightarrow}  H^0(X,{\pi'}^{\ast}(\O_{\sS}(1))) \\
 & \stackrel{\simeq}{\rightarrow} & H^0(\sS,\O_{\sS}(1)) \stackrel{\simeq}{\rightarrow}  H^0(\PP^{n+1},\O_{\PP^{n+1}}(1)).
\end{eqnarray*}
All these inclusions are isomorphisms, since $\dim(V)=n+2=h^0(\PP^{n+1},\O_{\PP^{n+1}}(1))$.
\end{proof}
\begin{proposition}\label{prop: type 2-1} 
Let $\varphi,\B,\sS$ be as in Lemma \ref{prop: cohomology I2B}.
The following conditions are equivalent:
\begin{enumerate}
 \item\label{item: 1, 2-1} $h^0(\PP^n,\I_{\B}(1))\neq0$; 
 \item\label{item: 2, 2-1} $\B$ is a  quadric of codimension $2$;
 \item\label{item: 3, 2-1} $\B$ is a complete intersection;
 \item\label{item: 4, 2-1} $\varphi$ is of type $(2,1)$.
\end{enumerate}
If one of the previous conditions is satisfied, then 
$\sS$ is a quadric and
 $\rk(\B)=\rk(\sS)-2$.     
\end{proposition}
\begin{figure}[htbp]
\centering
\includegraphics[width=0.55\textwidth]{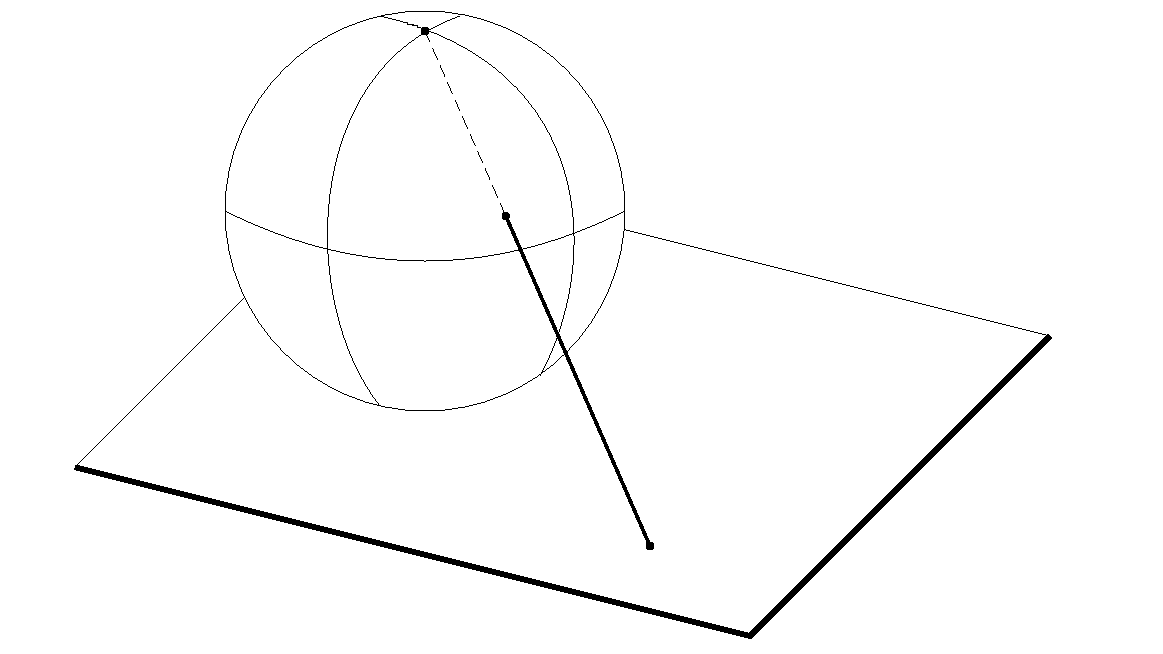}
\caption{Stereographic projection.}
\label{fig: stereographic}   
\end{figure}
\begin{proof} 
$(\ref{item: 1, 2-1}\Rightarrow\ref{item: 2, 2-1}\mbox{ and }\ref{item: 3, 2-1})$.
 If $f\in H^0(\PP^n,\I_{\B}(1))$ is a nonzero linear form, 
then the 
$n+1$
quadrics
 $x_0f,\ldots,x_n f$  
generate a subspace of codimension $1$
 of $H^0(\PP^n,\I_{\B}(2))$,   by Lemma \ref{prop: cohomology I2B}. 
Thus
there exists a 
quadric $F$ such that
$H^0(\PP^n,\I_{\B}(2))=\langle F, x_0f,\ldots,x_nf\rangle$
and $\B=V(F,f)\subset V(f)\subset\PP^{n}$.

$(\ref{item: 2, 2-1}\Rightarrow\ref{item: 1, 2-1}\mbox{ and }\ref{item: 4, 2-1})$. 
$\B$ is necessarily degenerate and,
modulo a change of coordinates
on $\PP^{n}$, 
we can suppose
 $\B=V(x_0^2+\cdots+x_s^2, x_n)$,   with $s\leq n-1$.
Hence 
$$
\varphi([x_0,\ldots,x_n])=
[x_0x_n,\ldots,x_n^2,x_0^2+\cdots+x_s^2]=[y_0,\ldots,y_{n+1}],
$$
$$
\sS=V(y_0^2+\cdots+y_s^2-y_{n}y_{n+1}),\quad \varphi^{-1}([y_0,\ldots,y_{n+1}])=[y_0,\ldots,y_{n}]. 
$$

$(\ref{item: 4, 2-1}\Rightarrow\ref{item: 2, 2-1})$.     We can suppose that
 $\varphi^{-1}$ is the projection from the point
 $[0,\ldots,0,1]$.     Thus, 
if $\varphi=(F_0,\ldots,F_{n+1})$,   then $F_0,\ldots,F_n$ have
a common factor
 $f$ and  $\B=V(F_{n+1},f)$.

$(\ref{item: 3, 2-1}\Rightarrow\ref{item: 1, 2-1})$.  If $h^0(\PP^n,\I_{\B}(1))=0$ and $\B$ is 
a complete intersection,  then every minimal system of generators of the ideal of
 $\B$ consists of forms of degree $2$, but then  
$h^0(\PP^n,\I_{\B}(2))=n-\dim(\B)<n+2$,   
absurd.    
\end{proof}

\section{Transformations of \texorpdfstring{$\PP^3$}{P3} and \texorpdfstring{$\PP^4$}{P4}}\label{sec: nonspecial case}
In this section 
we shall keep the following: 
\begin{notation}\label{notation: nonspecial case}
 Let $\varphi:\PP^n\dashrightarrow\overline{\varphi(\PP^n)}=\Q\subset\PP^{n+1}$ be
a quadratic birational transformation into 
an irreducible (hence normal, 
by \cite[\Rmnum{1} Exercise~5.12, \Rmnum{2} Exercise~6.5]{hartshorne-ag}) quadric 
 hypersurface $\Q$ and
moreover suppose that  its base locus $\emptyset\neq\B\subset\PP^n$ is reduced. 
\end{notation}
\begin{lemma}\label{prop: degenerate component}
Let  $X\subseteq\B$ be a degenerate irreducible component  of $\B$.
\begin{enumerate}
\item If $\mathrm{codim}_{\PP^n}(X)=2$   then $\deg(X)\leq2$ and, 
if $\deg(X)=2$ then $X=\B$ and $\B$ is a quadric.
\item If $\mathrm{codim}_{\PP^n}(X)=3$   then $\deg(X)\leq4$ and, 
if $\deg(X)=4$ then  we have 
 $h^0(\PP^n,\I_{X}(2))=h^0(\PP^n,\I_{\B}(2))+1$.
\end{enumerate}
\end{lemma}
\begin{proof}
We can choose coordinates $x_0,\ldots,x_n$ on $\PP^n$ such that
$X\subset V(x_n)\subset\PP^n$ 
and we consider the 
restriction 
map
$u:H^0(\PP^n,\I_{\B}(2)) \rightarrow H^0(\PP^{n-1},\I_{X}(2))$,
defined  by
$ 
u(F(x_0,\ldots,x_{n}))=F(x_0,\ldots,x_{n-1},0)$.  
 Note  that if $F\in H^0(\PP^n,\I_{\B}(2))$,  then
$
F\in \langle u(F) \rangle \oplus \langle x_0x_n,\ldots,x_{n-1}x_{n},x_n^2 \rangle
$
and, in particular, 
 $H^0(\PP^n,\I_{\B}(2))\subseteq 
\mathrm{Im}(u)\oplus \langle x_0x_n,\ldots,x_{n-1}x_{n},x_n^2 \rangle$.
Thus $\dim(\mathrm{Im}(u))\geq 1$ and, 
if $\dim(\mathrm{Im}(u))=1$ then   $\B$ is a quadric hypersurface in  $\PP^{n-1}$.     
Now suppose $\mathrm{codim}_{\PP^n}(X)=2$.     
From the above, there exists $\bar{F}\in H^0(\PP^{n-1},\I_{X}(2))$ and
$X$ has to be an irreducible component of $V(\bar{F})$.     
It follows that 
$\deg(X)\leq \deg(\bar{F})=2$ and, if $\deg(X)=2$ then  
$X=V(\bar{F})$ and 
$h^0(\PP^{n-1},\I_{X}(2))=\dim(\mathrm{Im}(u))= 1$.
 Suppose $\mathrm{codim}_{\PP^n}(X)=3$. 
 From the above, there exist  
 $\bar{F},\bar{F'}\in H^0(\PP^{n-1},\I_{X}(2))$ which are linearly independent
 and $X$ 
has to be contained in the complete intersection 
$V(\bar{F},\bar{F'})$.     
It follows  that 
 $\deg(X)\leq4$ and, if $\deg(X)=4$ then 
we have
 $h^0(\PP^{n-1},\I_{X}(2))= \dim(\mathrm{Im}(u))= 2$.
\end{proof}
\begin{proposition}\label{prop: P3}
If $n=3$,  
then either
\begin{enumerate}[(i)]
 \item\label{item: 1, P3} $\varphi$ is of type $(2,1)$,   or
 \item\label{item: 2, P3} $\varphi$ is of type $(2,2)$,   $\rk(\Q)=4$
 and $\B$ is the union of
a line  $r$ with two points $p_1,p_2$   such that
 $\langle p_1, p_2\rangle\cap r=\emptyset$. 
\end{enumerate}
\end{proposition}
\begin{figure}[htbp]
\begin{center}
\includegraphics[width=0.55\textwidth]{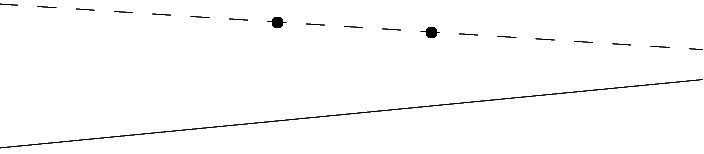}
\end{center}
\caption{Base locus when $n=3$.}
\label{fig: fig_base_loci_n=3}   
\end{figure}
\begin{proof}
Since $\mathrm{codim}_{\PP^3}(\B)\geq2$,   $\B$ 
has the following decomposition into irreducible components:
$$
\B=\bigcup_{j}C_j\cup\bigcup_{i}p_i ,
$$ 
where $C_j$ and $p_i$ are respectively curves and points, with $\deg(C_j)\leq3$,
by  B\'ezout's Theorem.
If one of $C_j$ is nondegenerate, 
then $\deg(C_j)=3$ and
 $C_j$ is the twisted cubic curve, by \cite[Proposition~18.9]{harris-firstcourse}. 
This would produce the absurd result that 
$5=h^0(\PP^3,\I_{\B,\PP^3}(2))\leq h^0(\PP^3,\I_{C_j,\PP^3}(2))=3$ and hence 
we have that every $C_j$ is degenerate.
By Lemma \ref{prop: degenerate component}, 
using the fact that a line
imposes $3$ conditions to the quadrics, 
it follows that 
 $\B$ is a (irreducible or not) conic
 or as asserted in  (\ref{item: 2, P3}).   
In the latter case, 
modulo a change of coordinates,
we can suppose
 $r=V(x_2,x_3)$,   $p_1=[0,0,1,0]$,   $p_2=[0,0,0,1]$ and  we get 
\begin{equation}
 \B=V(x_0x_2,\ x_0x_3,\ x_1x_2,\ x_1x_3,\ x_2x_3),\quad
 \Q= V(y_0y_3-y_1y_2).
\end{equation}
\end{proof}
\begin{lemma}\label{prop: isolated point}
 If $n\geq3$ and $\B$ has an isolated point, then  $\rk(\Q)\leq4$.
\end{lemma}
\begin{proof}
Let $p$ be an isolated (reduced) point of $\B$ and 
consider  the diagram
$$
\xymatrix{
\PP^{n-1}\simeq E \ar@{^{(}->}[r] \ar@{-->}@/^2.8pc/[rrd]^{\bar{\varphi}} &\Bl_{p}(\PP^{n}) \ar[d]^{\pi} \ar@{-->}[rd]\\
&\PP^{n} \ar@{-->}[r]^{\varphi} & \Q\subset\PP^{n+1}
}
$$
By the hypothesis on $p$, 
  $\bar{\varphi}$ is a linear morphism 
and so the quadric  $\Q=\Q^{n}$ contains a $\PP^{n-1}$.
This, since $n\geq3$,   implies $\rk(\Q)\leq 4$.
\end{proof}
\begin{proposition}\label{prop: P4}
If $n=4$ and
 $\rk(\Q)\geq 5$,   then either
\begin{enumerate}[(i)]
 \item\label{item: 1, P4} $\varphi$ is of type $(2,1)$,   or
 \item\label{item: 2, P4} $\varphi$ is of type $(2,2)$,   
$\Q$ is smooth and $\B$ 
is one of the following (see Figure \ref{fig: fig_base_loci_n=4}):
\begin{enumerate} 
 \item the rational normal quartic curve, 
 \item the union of the twisted cubic curve  in a hyperplane  $H\subset\PP^4$ 
       with  a line not contained  in $H$ and intersecting the twisted curve,
 \item the union of an irreducible conic with two skew lines that intersect it,
 \item the union of three skew lines with another line that intersects them.
\end{enumerate}
\end{enumerate}
\end{proposition}
\begin{figure}[htbp]
\centering
\includegraphics[width=\textwidth]{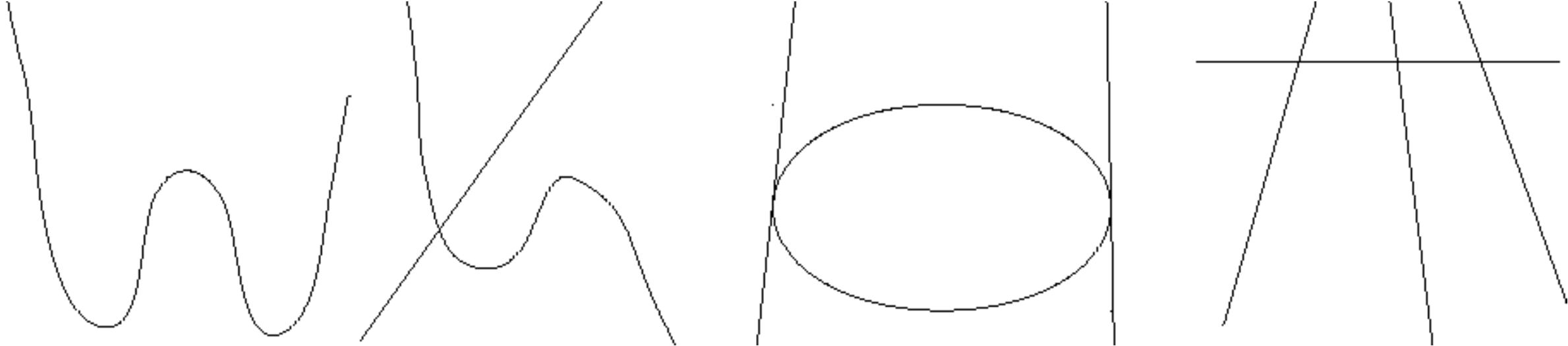}
\caption{Base loci when $n=4$.}
\label{fig: fig_base_loci_n=4}   
\end{figure}
\begin{proof}
By Lemma \ref{prop: isolated point} and since $\mathrm{codim}_{\PP^4}(\B)\geq2$,   $\B$
has the following decomposition into irreducible components:
$$
\B=\bigcup_{i}S_i\cup\bigcup_{j}C_j ,
$$ 
where $S_i$ and $C_j$ are respectively surfaces and curves. 
Let $S$ be one of $S_i$ and let $C$ be one of $C_j$.
We discuss all possible cases.
\begin{case}[$\deg(S)\geq 4$] This case is impossible 
by    B\'ezout's Theorem.
\end{case}
 \begin{case}[$\deg(S)=3$,   $S$ nondegenerate] 
Cutting $S$ with a general hyperplane $\PP^3\subset\PP^4$, we obtain 
the twisted cubic curve $\Gamma\subset\PP^3$.
Hence, from the exact sequence 
$0\rightarrow \I_{S,\PP^{4}}(-1)\rightarrow
 \I_{S,\PP^{4}}\rightarrow \I_{\Gamma, \PP^3} \rightarrow 0$, 
 we deduce the contradiction
$
6=h^0(\PP^{4},\I_{\B,\PP^{4}}(2))\leq h^0(\PP^{4},\I_{S,\PP^{4}}(2))\leq h^0(\PP^3,\I_{\Gamma, \PP^3}(2) )=3
$.
\end{case}
 \begin{case}[$\deg(S)\geq 2$,   $S$ degenerate] By 
Lemma \ref{prop: degenerate component} we have $\deg(S)=2$ and $S=\B$ is a quadric.
\end{case}
\begin{case}[$\deg(C)\geq 5$,   $C$ nondegenerate] Take a general 
hyperplane
 $\PP^3\subset\PP^4$
and put $\Lambda=\PP^3\cap C$. $\Lambda$ is a set of $\lambda\geq5$ points of $\PP^3$ 
in general position and therefore, by Lemma \ref{prop: castelnuovo argument}, 
we get the contradiction
$$
6=h^0(\PP^{4},\I_{\B}(2))\leq h^0(\PP^4,\I_{C}(2))\leq
h^0(\PP^3,\I_{\Lambda}(2))\leq
\left\{\begin{array}{ll}  10-\lambda\leq5, 
& \mbox{ if } \lambda\leq7 \\  10-7=3, & \mbox{ if } \lambda\geq7 . \end{array}   \right.
$$
\end{case}
 \begin{case}[$\deg(C)\geq5$,   $C$ degenerate]  
This is impossible by  Lemma \ref{prop: degenerate component}.
\end{case}
 \begin{case}[$\deg(C)=4$,   $C$ nondegenerate] By 
\cite[Proposition~18.9]{harris-firstcourse}, $C$ is the rational normal quartic curve  
and one of its parameterizations 
is $[s,t]\in\PP^1\mapsto [s^4,s^3t,s^2t^2,st^3,t^4]\in\PP^{4}$.
We have 
$h^0(\PP^4,\I_{C,\PP^4}(2))=6$ and hence $C=\B$ and    
\begin{eqnarray}
\B&=& V(x_2^2-x_1x_3,\ x_2x_3-x_1x_4,\ x_0x_4-x_1x_3,\ x_3^2-x_2x_4,\ x_0x_2-x_1^2,\ x_0x_3-x_1x_2), \\
\Q&=& V(y_0y_2-y_1y_5+y_3y_4).
\end{eqnarray}
\end{case}
\begin{case}[$\deg(C)=4$,   $C$ degenerate] This case is impossible 
by 
 Lemma \ref{prop: degenerate component} and Proposition \ref{prop: type 2-1}.
\end{case}
\begin{case}[$\deg(C)=3$,   $\langle C \rangle=\PP^3$] 
Modulo a change of coordinates, 
$C$ is the twisted cubic curve parameterized by
$[s,t]\in\PP^1\mapsto [s^3,s^2t,st^2,t^3,0]\in V(x_4)$  
and, 
by the reduction obtained, since
$ h^0(\PP^4,\I_{C}(2))=h^0(\PP^4,\I_{\B}(2))+2$,
it follows
$\B=C\cup r$, 
where $r$ is a line that intersects $C$ in a single point transversely.
We can choose the intersection point to be $p=[1,0,0,0,0]$
and then we have
$r=\{ [s+q_0t,q_1t,q_2t,q_3t,q_4t]: [s,t]\in\PP^1\}$,
for some  $q=[q_0,q_1,q_2,q_3,q_4]\in\PP^4$.     
By Proposition \ref{prop: type 2-1} we have $q_4\neq0$ and 
only for simplicity of notation we take
$q=[0,0,0,0,1]$,   hence $r=V(x_1,x_2,x_3)$.  So we obtain 
\begin{eqnarray}
\B&=&V( x_1^2-x_0x_2,\  x_3x_4,\  x_0x_3-x_1x_2,\  x_2x_4,\   x_1x_3-x_2^2,\  x_1x_4),\\
\Q&=& V(-y_4y_5+y_2y_3+y_0y_1).
\end{eqnarray}
\end{case}
\begin{case}[$\deg(C)=3$,   $\langle C \rangle=\PP^2$] It is impossible 
because otherwise 
$\B$ would contain
the entire plane spanned by $C$.
\end{case}
\begin{case}[$\deg(S)=1$] Since 
planes, conics and lines impose to the quadrics 
 respectively  $6$, $5$ and $3$ conditions, 
we have two subcases:
\begin{subcase}[$\B=S\cup\Gamma$,   $\Gamma$ conic, $\#(\Gamma\cap S)=2$] This
 is impossible
because otherwise there exists a line cutting $\B$ in exactly $3$ points.  
\end{subcase}
\begin{subcase}[$\B=S\sqcup l$,   $l$ line] This contradicts 
the birationality of $\varphi$. 
\end{subcase}
\end{case}
\begin{case}[$\deg(C)=2$,   $C$ irreducible] Let $C', l, l'$ be 
respectively 
eventual conic and eventual lines contained in $\B$.
We discuss the subcases:  
 \begin{subcase}[$\B=C\cup C'$,    $\#(C\cap C')=1$] In this case, 
we would have that  $\rk(\Q)=4$, against the  hypothesis.
 \end{subcase}
\begin{subcase}[$\B=C\cup C'\cup l$,    $\#(C\cap C')=2$,   $\#(C\cap l)=\#(C'\cap l)=1$] 
 $\B \supseteq C\cup C'\cup l$   implies that $\B$ is a quadric of codimension $2$.
\end{subcase}
\begin{subcase}[$\B=C\cup l\cup l'$,   $\#(l\cap l')=0$,   $\#(C\cap l)=\#(C\cap l')=1$] 
This case is really possible and, 
modulo a change of coordinates, 
we have
\begin{eqnarray}
 \B&=&V(x_2^2+x_0x_1,\ x_3x_4,\ x_1x_3,\ x_2x_4,\ x_2x_3,\ x_0x_4),\\
 \Q&=&V(-y_2y_5-y_3y_4+y_0y_1).
\end{eqnarray}
\end{subcase}
\end{case}
\begin{case}[$\deg(C)=2$,   $C$ reducible] By
  the reduction obtained,
$\B$ contains  two other skew lines, 
each of which intersects $C$ at a single point. 
This case is really possible and, 
modulo a change of coordinates, 
we have
\begin{eqnarray}
 \B&=&V(x_1x_2,\ x_3x_4,\ x_0x_3,\ x_2x_4,\ x_2x_3,\ x_0x_4-x_1x_4),\\
\Q&=&V(-y_4y_5+y_2y_3-y_0y_1).
\end{eqnarray}
\end{case}
\end{proof}
Example \ref{example: extra} suggests that a possible generalization
of  Proposition \ref{prop: P4}
 to the case where $\B$ in nonreduced and $rk(\Q)< 5$ 
may not be trivial.
\begin{example}\label{example: extra}
The rational map $\varphi:\PP^4\dashrightarrow\PP^5$ defined by 
$$\varphi([x_0,x_1,x_2,x_3,x_4])=
[x_0^2, -x_0x_1, -x_0x_2, x_1^2-x_0x_3, 2x_1x_2-x_0x_4, x_2^2] , $$
is birational 
into  its image, which is the quadric of rank $3$,
$\Q=V(y_0y_5-y_2^2)$.
If  $\pi:\PP^5\dashrightarrow \PP^4$ is the projection 
from the point  $[0,0,0,0,0,1]$, then the composition
 $\pi\circ\varphi:\PP^4\dashrightarrow \PP^4$ is an involution.
The base locus $\B$ of $\varphi$ is everywhere nonreduced, 
$(\B)_{\mathrm{red}}=V(x_0,x_1,x_2)$ and 
$P_{\B}(t)=4t+1$.
\end{example}

\chapter{On special quadratic birational transformations of a projective space into a hypersurface}\label{sec: transformations into a hypersurface}
   Consider a special birational transformation
   $\varphi:\PP^n\dashrightarrow\sS\subset\PP^{n+1}$ of type $(2,d)$
   from a complex projective space
   into a nonlinear and sufficiently regular hypersurface  $\sS$.
 The blow-up
 $\pi:X=\Bl_{\B}(\PP^n)\rightarrow\PP^n$ of $\PP^n$ along the base locus 
 $\B$
resolves the indeterminacies of the transformation
 $\varphi$. 
So,
comparing the two ways in which it is possible to write
 the canonical class
$K_X$, with respect to $\pi$ and $\pi':=\varphi\circ\pi$, 
we get 
a formula expressing  the dimension of $\B$ as a function of $n$, $d$ and $\deg(\sS)$,
see  Proposition \ref{prop: dimension formula}.
The primary advantage in dealing with the case of quadratic transformations
is that $\B$ is 
a $QEL$-variety (see Definition \ref{def: LQEL QEL CC}). 
Therefore it is possible to apply the main results  of the theory of  
 $QEL$-varieties; in particular,
 the Divisibility Theorem (Theorem \ref{prop: divisibility theorem}), 
together with the
formula on the dimension 
of $\B$, 
drastically reduces the set of quadruples 
 $(\dim(\B),n,d,\deg(\sS))$ for which  such a
 $\varphi$ exists.

    The $\varphi$ of type $(2,1)$   
     have been described in  Proposition \ref{prop: type 2-1} as
    the only transformations whose base locus is a quadric of codimension $2$; 
    in particular,  modulo projective transformations,  there is only one example.
With regard to the
 $\varphi$'s of type $(2,2)$ into a quadric hypersurface $\Q$,
it is known that a special Cremona 
 transformation 
$\PP^{n+1}\dashrightarrow\PP^{n+1}$ of type $(2,2)$ 
has as base locus a Severi variety (see 
Theorem \ref{prop: classification quadro-quadric special cremona})  and moreover,
modulo projective transformations, 
there exist only four Severi varieties
 (see Theorem \ref{prop: classification severi varieties}). 
Now, if we restrict these transformations 
to a general hyperplane $\PP^{n}\subset\PP^{n+1}$,
we clearly obtain 
special transformations 
$\varphi:\PP^n\dashrightarrow\Q\subset\PP^{n+1}$ of type $(2,2)$ 
(see Example \ref{example: d=2 Delta=2}).  
In Theorem \ref{prop: classification of type 2-2 into quadric} 
we  prove 
 that all examples of $\varphi$ of type $(2,2)$ into a quadric $\Q$ arise in this way;
in particular, their  base loci
are hyperplane sections of Severi varieties. 
Regarding special transformations of type $(2,2)$ into a cubic and quartic hypersurface
 we are able  to determine some  invariants of the base locus  as
the Hilbert polynomial and the Hilbert scheme of lines passing through a point
(Propositions \ref{prop: invariants d=2 Delta=3} and \ref{prop: invariants d=2 Delta=4}).

Another approach to  the study  
of all $\varphi$ of type $(2,d)$ is their classification according to
 the dimension of the base locus. 
In Table
 \ref{tab: cases with dimension leq 3}  
(which is constructed 
via Propositions \ref{prop: 2-fold in P6} and \ref{prop: 3-fold in P8}), 
we provide a list of all possible base loci when the dimension is at most $3$,
although in one case we do not know if it really exists.
As a consequence, in Corollary \ref{prop: classification type 2-3 into cubic}, we obtain
that a special transformation  
 $\varphi$ of type $(2,3)$ into a cubic hypersurface
 $\sS$ has as base locus  
the blow-up $\Bl_{\{p_1,\ldots,p_5\}}(Q)$ of $5$ points $p_1,\ldots,p_5$ 
in a smooth quadric $Q\subset\PP^4$, embedded in   
$\PP^8$ by the linear system $|2 H_{\PP^4}|_{Q}-p_1-\cdots-p_5|$ 
(see also Example \ref{example: d=3 Delta=3}).
 
 Throughout the chapter, unless otherwise specified, we shall keep the following:
\begin{notation}\label{notation: factorial hypersurface}
 Let $n\geq3$ and 
$\varphi:\PP^{n}\dashrightarrow\overline{\varphi(\PP^n)}=\sS\subset\PP^{n+1}$ be
a special birational transformation  of type $(2,d)$, with $d\geq 2$ and 
with $\sS$ a factorial hypersurface of degree $\Delta\geq 2$
(in particular, when $\Delta=2$, it is enough to require 
 $\rk(\sS)\geq 5$).  
Observe that
 $\Pic(\sS)=\ZZ\langle\O_{\sS}(1)\rangle$ (see \cite[\Rmnum{4} Corollary~3.2]{hartshorne-ample})
and $\omega_{\reg(\sS)}\simeq \O_{\reg(\sS)}(\Delta-n-2)$.
Moreover, denote by  $\B\subset\PP^n$ and $\B'\subset\sS\subset\PP^{n+1}$ 
respectively the base locus  of $\varphi$ and $\varphi^{-1}$
and assume\footnote{See Example  \ref{example: B2=singSred-3fold}
 for an explicit example of  special quadratic birational transformation
for which (\ref{eq: hypothesis on sing locus hypersurface}) is not satisfied.} 
\begin{equation}\label{eq: hypothesis on sing locus hypersurface}
(\B')_{\mathrm{red}}\neq (\sing(\sS))_{\mathrm{red}}.
\end{equation}
Put $r=\dim(\B)$, $r'=\dim(\B')$, $\delta=\delta(\B)$ the secant defect, 
$\lambda=\deg(\B)$, $g=g(\B)$ the sectional genus, $P_{\B}(t)$ the Hilbert polynomial,
$i(\B)$ and $c(\B)$ respectively index and coindex (when $\B$ is a Fano variety).
\end{notation}

\section{Properties of the base locus}\label{sec: properties base locus} 
\begin{proposition}\label{prop: dimension formula}  \hspace{1pt}
\begin{enumerate} 
\item $\B$ is a
$QEL$-variety of dimension and  type given by
\begin{displaymath}
r=\frac{d\,n-\Delta-3\,d+3}{2\,d-1},\quad
\delta =\frac{n-2\,\Delta-2\,d+4}{2\,d-1}.
\end{displaymath}
 \item $\B'$ is irreducible, generically reduced, of dimension
\begin{displaymath}
r'=\frac{2\left( d\,n-n+\Delta-d-1\right) }{2\,d-1}.
\end{displaymath}
\end{enumerate}
\end{proposition}
\begin{proof} See also \cite[Propositions~2.1 and 2.3]{ein-shepherdbarron}.
Consider the  diagram (\ref{eq: diagram resolving map}), 
where $\pi:X=\Bl_{\B}(\PP^n)\rightarrow\PP^n$
  and $\pi'=\varphi\circ\pi$.     
$X$   
can be identified
 with  the graphic
 $\Graph(\varphi)\subset\PP^n\times\sS$ and the
maps $\pi$,   $\pi'$ 
can be identified
 with the projections onto the factors.
It follows that $(\B)_{\mathrm{red}}$ (resp. $(\B')_{\mathrm{red}}$) is the set 
of the points $x$ such that the fiber $\pi^{-1}(x)$ (resp. ${\pi'}^{-1}(x)$) has 
positive dimension. Denote by $E$ the exceptional divisor of $\pi$, 
  $E'={\pi'}^{-1}(\B')$, 
$H\in|\pi^{\ast}(\O_{\PP^{n}}(1))|$,   
$H'\in|{\pi'}^{\ast}(\O_{\sS}(1))|$ 
and  note that, by the factoriality of 
 $\sS$ and by the proof of
\cite[Proposition~1.3]{ein-shepherdbarron}, it follows
that $E'$ is an irreducible divisor, in particular $\B'$ is irreducible.
Moreover we have the relations
\begin{equation}\label{eq: HH'}
H\sim d\,H'-E',\quad H'\sim 2\,H-E,
\end{equation}
from which we get
\begin{equation}\label{eq: EE'}
E\sim(2\,d-1)\,H'-2\,E',\quad E'\sim (2\,d-1)\,H-d\,E,
\end{equation}
and in particular by (\ref{eq: EE'}) it follows  $E'=(E')_{\mathrm{red}}$.
Put 
\begin{equation}\label{eq: positions}
U=\reg\left(\sS\right)\setminus\sing\left(\left(\B'\right)_{\mathrm{red}}\right),\quad V={\pi'}^{-1}(U),\quad Z=U \cap (\B')_{\mathrm{red}}.
\end{equation}
Observe 
 that, since $X$ is smooth and we have assumed 
  (\ref{eq: hypothesis on sing locus hypersurface}),
we have
$Z\neq\emptyset$. Thus, by \cite[Theorem~1.1]{ein-shepherdbarron}, 
 ${\pi'}|_{V}:V\rightarrow U$ coincides
 with the blow-up  of $U$ along $Z$;
in particular
$$
\Pic(X)=\ZZ\langle H\rangle \oplus\ZZ\langle E\rangle=\ZZ\langle H'\rangle\oplus\ZZ\langle E'\rangle=\Pic(V).
$$
Now, by \cite[\Rmnum{2} Exercise~8.5]{hartshorne-ag} and (\ref{eq: HH'}) and (\ref{eq: EE'}), we have
\begin{eqnarray}\label{eq: K_X}
K_X&\sim&\pi^{\ast}(K_{\PP^n}) + (\mathrm{codim}_{\PP^n}(\B)-1)\,E  \nonumber \\
 & \sim & (-n-1)\,H+(n-r-1)\,E  \nonumber  \\
 & \sim & (-2\,d\,r+r+d\,n-n-3\,d+1)\,H'+(2\,r-n+3)\,E' 
\end{eqnarray}
and also
\begin{eqnarray}\label{eq: K_V}
 K_X|V \sim K_V&\sim&{\pi'}^{\ast}(K_{\reg(\sS)}|U) + (\mathrm{codim}_{U}(Z)-1)\,E' \nonumber\\
   & \sim & (\deg(\sS)-(n+1)-1)\,H' + (\mathrm{codim}_{\sS}(\B')-1)\,E' \nonumber \\
   & \sim & (\Delta-n-2)\,H'+(n-r'-1)\,E'. 
\end{eqnarray}
By the comparison of
 (\ref{eq: K_X}) and (\ref{eq: K_V})
we obtain the expressions of
 $r=\dim(\B)$ and $r'=\dim(\B')$.
Moreover, the proof of \cite[Proposition~2.3(a),(b)]{ein-shepherdbarron} 
adapts to our case,
  producing that the secant  variety  $\Sec(\B)$ is a hypersurface of degree
 $2d-1$ in $\PP^n$ and  $\B$ is a $QEL$-variety of type $\delta=n-r'-2$.
Finally,
 we have  $\B'\cap U=(\B')_{\mathrm{red}}\cap U$ 
by the argument in \cite[\S  2.2]{ein-shepherdbarron}. 
\end{proof}
\begin{remark}\label{remark: dimension formula without hypothesis}
Proceeding as above and
interchanging the rules in (\ref{eq: positions}) by
\begin{equation}
U=\reg\left(\sS\right)\setminus\left(\B'\right)_{\mathrm{red}},\quad V={\pi'}^{-1}(U),\quad Z=U \cap (\B')_{\mathrm{red}}=\emptyset ,
\end{equation}
we deduce that $\pi'|_{V}:V\rightarrow U$ is an isomorphism  and in 
particular $\Pic(V)={\pi'}^{\ast}(\Pic(U))=\ZZ\langle H' \rangle$
and $K_V={\pi'}^{\ast}(K_U)=(\Delta-n-2) H'$.
Now, by (\ref{eq: K_X}),
$K_X|_V = K_{V} \sim (-2\,d\,r+r+d\,n-n-3\,d+1)H'$,
and hence, also without to assume 
 (\ref{eq: hypothesis on sing locus hypersurface}), 
we obtain that
$r=\left(dn-\Delta-3d+3\right)/\left(2d-1\right)$.
\end{remark}
Lemma \ref{prop: cohomology twisted ideal} is slightly stronger than
what is obtained by applying
directly the main result in \cite{bertram-ein-lazarsfeld}. However
it is essential to study $\B$  in the case $\delta=0$.  
Note that the assumptions on $\sS$ are not necessary.
\begin{lemma}[\cite{mella-russo-baselocusleq3}]\label{prop: cohomology twisted ideal}
 For $i>0$ and $t\geq n-2r-1$ we have
$
H^i(\PP^n, \I_{\B,\PP^n}(t))=0
$.
\end{lemma}
\begin{proof} 
The proof is located in
\cite[page~6]{mella-russo-baselocusleq3}, 
but we report it for the reader's convenience.
We use the notation of the proof of Proposition 
 \ref{prop: dimension formula}. For $i>0$ and $t\geq n-2r-1$ we have
 \begin{eqnarray*}
  H^i(\PP^n,\I_{\B,\PP^n}(t)) & = & H^i(\PP^n, \O_{\PP^n}(t)\otimes\I_{\B,\PP^n}) \\
 \mbox{(\cite[page~592]{bertram-ein-lazarsfeld})} &= & H^i(X, tH-E) \\
                                 &=& H^i(X, K_X+(tH-E-K_X)) \\
                                 &=& H^i(X, K_X+(t+n+1)H+(r-n)E) \\
                                 &=& H^i(X, K_X+(n-r)(2H-E)+(t-n+2r+1) H) \\
 \mbox{(\cite[page~20]{debarre})} &=& H^i(X, K_X+ \underbrace{
\underbrace{(n-r)H'}_{\mathrm{nef\ and\ big}}+\underbrace{(t-(n-2r-1))H}_{\mathrm{nef}}}_{\mathrm{nef\ and\ big}}) \\
 \mbox{(Kodaira Vanishing Theorem)}                 &=& 0.
 \end{eqnarray*}
\end{proof}
\begin{proposition}\label{prop: cohomology properties} The following statements hold:  
\begin{enumerate}
 \item\label{part: B is linearly normal} $\B$ is nondegenerate 
and projectively normal.
 \item\label{part: B is fano} If $\delta>0$, then either $\B$ is a Fano variety 
 of the first species  
of  index $i(\B)=(r+\delta)/2$,  or it is a hyperplane section  of the 
Segre variety $\PP^2\times\PP^2\subset\PP^8$.
 \item\label{part: hilbert polynomial}
If $\delta>0$, 
putting $i=i(\B)$ and $P(t)=P_{\B}(t)$,   we have
$$
\begin{array}{c}
\begin{array}{ccc}
 P(0)=1,& P(1)=n+1, & P(2)=\left({n}^{2}+n-2\right)/2,\\
\end{array}\\
P(-1)=P(-2)=\cdots=P(-i+1)=0,\\
\forall\,t\quad P(t)=(-1)^r P(-t-i).
\end{array}
$$
In particular, if $c(\B)=r+1-i\leq5$, it remains determined $P(t)$.     
\item\label{part: cohomology} Hypothesis as in part  
 \ref{part: hilbert polynomial}. We have 
$$
 h^j(\B,\O_{\B}(t))=\left\{
\begin{array}{cccc}   0 &\mbox{if}&j=0,&t<0\\
                    P(t) &\mbox{if}&j=0,&t\geq0\\     
                     0 &\mbox{if}&0<j<r,&t\in\ZZ\\
                     (-1)^r P(t)&\mbox{if}&j=r,&t<0\\
                     0 &\mbox{if}&j=r,& t\geq0
\end{array} \right.
$$
\end{enumerate}
\end{proposition}
\begin{proof}
(\ref{part: B is linearly normal}) $\B$ is nondegenerate 
by Proposition \ref{prop: type 2-1}.
$\B$ is projectively normal if and 
only if  $h^1(\PP^n,\I_{\B,\PP^n}(k))=0$ for every $k\geq1$,    and this,
by Lemma \ref{prop: cohomology twisted ideal},
is true whenever 
$1\geq n-2\,r-1$, 
i.e. whenever
 $\delta\geq0$.  
Note that for   
 $\delta>0$ the thesis follows also from  \cite[Corollary~2]{bertram-ein-lazarsfeld}.

(\ref{part: B is fano}) We know that
 $\B$ is a 
$QEL$-variety of type $\delta>0$.     If $\delta\geq3$ the thesis is 
contained in Theorem \ref{prop: main qel1};  
for $0<\delta\leq2$ we apply  
 Theorem \ref{prop: classification CC-varieties}.
Thus we have either that
$\B$ is a Fano variety with
 $\Pic(\B)\simeq\ZZ\langle\O_{\B}(1)\rangle$,   or
 $\B$ is one of the following:
\begin{enumerate}
 \item\label{item: veronese} $\nu_2(\PP^2)\subset\PP^5$,
 \item\label{item: rational normal scroll} a rational normal scroll $\PP_{\PP^1}(\O(1)^{\oplus r-1}\oplus\O(2))\subset\PP^{2r}$,
 \item a hyperplane section  of $\PP^2\times\PP^2\subset\PP^8$,
 \item\label{item: segre P2xP2} $\PP^2\times\PP^2\subset\PP^8$,
 \item\label{item: P2xP3inP11} $\PP^2\times\PP^3\subset\PP^{11}$.
\end{enumerate}
This follows only by imposing  that the pair $(r,n)$ corresponds to
that of
 a variety listed in 
 Theorem \ref{prop: classification CC-varieties}.
Of course  cases \ref{item: veronese}, \ref{item: segre P2xP2} and \ref{item: P2xP3inP11} are impossible,  
because $h^0(\PP^{n},\I_{\B}(2))=n+2$ but 
$h^0(\PP^{5},\I_{\nu_2(\PP^2)}(2))=6$,
$h^0(\PP^{8},\I_{\PP^2\times\PP^2}(2))=9$ and
 $h^0(\PP^{11},\I_{\PP^2\times\PP^3}(2))=18$.  
Case \ref{item: rational normal scroll} is excluded  because 
such a scroll satisfy\footnote{Actually 
one can easily see that if the base locus $\B\subset\PP^n$ 
is a rational normal scroll, then 
it holds $\delta=0$, $r=1$, $n=4$, $d=2$ and $\Delta=2$.}
 $\delta=1$ 
(see for example \cite[Proposition~1.5.3]{russo-specialvarieties}) 
and then we should have $n=2r+1\neq 2r$.

(\ref{part: hilbert polynomial}) By Lemma \ref{prop: cohomology twisted ideal} 
we have
$h^j(\PP^n,\I_{\B,\PP^n}(k))=0$, for every $j>0$, $k\geq 1-\delta$,
and by the structural sequence we get
$h^j(\B,\O_{\B}(k))=0$, for every $j>0$, $k\geq 1-\delta$.
Hence, by $\delta>0$, it follows
\begin{eqnarray*}
P(0)&=&h^0(\B,\O_{\B})=1, \\
 P(1)&=&h^0(\B,\O_{\B}(1))=h^0(\PP^n,\O_{\PP^n}(1))=n+1, \\
P(2)&=&h^0(\PP^n,\O_{\PP^n}(2))-h^0(\PP^n,\I_{\B,\PP^n}(2))=\left({n}^{2}+n-2\right)/2. 
\end{eqnarray*}
Moreover, if $t<0$, by  Kodaira Vanishing Theorem  
 and  Serre Duality,
$$
P(t)=(-1)^r h^r(\B,\O_{\B}(t))=(-1)^r h^0(\B,\O_{\B}(-t-i)),
$$
and hence for $-i<t<0$ we have $P(t)=0$ 
and for $t\leq -i$ (hence   for every $t$) we have $P(t)=(-1)^r P(-t-i)$.
In particular
$$
P(-i)=(-1)^r,\quad P(-i-1)=(-1)^r (n+1),\quad P(-i-2)=(-1)^r \left({n}^{2}+n-2\right)/2,
$$
 and we have at least $i+5$ independent conditions for $P(t)$.

(\ref{part: cohomology}) As the part (\ref{part: hilbert polynomial}).
\end{proof}

\section{Numerical restrictions}\label{sec: numerical constraints}
\begin{proposition}\label{prop: numerical constraint with Delta=2}
If $\Delta=2$, 
 then exactly one of the following holds:
\begin{enumerate}[(i)]
 \item\label{part: congr_i} $d$ is even and $\delta\in\{0,1,3,7\}$; 
 \item\label{part: congr_ii} $\delta$ is even, $d$ is odd and
  $d=u2^{\delta/2-1}+1$, for some $u\geq1$; 
 \item\label{part: congr_iii} $\delta=s2^e-1$,   $d=t2^{s2^{e-1}-e-1}+1$,   for some
 $e,s,t\geq1$,  $s$ odd.
\end{enumerate}
\end{proposition}
\begin{proof}
By Theorem \ref{prop: divisibility theorem},
for $\delta\geq3$,   we have
$r\equiv \delta \ \mathrm{mod}\ 2^{\left\lfloor(\delta-1)/2\right\rfloor}$,  i.e.
\begin{equation}\label{eq: divisibility theorem}
\Delta-2+(d-1)(\delta+1)\equiv 0 \ \mathrm{mod}\ 2^{\left\lfloor(\delta-1)/2\right\rfloor}.
\end{equation}
If $d$ is even, 
then (\ref{eq: divisibility theorem}) becomes
$\delta+1\equiv 0 \ \mathrm{mod}\ 2^{\left\lfloor(\delta-1)/2\right\rfloor}$
and we conclude that  $\delta\in\{0,1,2,3,7\}$.
Moreover, if $\delta\geq2$,
by Proposition 
\ref{prop: cohomology properties}
we get that $\B$ is a Fano variety with
\begin{equation}\label{eq: contradiction by conic-connected}
 2i(\B)=r+\delta=\Delta-4+\left(d+1\right)\left(\delta+1\right) \equiv 0\ \mathrm{mod}\ 2 ,
\end{equation}
and hence, if $\delta=2$,  
we would get a contradiction. 
If $d$ is odd and  $\delta$ is even, by (\ref{eq: divisibility theorem}) 
it immediately follows   (\ref{part: congr_ii}). 
Finally, if $d$ and $\delta$ are both odd,  
we write $d=q2^a+1$ and $\delta=s2^e-1$,  with $a,e,q,s\geq1$ and $q,s$ odd.
Then (\ref{eq: divisibility theorem}) is equivalent to
$2^{a+e}\equiv 0 \ \mathrm{mod}\ 2^{s2^{e-1}-1}$,
i.e.  $a\geq s2^{e-1}-e-1$.
Thus, putting $t=q2^{a-(s2^{e-1}-e-1)}$,  we have $d=t2^{s2^{e-1}-e-1}+1$.
\end{proof} 
\begin{corollary}\label{prop: constraint with Delta=2 and d even} 
 If $\Delta=2$ and $d$ is even, then  
the possible values of 
$n$,   $r$,   $r'$ and   $\delta$ 
are contained in Table \ref{tab: value when Delta=2 and d is even}.
\begin{table}[htbp]
\begin{center}
\begin{tabular}{|c|c|c|c|}
\hline
 $n$ & $r$ & $r'$ & $\delta$ \\
\hline \hline
 $2d$ & $d-1$ & $2(d-1)$ & $0$\\
 $4d-1$ & $2d-1$ & $4(d-1)$ & $1$\\
 $8d-3$ & $4d-1$ & $8(d-1)$ & $3$ \\
 $16d-7$ & $8d-1$ & $16(d-1)$ & $7$ \\
\hline
\end{tabular}
\end{center}
\caption{Values of $n$, $r$, $r'$ and $\delta$, when $\Delta=2$ and $d$ is even.}
\label{tab: value when Delta=2 and d is even}
\end{table}
\end{corollary}
\begin{proposition}\label{prop: d Delta both odd} 
 If $d$,   $\Delta$ are both odd, then 
$\delta=0$, $r=\Delta+d-3$, $n=2( \Delta+d-2)$.
\end{proposition} 
\begin{proof}
  If $\delta>0$,  by Proposition 
\ref{prop: cohomology properties}  
we would get  the contradiction  that 
(\ref{eq: contradiction by conic-connected}) holds.
Thus $\delta=0$ and there follow the expressions 
 of $r$ and   $n$ by Proposition \ref{prop: dimension formula}. 
\end{proof}
\begin{table}[htbp]
\tabcolsep=3pt
\begin{center}
\begin{tabular}{|c|c||cccccc|}
\hline 
$\Delta$ & $d$ & \multicolumn{6}{c|}{$(n,r,\delta,c)$} \\
\hline
\hline
\multirow{3}{*}{$2$} &$2$ & $(4,1,0,\_)$,& $(7,3,1,\_)$, & $(13,7,3,3)$, & $(25,15,7,5)$ & & \\ 
\cline{2-8}
& $3$ & $(6,2,0,\_)$, & $(16,8,2,4)$, & $(21,11,3,5)$, & $(26,14,4,6)$, & $(31,17,5,7)$, & $(41,23,7,9)$  \\
\cline{2-8}
&$4$ & $(8,3,0,\_)$, & $(15,7,1,4)$, & $(29,15,3,7)$, & $(57,31,7,13)$ & &   \\
\hline
\hline
\multirow{3}{*}{$3$} &$2$ & $(18,10,4,4)$, & $(24,14,6,5)$ & & & & \\ 
\cline{2-8}
& $3$ & $(8,3,0,\_)$ & & & & &  \\
\cline{2-8}
&$4$ &  $(10,4,0,\_)$, & $(24,12,2,6)$, & $(38,20,4,9)$ & & & \\
\hline
\hline
\multirow{3}{*}{$4$} &$2$ & $(17,9,3,4)$, & $(23,13,5,5)$ & & & & \\ 
\cline{2-8}
& $3$ & $(10,4,0,\_)$, & $(15,7,1,4)$, & $(20,10,2,5)$, & $(25,13,3,6)$, & $(30,16,4,7)$, & $(40,22,6,9)$ \\
\cline{2-8}
&$4$ & $(12,5,0,\_)$, & $(19,9,1,5)$, & $(33,17,3,8)$, & $(47,25,5,11)$, & $(75,41,9,17)$ &  \\
\hline
\end{tabular}
\end{center}
\caption{All cases with $\Delta\leq4$ and $d\leq 4$.}
\label{tab: all cases 4x4}
\end{table}
\begin{remark}
Let us say a few words about the construction of Table \ref{tab: all cases 4x4}.
If $\Delta=2$ and either $d=2$ or $d= 4$, then 
the list of cases is obtained by Corollary \ref{prop: constraint with Delta=2 and d even}.
If $\Delta=2$ and $d=3$, then by (\ref{eq: divisibility theorem}) we obtain 
$(n,r,\delta)\in\{ (6,2,0), (11,5,1), (16,8,2), (21,11,3), (26,14,4), (31,17,5), (41,23,7)  \}$;
the case  $(n,r,\delta)=(11,5,1)$ is impossible,  since otherwise 
by Proposition \ref{prop: cohomology properties} part \ref{part: hilbert polynomial},
we would get incompatible conditions 
for $P_{\B}(t)$;  by the same Proposition, we also get 
that in the case $(n,r,\delta)=(16,8,2)$ (resp. $(n,r,\delta)=(21,11,3)$)
we have $\lambda=36$ (resp. $\lambda=86$).   
For the cases with $d=2$ and either $\Delta=3$ or $\Delta= 4$, 
we refer to Propositions \ref{prop: invariants d=2 Delta=3} 
and \ref{prop: invariants d=2 Delta=4}, below.
The unique case with $\Delta=3$ and $d=3$ is obtained 
by Proposition \ref{prop: d Delta both odd}.
Finally, cases with either $(\Delta,d)=(3,4)$, $(\Delta,d)=(4,3)$ or $(\Delta,d)=(4,4)$
are obtained by (\ref{eq: divisibility theorem}) 
and (\ref{eq: contradiction by conic-connected}).
\end{remark}

\subsection{Sequences of impossible values of \texorpdfstring{$n$}{n}} 
If we fix $\Delta$, then for each $n$
there are only finitely many possible values ​​of $d$.
So we can ask:
given $\Delta$, 
what are the values ​​of $n$ such that no $d$ works?
In other words, 
\emph{What are the values ​​of $n$ such that
there are no special quadratic birational transformations 
$\varphi:\PP^n\dashrightarrow\sS\subset\PP^{n+1}$ 
as in Assumption \ref{notation: factorial hypersurface}?}

In particular, let $\Delta=2$. Then, 
 for $n\leq 60$ all (non-excluded) values of $d$ are 
contained in Table \ref{tab: value of d when Delta=2},
which is constructed via Proposition 
\ref{prop: numerical constraint with Delta=2} and where
an asterisk denotes cases further excluded by  
Proposition \ref{prop: cohomology properties} part \ref{part: hilbert polynomial}.
 \begin{table}[htbp]
 \begin{center}
 \begin{tabular}{||c|c||c|c||c|c||c|c||c|c||c|c||}
 \hline
 $n$ &$d$& $n$ &$d$&$n$&$d$&$n$&$d$&$n$&$d$&$n$&$d$\\
 \hline \hline
  1 & & 11 & $3^{\ast}$ & 21 & 3 & 31 & 3, 8 &41 & 3 & 51 & 13\\
 \hline
  2 & & 12 &6& 22&11&32&16&42&21&52&9, 26\\
 \hline
  3 & & 13&2 &23&6&33&&43&11&53&7\\
 \hline
 4 &2 & 14&7&24&12&34&17&44&22&54&27\\
 \hline
  5 & & 15&4&25&2&35&9&45&6&55&5, 14\\
 \hline
 6 &3 & 16&3, 8&26&3, 13&36&18&46&5, 23&56&28\\
 \hline
 7 &2 & 17&&27&7&37&5&47&12&57&4\\
 \hline
 8 &4 & 18&9&28&5, 14&38&19&48&24&58&29\\
 \hline
 9 & & 19& 5 &29&4&39&10&49&&59&15\\
 \hline
 10 & 5& 20&10&30&15&40&7, 20&50&25&60&30\\
 \hline
 \end{tabular}
 \end{center}
 \caption{Values of $d$, when $\Delta=2$ and $n\leq 60$.}
\label{tab: value of d when Delta=2}
\end{table}
We then see that for $n\in\{3,5,9,17,33\}$ no $d$ is possible.
More generally, if we define  the sequence $\xi_2(k)=1+2^k$, then
we have that 
$n\neq \xi_2(k)$ for any $k$.
In fact,
 if $n=\xi_2(k)$ for some $k$, 
then $(d-1)2^k=(2d-1)(r-\delta)$ and hence 
the contradiction that $(2d-1)$ divides $(d-1)$. 

Obviously, when $n>33$, there are many values ​​of $n$ 
that are not possible and that do not belong to the image of $\xi_2$. So
we define another sequence
 $\xi_2'(k)$, $k\geq0$, as follows: $\xi_2'(0)=33$,
and, for $k\geq1$, 
$$
\xi_2'(k)=\left\{\begin{array}{ll}
                \xi_2'(k-1)+16, & \mbox{if } k\not{\!\!\equiv}\ 0\ \mathrm{mod}\ 15,\\
                \xi_2'(k-1)+32, & \mbox{if } k \equiv 0\ \mathrm{mod}\ 15 ;\\
                \end{array}
         \right.
$$
equivalently, for $k\geq0$, we have
$\xi_2'(k)=33+16\left\lfloor k/15 \right\rfloor+16k$.
 \begin{proposition} Let $\Delta=2$.
\begin{enumerate}
\item\label{part: a - sufficient cond for type 2 1}  $n=\xi_2'(k)$, for some $k\geq0$, if and only if 
 case (\ref{part: congr_iii}) of 
 Proposition \ref{prop: numerical constraint with Delta=2} occurs, with either $e\geq5$   
or $e=4$ and   $s\not{\!\!\equiv} 1\ \mathrm{mod}\ 16$. 
\item\label{part: b - sufficient cond for type 2 1} If  $n=\xi_2'(k)$, then $k\geq 3840$ and $n\geq 65569$.
\item\label{part: c - sufficient cond for type 2 1}  $n\neq\xi_{2}'(k)$, for infinitely many $k$. 
\end{enumerate}
 \end{proposition}
\begin{proof}
By straightforward calculation we see that  
$n=\xi_2'(k)$, for some  $k\geq 0$, if and only if 
$$
a\left( n\right) :=\frac{n-33}{16}\in\ZZ\quad \mbox{and}\quad b\left( n\right) :=\frac{a\left( n\right) -15}{16}=\frac{n-273}{256}\notin\ZZ.
$$
It follows that $\{\xi_2(k):k\geq5\}\subset\{\xi_2'(k):k\geq0\}$
 and hence 
we have statement \ref{part: c - sufficient cond for type 2 1}.
Statement \ref{part: b - sufficient cond for type 2 1}  
 follows from statement \ref{part: a - sufficient cond for type 2 1}
and it can also be shown by computer. 
We now show statement \ref{part: a - sufficient cond for type 2 1}.  
 $\delta=0,1,3,7$ respectively implies  $n=d2^i-\delta$, with $i=1,2,3,4$, 
and for these  $n$ and $\delta$, 
$a(n)=\left(d{2}^{i}-\delta-33\right)/16= 
{2}^{i-1}(2d-1)/16-2\notin\ZZ$.
If $\delta$ is even, then $n=(2d-1)(\delta+1)+1$ is even, 
but the values of   $\xi_2'(k)$   
are always odd.
Thus, if  
 $n=\xi_2'(k)$, then  case (\ref{part: congr_iii}) of 
Proposition \ref{prop: numerical constraint with Delta=2} occurs  and we have
\begin{align*}
n&= s\,{2}^{{2}^{e-1}\,s}\,t+{2}^{e}\,s+1 ,& a(n)&= {2}^{e-4}\,s\,\left( {2}^{{2}^{e-1}\,s-e}\,t+1\right) -2 ,\\ 
d&= {2}^{{2}^{e-1}\,s-e-1}\,t+1 ,& b(n)&= s\,{2}^{{2}^{e-1}\,s-8}\,t+{2}^{e-8}\,s-\frac{17}{16}  . 
\end{align*}
If $e=1$,   $$a(n)=\frac{s\,\left( {2}^{s-1}\,t+1\right) }{8}-2\in\ZZ\ \Rightarrow\ s=1 \mbox{ and } t\mbox{ odd}\ \Rightarrow\ d=\frac{t}{2}+1\notin\ZZ,$$
from which we obtain that $a(n)\notin\ZZ$. 
Similarly, we obtain that $a(n)\notin\ZZ$, if $e=2$.     If $e=3$,   since ${2}^{e-1}\,s-e\geq1$, we again obtain that  
$a(n)\notin\ZZ$.   While  for $e\geq4$,   $a(n)$ is always integer. 
Finally, for $e=4$, 
$$
b(n)=s\,{2}^{8\,s-8}\,t+\frac{s-1}{16}-1\notin\ZZ\ \Leftrightarrow\ \frac{s-1}{16}\notin\ZZ,
$$
and for $e\geq5$,   $b(n)$ is never integer. 
\end{proof}
Now we consider the case in which $\Delta\neq2$.
Firstly, we observe that $n\geq2\Delta$,
since otherwise we should have
$\delta=(n-2\Delta-2d+4)/(2d-1)\leq -1+2/(2d-1)<0$.
Define the sequences $n=\xi_{\Delta}(k)$, $k\geq0$, 
 as follows:
$$
\xi_{\Delta}(k):=2\Delta-3+2^{j(\Delta)}+2^{j(\Delta)}k, 
$$
where
$$
j(\Delta)=\left\{\begin{array}{ll}
                2 &\mbox{if } \Delta\equiv 1\ \mathrm{mod}\ 2,\\
                \min\{j\geq3: \Delta\not{\!\!\equiv} 2\ \mathrm{mod}\ 2^j \} & \mbox{if } \Delta \equiv 0\ \mathrm{mod}\ 2 .\\
                \end{array}
         \right.
$$
Then we have the following:
\begin{proposition}\label{prop: sufficient conditions for the nonexistence}
If $\Delta\neq2$, then $n\neq\xi_{\Delta}(k)$, for any $k\geq0$. 
\end{proposition}
\begin{proof}
Suppose, by contradiction, that $n=\xi_{\Delta}(k)$, for some $k\geq0$.
We have 
$n=(2d-1)(\delta+1)+2\Delta-3 $, so that $\delta$ must be odd.
Now suppose that $\Delta\equiv 1\ \mathrm{mod}\ 2$.
By Proposition \ref{prop: d Delta both odd} it follows that $d$ is even.
If $\delta\geq 3$,  by (\ref{eq: divisibility theorem}) we obtain the 
contradiction $0\ \mathrm{mod}\ 2
= (d-1)(\delta+1)+\Delta-2\ \mathrm{mod}\ 2
=1\ \mathrm{mod}\ 2$,
so we must have $\delta=1$. But then $n=2\Delta-5+4d=2\Delta+1+4k$ and we
 obtain another contradiction, namely that  $2(d-k)=3$.
Now suppose that  $\Delta\equiv 0\ \mathrm{mod}\ 2$. 
By the fact that $\xi_{\Delta}(k)=n$, we deduce 
$(2d-1)(\delta+1)=2^{j(\Delta)}(k+1)$, from which
it follows that  $\delta=-1+2^{j(\Delta)}h$, for some $h\geq1$. So
$\lfloor(\delta-1)/2\rfloor=(\delta-1)/2 = 2^{j(\Delta)-1}h-1\geq
 2^{j(\Delta)-1}-1\geq j(\Delta)$ and by (\ref{eq: divisibility theorem})
we obtain 
$0\ \mathrm{mod}\ 2^{j(\Delta)}
= (d-1)(\delta+1)+\Delta-2\ \mathrm{mod}\ 2^{j(\Delta)}
=(d-1)2^{j(\Delta)}h+\Delta-2 \ \mathrm{mod}\ 2^{j(\Delta)}
=\Delta-2 \ \mathrm{mod}\ 2^{j(\Delta)}$, against the definition of $j(\Delta)$.
\end{proof}
\begin{remark}
Although our original hypothesis is $\Delta\geq2$, 
note that also when $\Delta=1$,
Proposition \ref{prop: sufficient conditions for the nonexistence}
 is still valid. In other words,  
\emph{ There are no special quadratic Cremona transformations of $\PP^n$,
whenever $n\in\{3,7,11,15,19,\ldots\}$.}
We should also note that, when  
$j(\Delta)\geq5$, then Proposition \ref{prop: sufficient conditions for the nonexistence}
is very far from being sharp. 
\end{remark}
\begin{table}[htbp]
\tabcolsep=4pt
\small
\begin{center}
\begin{tabular}{|c|cccccccccccccccccccc|}
$\Delta $ \\
 \hline
2 & 3 & 5 & 9 & 17 & 33 & 49 & 65 & 81 & 97 & 113 & 129 & 145 & 161 & 177 & 193 & 209 & 225 & 241 & 257 & 289 \\
\hline
3 & 3 & 4 & 5 & 7 & 11 & 15 & 19 & 23 & 27 & 31 & 35 & 39 & 43 & 47 & 51 & 55 & 59 & 63 & 67 & 71 \\
\hline 
4 & 3 & 4 & 5 & 6 & 7 & 9 & 13 & 21 & 29 & 37 & 45 & 53 & 61 & 69 & 77 & 85 & 93 & 101 & 109 & 117 \\
\hline
5 & 3 & 4 & 5 & 6 & 7 & 8 & 9 & 11 & 15 & 19 & 23 & 27 & 31 & 35 & 39 & 43 & 47 & 51 & 55 & 59 \\
\hline 
6 & 3 & 4 & 5 & 6 & 7 & 8 & 9 & 10 & 11 & 13 & 17 & 25 & 33 & 41 & 49 & 57 & 65 & 73 & 81 & 89 \\
\hline 
7 & 3 & 4 & 5 & 6 & 7 & 8 & 9 & 10 & 11 & 12 & 13 & 15 & 19 & 23 & 27 & 31 & 35 & 39 & 43 & 47 \\
\hline 
8 & 3 & 4 & 5 & 6 & 7 & 8 & 9 & 10 & 11 & 12 & 13 & 14 & 15 & 17 & 21 & 29 & 37 & 45 & 53 & 61 \\
\hline 
9 & 3 & 4 & 5 & 6 & 7 & 8 & 9 & 10 & 11 & 12 & 13 & 14 & 15 & 16 & 17 & 19 & 23 & 27 & 31 & 35 \\
\hline 
10 & 3 & 4 & 5 & 6 & 7 & 8 & 9 & 10 & 11 & 12 & 13 & 14 & 15 & 16 & 17 & 18 & 19 & 21 & 25 & 33 \\
\hline 
 11 & 3 & 4 & 5 & 6 & 7 & 8 & 9 & 10 & 11 & 12 & 13 & 14 & 15 & 16 & 17 & 18 & 19 & 20 & 21 & 23 \\
\hline
\end{tabular}
\end{center}
\caption{Some impossible values of $n$, when $2\leq\Delta\leq 11$.}
\label{tab: value of xi}
\end{table}

\section{Examples}\label{sec: examples hypersurface}
The calculations in the following examples can be verified 
 with the aid of  the computer algebra systems: \cite{macaulay2} and \cite{sagemath}.
\begin{example}[$\Delta=2$, $d=2$]\label{example: d=2 Delta=2}
Let $\psi:\PP^{n+1}\dashrightarrow\PP^{n+1}$ be a Cremona  transformation  of 
type $(2,2)$. 
If $H\simeq\PP^{n}\subset\PP^{n+1}$ is a general hyperplane, 
then $\Q:=\overline{\psi(H)}\subset\PP^{n+1}$ is a quadric hypersurface 
 and the restriction  
$\psi|_{H}:\PP^n\dashrightarrow\Q\subset\PP^{n+1}$ is 
a birational transformation
of type $(2,2)$, with base locus 
 $
\mathrm{Bs}(\psi|_H)=\mathrm{Bs}(\psi)\cap H$.
If $\psi$ is special, i.e. if its base locus is a Severi variety, 
then also $\psi|_H$ is special and moreover
it is possible to verify that
the quadric $\Q$ is smooth, for example by determining explicitly the equation. 
\end{example}

\begin{example}[$\Delta=2$, $d=3$, $\delta=0$]\label{example: d=3 Delta=2}
We construct the Edge variety
 $\mathcal{X}$ of 
dimension $3$ and degree $7$ (see also \cite{edge} and \cite[Example~2.4]{ciliberto-mella-russo}).
Consider the Segre variety
$S^{1,3}=\PP^1\times\PP^3\subset\PP^7$ and  choose 
coordinates $x_0,\ldots,x_7$ on $\PP^7$ such that the equations 
of $S^{1,3}$ in $\PP^7$ are given by
$$
\begin{array}{ccc}Q_0=- x_{1} x_{4} + x_{0} x_{5},&Q_1= - x_{2} x_{4} + x_{0} x_{6},&Q_2= - x_{3} x_{4} + x_{0} x_{7},\\ 
                  Q_3= - x_{2} x_{5} + x_{1} x_{6},&Q_4= - x_{3} x_{5} + x_{1} x_{7},&Q_5= - x_{3} x_{6} + x_{2} x_{7}.
\end{array}
$$
Take a general quadric 
 $V(Q)$ containing the linear space 
$P=V(x_0,x_1,x_2,x_3)\subset S^{1,3},$
i.e. $Q=\sum b_{ij}x_i x_j$, for suitable coefficients $b_{ij}\in\CC$, with 
$i\leq j$,  $0\leq i\leq 3$ and  $0\leq j \leq 7$.
The intersection $\mathcal{Y}=S^{1,3}\cap V(Q)$ is an equidimensional variety
of dimension $3$ that has $P$ as irreducible component of multiplicity $1$.
Hence it  defines a variety
 $\mathcal{X}$ of dimension $3$ and degree $7$
 such that
$\mathcal{Y}=P\cup \mathcal{X}$ and  $P\nsubseteq \mathcal{X}$.
Since we are interested in
 $\mathcal{Y}$ and not in $Q$,
we can suppose 
$
b_{14}=b_{24}=b_{34}=b_{25}=b_{35}=b_{36}=0
$.
Thus it is easy to verify that
 $\mathcal{X}$ is the scheme-theoretic intersection of  $\mathcal{Y}$ with  the quadric $V(Q')$,   where
 \begin{eqnarray*}
  Q'&=&b_{37}x_7^2+b_{27}x_6x_7+b_{17}x_5x_7+b_{07}x_4x_7+b_{33}x_3x_7  +  b_{26}x_6^2+b_{16}x_5x_6 + \\ &+& 
       b_{06}x_4x_6+b_{23}x_3x_6+b_{22}x_2x_6 +b_{15}x_5^2+b_{05}x_4x_5+b_{13}x_3x_5+b_{12}x_2x_5 + \\ &+&
       b_{11}x_1x_5  +  b_{04}x_4^2+b_{03}x_3x_4+b_{02}x_2x_4+b_{01}x_1x_4+b_{00}x_0x_4.
 \end{eqnarray*}
The rational map
$\psi:\PP^7\dashrightarrow\PP^7$,
defined by $\psi([x_0,\ldots,x_7])=[Q_0,\ldots,Q_5,Q,Q']$,
has as image the 
 rank $6$ quadric
$\Q=V(y_{0} y_{5} - y_{1} y_{4} + y_{2} y_{3})$
and the closure of its general fiber is a $\PP^1\subset\PP^7$. 
Hence restricting 
$\psi$ to a general $\PP^6\subset\PP^7$, we get a birational transformation  
$
\PP^6\dashrightarrow\Q\subset\PP^7$.  
\end{example}
\begin{example}\label{example: d=3 Delta=2 continuing} Continuing with the Example
 \ref{example: d=3 Delta=2},
we set
$b_{00}=b_{07}=b_{15}=b_{16}=b_{22}=b_{33}=1$ and for all the other indices
 $b_{ij}=0$.
Substitute in the quadrics
 $Q_0,\ldots,Q_5,Q,Q'$ instead of the variable $x_7$ 
the variable $x_0$, i.e. we consider the intersection of $\mathcal{X}$ 
with the hyperplane $V(x_7-x_0)$.
Denote by 
 $X\subset\PP^6$ the scheme so obtained.
$X$ is irreducible, smooth, it is defined by the $8$ independent quadrics:
\begin{equation}\label{eq: equations of B, d=3 Delta=2}
\begin{array}{c}
x_3x_6-x_0x_2,\ x_5x_6+x_2x_6+x_5^2+2x_0x_4+x_0x_3,\ x_3x_5-x_0x_1,\ x_2x_5-x_1x_6,\\
x_3x_4-x_0^2,\ x_2x_4-x_0x_6,\ x_1x_4-x_0x_5,\ x_1x_6+x_1x_5+x_3^2+x_2^2+2x_0^2,
\end{array}
\end{equation}
and its Hilbert polynomial is
 $P_{X}(t)=(7t^2+5t+2)/2$.
The quadrics (\ref{eq: equations of B, d=3 Delta=2}) 
define a birational map 
 $\psi:\PP^6\dashrightarrow\Q\subset\PP^7$ into the quadric 
 $\Q=V(y_0y_6-y_2y_5+y_3y_4)$ and the inverse of $\psi$
 is defined by the cubics:
\begin{equation}\label{eq: inverse, d=3 Delta=2}
\begin{array}{c}
y_0y_5y_7-y_1y_4y_7+y_1y_2y_6+y_0y_1y_6+y_2y_3y_5 + y_0y_3y_5+y_2^2y_4+y_0y_2y_4, \\
-y_6y_7^2-2y_4y_6y_7-y_3y_6y_7+y_0y_3y_7-y_1y_2y_7+2y_2y_6^2+2y_0y_6^2+y_2y_3^2  +y_0y_3^2+y_2^3+y_0y_2^2, \\
-y_5y_7^2-2y_4y_5y_7-y_3y_5y_7-y_0y_1y_7+2y_2y_5y_6+2y_0y_5y_6-y_1y_2y_3-y_0y_1y_3+y_0y_2^2+y_0^2y_2, \\
-y_4y_7^2+y_0y_6y_7-y_2y_5y_7-2y_4^2y_7-y_0^2y_7+2y_2y_4y_6+2y_0y_4y_6-y_0y_2y_3-y_0^2y_3-y_1y_2^2   -y_0y_1y_2, \\
-y_5^2y_7-y_4^2y_7-y_1y_6^2-y_3y_5y_6-y_1y_5y_6-y_2y_4y_6-2y_4y_5^2-y_3y_5^2-y_2y_4y_5-2y_4^3-y_1^2y_4   -y_0^2y_4, \\
-y_1y_6y_7-y_3y_5y_7-y_2y_4y_7-y_1y_3y_6+y_0y_2y_6-2y_2y_5^2-y_3^2y_5-y_2^2y_5-2y_2y_4^2-y_1^2y_2-y_0^2y_2, \\
-y_1y_5y_7-y_0y_4y_7+y_1y_3y_6-y_0y_2y_6-2y_0y_5^2+y_3^2y_5+y_2^2y_5 -2y_0y_4^2-y_0y_1^2-y_0^3 .
\end{array}
\end{equation}
The base locus  
 $Y\subset\Q\subset\PP^7$ of $\psi^{-1}$ 
is obtained by intersecting the scheme defined by
 (\ref{eq: inverse, d=3 Delta=2}) with the quadric  $\Q$.
$Y$ is irreducible with Hilbert polynomial  
$$P_{Y}(t)=\frac{9t^4+38t^3+63t^2+58t+24}{4!}$$
and its singular locus has dimension  $0$.
\end{example}

\begin{example}[$\Delta=2$, $d=4$, $\delta=0$]\label{example: d=4 Delta=2} 
The $10$ quadrics (\ref{eq: equations S10})  define a rational map  
 $\psi:\PP^{15}\dashrightarrow\PP^{9}$, with base locus  
 $S^{10}$ and image in $\PP^9$
the smooth quadric 
$$\Q=V({y_8}{ y_9}-{ y_6}{ y_7}-{y_0}{ y_5}+ { y_2}{ y_4}+{ y_1}{ y_3}).$$
By \cite[page~798]{ein-shepherdbarron} the closure of 
the general fiber of  $\psi$ is
 a $\PP^7\subset\PP^{15}$   and hence,
by restricting 
 $\psi$ to a general  $\PP^8\subset\PP^{15}$, 
 we get a special birational transformation 
$
\PP^8\dashrightarrow\Q 
$
(necessarily of type $(2,4)$, by Remark \ref{remark: dimension formula without hypothesis}).
\end{example}
\begin{example}[$\Delta=3$, $d=3$]\label{example: d=3 Delta=3}
Let
$u:\PP^3\stackrel{\simeq}{\dashrightarrow}Q=V(u_0\,u_4-u_1^2-u_2^2-u_3^2)\subset\PP^4$
be defined by
$u([z_0,z_1,z_2,z_3]) =[{z}_{0}^{2},{z}_{0}{z}_{1},{z}_{0}{z}_{2},{z}_{0}{z}_{3},{z}_{1}^{2}+{z}_{2}^{2}+{z}_{3}^{2}]$.
Consider the composition
$ v^0:\PP^3\stackrel{u}{\dashrightarrow}\PP^4\stackrel{\nu_2}{\longrightarrow}\PP^{14}\dashrightarrow\PP^{13} $,
where $\nu_2$ is the Veronese map (with lexicographic order) 
and the last map is the projection onto the hyperplane
 $V(v_4)\simeq\PP^{13}\subset\PP^{14}$. 
The map $v^0$ parameterizes a nondegenerate variety in $\PP^{13}$, of degree 
 $16$ and isomorphic to the quadric $Q$. 
Take the point  $p_1=[1,0,0,0]\in\PP^3$ and  consider the composition
$v^1:\PP^3\stackrel{v^0}{\dashrightarrow}\PP^{13}\stackrel{\pi_1}{\dashrightarrow}\PP^{12} 
$,
where $\pi_1$ is the projection from the point $v^0(p_1)$ (precisely,
 if $j$ is  the index of the last nonzero coordinate
of $v^0(p_1)$, 
 we exchange the coordinates 
 $v_j$,   $v_{13}$ and 
 we project onto the hyperplane $V(v_{13})$ from the point $v^0(p_1)$).
Repeat the construction with the points $p_2=[0,0,0,1]$, $p_3=[1,0,0,1]$, $p_4=[0,1,0,1]$, $p_5=[0,0,1,1]$, 
obtaining the maps  
$v^2:\PP^3\dashrightarrow\PP^{11}$, $v^3:\PP^3\dashrightarrow\PP^{10}$, $v^4:\PP^3\dashrightarrow\PP^{9}$, $v^5:\PP^3\dashrightarrow\PP^{8}$.
       The map $v^5$ is given by\footnote{Note that $v^5:\PP^3\dashrightarrow\PP^8$ 
can be obtained by choosing a suitable basis 
of $H^0(\PP^3, \I_{Z}(3))$, where $Z$ is the union 
of the conic $V(z_0,z_1^2+z_2^2+z_3^2)$ and the four points: 
$[1,0,0,0]$, $[1,1,0,0]$, $[1,0,1,0]$, $[1,0,0,1]$.}
\begin{eqnarray*}
v^5([z_0,z_1,z_2,z_3])&=&[z_3^3-z_0\,z_3^2+z_2^2\,z_3+z_1^2\,z_3, \  
                          -z_1\,z_3^2-z_1\,z_2^2-z_1^3+z_0^2\,z_1, \\  
                        &&  -z_2\,z_3^2-z_2^3-z_1^2\,z_2+z_0^2\,z_2, \   
                          z_0^2\,z_3-z_0\,z_3^2, \  
                          -z_1\,z_3^2-z_1\,z_2^2-z_1^3+z_0\,z_1^2, \\  
                       &&   z_0\,z_1\,z_2, \  
                            z_0\,z_1\,z_3,  \   
                            z_0\,z_2\,z_3, \  
                          -z_2\,z_3^2-z_2^3+z_0\,z_2^2-z_1^2\,z_2]
\end{eqnarray*}
       and parameterizes the smooth variety $X\subset\PP^8$ defined by the $10$ independent quadrics:
       \begin{equation}\label{eq: equations of B, d=3 Delta=3}
       \begin{array}{c}
       x_5\,x_8-x_4\,x_8+x_1\,x_8-x_4\,x_7+x_5\,x_6+x_2\,x_6+x_5^2+x_4\,x_5-x_3\,x_5-x_2\,x_5,  \\
       x_5\,x_8-x_4\,x_8-x_4\,x_7+x_5\,x_6+x_2\,x_6+x_5^2+x_4\,x_5-x_3\,x_5-x_1\,x_5+x_2\,x_4,\\
       -x_7\,x_8-x_0\,x_8-x_7^2+x_3\,x_7-x_5\,x_6+x_0\,x_2,\\
        x_6\,x_8-x_4\,x_7+x_6^2+x_5\,x_6+x_4\,x_6-x_1\,x_6-x_0\,x_6+x_3\,x_4,\\
        x_6\,x_7+x_4\,x_7-x_5\,x_6-x_2\,x_6+x_3\,x_5,\\
        x_7\,x_8+x_3\,x_8+x_7^2-x_2\,x_7-x_0\,x_7+x_5\,x_6,\\
        x_1\,x_7-x_2\,x_6,\\
        x_5\,x_7+x_6^2+x_4\,x_6-x_1\,x_6-x_0\,x_6+x_3\,x_4,\\
       x_2\,x_6-x_3\,x_5+x_0\,x_5,\\
       x_3\,x_6-x_1\,x_6-x_0\,x_6+x_3\,x_4-x_0\,x_4+x_0\,x_1. 
       \end{array}
       \end{equation}
The quadrics (\ref{eq: equations of B, d=3 Delta=3}) 
define a special birational 
transformation 
 $\psi:\PP^8\dashrightarrow\sS\subset\PP^9$ of type $(2,3)$ 
   into the cubic $\sS\subset\PP^9$
defined by
\begin{equation}
\begin{array}{l}
 - y_7\,y_8\,y_9 + y_6\,y_8\,y_9 + y_3\,y_8\,y_9 + y_7^2\,y_9 - y_6\,y_7\,y_9 - 2\,y_3\,y_7\,y_9 - y_2\,y_7\,y_9 +\\
+ y_0\,y_7\,y_9 + y_4\,y_6\,y_9 + y_3\,y_6\,y_9 - y_4\,y_5\,y_9 - y_1\,y_5\,y_9 - y_4^2\,y_9 - y_0\,y_4\,y_9  +\\
+ y_3^2\,y_9 + y_2\,y_3\,y_9 - y_0\,y_3\,y_9 + y_7\,y_8^2 - y_3\,y_8^2 - y_6\,y_7\,y_8 + 2\,y_4\,y_7\,y_8  +\\
+ y_3\,y_7\,y_8 - y_2\,y_7\,y_8 - y_4\,y_6\,y_8 + y_2\,y_6\,y_8 - y_0\,y_6\,y_8 + y_4\,y_5\,y_8 + y_0\,y_5\,y_8 +\\
- 3\,y_3\,y_4\,y_8 - y_1\,y_4\,y_8 + y_0\,y_4\,y_8 - y_3^2\,y_8 + y_2\,y_3\,y_8 - y_1\,y_3\,y_8 + y_5\,y_7^2  +\\
+ y_2\,y_7^2 + y_1\,y_7^2 - y_0\,y_7^2 + y_6^2\,y_7 + y_3\,y_6\,y_7 + y_0\,y_6\,y_7 - y_3\,y_5\,y_7 +\\
+ y_1\,y_5\,y_7 - y_2\,y_3\,y_7 - y_1\,y_3\,y_7 + y_0\,y_3\,y_7 + y_0\,y_2\,y_7 - y_4\,y_6^2 - y_3\,y_6^2 +\\
- y_3\,y_5\,y_6 - y_3\,y_4\,y_6 - y_0\,y_4\,y_6 - y_3^2\,y_6 - y_0\,y_3\,y_6 - y_1\,y_2\,y_6 + y_4^2\,y_5 +\\
+ y_3\,y_4\,y_5 + y_0\,y_4\,y_5 + y_2\,y_4^2 - y_1\,y_4^2 + y_0\,y_4^2 + y_2\,y_3\,y_4 - y_1\,y_3\,y_4 +\\
+ y_0\,y_3\,y_4 + y_1\,y_2\,y_4 .
\end{array}
\end{equation}
The singular locus of $\sS$ has dimension $3$, 
from which it follows the factoriality of $\sS$ 
(see  \cite[\Rmnum{11} Corollaire~3.14]{sga2} and Remark \ref{remark: samuel conjecture} below).
\end{example}
\begin{example}[$\Delta=4$, $d=2$, $\delta=0$]\label{example: B2=singSred-3fold} 
This is an example for which (\ref{eq: hypothesis on sing locus hypersurface}) is not satisfied;
essentially, it gives an example for case (\ref{part: case 0.5, 3-fold}) of 
Proposition \ref{prop: 3-fold in P8 - S nonlinear}.
Consider the irreducible
smooth $3$-fold $X\subset\PP^8$ defined as the intersection of
 $\PP^2\times\PP^2\subset\PP^8$, which is defined by (\ref{eq: P2xP2inP8}),
 with the quadric hypersurface defined by:
 \begin{equation}
2x_0^2+3x_1^2+5x_2^2+x_3^2+x_4^2+x_5^2+5x_6^2+3x_7^2+2x_8^2.
 \end{equation}
Thus $X$ is defined by $10$ quadrics and we can consider the associated 
rational map $\psi:\PP^8\dashrightarrow\sS=\overline{\psi(\PP^8)}\subset\PP^9$.
We have that $\sS$ is the quartic hypersurface defined by:
 \begin{equation}
\begin{array}{l}
  2y_1^2y_2^2+3y_1^2y_3^2-4y_0y_1y_2y_4+2y_0^2y_4^2+5y_3^2y_4^2-6y_0y_1y_3y_5-10y_2y_3y_4y_5+3y_0^2y_5^2 \\ +5y_2^2y_5^2+y_2^2y_6^2+y_3^2y_6^2+5y_4^2y_6^2+3y_5^2y_6^2-2y_0y_2y_6y_7-10y_1y_4y_6y_7+y_0^2y_7^2 \\ +5y_1^2y_7^2+y_3^2y_7^2+2y_5^2y_7^2-2y_0y_3y_6y_8-6y_1y_5y_6y_8-2y_2y_3y_7y_8-4y_4y_5y_7y_8+y_0^2y_8^2 \\ +3y_1^2y_8^2+y_2^2y_8^2+2y_4^2y_8^2-y_3y_4y_6y_9+y_2y_5y_6y_9+y_1y_3y_7y_9-y_0y_5y_7y_9-y_1y_2y_8y_9+y_0y_4y_8y_9,  
\end{array}
 \end{equation}
and $\psi$ is birational with inverse defined by:
\begin{equation}
\begin{array}{c}
 y_5y_7-y_4y_8,\ 
y_5y_6-y_1y_8,\  
y_4y_6-y_1y_7,\  
y_3y_7-y_2y_8,\  
y_3y_6-y_0y_8, \\
y_2y_6-y_0y_7,\  
y_3y_4-y_2y_5,\  
y_1y_3-y_0y_5,\  
y_1y_2-y_0y_4.
\end{array}
\end{equation}
Moreover $X$ is a Mukai variety of degree $12$, sectional genus $7$ and Betti numbers
$b_2=2$, $b_3=18$. 
The secant variety $\Sec(X)$ 
is a cubic hypersurface and for the general point $p\in\Sec(X)$
there are two secant lines of $X$ through $p$.
Finally, denoting by $Y\subset\sS$ the base locus of $\psi^{-1}$, we have 
$\dim(Y)=5$, $\dim(\sing(Y))=0$ and
$
Y=(Y)_{\mathrm{red}}=(\sing(\sS))_{\mathrm{red}}
$.
\end{example}
\section{Transformations of type \texorpdfstring{$(2,2)$}{(2,2)} into a quadric}\label{sec: type 2-2 into quadric} 
\begin{theorem}\label{prop: classification of type 2-2 into quadric}
If $\Delta=2$, $\varphi:\PP^n\dashrightarrow \Q:=\sS\subset\PP^{n+1}$
is of type $(2,2)$ and $\Q$ is smooth,
then 
$\B$ is a hyperplane section of a Severi variety.
\end{theorem}
\begin{proof}
By Corollary \ref{prop: constraint with Delta=2 and d even} we get
 $\delta\in\{0,1,3,7\}$.  
If $\delta=0$,
we have $r=1$,   $n=4$  and the 
thesis follows from Proposition \ref{prop: P4}. 
Alternatively, we can apply 
 Lemma \ref{prop: cohomology twisted ideal} 
to  determine the Hilbert polynomial of the curve  $\B$.
If $\delta=1$,
we have $r=3$,   $n=7$ and,
by  Proposition \ref{prop: cohomology properties} 
part \ref{part: B is fano},
 $\B$ is a hyperplane section of the Segre variety 
 $\PP^2\times\PP^2\subset\PP^8$,  or 
it is a Fano variety of the first species of     
index $i(\B)=r-1$.
The latter case cannot occur
by the classification of del Pezzo varieties
in Theorem \ref{prop: classification del pezzo varieties}.  
If $\delta=3$,
we have $r=7$,   $n=13$  
and $\B$ is a Mukai variety of the first specie.
   By Proposition
 \ref{prop: cohomology properties} part \ref{part: hilbert polynomial}, 
we can 
determine the Hilbert polynomial  of $\B$ and in particular  
to get that the sectional genus is $g=8$.     
Hence, applying
 the classification of Mukai varieties in Theorem \ref{prop: classification first species mukai varieties},
we get that
$\B$ is 
a hyperplane section 
 of the Grassmannian $\GG(1,5)\subset\PP^{14}$.  
Now, suppose
 $\delta=7$ and hence 
$r=15$, $n=25$, $r'=16$. 
Keep the notation as
 in the proof of Proposition \ref{prop: dimension formula} and  put
$Y=(\B')_{\mathrm{red}}$. We shall show that $\B\subset\PP^{25}$ 
is a hyperplane section of $E_6\subset\PP^{26}$ in several steps.
\begin{claim}
$\pi'(E)=\left(\Sec(\B')\cap\Q\right)_{\mathrm{red}}=\left(\Sec(Y)\cap\Q\right)_{\mathrm{red}}$ and for the general point $p\in\pi'(E)$ 
 the entry locus $\Sigma_p(\B')$ is a quadric of dimension $8$.
\end{claim}
\begin{proof}[Proof of the Claim]
See also \cite[Proposition~2.3]{ein-shepherdbarron}.
If $x\in\B$ is a point, then $\pi'(\pi^{-1}(x))\simeq\PP^9$ and by the relations (\ref{eq: EE'}) 
we get that
$\pi'(\pi^{-1}(x))\cap\B'$  is a quadric hypersurface in $\PP^9$. 
Then, for every point $p\in \pi'(\pi^{-1}(x))\setminus \B'$, every line contained
in $\pi'(\pi^{-1}(x))$ and through  
 $p$ is a secant line of $\B'$  and therefore
 it is contained in $L_p(\B')$. Thus 
$$\pi'(\pi^{-1}(x))\subseteq L_p(\B')\subseteq \Sec(\B')\cap\Q$$
and varying $x\in\B$ we get
\begin{equation}\label{eq: first inclusion, type 2-2}
\pi'(E)=\bigcup_{x\in \B} \pi'(\pi^{-1}(x))\subseteq \Sec(\B')\cap \Q.
\end{equation}
Conversely,
if
 $p\in\Sec(\B')\cap\Q\setminus\B'$, 
then every secant line  
 $l$ of $\B'$ through  $p$ is contained 
 in $\Q$ and $\varphi^{-1}(l)=\varphi^{-1}(p)=:x$. 
Hence the strict transform $\widetilde{l}=\overline{\pi'^{-1}(l\setminus\B')}$ 
is contained
in $\pi^{-1}(x)\subseteq E$ 
 and then $l$ is contained in $\pi'(\pi^{-1}(x))\subseteq\pi'(E)$.
Thus 
$$L_p(\B')\subseteq \pi'(\pi^{-1}(x))\subseteq\pi'(E)$$
 and varying $p$ 
(since $p$ lies on at least a secant line of $\B'$)
we get 
\begin{equation}\label{eq: second inclusion, type 2-2}
\left(\Sec(\B')\cap\Q\right)_{\mathrm{red}} = \overline{\left(\Sec(\B')\cap\Q\setminus\B'\right)_{\mathrm{red}}}\subseteq \pi'(E).
\end{equation}
The conclusion follows from (\ref{eq: first inclusion, type 2-2}) and (\ref{eq: second inclusion, type 2-2}) and 
by observing that, since $\B'$ is irreducible and generically reduced, we have
 $(\Sec(\B'))_{\mathrm{red}}=\Sec(Y)$.
\end{proof} 
\begin{claim}\label{step: tang proj and entry locus}
For the general point $x\in\B$, 
$\overline{\tau_{x,\B}(\B)}=W_{x,\B}\simeq \Sigma_{p}(\B')$, 
with $p$ general point in $\pi'(\pi^{-1}(x))$.
\end{claim}
\begin{proof}[Proof of the Claim]
See also \cite[Theorem~1.4]{mella-russo-baselocusleq3}.
The projective space $\PP^9\subset\PP^{25}$ skew to $T_x(\B)$ (i.e.
the codomain of $\tau_{x,\B}$) identifies with
 $\pi^{-1}(x)=\PP({\mathcal{N}}_{x,\B}^{\ast})$.
For $z\in\B\setminus T_x(\B)$, the line 
$l=\langle z,x \rangle$ 
corresponds to a point $w\in \pi^{-1}(x)=\PP({\mathcal{N}}_{x,\B}^{\ast})$.
Moreover $\pi'(w)=\varphi(l)\in \B'$ and 
 $\tau_{x,\B}(z)=\tau_{x,\B}(l)=w$, so that 
$\pi'(w)\in \pi'(\pi^{-1}(x))\cap \B'$
and hence $W_{x,\B}$ identifies with an irreducible component of
 $\pi'(\pi^{-1}(x))\cap \B'$.
Since $W_{x,\B}$ is a nondegenerate hypersurface  
in $\PP({\mathcal{N}}_{x,\B}^{\ast})$ (see Proposition \ref{prop: nondegenerate tang projection}),   
we have 
\begin{equation}
W_{x,\B}\simeq\pi'(\pi^{-1}(x))\cap \B'\subset\pi'(\pi^{-1}(x))\simeq\PP^9.
\end{equation}
\end{proof} 
\begin{claim}\label{step: smooth image tangential projection}
For the general point $x\in\B$, $W_{x,\B}\subset\PP^9$ is 
a smooth quadric hypersurface.
\end{claim}
\begin{proof}[Proof of the Claim]
By Theorem \ref{prop: main qel1}
it follows that 
$\L_{x,\B}\subset\PP^{14}$ is a smooth 
$QEL$-variety of dimension $\dim(\L_{x,\B})=9$ and of
type $\delta(\L_{x,\B})=5$.
So, applying  Theorem \ref{prop: LQEL of higher type}, we get that
$\L_{x,\B}\subset\PP^{14}$ is projectively equivalent to a hyperplane section 
 of the spinorial variety $S^{10}$.
 Hence,
by  Proposition \ref{prop: surjective second fundamental form}
and Example \ref{example: d=4 Delta=2},
it follows that 
 $W_{x,\B}\subset\PP^9$ is
a smooth quadric 
 of dimension $r-\delta=8$. 
\end{proof} 
Summing up the previous claims, we obtain that for the general point
$p\in\pi'(E)\setminus Y$ the entry locus $\Sigma_p(\B')$ is a smooth quadric;
moreover,   one can be easily convinced that we also have
$\Sigma_{p}(\B')=L_p(\B')\cap \B'=L_p(\B')\cap Y=L_p(Y)\cap Y = \Sigma_p(Y)$.
Now, by the proof of  Proposition \ref{prop: dimension formula}, 
it follows that $\pi'(E)$ is a reduced and irreducible divisor 
in $|\O_{\Q}(3)|$ and,
since $\pi'(E)=\left(\Sec(Y)\cap \Q\right)_{\mathrm{red}}$, 
we have that $\Sec(Y)$ is a hypersurface of degree a multiple of $3$. In particular
$Y\subset\PP^{26}$ is an irreducible nondegenerate variety
 with secant defect $\delta(Y)=8$.
Note also that $Y$ is different from a cone: if
 $z_0$ is a vertex of $Y$,
since $I_{\B',\PP^{26}}\subseteq I_{Y,\PP^{26}}$, 
we have that $z_0$ is a vertex of every quadric defining
 $\B'$, in particular $z_0$ is a vertex of $\B'$ and $\Q$;
then, for a general point $z\in\Q$, we have 
$\langle z, z_0 \rangle\subseteq\overline{(\varphi^{-1})^{-1}(\varphi^{-1}(z))}$, 
against the birationality  of $\varphi^{-1}$. 
\begin{claim}\label{step: Y is QEL}
$Y$ is a $QEL$-variety. 
\end{claim}
\begin{proof}[Proof of the Claim]
Consider the rational map
$\psi:\Sec(Y)\dashrightarrow \PP^{26}$, 
defined by the linear system
 $|\I_{\B',\PP^{26}}(2)|$ restricted to $\Sec(Y)$. 
Since $\overline{\psi(\Sec(Y)\cap \Q)}=\overline{\varphi^{-1}(\pi'(E))}=\B$ and 
the image of $\psi$ is 
of course nondegenerate, it  follows that the dimension
of the image of
 $\psi$ is at least $16$ and hence that the dimension 
of its general fiber is at most $9$. 
Now, if $q\in \Sec(Y)\setminus Y$ is a general point, denote by
 $\widetilde{\psi^{-1}(\psi(q))}$ the irreducible component
 of $\overline{\psi^{-1}(\psi(q))}$ through
  $q$ and by $\widetilde{\Sigma_q(Y)}$ an any irreducible component
 of $\Sigma_q(Y)$; note
 that, from generic smoothness \cite[\Rmnum{3} Corollary~10.7]{hartshorne-ag}, it follows 
that $\overline{\psi^{-1}(\psi(q))}$ 
(and at fortiori $\widetilde{\psi^{-1}(\psi(q))}$) is smooth in its general point  $q$.  
We have $S(q,\widetilde{\Sigma_q(Y)})\subseteq \widetilde{\psi^{-1}(\psi(q))}$, and since 
$\dim(\widetilde{\Sigma_q(Y)})=8$, 
it follows $S(q,\widetilde{\Sigma_q(Y)})= \widetilde{\psi^{-1}(\psi(q))}$. 
Thus the cone $S(q,\widetilde{\Sigma_q(Y)})$ is smooth in its vertex and  
necessarily it follows  
$S(q,\widetilde{\Sigma_q(Y)})=\PP^9$ and $\langle \Sigma_q(Y) \rangle=\PP^9$.
Finally, by  Trisecant Lemma \cite[Proposition~1.3.3]{russo-specialvarieties}, it
follows that $\Sigma_q(Y)\subset\PP^9$ is a quadric hypersurface.
\end{proof} 
\begin{claim}\label{step: gammaY=deltaY} 
$\widetilde{\gamma}(Y)=0$.
\end{claim}  
\begin{proof}[Proof of the Claim]
We first show that 
for the general point $q\in\Sec(Y)\setminus Y$ 
the entry locus $\Sigma_q(Y)$ is smooth,  by discussing two cases:
\begin{case}[Suppose $\pi'(E)\nsubseteq \sing(\Sec(Y))$]
Denote by
 $\Hilb(Y)$ the Hilbert scheme of  $8$-dimensional quadrics contained in $Y$
and by $V$ a nonempty open set of $\Sec(Y)\setminus Y$ such that for every  
$q\in V$ we have $\Sigma_q(Y)\in \Hilb(Y)$. 
If $\varrho:Y\times \left(Y\times\PP^{26}\right)\longrightarrow Y\times\PP^{26}$ is the projection, 
at the closed subscheme of $Y\times V$,
\begin{displaymath}
 \varrho\left(\overline{\{(w,z,q)\in Y\times Y\times \PP^{26}: 
w\neq z\mbox{ and } q\in\langle w,z\rangle\}}\right)\cap Y\times V    
=  \{ (z,q)\in Y\times V : z\in \Sigma_q(Y) \},
\end{displaymath}
corresponds a rational map 
$\nu:\Sec(Y)\dashrightarrow \Hilb(Y)$ which sends the point
 $q\in V$ to the quadric $\nu(q)=\Sigma_q(Y)$;
denote by  
$\mathrm{Dom}(\nu)$ the largest open set  of $\Sec(Y)$ where
 $\nu$ can be defined.
By assumption we have
 $D':=\pi'(E)\cap\reg(\Sec(Y))\neq\emptyset$.
It follows that the rational map 
$\nu':=\nu|_{\reg(\Sec(Y))}:\reg(\Sec(Y))\dashrightarrow \Hilb(Y)$,
having indeterminacy locus of codimension $\geq 2$, is defined in the general point 
 of $D'$, i.e.
\begin{equation}\label{eq: Dom intersects divisor}
\emptyset\neq\mathrm{Dom}(\nu')\cap D'=\mathrm{Dom}(\nu)\cap\pi'(E)\cap \reg(\Sec(Y))\subseteq \mathrm{Dom}(\nu)\cap\pi'(E).
\end{equation}
Now, consider the natural map
$\rho:\Hilb(Y)\longrightarrow\GG(9,26)$, defined by $\rho(Q)=\langle Q \rangle$
and  the closed subset 
$C:=\{ (q,L)\in \PP^{26}\times \GG(9,26): q\in L  \}$ 
of $\PP^{26}\times \GG(9,26)$.
We have
\begin{eqnarray*}
&& \left(\left(\mathrm{Id}_{\PP^{26}}\times\rho\right)^{-1}(C)\cap \mathrm{Dom}(\nu)\times \Hilb(Y)\right)\bigcap \Graph(\nu:\mathrm{Dom}(\nu)\to\Hilb(Y)) \\
&=&\left\{(q,Q)\in \mathrm{Dom}(\nu)\times \Hilb(Y): Q=\nu(q)\mbox{ and }q\in\langle Q\rangle \right\} \\
&\simeq& \left\{ q\in \mathrm{Dom}(\nu): q\in \langle \nu(q)\rangle\right\}=:T,
\end{eqnarray*}
from which it follows that the set  $T$ 
is closed in $\mathrm{Dom}(\nu)$ and,
since $\mathrm{Dom}(\nu)\supseteq T\supseteq V$, we have 
\begin{equation}\label{eq: property of Dom}
\mathrm{Dom}(\nu)=\overline{T}=T=\left\{ q\in \mathrm{Dom}(\nu): q\in \langle \nu(q)\rangle\right\}.
\end{equation}
By (\ref{eq: Dom intersects divisor}) and (\ref{eq: property of Dom}) 
it follows that, for the general point $p\in\pi'(E)\setminus Y$, we have $p\in\langle\nu(p)\rangle$, 
hence $\langle\nu(p)\rangle\subseteq L_p(Y)$  and then $\nu(p)=\Sigma_p(Y)$.
Thus $\nu(\mathrm{Dom}(\nu))$ intersects the open  $U(Y)$ of 
$\Hilb(Y)$ consisting of the smooth $8$-dimensional quadrics 
 contained in $Y$ and then,  for a general point 
 $q$ in the nonempty open set $\nu|_{\mathrm{Dom}(\nu)}^{-1}(U(Y))$, 
the entry locus $\Sigma_q(Y)$ is a smooth $8$-dimensional quadric.
\end{case}
\begin{case}[Suppose $\pi'(E)\subseteq \sing(\Sec(Y))$]\footnote{Note that 
this is the only place in the proof where we use the smoothness of $\Q$.
Note also that 
this case does not arise if one first shows that $\deg(\Sec(Y))=3$.
Indeed,  let $\mathcal{C}\subset\PP^N$ be a cubic hypersurface
 and $D\subset \mathcal{C}$ an irreducible divisor contained in $\sing(\mathcal{C})$.
Then, for a general plane $\PP^2\subset\PP^N$, putting $C=\mathcal{C}\cap \PP^2$ and $\Lambda=D\cap \PP^2$,
we have $\Lambda\subseteq \sing(C)$. Hence, by   B\'ezout's Theorem, being $C$ 
a plane cubic curve, it follows
 $\deg(D)=\#(\Lambda)=1$.} Let $z\in Y$ be a general point,
 $\tau_{z,Y}:Y\dashrightarrow W_{z,Y}\subset \PP^{9}$ 
the tangential projection (note that $W_{z,Y}$ is a nonlinear hypersurface, by
Proposition \ref{prop: nondegenerate tang projection}) and   
let $q\in \Sec(Y)\setminus Y$ be a general point. 
Arguing as in \cite[Claim~8.8]{chiantini-ciliberto} we obtain that 
$T_z(Y)\cap L_q(Y)=T_z(Y)\cap \langle \Sigma_q(Y) \rangle=\emptyset$, so we deduce 
that $\tau_{z,Y}$ isomorphically maps $\Sigma_q(Y)$ to $W_{z,Y}$. 
From this and Terracini Lemma it follows that 
\begin{equation}\label{eq: identity of joints}
\Sec(Y)=S(\Sigma_q(Y),Y):=\overline{\bigcup_{\substack{(z_1,z_2)\in \Sigma_q(Y)\times Y \\  z_1\neq z_2}} \langle z_1, z_2\rangle} .
\end{equation}
Now suppose by contradiction that there exists $w\in \mathrm{Vert}(\Sigma_q(Y))$.
By (\ref{eq: identity of joints}) it follows that $w\in \mathrm{Vert}(\Sec(Y))$ and 
by assumption we obtain $\pi'(E)\subseteq S(w,\pi'(E))\subseteq\sing(\Sec(Y))$.
Thus $w\in\mathrm{Vert}(\pi'(E))$ and hence $T_w(\pi'(E))=\PP^{26}$. This yields that
$\Q$ is singular in $w$, a contradiction.
\end{case}
Finally, since for general points $z\in Y$ and $q\in\Sec(Y)\setminus Y$, we have 
$W_{z,Y}=\overline{\tau_z(Y)}\simeq\Sigma_q(Y)$ and $\Sigma_q(Y)$ is a smooth quadric, 
we deduce that $W_{z,Y}$ is smooth.
It follows that the Gauss map 
$G_{W_{z,Y}}:W_{z,Y}\dashrightarrow\GG(8, 9)=\left(\PP^{9}\right)^{\ast}$ 
is birational onto its image  (see Corollary \ref{prop: birational gauss})
 and hence the dimension 
of its general fiber is $\widetilde{\gamma}(Y)=\gamma(Y)-\delta(Y)=0$.
\end{proof}
Now, consider two  general points $z_1,z_2\in Y$ ($z_1\neq z_2$),
a general point 
 $q\in\langle z_1, z_2\rangle$ and put
 $\Sigma_q=\Sigma_q(Y)$, $L_q=L_q(Y)$, $H_q=T_q(\Sec(Y))$.
Let $\pi_{L_q}:Y\dashrightarrow\PP^{16}$ be the linear projection from $L_q$ and
$H'_q\simeq\PP^{25}\subset\PP^{16}$ the projection of $H_q$ from $L_q$.
We note that,  by the proof of the
Scorza Lemma in \cite{russo-specialvarieties}, it follows that 
$\Sigma_q=\overline{\tau_{z_1,Y}^{-1}(\tau_{z_1,Y}(z_2))}$ and,
 since the image
 of the general tangential projection  is smooth of 
dimension $8$ and the general entry locus is smooth 
of dimension $8=\dim(Y)-\dim(W_{z_1,Y})$, it follows
 that $Y$ is smooth along $\Sigma_q(Y)$.
In particular,  it 
follows that $H_q$ is tangent to
$Y$ along $\Sigma_q$. 
\begin{claim}\label{step: birationality}
$\pi_{L_q}$ is birational.
\end{claim}
\begin{proof}[Proof of the Claim]
See also \cite[Proposition~3.3.14]{russo-specialvarieties}.
By the generality of the points
 $z_1,z_2,q$ it follows 
\begin{equation}\label{eq: 1, step birationality}
 L_q=\langle T_{z_1}(Y),z_2 \rangle \cap \langle T_{z_2}(Y),z_1 \rangle.
\end{equation}
The projection from the linear space
 $\langle T_{z_1}(Y),z_2 \rangle$ can be obtained as the 
composition of the tangential projection 
$\tau_{z_1,Y}:Y\dashrightarrow W_{z_1,Y}\subset\PP^9$ and 
of the projection of $W_{z_1,Y}$ from the point $\tau_{z_1,Y}(z_2)$.
 Thus the projection from $\langle T_{z_1}(Y),z_2 \rangle$, 
$\pi_{z_1,z_2}:Y\dashrightarrow\PP^8$ is dominant and 
 for the general point $z\in Y$ we get 
\begin{equation}\label{eq: 2, step birationality}
  \langle T_{z_1}(Y),z_2, z \rangle \cap Y \setminus \langle T_{z_1}(Y),z_2 \rangle 
= \pi_{z_1,z_2}^{-1}(\pi_{z_1,z_2}(z))=Q_{z_1,z}\setminus \langle T_{z_1}(Y),z_2 \rangle,
\end{equation}
where $Q_{z_1,z}$ denotes the entry locus of $Y$ with respect to a general point on $\langle z_1, z \rangle$.
Similarly  
\begin{equation}\label{eq: 3, step birationality}
  \langle T_{z_2}(Y),z_1, z \rangle \cap Y \setminus \langle T_{z_2}(Y),z_1 \rangle 
= \pi_{z_2,z_1}^{-1}(\pi_{z_2,z_1}(z))=Q_{z_2,z}\setminus \langle T_{z_2}(Y),z_1 \rangle.
\end{equation}
Now 
$
 \pi_{L_q}^{-1}(\pi_{L_q}(z))=\langle L_q,z \rangle\cap Y \setminus \Sigma_q 
$
and, by the generality of $z$, 
\begin{equation}\label{eq: 4, step birationality}
\pi_{L_q}^{-1}(\pi_{L_q}(z))=\langle L_q,z \rangle\cap Y \setminus H_q.
\end{equation}
By (\ref{eq: 1, step birationality}), (\ref{eq: 2, step birationality}), (\ref{eq: 3, step birationality}) and (\ref{eq: 4, step birationality}) 
and observing that the spaces $\langle T_{z_1}(Y),z_2\rangle$, 
$\langle T_{z_2}(Y),z_1\rangle$ are contained in $H_q$,
it follows 
\begin{equation}\label{eq: 5, step birationality}
\{z\}\subseteq \pi_{L_q}^{-1}(\pi_{L_q}(z))\subseteq Q_{z_1,z}\cap Q_{z_2,z}.
\end{equation}
Finally, as
 we have already observed  in  Claim \ref{step: gammaY=deltaY}, the restriction
of the tangential projection
 $\tau_{z_1,Y}$ to $Q_{z_{2},z}$ is an isomorphism
$\bar{\tau}:=(\tau_{z_1,Y})|_{Q_{z_{2},z}}:Q_{z_2,z}\rightarrow W_{z_1,Y}$;
hence 
\begin{equation}\label{eq: 6, step birationality}
\{z\}=\bar{\tau}^{-1}(\bar{\tau}(z))=\tau_{z_1,Y}^{-1}(\tau_{z_1,Y}(z))\cap Q_{z_2,z}=Q_{z_1,z}\cap Q_{z_2,z}. 
\end{equation}
By (\ref{eq: 5, step birationality}) and (\ref{eq: 6, step birationality}), it follows $\pi_{L_q}^{-1}(\pi_{L_q}(z))=\{z\}$
and hence the birationality of $\pi_{L_q}$.
\end{proof}
\begin{claim}\label{step: isomorphism}
$\pi_{L_q}$ induces an isomorphism
$Y\setminus H_q\stackrel{\simeq}{\longrightarrow} \PP^{16}\setminus H'_q$.
\end{claim}
\begin{proof}[Proof of the Claim]
We resolve  the indeterminacies of $\pi_{L_q}$ with the diagram
\begin{displaymath}
\xymatrix{& \Bl_{\Sigma_{q}}(Y) \ar[dl]_{\alpha} \ar[dr]^{\widetilde{\pi_{L_q}}}\\
Y\ar@{-->}[rr]^{\pi_{L_q}}& & \PP^{16}}
\end{displaymath}
The morphism $\widetilde{\pi_{L_q}}$ is projective and birational
 and hence surjective. Moreover, the
  points of the base locus  of $\pi_{L_q}^{-1}$ are the points for which
  the fiber of $\widetilde{\pi_{L_q}}$ has positive dimension.
Since $H_q\supseteq L_q$ and $H_q\cap Y=\overline{\pi_{L_q}^{-1}(H'_q)}$,
in order to prove the assertion, it suffices to
show  that, for every $w\in\PP^{16}\setminus H'_q$,
 $\dim\left(\widetilde{\pi_{L_q}}^{-1}(w)\right)=0$. 
Suppose by contradiction that there exists $w\in \PP^{16}\setminus H'_q$ 
such that $Z:=\widetilde{\pi_{L_q}}^{-1}(w)$ has positive dimension. 
Then, for the choice  of $H_q$, 
we have 
$\emptyset=Z\cap \alpha^{-1}(H_q\cap Y)\supseteq Z\cap \alpha^{-1}(\Sigma_{q})$ and 
therefore 
$\alpha(Z)$ contains an irreducible curve $C$ 
with $\pi_{L_q}(C)=w$ and $C\cap L_q=\emptyset$, against the fact
that a linear projection, when is defined everywhere, is a finite morphism. 
\end{proof}
\begin{claim}\label{step: Y smooth}
$Y$ is smooth.
\end{claim}
\begin{proof}[Proof of the Claim]
Suppose that there exists a point
 $z_0$ with
$$z_0\in
\bigcap_{\substack{ q\in\Sec(Y) \\ \mathrm{generale}} } T_q\left(\Sec(Y)\right)\cap Y = 
  \bigcap_{q\in\Sec(Y)} T_q\left(\Sec(Y)\right)\cap Y =
  \mathrm{Vert}\left(\Sec(Y)\right)\cap Y .$$
If $z\in Y$ is a general point, since $Y$ is not a cone, 
the tangential projection $\tau_{z,Y}$ is defined in $z_0$ and
it follows that 
$\tau_{z,Y}(z_0)$ is a vertex of $W_{z,Y}$.
This contradicts Claim \ref{step: gammaY=deltaY} and
 hence we have
$\bigcap_{\substack{ q\in\Sec(Y) \\ \mathrm{generale}} } T_q\left(\Sec(Y)\right)\cap Y=\emptyset$,
from which we conclude by Claim \ref{step: isomorphism}.
\end{proof}
Now we can conclude the proof of
 Theorem \ref{prop: classification of type 2-2 into quadric}.
 By  Claim \ref{step: Y smooth} it follows that $Y\subset\PP^{26}$ 
is a Severi variety, so by their 
classification (see Theorem \ref{prop: classification severi varieties}) it follows
 that $Y=E_6$;
moreover, since
$
27=h^0(\PP^{26}, \I_{\B',\PP^{26}}(2))\leq 
h^0(\PP^{26}, \I_{Y,\PP^{26}}(2))=27,
$
we have $Y=\B'$. Now, by the
classification of the special Cremona transformations of type $(2,2)$
 in Theorem \ref{prop: classification quadro-quadric special cremona} 
(or also by a direct calculation), it follows that the lifting 
$\psi:\PP^{26}\dashrightarrow\PP^{26}$ 
of $\varphi^{-1}:\Q\dashrightarrow\PP^{25}=:H\subset\PP^{26}$ is
a birational transformation of type $(2,2)$ and therefore 
the base locus  $\widehat{\B}\subset\PP^{26}$ of the inverse of 
 $\psi$ is again the variety $E_6$.
Of course $\widehat{\B}\cap H = \B$ and hence the thesis. 
\end{proof} 
\begin{remark}
We observe that from
 Claim \ref{step: gammaY=deltaY} and Proposition \ref{prop: criterion R1 property} 
it follows
that $Y$ is a $R_1$-variety and hence by 
Theorem \ref{prop: R1 varieties and severi varieties} 
it follows
Claim \ref{step: Y smooth}.
\end{remark}

\section{Transformations whose base locus has dimension \texorpdfstring{$\leq3$}{<=3}}\label{sec: small dimension of B}
Let $\varphi$ be a special transformation  as
in Assumption \ref{notation: factorial hypersurface} and let $r\leq3$.  
From  Proposition \ref{prop: dimension formula} 
we get the following possibilities for $(r,n)$:    $(1,4)$; $(2,6)$;  $(3,7)$; $(3,8)$.
If $(r,n)\in\{(1,4),(3,7)\}$ then $(d,\Delta)=(2,2)$
and these cases  
have already been classified in 
Theorem \ref{prop: classification of type 2-2 into quadric}.
\subsection{Case \texorpdfstring{$(r,n)=(2,6)$}{(r,n)=(2,6)}}
\begin{proposition}\label{prop: 2-fold in P6}
 Let $\varphi:\PP^6\dashrightarrow\overline{\varphi(\PP^6)}=\sS\subset\PP^7$ be
 birational and special of type $(2,d)$,   with
$\sS$ a factorial hypersurface of degree $\Delta\geq2$.  
If $r=\dim(\B)=2$, then
 $\B$ is the blow-up $\sigma:\Bl_{\{p_0,\ldots,p_5\}}(\PP^2)\rightarrow\PP^2$ 
of $6$ points in the plane with $H_{\B}\sim \sigma^{\ast}(4H_{\PP^2})-2E_0-E_1-\cdots-E_5$ ($E_0,\ldots,E_5$ are the exceptional divisors).     
Moreover, we have $d=3$ and $\Delta=2$.
\end{proposition}
\begin{proof}
 By Lemma \ref{prop: cohomology twisted ideal}, it follows 
$\chi(\B,\O_{\B}(1))=7$ and $\chi(\B,\O_{\B}(2))=20$,   from which we deduce
$$
P_{\B}(t)=\left(\lambda\,{t}^{2}+\left( 26-3\,\lambda\right) \,t+2\,\lambda-12\right)/2=\left(\left( g+12\right)\,{t}^{2} +\left( 16-3\,g\right)\,t+2\,g \right)/4
$$
and hence 
$g=2(\lambda-6)$.     In 
particular $\lambda\geq6$,   being $g\geq0$.
Now, if $\lambda\leq 2\,\mathrm{codim}_{\PP^6}(\B)+1=9$,   cutting 
 $\B$ with a general $\PP^4\subset\PP^6$,
we obtain a set 
$\Lambda\subset\PP^4$ of $\lambda$ points
 that imposes independent conditions to the quadrics
 of $\PP^4$ (Lemma \ref{prop: castelnuovo argument});
hence 
$
 h^0(\PP^6,\I_{\B,\PP^6}(2))\leq h^0(\PP^4,\I_{\Lambda,\PP^4}(2))=h^0(\PP^4,\O_{\PP^4}(2))-\lambda 
$, 
 i.e.  $\lambda\leq7$.   
  Moreover, if $\lambda\geq 9$,   $\Lambda$ would impose at 
least 
 $9$ conditions to the quadrics  and hence we would get the contradiction 
$h^0(\PP^6,\I_{\B
}(2))\leq 6$.
Hence $\lambda=6$ or $\lambda=7$,   and in both cases, knowing 
the expression of the Hilbert polynomial, we conclude
applying 
 \cite{ionescu-smallinvariants}:
if $\lambda=6$,   such a variety does not exist; 
if $\lambda=7$,   then $\B$ is as asserted. 
Finally,  by
Remark \ref{remark: dimension formula without hypothesis}, 
we get either $(d,\Delta)=(3,2)$ or $(d,\Delta)=(2,3)$, 
but the latter case  is impossible by Example \ref{example: d=3 Delta=2}
(the pair $(d,\Delta)$ may also be determined 
by calculating the Chern classes of $\B$).
\end{proof}
\subsection{Case \texorpdfstring{$(r,n)=(3,8)$}{(r,n)=(3,8)}}
Firstly we observe that if $(r,n)=(3,8)$  by Remark \ref{remark: dimension formula without hypothesis} 
it follows  $d+\Delta=6$ and hence we have $(d,\Delta)\in\{(2,4),(3,3),(4,2)\}$.
\begin{remark}\label{rem: chern classes}
See also Chap. \ref{chapter: transformations whose base locus has dimension at most three}, p. \pageref{prop: segre and chern classes}.
Let notation be as in the proof of 
Proposition \ref{prop: dimension formula} and  let $(r,n)=(3,8)$. 
Denote by $c_j:=c_j(\mathcal{T}_{\B})\cdot H_{\B}^{3-j}$ 
(resp. $s_j:=s_j(\mathcal{N}_{\B,\PP^8})\cdot H_{\B}^{3-j}$), for $1\leq j\leq 3$,
the degree of the $j$-th Chern class (resp. Segre class) of $\B$.
From the exact sequence
$0\rightarrow\mathcal{T}_{\B}\rightarrow \mathcal{T}_{\PP^8}|_{\B}\rightarrow\mathcal{N}_{\B,\PP^8}\rightarrow0$
we get: $s_1=c_1-9\lambda$,
$s_2=c_2-9c_1+45\lambda$, 
$s_3=c_3-9c_2+45c_1-165\lambda$.
Moreover
\begin{eqnarray*}
\lambda&=& H_{\B}^3=-K_{\B}\cdot H_{\B}^2+2g-2-\lambda 
= s_1+8\lambda+2g-2, \\
d \Delta &=& d{H'}^{8}={H'}^{7}\cdot(dH'-E')=(2H-E)^7\cdot H \\
&=& -H\cdot E^7+14H^2\cdot E^6-84H^3\cdot E^5 + 128H^8 
= -s_2 -14 s_1 -84 \lambda +128, \\
\Delta &=& {H'}^8=(2H-E)^8  
= E^8-16H\cdot E^7+112H^2\cdot E^6-448H^3\cdot E^5+256H^8 \\
&=& -s_3-16s_2-112s_1 -448\lambda+256,
\end{eqnarray*}
and hence
\begin{displaymath}
\left\{
 \begin{array}{l} 
  s_1=-7\lambda-2g+2, \\
s_2=14\lambda+28g-d\Delta+100, \\
s_3=112\lambda-224g+(16d-1)\Delta-1568, 
 \end{array}
\right.
\left\{
 \begin{array}{l} 
c_1=2\lambda-2g+2, \\
c_2=-13\lambda+10g-d\Delta+118, \\
c_3=70\lambda-44g+(7d-1)\Delta-596. 
 \end{array}
\right.
\end{displaymath}
\end{remark}

\begin{proposition}\label{prop: 3-fold in P8 - S nonlinear}
Let $\varphi:\PP^8\dashrightarrow\overline{\varphi(\PP^8)}=\sS\subset\PP^9$ be
 birational and special of type $(2,d)$, 
with $\sS$ a factorial hypersurface of degree $\Delta\geq2$.
 If $r=\dim(\B)=3$,
 then  one of the following cases holds:
\begin{enumerate}[(i)]
 \item\label{part: case 0, 3-fold} $\lambda=12$, $g=7$, $d=4$, $\Delta=2$, $\B$ is a linear section of the spinorial variety $S^{10}\subset\PP^{15}$;
 \item\label{part: case 0.5, 3-fold} $\lambda=12$, $g=7$, $d=2$, $\Delta=4$, $\B$ is a Mukai variety with Betti numbers  $b_2=2$, $b_3=18$;
 \item\label{part: case 1, 3-fold} $\lambda=11$, $g=5$, $d=3$, $\Delta=3$, $\B$ is the variety $\mathfrak{Q}_{p_1\ldots,p_5}$ defined as the blow-up 
of  $5$ points $p_1,\ldots,p_5$ (possibly infinitely near) in a smooth quadric $Q\subset\PP^4$,   
  with $H_{\mathfrak{Q}_{p_1,\ldots,p_5}}\sim \sigma^{\ast}(2{H_{\PP^4}}|_{Q})-E_1-\cdots-E_5$, 
where $\sigma$ is the blow-up map and $E_1,\ldots,E_5$ are 
the exceptional divisors; 
 \item\label{part: case 2, 3-fold} $\lambda=11$, $g=5$, $d=4$, $\Delta=2$, $\B$ is a scroll over 
                                    $\PP_{\PP^1}(\O\oplus\O(-1))$.  
\end{enumerate}
\end{proposition}
\begin{proof}
 By Proposition \ref{prop: cohomology properties},  
 $\B\subset\PP^8$ is nondegenerate and linearly normal. 
 Let $\Lambda\subset C\subset S \subset \B \subset \PP^8$ be a sequence of
 general linear sections 
 of $\B$.
By the exact sequence 
$ 0\rightarrow \I_{\B,\PP^{8}}(-1)\rightarrow\I_{\B,\PP^{8}}\rightarrow \I_{S, \PP^7} 
\rightarrow 0 $
and those similar for
 $S$ and $C$,  using that $\B$,   $S$,   $C$ are nondegenerate, we get
the inequality
$h^0(\PP^8,\I_{\B,\PP^8}(2))\leq h^0(\PP^5,\I_{\Lambda,\PP^5}(2))$.
In particular, putting
$h_{\Lambda}(2):=\dim(\mathrm{Im}(H^0(\PP^5,\O_{\PP^5}(2))\rightarrow 
H^0(\Lambda,\O_{\Lambda}(2))))$,
we have
\begin{equation}\label{eq: leq_h2}
h_{\Lambda}(2)\leq h^0(\PP^5,\O_{\PP^5}(2))-h^0(\PP^8,\I_{\B,\PP^8}(2))=11.
\end{equation}
If now $\#(\Lambda)=\lambda\geq 11$,   
taking $\Lambda'\subseteq\Lambda$ with  $\#(\Lambda')=11$,   
by Lemma \ref{prop: castelnuovo argument} we obtain 
\begin{equation}\label{eq: geq_h2}
h_{\Lambda}(2)\geq h^0(\PP^5,\O_{\PP^5}(2))-h^0(\PP^5,\I_{\Lambda',\PP^5}(2))=\#(\Lambda')=11.
\end{equation}
The inequalities
 (\ref{eq: leq_h2}) and (\ref{eq: geq_h2}) yield
$ h_{\Lambda}(2)=2\cdot6-1 $
and this, by Proposition \ref{prop: castelnuovo lemma},
yields a contradiction if $\lambda\geq 13$.
Thus we have 
$ \lambda\leq 12$
 and, by  Castelnuovo's bound (Proposition \ref{prop: castelnuovo bound}),
we also have 
\begin{equation}\label{eq: K_S.H_Sleq0}
K_S\cdot H_S = (K_{\B}+H_{\B})\cdot H_{\B}^2=2g-2-\lambda\leq0.
\end{equation}
We discuss two cases.
\begin{case}[Suppose $K_S\nsim 0$]\label{case: KS nsim 0}
By (\ref{eq: K_S.H_Sleq0}) and 
by the proof of \cite[\Rmnum{5} Lemma~1.7]{hartshorne-ag},
it follows  that
$ h^2(S,\O_S)=h^2(S,\O_S(1))=0 $.
Consequently, by Lemma \ref{prop: cohomology twisted ideal} and by
the exact sequence
$ 0\rightarrow\O_{\B}(-1)\rightarrow\O_{\B}\rightarrow\O_S\rightarrow0 $,
 we obtain
$$h^2(\B,\O_{\B})=h^3({\B},\O_{\B})=h^3({\B},\O_{\B}(-1))=0.$$
Moreover $h^1(\B,\O_{\B})=h^1(S,\O_S)=:q$ and using again  
 Lemma \ref{prop: cohomology twisted ideal} we obtain
\begin{equation}\label{eq: conditions hilbert pol}
\begin{array}{cc}
\chi({\B},\O_{\B}(-1))=0, & \chi({\B},\O_{\B})=1-q,\\
\chi({\B},\O_{\B}(1))=9, &\chi({\B},\O_{\B}(2))=35.
\end{array} \end{equation}
Now, the conditions (\ref{eq: conditions hilbert pol}) determine
 $P_{\B}(t)$
in function of $q$, 
 from which in particular we obtain 
 $\lambda=11-3q$,  $g=5-5q$.
Being $g\geq0$,   we have $(q,\lambda,g)=(0,11,5)$ or $(q,\lambda,g)=(1,8,0)$, 
but the latter case
is impossible by \cite[Theorems~10.2 and 12.1, Remark~12.2]{fujita-polarizedvarieties}.
 Thus we have
\begin{equation}\label{eq: irregularity is 0}
q=0,\ P_{\B}(t)=\left( 11{t}^{3}+21{t}^{2}+16t+6 \right)/6,\ K_S\cdot H_S=-3,\ g=5.
\end{equation}
Applying the main result in \cite{ionescu-degsmallrespectcodim} and the numerical constraints 
 in \cite{besana-biancofiore-deg11} and \cite{besana-biancofiore-numerical}, it follows immediately that
 $\B$ is one of the following:
\begin{enumerate}[(a)] 
 \item\label{part: case 1, 3-fold - proof} the variety $\mathfrak{Q}_{p_1,\ldots,p_5}$; 
 \item\label{part: case 2, 3-fold - proof} a scroll over a surface $Y$, where  $Y$ is either the 
blow-up  of $5$ points in $\PP^2$, or the rational ruled surface $\PP_{\PP^1}(\O\oplus\O(-1))$;
 \item\label{part: case 3, 3-fold - proof} a quadric fibration over $\PP^1$.     
\end{enumerate}
Now, if $\B$ is as in case (\ref{part: case 2, 3-fold - proof}),
we use the well-known 
relation (multiplicativity of the topological Euler characteristic)
 $c_3(\B)=c_1(\PP^1)c_2(Y)$, and by Remark \ref{rem: chern classes} we deduce  
$\Delta=(2c_2(Y)+46)/(7d-1)$.
Moreover, if $Y$ is the blow-up of $5$ points in $\PP^2$, we have 
$c_2(Y)=12\chi(\O_Y)-K_Y^2=12\chi(\O_{\PP^2})-(K_{\PP^2}^2-5)=8$,
while if $Y$ is $\PP_{\PP^1}(\O\oplus\O(-1))$, we have $c_2(Y)=4$.
Thus, if $\B$ is as in case (\ref{part: case 2, 3-fold - proof}), 
we have $Y=\PP_{\PP^1}(\O\oplus\O(-1))$, $d=4$ and $\Delta=2$.
If $\B$ is as in case (\ref{part: case 3, 3-fold - proof}), 
we easily deduce that $c_2(\B)=20$ and hence, again by Remark \ref{rem: chern classes},
we obtain the contradiction $d\Delta=5$.
Now suppose $\B$ as in case (\ref{part: case 1, 3-fold - proof}), namely
 $\B$ is obtained as a sequence 
$$
\B=Z_5\stackrel{\sigma_5}{\longrightarrow}Z_4\stackrel{\sigma_4}{\longrightarrow}\cdots 
\stackrel{\sigma_1}{\longrightarrow} Z_0=Q,
$$
 where $\sigma_j$ is the blow-up at a point $p_j\in Z_{j-1}$, 
$H_{Z_{j}}=\sigma_j^{\ast}(H_{Z_{j-1}})-E_j$, $E_j$ is the exceptional divisor
 and $H_{Z_0}=H_Q=2H_{\PP^4}|_{Q}$.
By \cite[page~609]{griffiths-harris} it follows that 
$c_2(Z_j)=\sigma_j^{\ast}(c_2(Z_{j-1}))$ and hence 
$$ c_2(Z_j)\cdot H_{Z_j}=
 \sigma_j^{\ast}(c_2(Z_{j-1}))\cdot \sigma_j^{\ast}(H_{Z_{j-1}})-\sigma_j^{\ast}(c_2(Z_{j-1}))\cdot E_j=
 c_2(Z_{j-1})\cdot H_{Z_{j-1}}.$$
In particular,  we obtain 
$c_2(\B)\cdot H_{\B}=2c_2(Q)\cdot H_{\PP^4}|_Q=16$. 
On the other hand, by Remark \ref{rem: chern classes}, we obtain 
that $c_2(\B)\cdot H_{\B}=25-d\Delta$, hence $d\Delta=9$.
 \end{case}
\begin{case}[Suppose $K_S\sim 0$] 
By Castelnuovo's bound, since $K_S\cdot H_S=0$, it follows that 
$(\lambda,g)=(12,7)$ and hence also that $\chi(S,\O_{S})=-3\lambda+2g+24=2$ (note that,
 as in the previous case, we know the values of $P_{\B}(1)$ and $P_{\B}(2)$).
We have $q=h^1(S,\O_S)=1-\chi(S,\O_S)+h^2(S,\O_S)=-1+h^2(S,K_S)=0$ and hence 
$S$ is a $K3$-surface, $C$ is a canonical curve
  and $\B$ is a Mukai variety.
Now we denote by $b_j=b_j(\B)$ the $j$-th Betti number of $\B$.
By Poincar\'e-Hopf index formula and Poincar\'e duality 
(see for example \cite{griffiths-harris}) 
we have $c_3(\B)=\sum_{j} (-1)^j b_j= 2+2b_2-b_3$ 
and, by Remark \ref{rem: chern classes}, we also have 
$c_3(\B)=-7d^2+43d-70$. Moreover, by \cite{mori-mukai}, if $b_2\geq 2$ then 
$(b_2,b_3)\in\{(2,12),(2,18),(3,16),(9,0)\}$. Thus, 
if $b_2\geq 2$ we have $b_2=2$, $b_3=18$, $d=2$, $\Delta=4$.
Finally, by Theorem \ref{prop: classification first species mukai varieties},
 if $b_2=1$ then $\B$ is a linear section 
of the spinorial variety $S^{10}\subset\PP^{15}$. 
Thus, we have a natural inclusion 
$\iota:H^0(\PP^{15},\I_{S^{10}}(2))\hookrightarrow H^0(\PP^{8},\I_{\B}(2))$
and, since $h^0(\PP^{15},\I_{S^{10}}(2))=h^0(\PP^{8},\I_{\B}(2))$,
we see that $\iota$ is an isomorphism.
This says that $\varphi$ is the restriction of the map 
$\psi:\PP^{15}\dashrightarrow\Q\subset\PP^9$ given in Example \ref{example: d=4 Delta=2}.
\end{case} 
\end{proof}

\begin{remark}[on case (\ref{part: case 2, 3-fold}) 
               of Proposition \ref{prop: 3-fold in P8 - S nonlinear}]
More precisely, from \cite[Proposition~4.2.3]{besana-biancofiore-deg11} it follows 
that $\B=\PP_{\mathbb{F}_1}(\mathcal{E})$,
where $(\mathbb{F}_1,H_{\mathbb{F}_1}):=(\PP_{\PP^1}(\O\oplus\O(-1)),C_0+2f)$ 
(notation as in \cite[page~373]{hartshorne-ag}) 
and $\mathcal{E}$ is a locally free sheaf of rank $2$ 
on $\mathbb{F}_1$, with $c_2(\mathcal{E})=10$.
$\mathbb{F}_1$ is thus the cubic surface of $\PP^4$ with ideal generated by: 
$x_0x_3-x_2x_1, x_0x_4-x_3x_1, x_2x_4-x_3^2$ and it is 
  isomorphic to
$\PP^2$ with one point blown up. 
We point out that 
the problem of the existence of an example for case (\ref{part: case 2, 3-fold}) 
of Proposition \ref{prop: 3-fold in P8 - S nonlinear}
  is essentially reduced to showing that 
such a scroll over $\FF_1$ must be cut out by quadrics. 
For further details we refer to Example \ref{example: 14} below.
\end{remark}

\begin{proposition}\label{prop: 3-fold in P8}
 Let $\varphi:\PP^8\dashrightarrow\overline{\varphi(\PP^8)}=\sS\subset\PP^9$ be
 birational and special of type $(2,d)$, with   
  $\sS$ a hypersurface as in Assumption \ref{notation: factorial hypersurface}.
If $r=\dim(\B)=3$,
then either 
case (\ref{part: case 0, 3-fold}), 
case (\ref{part: case 1, 3-fold}), or 
case (\ref{part: case 2, 3-fold}) of 
Proposition \ref{prop: 3-fold in P8 - S nonlinear} 
holds.
\end{proposition}
\begin{proof}[First proof of Proposition \ref{prop: 3-fold in P8}]
 We have to exclude case (\ref{part: case 0.5, 3-fold}) 
of proposition \ref{prop: 3-fold in P8 - S nonlinear}, so
we just assume that $\B$ is as in this case.
Using the fact that $K_{\B}\sim -H_{\B}$ and $(\lambda,g)=(12,7)$ 
we can compute
the Segre classes of the tangent bundle of $\B$:
\begin{eqnarray*}
 s_1(\mathcal{T}_{\B})\cdot H_{\B}^2 &=& -\lambda =-12 , \\
 s_2(\mathcal{T}_{\B})\cdot H_{\B}   &=&  -24+\lambda=-12 , \\
 s_3(\mathcal{T}_{\B})               &=&  -c_3(\B)+48-\lambda=100-(7d-1)\Delta. \\
\end{eqnarray*}
Since $\B$ is a $QEL$-variety of type $\delta=0$ we apply the
\emph{double point formula} 
(see  for example \cite{peters-simonis} and \cite{laksov} and also Proposition \ref{prop: double point formula} below)
\begin{equation}
 2(2d-1) = 2\deg(\Sec(\B)) 
= \lambda^2 - \sum_{j=0}^{3}\binom{7}{j} s_{3-j}(\mathcal{T}_{\B})\cdot H_{\B}^{j} 
= (7d-1)\Delta-40,
\end{equation}
from which we deduce 
$\Delta=(4d+38)/(7d-1)$ i.e. $d=4$ and $\Delta=2$.
\end{proof}
\begin{proof}[Second proof of Proposition \ref{prop: 3-fold in P8}]
Let $x\in\B$ be a general point and put $k=\#(\L_{x,\B})$.
Since the variety $\B$ is not a scroll over a curve  
(otherwise it would happen $\lambda^2\geq (2r+1)\lambda+r(r+1)(g-1)$, by \cite{besana-biancofiore-numerical}) 
and it is defined by
 quadrics, by \cite[Proposition~5.2]{ciliberto-mella-russo}, 
it follows that 
the support of the base locus
 of the tangential projection $\tau_{x,\B}:\B\dashrightarrow W_{x,\B}\subset\PP^4$, 
  i.e. $(T_x(\B)\cap\B)_{\mathrm{red}}$,   consists of $0\leq k< \infty$ lines through  $x$.
Now, by Proposition \ref{prop: dimension formula}, 
$\B$ is a $QEL$-variety of type $\delta=0$ and 
repeating the argument  
in \cite[\S  5]{ciliberto-mella-russo} (keeping also in mind 
Theorem \ref{prop: birational tangential projection when delta is 0}) we get  the relation 
$\lambda-8+k=\deg(W_{x,\B})$.
On the other hand, 
by proceeding as in Claim \ref{step: tang proj and entry locus}
or in \cite[Theorem~1.4]{mella-russo-baselocusleq3}, we also obtain
$\deg(W_{x,\B})\leq d$. Hence,  we deduce
\begin{equation}\label{eq: lambda-8+k leq d}
\lambda-8+k\leq d,
\end{equation}
from which the conclusion follows.
\end{proof}
\begin{remark}
Note that in case (\ref{part: case 1, 3-fold}) 
of Proposition \ref{prop: 3-fold in P8 - S nonlinear},
 by (\ref{eq: lambda-8+k leq d}) it follows that
$k=\#(\L_{x,\B})=0$.
We show directly that for a general point 
$x\in\mathfrak{Q}=\mathfrak{Q}_{p_1,\ldots,p_5}$,
we have $\L_{x,\mathfrak{Q}}=\emptyset$.
Suppose by contradiction that there exists $[l]\in\L_{x,\mathfrak{Q}}$.
Then, by (\ref{eq: dim LxX}),
$ 0=\dim_{[l]}(\L_{x,\mathfrak{Q}})
= -K_{\mathfrak{Q}}\cdot l -2$
and hence
\begin{equation}\label{eq: (K+2H)l=0}
(K_{\mathfrak{Q}}+2H_{\mathfrak{Q}})\cdot l = 0.
\end{equation}
Moreover, by \cite[\S  0.3]{ionescu-degsmallrespectcodim}, the adjunction map
 $\psi_{\mathfrak{Q}}$, i.e. the map defined by the complete linear system
 $|K_{\mathfrak{Q}}+2\,H_{\mathfrak{Q}}|$,    is everywhere defined  and we have  
a commutative diagram of adjunction maps
$$
\xymatrix{
 (\mathfrak{Q},H_{\mathfrak{Q}})\ar[r]^{\psi_{\mathfrak{Q}}}  \ar[d]_{\sigma} & \psi_{\mathfrak{Q}}(\mathfrak{Q}) \\ (Q, H_{Q}) \ar[ur]_{\psi_{Q}}
}
$$
where $\sigma$ is the blow-up map and $(Q,H_Q)=(Q^3\subset\PP^4,2\,H_{\PP^4}|_{Q})$.
Now
$ K_{Q}+2\,H_{Q}  \sim (K_{\PP^4}+Q)|_{Q}+2\,(2\,H_{\PP^4})|_{Q}  
 \sim (-5\,H_{\PP^4}+2\,H_{\PP^4}+4\,H_{\PP^4})|_{Q}\sim H_{\PP^4}|_{Q} $,
but this is in contradiction with (\ref{eq: (K+2H)l=0}). 
\end{remark}

\subsection{Summary results}
\begin{theorem} Table \ref{tab: cases with dimension leq 3} 
classifies all special quadratic birational transformations 
as in Assumption \ref{notation: factorial hypersurface} and with $r\leq 3$. 
\begin{table}[htbp]
\centering
\tabcolsep=10.3pt
\begin{tabular}{|c|c|c|c|c|c|c|c|}
\hline
 $r$ & $n$ & $\Delta$ & $d$ & $\delta$ & $\lambda$  & Abstract structure of $\B$ & Examples\\
\hline
\hline
 $1$ & $4$ & $2$      &  $2$& $0$      & $4$       & $(\PP^1,\O(4))$ & exist\\
\hline
 $2$ & $6$ & $2$      &  $3$& $0$      & $7$       & Hyperplane section  of an Edge variety & exist\\
\hline
 $3$ & $7$ & $2$      &  $2$& $1$      & $6$       & Hyperplane section  of $\PP^2\times\PP^2\subset\PP^8$ & exist\\
\hline
 $3$ & $8$ & $2$      & $4$ & $0$     & $12$       & Linear section of $S^{10}\subset\PP^{15}$ & exist\\
\hline
$3$ & $8$ & $3$      & $3$ & $0$     & $11$        &  $\mathfrak{Q}_{p_1,\ldots,p_5}$ & exist \\
\hline
$3$ & $8$ & $2$      & $4$ & $0$     & $11$        &  Scroll over $\PP_{\PP^1}(\O\oplus \O(-1))$ & not know\\
\hline
\end{tabular}
\caption{All transformations $\varphi$ as in Assumption \ref{notation: factorial hypersurface} and with $r\leq3$.}
\label{tab: cases with dimension leq 3}
\end{table}
\end{theorem}
\begin{corollary}\label{prop: classification type 2-3 into cubic}
Let $\varphi$ be of type $(2,3)$ and let $\Delta=3$.
Then $\B$ is the variety $\mathfrak{Q}_{p_1,\ldots,p_5}\subset\PP^8$.
\end{corollary}
\begin{proof}
By Propositions \ref{prop: dimension formula} and \ref{prop: d Delta both odd},
 $\B$ is a $QEL$-variety of type $\delta=0$ and dimension $3$ in $\PP^{8}$.
Hence we apply Proposition \ref{prop: 3-fold in P8}.
\end{proof}
In Chap. \ref{chapter: transformations whose base locus has dimension at most three}, 
 we will extend these results
in the case where $\sS$ is not necessarily a hypersurface.

\section
{Invariants of transformations of type \texorpdfstring{$(2,2)$}{(2,2)}  into a cubic and a quartic}\label{sec: invariants 2-2}

\begin{remark}\label{rem: LxB linearly normal}
Let $\delta\geq 3$ and consider $\L_{x,\B}\subset\PP^{r-1}$, 
where  $x\in\B$ is a general point. 
By Theorem \ref{prop: main qel1} and 
Proposition \ref{prop: relation between hilbert scheme and second funfamental form} 
it follows that 
$\L_{x,\B}$ is a smooth irreducible nondegenerate
variety  
of codimension $(r-\delta+2)/2$ and it is scheme-theoretic intersection of
 quadrics.
Then, applying \cite[Corollary~2]{bertram-ein-lazarsfeld}, we get that $\L_{x,\B}$ is
linearly normal.
\end{remark}
In  Propositions \ref{prop: invariants d=2 Delta=3} and \ref{prop: invariants d=2 Delta=4} we write 
$P=a_0,a_1,\ldots, a_r$ to indicate that 
$
P_{\B}(t)=a_0 \binom{ t }{ r } + a_1 \binom{ t }{ r-1 } + \cdots + a_r 
$.
\begin{proposition}\label{prop: invariants d=2 Delta=3} Let $\varphi$ be of type $(2,2)$ 
and let $\Delta=3$. 
Then
$\B$ is a $QEL$-variety of type $\delta$ 
and a Fano variety of the first species 
of  coindex $c$,   as one of the following cases: 
\begin{enumerate}[(i)]
 \item\label{part: first, invariants d=2 Delta=3} $n=18$,   $r=10$,   $\delta=4$,   $c=4$, 
$P=$ $34$, $272$, $964$, $1988$, $2633$, $2330$, $1387$, $544$, $133$, $18$, $1$;
 for the general point $x\in\B$, $\L_{x,\B}\subset\PP^9$ is projectively equivalent to 
$\PP^1\times\PP^4\subset\PP^9$.
\item\label{part: second, invariants d=2 Delta=3} $n=24$,   $r=14$,   $\delta=6$,   $c=5$, 
$P=$ $80$, $920$, $4866$, $15673$, $34302$, $53884$, $62541$, $54366$, $35472$, $17228$, $6104$, $1521$, $250$, $24$, $1$; 
 for the general point $x\in\B$, $\L_{x,\B}\subset\PP^{13}$ is 
projectively equivalent to a smooth $8$-dimensional linear section  
  of $S^{10}\subset\PP^{15}$.
\end{enumerate}
\end{proposition}
\begin{proof}
By Proposition \ref{prop: dimension formula} and Proposition \ref{prop: cohomology properties} 
parts \ref{part: B is linearly normal} and \ref{part: B is fano}, 
and applying  
Theorems \ref{prop: classification CC-varieties} and \ref{prop: divisibility theorem},
it follows that
 $\B$ is a $QEL$-variety of type $\delta$ and 
$$(n,r,\delta)\in\{(6,2,0),
(12,6,2), (18,10,4), (24,14,6)\}.$$
 The tern $(6,2,0)$ is excluded by Proposition  \ref{prop: 2-fold in P6}; the tern 
$(12,6,2)$ is excluded since otherwise 
by Proposition \ref{prop: cohomology properties} part \ref{part: hilbert polynomial},
we would get incompatible conditions 
for $P_{\B}(t)$. 
The statement on $\L_{x,\B}$, in the case (\ref{part: first, invariants d=2 Delta=3}) 
follows from Theorem \ref{prop: classification CC-varieties}, 
while in the case (\ref{part: second, invariants d=2 Delta=3}) 
it follows  from Theorem \ref{prop: classification first species mukai varieties}  
or Theorem \ref{prop: LQEL of type delta=r/2}.
\end{proof}

\begin{proposition}\label{prop: invariants d=2 Delta=4}
Let $\varphi$ be of type $(2,2)$ and let $\Delta=4$. 
 Then $\B$ 
is a $QEL$-variety of type $\delta$ 
and a Fano variety 
of the first species
of coindex $c$, as one of the following cases:
\begin{enumerate}[(i)]
\item\label{part: first, invariants d=2 Delta=4} $n=17$,   $r=9$,   $\delta=3$,   $c=4$, 
$P=$ $35$, $245$, $747$, $1297$, $1406$, $980$, $435$, $117$, $17$, $1$;
 for the general point $x\in\B$, $\L_{x,\B}\subset\PP^{8}$ is 
projectively equivalent to 
$\PP_{\PP^1}(\O(1)\oplus\O(1)\oplus\O(1)\oplus\O(2))\subset\PP^8$.
\item\label{part: second, invariants d=2 Delta=4} $n=23$,   $r=13$,   $\delta=5$,   $c=5$, 
$P=$ $82$, $861$, $4126$, $11932$, $23195$, $31943$, $31984$, $23504$, $12628$, $4875$, $1306$, $228$, $23$, $1$; 
  for the general point  $x\in\B$, $\L_{x,\B}\subset\PP^{12}$ is 
projectively equivalent to a 
 smooth $7$-dimensional linear section  of $S^{10}\subset\PP^{15}$.
\end{enumerate}
\end{proposition}
\begin{proof}
 As in the proof  of Proposition \ref{prop: invariants d=2 Delta=3}, we get that $\B$ is a 
$QEL$-variety of type $\delta$ and
dimension $r$ with
$$(n,r,\delta)\in\{(8,3,0), (11,5,1), (17,9,3), (23,13,5)\}.$$
The case with $\delta=0$ is excluded by  Proposition \ref{prop: 3-fold in P8};
the case with $\delta=1$ is excluded by  Proposition 
\ref{prop: cohomology properties} part \ref{part: hilbert polynomial}; 
by the same Proposition, we get the expression of the Hilbert polynomials in the 
cases
 with $\delta\geq3$.     
Finally, the statement on $\L_{x,\B}$, in the case (\ref{part: first, invariants d=2 Delta=4}) 
follows from Theorem \ref{prop: classification CC-varieties}, 
while in the case (\ref{part: second, invariants d=2 Delta=4}) it follows 
from  Theorem \ref{prop: classification first species mukai varieties} 
(for the latter case, by  Kodaira Vanishing Theorem  and  Serre Duality, 
we get $g(\L_{x,\B})=7$).   
\end{proof}

\begin{appendices}
\chapter{Further remarks on quadro-quadric birational transformations into a quadric}

In this short appendix we provide a few more arguments to show 
Theorem \ref{prop: classification of type 2-2 into quadric}, 
under an additional assumption.

Let $\varphi:\PP^n\dashrightarrow\Q\subset\PP^{n+1}$ be a birational 
transformation of type $(2,2)$ into a smooth quadric $\Q$  and
let $\B\subset\PP^n$ and $\B'\subset\Q\subset\PP^{n+1}$ be respectively 
the base locus  of $\varphi$ and $\varphi^{-1}$.
\begin{definition}
We shall say that $\varphi$ is \emph{liftable} if there exists a Cremona transformation 
$\widehat{\varphi}:\PP^{n+1}\dashrightarrow\PP^{n+1}$  of type $(2,2)$ such that 
$\varphi$ is the restriction of $\widehat{\varphi}$ 
to a hyperplane and therefore 
$\B$ is a hyperplane section of the base locus $\widehat{\B}$ of $\widehat{\varphi}$.
\end{definition}
From Propositions \ref{prop: P3} and \ref{prop: P4} and 
by straightforward calculations follows that
if $\B$ is reduced and $n\leq 4$, then $\varphi$ is liftable;
further, from Theorem \ref{prop: classification of type 2-2 into quadric} 
it follows that $\varphi$ is liftable whenever it is special.
So we feel motivated to make the following:
\begin{conjecture}\label{conjecture: liftable 2-2}
Each birational 
transformation $\PP^n\dashrightarrow\Q\subset\PP^{n+1}$ 
of type $(2,2)$ into a smooth quadric $\Q$ is liftable.
\end{conjecture}
Recall that in \cite{pirio-russo} it was proved that 
\emph{each Cremona transformation 
of type $(2,2)$ is,
modulo changes of coordinates in the source and target space,
 an involution which is the adjoint of a rank $3$ Jordan algebra.}

This has strong consequences in the event that Conjecture \ref{conjecture: liftable 2-2} 
is true.   For example, it implies that for every $\varphi$ as above, 
$\B'$ can be considered projectively equivalent to $\widehat{\B}$ and hence 
$\B$ can be considered as a hyperplane section of $\B'$.
Further,  if $\varphi$ is special, 
we have that $\widehat{\B}$ is irreducible and 
  $\sing(\widehat{\B})$ is a finite set.\footnote{Otherwise  
the very ample divisor  $\B$ of $\widehat{\B}$ would intersect an 
 irreducible curve contained in    
 $\sing({\widehat{\B}})$ and then $\B$ would be singular.}   
    Then, restricting $\widehat{\varphi}$ to a general hyperplane,
we get a special quadratic birational transformation 
into a general quadric containing $\B'$. 
So we can treat $\B$ like a  general  hyperplane section of
 $\widehat{\B}$.

From these remarks it follows that
the proof of Theorem \ref{prop: classification of type 2-2 into quadric} 
could  be  simplified
if one had shown  a priori  that Conjecture \ref{conjecture: liftable 2-2}
holds in the special case.
In fact, if it were the  case, one could deduce that
$\widetilde{\gamma}(\B)\geq\widetilde{\gamma}(\widehat{\B})\geq 0$ and,
as in Claim \ref{step: smooth image tangential projection}, 
one would deduce that $\widetilde{\gamma}(\B)=0$. 
Hence $\widehat{\B}$ would be a Severi variety, by
Theorem \ref{prop: R1 varieties and severi varieties}.

Assuming again that 
Conjecture \ref{conjecture: liftable 2-2}
holds in the special case,
we can deduce Theorem \ref{prop: classification of type 2-2 into quadric} also by 
Proposition 
\ref{prop: alternate proof classification type 2-2 into quadric}, below. In fact, 
 cases with $\delta\neq7$ are clear and we can assume $\delta=7$ and 
$\varphi:\PP^{25}\dashrightarrow\PP^{26}$;
thus we obtain $\widehat{\B}_{\mathrm{red}}=E_6\subset\PP^{26}$ and 
then $\widehat{\B}=E_6$, since
$h^0(\PP^{26},\I_{\widehat{\B},\PP^{26}}(2))=h^0(\PP^{26},\I_{E_6,\PP^{26}}(2))$.    
\begin{remark}\label{remark: LxX like hyperplane section E6}
Let $Y\subset\PP^n$ be a smooth $LQEL$-variety, with  $\dim(Y)=15$, 
$\dim(\mathrm{Sec}(Y))=24<n$ and
$\delta(Y)=7$. 
Let $\widetilde{Y}\subset\PP(H^0(Y,\O_Y(1)))=\PP^{\widetilde{n}}$ be 
the linear normalization of $Y$. 
We have $\dim(\widetilde{Y})=\dim(Y)$, 
$\dim(\mathrm{Sec}(\widetilde{Y}))=\dim(\mathrm{Sec}(Y))$,   
$\delta(\widetilde{Y})=\delta(Y)$ and $\widetilde{Y}$ is an $LQEL$-variety. 
As in Claim \ref{step: smooth image tangential projection} we obtain that,
for the general point $x\in \widetilde{Y}$, 
$\L_{x,\widetilde{Y}}\subset\PP^{14}$ is  a hyperplane section 
 of the spinorial variety $S^{10}\subset\PP^{15}$;
so, by Propositions \ref{prop: surjective second fundamental form} and 
 \ref{prop: relation between hilbert scheme and second funfamental form},
we deduce that 
$\widetilde{n}-16 = \dim(|II_{x,\widetilde{Y}}|)\leq 
h^0(\PP^{15},\I_{S^{10},\PP^{15}}(2))-1 = 9$,
from which $\widetilde{n}\leq 25$. Thus $n=\widetilde{n}=25$, $Y=\widetilde{Y}$ and 
$|II_{x,Y}|$ coincides with the second fundamental form of 
a  hyperplane section of the Cartan variety $E_6\subset\PP^{26}$.
\end{remark}
\begin{proposition}\label{prop: alternate proof classification type 2-2 into quadric}
 Let $X\subset\PP^{26}$ be an irreducible closed subscheme with 
 $\mathrm{Vert}(X)
=\emptyset$.
Suppose that the intersection 
 $Y=X\cap H\subset H$ of $X$ 
with a general hyperplane 
 $H\subset\PP^{26}$ is a smooth nondegenerate  $LQEL$-variety   
of dimension $15$ and of type $\delta=7$.
Then $X_{\mathrm{red}}=E_6\subset\PP^{26}$.
\end{proposition}
\begin{proof} 
Since
$X\cap H= (X_{\mathrm{red}}\cap H)_{\mathrm{red}}=X_{\mathrm{red}}\cap H$ and 
$\mathrm{Vert}(X_{\mathrm{red}})\subseteq\mathrm{Vert}(X)$,
we can reduce to the case where 
 $X$ is a reduced scheme.
Hence $X\subset\PP^{26}$ is an irreducible nondegenerate $16$-dimensional variety
which is
singular at most in finitely many points, and  
$\Sec(X)$ is a nonlinear hypersurface with $\Sec(X)\cap H=\Sec(Y)$.
For $x\in\reg(X)$, put
$S_{x}=\{ [L]\in \L_{x,X}: L\cap \sing(X)\neq\emptyset\}$ and consider 
the set $U=\{x\in \reg(X): S_{x}=\emptyset\}$.
We have
\begin{eqnarray*}
 \reg(X)\setminus U &=&\left\{x\in\reg(X):\exists \, p\in\sing(X)\mbox{ with }\langle x,p \rangle\subseteq X  \right\}\\
 &=&\reg(X)\cap \bigcup_{p\in\sing(X)} \left( \bigcup_{[L]\in \L_{p,X}} L  \right),
\end{eqnarray*}
and $\bigcup_{[L]\in \L_{p,X}} L$ is the cone $S(p,\L_{p,X})\subseteq\PP^{26}$ over the closed 
$\L_{p,X}\subseteq \L_{p,\PP^{26}}\simeq \PP^{25}$ with vertex $p$.
Hence $U\subseteq X$ is an open set and moreover $U\neq\emptyset$, by the hypothesis $\mathrm{Vert}(X)=\emptyset$.
Thus, applying 
Proposition \ref{prop: sing LxX},    we get that,    for the general point $x\in X$, 
$\sing(\L_{x,X})=\emptyset$.
Now,  let $x\in Y$ be a general point (and hence also a general point of $X$).
By Remark \ref{remark: LxX like hyperplane section E6}, 
we obtain that $\L_{x,Y}\subset\PP^{14}$ is 
a smooth hyperplane section 
 of $S^{10}\subset\PP^{15}$. Since
$\L_{x,X}$ is a smooth extension  of $\L_{x,Y}$, applying 
Theorem \ref{prop: classification first species mukai varieties},   
it follows that   $S^{10}$ is an irreducible component  of $\L_{x,X}$.
Now,  by Proposition \ref{prop: relation between hilbert scheme and second funfamental form},
it follows   that $\L_{x,X}$ is a closed subscheme of  
$B_{x,X}=\mathrm{Bs}(|II_{x,X}|)$ and hence
 $$
|II_{x,X}|\subseteq \PP(H^0(\PP^{15},\I_{B_{x,X}}(2)))\subseteq \PP(H^0(\PP^{15},\I_{S^{10}}(2)))=|II_{x,E_6}|.
$$
Moreover, by the commutativity of the diagram 
$$
\xymatrix{
E_X\ar@{^{(}->}[r] \ar@{-->}@/^3.5pc/[rrrrd]^{\phi_{|II_{x,X}|}} & \Bl_{x}(X)\ar[r] & X \ar@{-->}[rrd]^{\tau_{x,X}} \\
E_Y\ar@{^{(}->}[r] \ar@{-->}@/_1.5pc/[rrrr]_{\phi_{|II_{x,Y}|}} \ar@{^{(}->}[u] & \Bl_{x}(Y)\ar[r] \ar@{^{(}->}[u] & Y \ar@{-->}[rr]^{\tau_{x,Y}} \ar@{^{(}->}[u] && \PP^9
}
$$
and by 
Proposition \ref{prop: surjective second fundamental form}, 
it follows
$$
 \PP^9 \supseteq \left\langle\overline{\phi_{|II_{x,X}|}(E_X)}\right\rangle\supseteq 
 \left\langle\overline{\phi_{|II_{x,X}|}(E_Y)}\right\rangle = 
\left\langle\overline{\phi_{|II_{x,Y}|}(E_Y)}\right\rangle=\PP^9,
$$
from which $\dim(|II_{x,X}|)=9$ and then
$|II_{x,X}|=|II_{x,E_6}|$.
Finally, by \cite[p.~57, case~E]{zak-tangent}, 
$E_6$ is a  Hermitian symmetric space 
 and it follows $X=E_6$
by the main result in \cite{landsberg-rigidity}.
\end{proof}

Proposition \ref{prop: alternate proof classification type 2-2 into quadric} and 
Theorem \ref{prop: classification of type 2-2 into quadric} are
 in connection with a fairly well-known open problem, namely 
to classify all smooth $(L)QEL$-varieties $X$ of type $\delta(X)=(\dim(X)-1)/2$.

The following result is contained in 
\cite{fujita-3-fold} and \cite{ohno},
but it can also be inferred from 
Theorems \ref{prop: divisibility theorem}, 
\ref{prop: classification CC-varieties}, 
\ref{prop: classification del pezzo varieties},  
\ref{prop: classification first species mukai varieties}.  
\begin{proposition}\label{prop: QEL type delta-1 mezzi}
 Let $X\subset\PP^n$ be a smooth linearly normal $r$-dimensional $QEL$-variety 
of type $\delta=(r-1)/2>0$ with $\Sec(X)\subsetneq \PP^n$. 
Then either $n=25$, $r=15$, $\delta=7$  
  or $X$ is projectively equivalent to one of the following:
\begin{enumerate}
 \item the Veronese $3$-fold $\nu_2(\PP^3)\subset\PP^9$;
 \item the blow-up of $\PP^3$ at a point $\mathrm{Bl}_p(\PP^3)\subset\PP^8$;
 \item a hyperplane section of the Segre $4$-fold $\PP^2\times\PP^2\subset\PP^8$;
 \item the Segre $5$-fold $\PP^2\times\PP^3\subset\PP^{11}$;
 \item a hyperplane section of the Grassmannian $\GG(1,5)\subset\PP^{14}$.
\end{enumerate}
\end{proposition}
We feel motivated to make the following: 
\begin{conjecture}
Let $X\subset\PP^n$ be as in Proposition \ref{prop: QEL type delta-1 mezzi} with 
$n=25$, $r=15$ and $\delta=7$. Then 
 $X$ is a hyperplane section of $E_6\subset\PP^{26}$.
\end{conjecture}

\end{appendices}
\chapter{On special quadratic birational transformations whose base locus has dimension at most three}\label{chapter: transformations whose base locus has dimension at most three}
In this chapter we continue the study of special quadratic birational transformations 
$\varphi:\PP^n\dashrightarrow\sS:=\overline{\varphi(\PP^n)}\subseteq\PP^{N}$
started in Chapter \ref{sec: transformations into a hypersurface},  
by reinterpreting techniques and well-known results on 
special Cremona transformations. 
While in Chapter \ref{sec: transformations into a hypersurface} we required that $\sS$ was a hypersurface,
here we allow more freedom in the choice of $\sS$, 
but we only treat the case in which 
the dimension of the base 
locus $\B$ is $r=\dim(\B)\leq3$.

Note that for every closed subscheme $X\subset\PP^{n-1}$ 
cut out by the quadrics containing it, we can 
consider $\PP^{n-1}$ as a hyperplane  in $\PP^n$
and hence $X$ as a subscheme of $\PP^n$. So 
the linear system $|\I_{X,\PP^n}(2)|$ of all quadrics in $\PP^n$
containing $X$ defines a quadratic rational map
 $\psi:\PP^n\dashrightarrow\PP^N$ 
($N=h^0(\I_{X,\PP^n}(2))-1=n+h^0(\I_{X,\PP^{n-1}}(2))$),
which is birational onto the image 
and whose inverse is defined by linear forms, 
i.e. $\psi$ is of type $(2,1)$.
Conversely, every birational transformation 
$\psi:\PP^n\dashrightarrow\overline{\psi(\PP^n)}\subseteq\PP^N$
of type $(2,1)$ whose image is nondegenerate, normal and linearly normal 
arise in this way.
From this it follows that there are many (special) quadratic transformations. 
However, 
when the image $\sS$ of the 
 transformation $\varphi$
 is sufficiently regular,
by straightforward generalization of Proposition \ref{prop: dimension formula},
we obtain  
strong numerical and geometric restrictions on the base locus $\B$. 
For example, as soon as $\sS$ is not too much singular,
 the secant variety $\Sec(\B)\subset\PP^n$ has to be a hypersurface and $\B$
has to be a  $QEL$-variety of type $\delta=\delta(\B)=2\dim(\B)+2-n$; in particular 
$n\leq 2\dim(\B)+2$ and $\Sec(\B)$ is a hyperplane if and only if 
 $\varphi$ is of type $(2,1)$.
So the classification of 
 transformations $\varphi$ of type $(2,1)$ whose base locus 
has dimension $\leq 3$ 
essentially follows  
from classification results on
 $QEL$-manifold:
 Proposition \ref{prop: LQEL with delta=r and delta=r-1}, 
 Theorem \ref{prop: classification CC-varieties} and
\cite[Theorems~4.10 and 7.1]{ciliberto-mella-russo}.

When $\varphi$ is of type $(2,d)$ with $d\geq2$, then $\Sec(\B)$ 
is a nonlinear hypersurface 
and it is not so easy to exhibit examples. 
The most difficult cases of this kind 
are those for which $n=2r+2$ i.e. $\delta=0$.
In order to classify these transformations, 
we proceed as in Propositions 
\ref{prop: 2-fold in P6} and \ref{prop: 3-fold in P8 - S nonlinear} 
(see also \cite{mella-russo-baselocusleq3}). 
That is, we first determine the Hilbert
polynomial of $\B$ in Lemmas \ref{lemma: r=2 B nondegenerate} and 
\ref{lemma: r=3 B nondegenerate}, 
by using  
the usual Castelnuovo's argument, Castelnuovo's bound and 
some refinement of Castelnuovo's bound 
(see Chap. \ref{cap: castelnuovo theory}); 
consequently we deduce
 Propositions \ref{prop: r=2 B nondegenerate}
and \ref{prop: r=3 B nondegenerate} by applying the classification 
of smooth varieties of low degree: 
\cite{ionescu-smallinvariants},
\cite{ionescu-smallinvariantsII},
\cite{ionescu-smallinvariantsIII},
\cite{fania-livorni-nine},
\cite{fania-livorni-ten},
\cite{besana-biancofiore-deg11},
\cite{ionescu-degsmallrespectcodim}.
We also apply the 
double point formula in Lemmas:
\ref{lemma: double point formula r=2},
\ref{lemma: double point formula}, 
\ref{lemma: quadric fibration},
\ref{lemma: scroll over surface} and
\ref{lemma: scroll over curve},
 in order to obtain  additional informations on
 $d$ and $\Delta=\deg(\sS)$.  

We summarize our classification 
results in Table \ref{tabella: all cases 3-fold}.
In particular, we provide 
an answer to a question left open
 in the recent preprint \cite{alzati-sierra}. 
\section{Notation and general results}\label{sec: notation}
Throughout the chapter we work over $\CC$ and keep the following setting.
\begin{assumption}\label{assumption: base}
Let $\varphi:\PP^n\dashrightarrow\sS:=\overline{\varphi(\PP^n)}\subseteq\PP^{n+a}$ be
a quadratic birational transformation 
with smooth connected base locus $\B$  
and with  $\sS$ nondegenerate, 
linearly normal and factorial.
\end{assumption}
Recall that we can resolve the
 indeterminacies of
 $\varphi$ with the diagram 
\begin{equation} 
\xymatrix{ & \widetilde{\PP^n} \ar[dl]_{\pi} \ar[dr]^{\pi'}\\ \PP^n\ar@{-->}[rr]^{\varphi}& & \sS } 
\end{equation}
where $\pi:\widetilde{\PP^n}=\Bl_{\B}(\PP^n)\rightarrow\PP^n$ is the blow-up of
 $\PP^n$ along $\B$ and  
$\pi'=\varphi\circ\pi:\widetilde{\PP^n}\rightarrow\sS$.   
Denote by $\B'$ the base locus of $\varphi^{-1}$, 
 $E$ the exceptional divisor of $\pi$, 
 $E'=\pi'^{-1}(\B')$, 
$H=\pi^{\ast}(H_{\PP^n})$,
$H'={\pi'}^{\ast}(H_{\sS})$, and note that, since
$\pi'|_{\widetilde{\PP^n}\setminus E'}:\widetilde{\PP^n}\setminus E'\rightarrow \sS\setminus\B'$
is an isomorphism,
we have $(\sing(\sS))_{\mathrm{red}}\subseteq (\B')_{\mathrm{red}}$.
We also put 
$r=\dim(\B)$, 
$r'=\dim(\B')$,
$\lambda=\deg(\B)$, 
$g=g(\B)$ the sectional genus of $\B$,
$c_j=c_j(\T_{\B})\cdot H_{\B}^{r-j}$ (resp. $s_j=s_j(\N_{\B,\PP^n})\cdot H_{\B}^{r-j}$) 
the degree of the $j$-th Chern class (resp. Segre class) of $\B$, 
$\Delta=\deg(\sS)$,
$c=c(\sS)$ the \emph{coindex} of $\sS$ 
(the last of which is defined 
by $-K_{\reg(\sS)}\sim (n+1-c)H_{\reg(\sS)}$, 
whenever $\Pic(\sS)=\ZZ\langle H_{\sS}\rangle$).
\begin{assumption}\label{assumption: liftable}
 We suppose that  there exists a rational map 
$\widehat{\varphi}:\PP^{n+a}\dashrightarrow\PP^n$ 
defined by a sublinear system of $|\O_{\PP^{n+a}}(d)|$ and having 
base locus $\widehat{\B}$ 
such that $\varphi^{-1}=\widehat{\varphi}|_{\sS}$ and $\B'=\widehat{\B}\cap\sS$.
We then will say that $\varphi^{-1}$ is \emph{liftable}\footnote{If $a\geq 2$ 
and $\psi:\PP^n\dashrightarrow\mathbf{Z}:=\overline{\psi(\PP^n)}\subset\PP^{n+a}$ is
a birational transformation with $\mathbf{Z}$ factorial,
 from \cite{mella-polastri} it follows that there exists a Cremona transformation
$\widetilde{\psi}:\PP^{n+a}\dashrightarrow\PP^{n+a}$ 
such that $\overline{\widetilde{\psi}(\mathbf{Z})} \simeq\PP^n\subset\PP^{n+a}$ and 
$\psi^{-1}=\widetilde{\psi}|_{\mathbf{Z}}$; in particular, 
if $\varpi$ denotes the linear projection
 of $\PP^{n+a}$ onto $\overline{\widetilde{\psi}(\mathbf{Z})}$, we have 
  $\psi^{-1}=(\varpi\circ\widetilde{\psi})|_{\mathbf{Z}}$. But this 
in general  does not ensure the liftability of $\psi^{-1}$, because 
we only have that $\mathrm{Bs}(\psi^{-1})\subseteq \mathrm{Bs}(\varpi\circ\widetilde{\psi}) \cap \mathbf{Z}$.}
and that $\varphi$ is \emph{of type} $(2,d)$.
\end{assumption}
   The above assumption 
yields the relations:  
\begin{equation}\label{eq: lift}
\begin{array}{ll}
 H' \sim  2H-E,   &  H  \sim  dH'-E',  \\
 E'\sim  (2d-1)H-dE, & E \sim  (2d-1)H'-2E' ,
\end{array}
\end{equation}
and hence also
$ \Pic(\widetilde{\PP^n})\simeq \ZZ\langle H \rangle\oplus \ZZ\langle E \rangle
\simeq \ZZ\langle H'\rangle\oplus \ZZ\langle E'\rangle $.
Note that,  by the proofs of
\cite[Proposition~1.3 and 2.1(a)]{ein-shepherdbarron} and
by factoriality of $\sS$, we obtain that $E'$ 
is a reduced and irreducible divisor.
Moreover we have
$\Pic(\sS)\simeq \Pic(\sS\setminus\B')\simeq \Pic(\widetilde{\PP^n}\setminus E')
\simeq \ZZ\langle H'\rangle\simeq \ZZ\langle H_{\sS}\rangle$.
Finally, we require the following: 
\begin{assumption}\label{assumption: ipotesi}
$(\sing(\sS))_{\mathrm{red}}\neq (\B')_{\mathrm{red}}$.
\end{assumption}
Now we point out that, just as in Proposition \ref{prop: dimension formula}, 
 since $E'$ is irreducible, 
by Assumption \ref{assumption: ipotesi} and \cite[Theorem~1.1]{ein-shepherdbarron},
 we deduce that 
$\pi'|_V:V\rightarrow U$ coincides with the blow-up of $U$ along $Z$,
where $U=\reg(\sS)\setminus\sing((\B')_{\mathrm{red}})$, 
$V=\pi'^{-1}(U)$ and  $Z=U\cap (\B')_{\mathrm{red}}$.
It follows that
$K_{\widetilde{\PP^n}} \sim (-n-1)H+(n-r-1)E \sim (c-n-1)H'+(n-r'-1)E'$, 
from which, together with (\ref{eq: lift}), we obtain
$2r+3-n=n-r'-1$ and 
$c=\left( 1-2d\right) r+dn-3d+2$.
One can also easily see that,
for the general point 
$x\in\Sec(\B)\setminus \B$,  
$\overline{\varphi^{-1}\left(\varphi\left(x\right)\right)}$
is a linear space of dimension $n-r'-1$   and
$\overline{\varphi^{-1}\left(\varphi\left(x\right)\right)}\cap \B$ 
is a quadric hypersurface,  which coincides with 
the entry locus $\Sigma_{x}(\B)$ of $\B$ with respect to $x$.
So we can generalize Proposition \ref{prop: dimension formula}, 
obtaining one of the main results useful for purposes of this chapter:
\begin{proposition}\label{prop: B is QEL}
 $\Sec(\B)\subset\PP^n$ is a hypersurface of degree $2d-1$ and 
       $\B$ is a $QEL$-variety of type $\delta=2r+2-n$.
\end{proposition}
In many cases, $\B$ has a much stronger property of being $QEL$-variety.
Recall that a subscheme $X\subset\PP^n$ is said to have the $K_2$ property if 
$X$ is cut out by quadratic forms $F_0,\ldots,F_N$ 
such that the  Koszul relations among the $F_i$ are generated by linear syzygies.
We have the following fact (see \cite{vermeire} and \cite{alzati-syz}):
\begin{fact}\label{fact: K2 property}
Let $X\subset\PP^n$ be 
 a smooth  variety 
cut out by quadratic forms $F_0,\ldots,F_N$
satisfying $K_2$ property and let 
$F=[F_0,\ldots,F_N]:\PP^n\dashrightarrow\PP^N$ be the 
induced rational map. Then 
for every $x\in\PP^n\setminus X$, 
$\overline{F^{-1}\left(F\left(x\right)\right)}$ is 
a linear space of dimension 
$n+1-
\mathrm{rank}\left(\left({\partial F_i}/{\partial x_j}(x)\right)_{i,j}\right)$;
moreover,
$\dim(\overline{F^{-1}\left(F\left(x\right)\right)})>0$ 
if and only if $x\in\Sec(X)\setminus X$ and in this case  
$\overline{F^{-1}\left(F\left(x\right)\right)}\cap X$ is a quadric hypersurface,
which coincides with 
the entry locus $\Sigma_{x}(X)$ of $X$ with respect to $x$.
\end{fact}
We have a simple sufficient 
condition for the $K_2$ property (see \cite{saint-donat}, \cite{green-lazarsfeld-1988} and \cite[Proposition~2]{alzati-russo-subhomaloidal}):
\begin{fact}\label{fact: test K2}
 Let $X\subset\PP^n$ be a smooth linearly normal variety and suppose $h^1(\O_X)=0$
if $\dim(X)\geq2$. Putting $\lambda=\deg(X)$ and $s=\mathrm{codim}_{\PP^n}(X)$ we have:
\begin{itemize}
\item if $\lambda\leq 2s+1$,  then $X$ is arithmetically Cohen-Macaulay;
\item if $\lambda\leq 2s$,  then the homogeneous ideal of $X$ 
is generated by quadratic forms;
\item if $\lambda\leq2s-1$,  then the syzygies of the generators 
of the homogeneous ideal of $X$ are generated by the linear ones.
\end{itemize}
\end{fact}
\begin{remark}\label{remark: samuel conjecture}
Let $\psi:\PP^n\dashrightarrow\mathbf{Z}:=\overline{\psi(\PP^n)}\subseteq\PP^{n+a}$ be
a birational transformation ($n\geq3$). 

We point out that, from Grothendieck's Theorem on parafactoriality (Samuel's Conjecture) 
\cite[\Rmnum{11} Corollaire~3.14]{sga2} it follows that $\mathbf{Z}$ is factorial 
whenever it is a local complete intersection 
 with $\dim(\sing(\mathbf{Z}))<\dim(\mathbf{Z})-3$.
Of course, every complete intersection in a smooth variety 
is a local complete intersection.

Moreover, $\psi^{-1}$ is liftable whenever
  $\Pic(\mathbf{Z})=\ZZ\langle H_{\mathbf{Z}}\rangle$ 
and $\mathbf{Z}$ is factorial and projectively normal. 
So, from \cite{larsen-coomology} and \cite[\Rmnum{4} Corollary~3.2]{hartshorne-ample}, 
$\psi^{-1}$ is liftable whenever
$\mathbf{Z}$ is either 
  smooth and projectively normal with $n\geq a+2$ 
or a factorial 
complete intersection. 
\end{remark}
\section{Numerical restrictions}
Proposition \ref{prop: B is QEL} already provides 
a restriction on the invariants of the transformation $\varphi$;
 here we give further restrictions of this kind.
\begin{proposition}\label{prop: hilbert polynomial}
 Let $\epsilon=0$ if $\langle \B \rangle =\PP^n$ and let  $\epsilon=1$ otherwise.  
\begin{itemize} 
\item   If $r=1$ we have: 
  \begin{eqnarray*}
   \lambda &=&  \frac{n^2-n+2\,\epsilon-2\,a-2}{2}  ,  \\
         g &=&  \frac{n^2-3\,n+4\,\epsilon-2\,a-2}{2}  .
  \end{eqnarray*}
\item If $r=2$ we have: 
 \begin{eqnarray*}
   \chi(\O_{\B}) &=&  \frac{2\,a-n^2+5\,n+2\,g-6\,\epsilon+4}{4} , \\
         \lambda &=&  \frac{n^2-n+2\,g+2\,\epsilon-2\,a-4}{4} . \\
 \end{eqnarray*}
\item If $r=3$ we have: 
 \begin{eqnarray*}
   \chi(\O_{\B}) &=& \frac{4\lambda-n^2+3n-2g-4\epsilon+2a+6}{2}  .
 \end{eqnarray*}
\end{itemize}
\end{proposition}
\begin{proof}
By Proposition \ref{prop: B is QEL} we have
$h^0(\PP^n,\I_{\B}(1))=\epsilon$.
Since $\sS$ is normal and linearly normal, 
we have $h^0(\PP^n,\I_{\B}(2))=n+1+a$ (see Lemma \ref{prop: cohomology I2B}).
Moreover,
 since $n\leq 2r+2$ (being $\delta\geq0$), 
proceeding as in  Lemma \ref{prop: cohomology twisted ideal}
 (or applying
 \cite[Proposition~1.8]{mella-russo-baselocusleq3}), 
 we obtain
$h^j(\PP^n,\I_{\B}(k))=0$ for every $j,k\geq1$.
So we obtain
$\chi(\O_{\B}(1))=n+1-\epsilon$ and 
$\chi(\O_{\B}(2))= (n+1)(n+2)/2 - (n+1+a)$.
\end{proof}
\begin{proposition}\label{prop: segre and chern classes}  \hspace{1pt}
\begin{itemize}
\item  If $r=1$ we have:
\begin{eqnarray*}
{c}_{1} &=&  2-2\,g , \\  
{s}_{1} &=&  -n\,\lambda-\lambda-2\,g+2 , \\  
d &=& \frac{ 2\,\lambda-{2}^{n} }{ 2\,n\,\lambda-2\,\lambda-{2}^{n+1}-4\,g+4 } , \\  
\Delta &=&  -n\,\lambda+\lambda+{2}^{n}+2\,g-2 . 
\end{eqnarray*}
\item If $r=2$ we have:
\begin{eqnarray*}
c_1 &=& \lambda-2\,g+2 , \\ 
c_2 &=&  -\frac{{n}^{2}\,\lambda}{2}+\frac{3\,n\,\lambda}{2}+{2}^{n}+2\,g\,n-2\,n-2\,g-\Delta+2 , \\ 
s_1 &=&  -n\,\lambda-2\,g+2 , \\ 
s_2 &=&  2\,n\,\lambda+{2}^{n}+4\,g\,n-4\,n-\Delta , \\ 
d\,\Delta &=& -n\,\lambda+2\,\lambda+{2}^{n-1}+2\,g-2 .
\end{eqnarray*}
\item If $r=3$ we have: 
\begin{eqnarray*}
c_1 &=& 2\,\lambda-2\,g+2 , \\
c_2 &=&  -\frac{{n}^{2}\,\lambda}{2}+\frac{5\,n\,\lambda}{2}-\lambda+{2}^{n-1}+2\,g\,n-2\,n-6\,g-d\,\Delta+6 , \\
c_3 &=&  \frac{{n}^{3}\,\lambda}{3}-2\,{n}^{2}\,\lambda+\frac{11\,n\,\lambda}{3}-2\,\lambda-n\,{2}^{n-1}+3\,{2}^{n-1}-g\,{n}^{2}+{n}^{2}+3\,g\,n+  \\ &&  + d\,\Delta\,n-3\,n-4\,g-d\,\Delta-\Delta+4 , \\ 
s_1 &=&  -n\,\lambda+\lambda-2\,g+2 , \\
s_2 &=&  2\,n\,\lambda-2\,\lambda+{2}^{n-1}+4\,g\,n-4\,n-4\,g-d\,\Delta+4 , \\
s_3 &=&  \frac{2\,{n}^{3}\,\lambda}{3}-4\,{n}^{2}\,\lambda+\frac{10\,n\,\lambda}{3}-n\,{2}^{n}+{2}^{n}-4\,g\,{n}^{2}+4\,{n}^{2}+4\,g\,n+2\,d\,\Delta\,n-4\,n-\Delta .
\end{eqnarray*}
\end{itemize}
\end{proposition}
\begin{proof}
See also \cite{crauder-katz-1989} and \cite{crauder-katz-1991}.
By \cite[page~291]{crauder-katz-1989} we see that 
\begin{displaymath}
 H^j\cdot E^{n-j}=
\left\{ 
\begin{array}{ll} 
1, & \mbox{if } j= n ; \\
0, & \mbox{if } r+1\leq j \leq n-1 ; \\
(-1)^{n-j-1} s_{r-j}, & \mbox{if } j \leq r .
\end{array} 
\right.
\end{displaymath}
Since $H'=2H-E$ and $H=dH'-E'$ we have
\begin{eqnarray}
\Delta &=& {H'}^n=(2H-E)^n,  \\
d \Delta &=& d{H'}^{n}={H'}^{n-1}\cdot(dH'-E')=(2H-E)^{n-1}\cdot H.  
\end{eqnarray}
From the exact sequence
$0\rightarrow\mathcal{T}_{\B}\rightarrow \mathcal{T}_{\PP^n}|_{\B}\rightarrow\mathcal{N}_{\B,\PP^n}\rightarrow0$
we get: 
\begin{eqnarray}
s_1 &=& - \lambda \left( n+1\right) + {c}_{1} ,\\
s_2 &=&  \lambda \begin{pmatrix}n+2\cr 2\end{pmatrix}-{c}_{1} \left( n+1\right) +{c}_{2} ,\\ 
s_3 &=& -\lambda \begin{pmatrix}n+3\cr 3\end{pmatrix}+{c}_{1} \begin{pmatrix}n+2\cr 2\end{pmatrix}-{c}_{2} \left( n+1\right) +{c}_{3} ,\\ 
  & \vdots & \nonumber
\end{eqnarray}
Moreover $c_1=-K_{\B}\cdot H_{\B}^{r-1}$ 
and it can be expressed as a function of $\lambda$ and $g$.
Thus we found $r+3$ 
 independent equations on the $2r+5$ variables: 
$c_1,\ldots,c_r,s_1,\ldots,s_r,d,\Delta,\lambda,g,n$.
\end{proof}
\begin{remark}
Proposition \ref{prop: segre and chern classes} 
holds under less restrictive assumptions, as shown in the above proof.
Here we treat the special case:
let $\psi:\PP^8\dashrightarrow\mathbf{Z}:=\overline{\psi(\PP^8)}\subseteq\PP^{8+a}$ be a quadratic 
rational map whose base locus 
is a smooth irreducible $3$-dimensional variety $X$.
Without any other restriction on $\psi$, denoting
with $\pi:\Bl_X(\PP^8)\rightarrow \PP^8$ 
the blow-up of $\PP^8$ along $X$ and with $s_i(X)=s_i(\N_{X,\PP^8})$, we have
\begin{equation}\label{eq: grado mappa razionale}
\deg(\psi)\deg(\mathbf{Z}) = (2\pi^{\ast}(H_{\PP^8})-E_X)^8 =
-s_3(X)-16s_2(X)-112s_1(X)-448\deg(X)+256.
\end{equation}
 Moreover, if $\psi$ is birational with liftable inverse and 
$\dim(\sing(\mathbf{Z}))\leq 6$,
 we also have
\begin{equation}\label{eq: sollevabile}
d \deg(\mathbf{Z}) =(2\pi^{\ast}(H_{\PP^8})-E_X)^7\cdot \pi^{\ast}(H_{\PP^8}) = -s_2(X) -14 s_1(X) -84 \deg(X) +128,
\end{equation}
where $d$ denotes the degree of the linear system defining $\psi^{-1}$.
%
%
\end{remark}
Proposition \ref{prop: double point formula}
is a translation of the well-known 
\emph{double point formula} (see for example \cite{peters-simonis} and \cite{laksov}),
taking into account Proposition \ref{prop: B is QEL}.
\begin{proposition}\label{prop: double point formula}
If $\delta=0$ 
then 
$$
2(2d-1)=
\lambda^2 - 
\sum_{j=0}^{r}\begin{pmatrix} 2r+1 \cr j \end{pmatrix} s_{r-j}(\T_{\B})\cdot H_{\B}^{j}.
$$
\end{proposition}
\section{Case of dimension 1}\label{sec: dim 1}
Lemma \ref{lemma: numerical 1-fold}  
directly follows from  
Propositions \ref{prop: hilbert polynomial} 
and \ref{prop: segre and chern classes}.
\begin{lemma}\label{lemma: numerical 1-fold}
If $r=1$, then 
one of the following cases holds:
\begin{enumerate}[(A)]
 \item $n=3$, $a=1$, $\lambda=2$, $g=0$, $d=1$, $\Delta=2$; 
 \item $n=4$, $a=0$, $\lambda=5$, $g=1$, $d=3$, $\Delta=1$; 
 \item $n=4$, $a=1$, $\lambda=4$, $g=0$, $d=2$, $\Delta=2$; 
 \item\label{case: escluso 1-fold} $n=4$, $a=2$, $\lambda=4$, $g=1$, $d=1$, $\Delta=4$; 
 \item $n=4$, $a=3$, $\lambda=3$, $g=0$, $d=1$, $\Delta=5$.
\end{enumerate}
\end{lemma}
\begin{proposition}\label{prop: possibili casi 1-fold}
If $r=1$, then 
one of the following cases holds:
\begin{enumerate}[(I)]
 \item $n=3$, $a=1$,  $\B$ is a conic;
 \item $n=4$, $a=0$,  $\B$ is an elliptic curve of degree $5$;
 \item $n=4$, $a=1$,  $\B$ is the rational normal quartic curve;
 \item $n=4$, $a=3$,  $\B$ is the twisted cubic curve.
\end{enumerate}
\end{proposition}
\begin{proof}
 From Lemma \ref{lemma: numerical 1-fold}
 it remains only to exclude case (\ref{case: escluso 1-fold}).
In this case $\B$ is a complete intersection of two quadrics in $\PP^3$ 
and also 
it is an $OADP$-curve. 
This is absurd because the only $OADP$-curve is the twisted cubic curve.
\end{proof}
\section{Case of dimension 2}\label{sec: dim 2}
Proposition \ref{prop: possibili casi 2-fold} follows from 
Proposition \ref{prop: LQEL with delta=r and delta=r-1} and
            \cite[Theorem~4.10]{ciliberto-mella-russo}.
\begin{proposition}\label{prop: possibili casi 2-fold} 
If $r=2$, then either $n=6$, $d\geq2$, $\langle \B \rangle = \PP^6$, or 
one of the following cases holds:
\begin{enumerate}[(I)] 
  \setcounter{enumi}{4} 
 \item\label{case 2-fold a}  $n=4$,  $d=1$, $\delta=2$, 
              $\B=\PP^1\times\PP^1\subset\PP^3\subset\PP^4$;
 \item\label{case 2-fold b}  $n=5$,  $d=1$, $\delta=1$, 
              $\B$ is a hyperplane section of $\PP^1\times\PP^2\subset\PP^5$; 
 \item\label{case 2-fold c}  $n=5$,  $d=2$, $\delta=1$, 
              $\B=\nu_2(\PP^2)\subset\PP^5$ is the Veronese surface;
 \item\label{case 2-fold d}  $n=6$,  $d=1$, $\delta=0$, $\B\subset\PP^5$ is an $OADP$-surface, i.e. $\B$ is as in one of the following cases: 
         \begin{enumerate}[($\ref{case 2-fold d}_1$)]
         \item\label{case 2-fold d1} $\PP_{\PP^1}(\O(1)\oplus\O(3))$ or
                                     $\PP_{\PP^1}(\O(2)\oplus\O(2))$;        
         \item\label{case 2-fold d2} del Pezzo surface of degree  $5$ 
                              (hence the blow-up of  
                 $\PP^2$ at $4$ points 
                 $p_1,\ldots,p_4$ and $|H_{\B}|=|3H_{\PP^2}-p_1-\cdots-p_4|$).
         \end{enumerate}
\end{enumerate}
\end{proposition}
\begin{lemma}\label{lemma: r=2 B nondegenerate}
If $r=2$, $n=6$ and $\langle \B \rangle = \PP^6$, 
then  one of the following cases holds:
\begin{enumerate}[(A)]
\item \label{case 2-fold a=0 lambda=7}
 $a=0$,  $\lambda=7$, $g=1$, $\chi(\O_{\B})=0$;
\item \label{case 2-fold a leq 3}
 $0\leq a \leq 3$, $\lambda=8-a$, $g=3-a$, $\chi(\O_{\B})=1$.
\end{enumerate}
\end{lemma}
\begin{proof}
By Proposition \ref{prop: hilbert polynomial}
 it follows that
$g=2\lambda+a-13$ and $\chi(\O_{\B})=\lambda+a-7$.
By Lemma \ref{prop: castelnuovo argument} and using that $g\geq0$
(proceeding as in Proposition \ref{prop: 2-fold in P6}),
we obtain $(13-a)/2 \leq \lambda \leq 8-a$.
\end{proof}
\begin{lemma}\label{lemma: double point formula r=2}
 If $r=2$, $n=6$ and $\langle \B \rangle = \PP^6$, then
one of the following cases holds:
\begin{itemize}
 \item $a=0$, $d=4$, $\Delta=1$;
 \item $a=1$, $d=3$, $\Delta=2$;
 \item $a=2$, $d=2$, $\Delta=4$;
 \item $a=3$, $d=2$, $\Delta=5$.
\end{itemize}
\end{lemma}
\begin{proof}
 We have $s_1(\T_{\B})\cdot H_{\B} = -c_1 $ and 
 $ s_2(\T_{\B}) = c_1^2-c_2=12\chi(\O_{\B}) -2c_2 $. 
So, by Proposition \ref{prop: double point formula}, we obtain 
\begin{equation}
2(2d-1) = \lambda^2-10\lambda-12\chi(\O_{\B})+2c_2+5c_1 .
\end{equation}
Now, by Propositions \ref{prop: hilbert polynomial} and
 \ref{prop: segre and chern classes}, 
we obtain
\begin{equation}
d\Delta = 2a+4,\quad  
 \Delta = (g^2+(-2a-4)g-16d+a^2-4a+75)/8 ,
\end{equation}
and then we conclude by Lemma \ref{lemma: r=2 B nondegenerate}.
\end{proof}
\begin{proposition}\label{prop: r=2 B nondegenerate}
If $r=2$,  $n=6$ and $\langle \B\rangle=\PP^6$ 
 then  one of the following cases holds:
\begin{enumerate}[(I)]
\setcounter{enumi}{8} 
\item $a=0$,  $\lambda=7$,  $g=1$, $\B$ is an elliptic scroll $\PP_{C}(\E)$ with $e(\E)=-1$;
\item $a=0$,  $\lambda=8$,  $g=3$, $\B$ is the blow-up of $\PP^2$ at $8$ points $p_1\ldots,p_8$, $|H_{\B}|=|4H_{\PP^2}-p_1-\cdots-p_8|$;
\item $a=1$,  $\lambda=7$,  $g=2$, $\B$ is the blow-up of $\PP^2$ at $6$ points $p_0\ldots,p_5$, $|H_{\B}|=|4H_{\PP^2}-2p_0-p_1-\cdots-p_5|$;
\item $a=2$,  $\lambda=6$,  $g=1$, $\B$ is the blow-up of $\PP^2$ at $3$ points $p_1,p_2,p_3$, $|H_{\B}|=|3H_{\PP^2}-p_1-p_2-p_3|$;
\item $a=3$,  $\lambda=5$,  $g=0$, $\B$ is a rational normal scroll.
\end{enumerate}
\end{proposition}
\begin{proof}
For $a=0$, $a=1$ and $a\in\{2,3\}$ 
the statement follows, respectively, from \cite{crauder-katz-1989},
 Proposition \ref{prop: 2-fold in P6} and \cite{ionescu-smallinvariants}. 
\end{proof}
\section{Case of dimension 3}\label{sec: dim 3}
Proposition \ref{prop: possibili casi C1} follows from: 
            Proposition \ref{prop: LQEL with delta=r and delta=r-1},
            \cite{fujita-3-fold},
             Theorems \ref{prop: classification CC-varieties} and
             \ref{prop: classification del pezzo varieties}
            and
            \cite{ciliberto-mella-russo}.
\begin{proposition}\label{prop: possibili casi C1} 
If $r=3$, then either $n=8$, $d\geq2$, $\langle \B \rangle = \PP^8$, or  
one of the following cases holds:
\begin{enumerate}[(I)]
 \setcounter{enumi}{13} 
 \item\label{case C1 a}  $n=5$, $d=1$, $\delta=3$, $\B=Q^3\subset\PP^4\subset\PP^5$ is a quadric;
 \item\label{case C1 b}  $n=6$, $d=1$, $\delta=2$, $\B=\PP^1\times\PP^2\subset\PP^5\subset\PP^6$;
 \item\label{case C1 c}  $n=7$, $d=1$, $\delta=1$, $\B\subset\PP^6$ is as 
in one of the following cases:
         \begin{enumerate}[($\ref{case C1 c}_1$)]
         \item\label{case C1 c1} $\PP_{\PP^1}(\O(1)\oplus\O(1)\oplus\O(2))$;
         \item\label{case C1 c2} linear section of $\GG(1,4)\subset\PP^9$; 
         \end{enumerate}
 \item\label{case C1 d}  $n=7$, $d=2$, $\delta=1$, $\B$ is a hyperplane section of $\PP^2\times\PP^2\subset\PP^8$;
 \item\label{case C1 e}  $n=8$,  $d=1$, $\delta=0$, $\B\subset\PP^7$ is an $OADP$-variety, 
 i.e. $\B$ is as in one of the following cases:
 \begin{enumerate}[($\ref{case C1 e}_1$)]
 \item\label{case C1 e1} $\PP_{\PP^1}(\O(1)\oplus\O(1)\oplus\O(3))$ or
                         $\PP_{\PP^1}(\O(1)\oplus\O(2)\oplus\O(2))$; 
 \item\label{case C1 e2} Edge variety of degree $6$ (i.e. $\PP^1\times\PP^1\times\PP^1$) or 
                         Edge variety of degree $7$;
 \item\label{case C1 e3} $\PP_{\PP^2}(\E)$, where $\E$ 
is a vector bundle with  $c_1(\E)=4$ and $c_2(\E)=8$, given as an extension by
the following exact sequence
$0\rightarrow\O_{\PP^2}\rightarrow\E\rightarrow 
\I_{\{p_1,\ldots,p_8\},\PP^2}(4)\rightarrow0$. 
 \end{enumerate}
\end{enumerate}
\end{proposition}
In the following we denote by 
$\Lambda\subsetneq C\subsetneq S\subsetneq \B$ 
a sequence of general linear sections of $\B$.
\begin{lemma}\label{lemma: r=3 B nondegenerate}
If $r=3$, $n=8$ and $\langle \B \rangle = \PP^8$,
 then one of the following cases holds:
\begin{enumerate}[(A)]
\item \label{a=0,lambda=13}
 $a=0$, $\lambda=13$, $g=8$, $K_S\cdot H_S=1$, $K_S^2=-1$; 
\item \label{a=1,lambda=12}
 $a=1$, $\lambda=12$, $g=7$, $K_S\cdot H_S=0$, $K_S^2=0$; 
\item \label{a geq2}
 $0\leq a\leq6$, $\lambda=12-a$, $g=6-a$, $K_S\cdot H_S=-2-a$.
\end{enumerate}
\end{lemma}
\begin{proof}
Firstly we note that,
from the exact sequence
 $0\rightarrow\mathcal{T}_{S}\rightarrow\mathcal{T}_{\B}|_{S}\rightarrow\O_{S}(1)\rightarrow0$,
 we deduce 
$c_2=c_2(S)+c_1(S)=12\chi(\O_{S})-K_{S}^2-K_{S}\cdot H_{S}$
and hence
\begin{equation}
K_S^2=14\lambda+12\chi(\O_S)-12g+d\Delta-116 = -22\lambda+12g+d\Delta-12a+184.
\end{equation}
Secondly we note that (see Lemma \ref{prop: castelnuovo argument}), putting 
 $h_{\Lambda}(2):=h^0(\PP^5,\O(2))-h^0(\PP^5,\I_{\Lambda}(2))$,
we have  
\begin{equation}\label{hilbert-function}
 \mathrm{min}\{\lambda,11\} \leq h_{\Lambda}(2)\leq 
         21-h^0(\PP^8,\I_{\B}(2))=12-a.
\end{equation}
Now we establish the following:
\begin{claim}\label{claim: KsHs<0} 
If $K_S\cdot H_S\leq0$ and $K_S\nsim 0$, then $\lambda=12-a$ and $g=6-a$.
\end{claim}
\begin{proof}[Proof of the Claim]
Similarly to Case \ref{case: KS nsim 0}, we obtain that
$P_{\B}(-1)=0$ and $P_{\B}(0)=1-q$, where 
$q:=h^1(S,\O_S)=h^1(\B,\O_{\B})$;  in particular $g=-5q-a+6$ and  $\lambda=-3q-a+12$.
Since $g\geq0$  we have $5q\leq 6-a$ and
the possibilities are:
if $a\leq1$ then $q\leq1$; if $a\geq 2$ then $q=0$.
If $(a,q)=(0,1)$ then $(g,\lambda)=(1,9)$ 
and the case is excluded by 
 \cite[Theorem~12.3]{fujita-polarizedvarieties}\footnote{Note that 
$\B$ cannot be a scroll over a curve 
(this follows from (\ref{eq: relation scroll}) and (\ref{eq: second relation scroll}) below
   and also it follows from \cite[Proposition~3.2(i)]{mella-russo-baselocusleq3}).}; 
if $(a,q)=(1,1)$ then $(g,\lambda)=(0,8)$ 
and the case is excluded by 
\cite[Theorem~12.1]{fujita-polarizedvarieties}.
Thus we have $q=0$ and hence
$g=6-a$ and $\lambda=12-a$;
in  particular we have $a\leq 6$.
\end{proof}
Now we discuss the cases according 
to the value of  $a$.
\begin{case}[$a=0$]
It is clear that $\varphi$ must be of type $(2,5)$ and hence 
$K_S^2=-22\lambda+12g+189$. 
By Claim \ref{claim: KsHs<0}, if $K_S\cdot H_S=2g-2-\lambda<0$,
we fall into case (\ref{a geq2}). 
So we suppose that $K_S\cdot H_S\geq0$, namely that $g\geq\lambda/2+1$.
From Castelnuovo's bound it follows that
 $\lambda\geq12$  and if $\lambda=12$ then  
$K_S\cdot H_S=0$,  $g=7$ and hence $K_S^2=9$.
Since this is impossible by Claim \ref{claim: KsHs<0}, we conclude that
 $\lambda\geq 13$.
Now by (\ref{hilbert-function}) it follows that
 $11\leq h_{\Lambda}(2)\leq12$, but
if $h_{\Lambda}(2)=11$ 
from Castelnuovo Lemma (Proposition \ref{prop: castelnuovo lemma})  
we obtain a contradiction.
Thus we have $h_{\Lambda}(2)=12$ and
$h^0(\PP^5,\I_{\Lambda}(2))=h^0(\PP^8,\I_{\B}(2))=9$.
So from Theorem \ref{prop: extension castelnuovo lemma by fano harris} we deduce 
that $\lambda\leq 14$ and furthermore, 
 by the refinement of Castelnuovo's bound contained in 
 Theorem \ref{prop: refinement castelnuovo bound by ciliberto}, 
we obtain $g\leq 2\lambda-18$.
In summary we have the following possibilities:
\begin{enumerate}[(i)]
 \item\label{case: T1} $\lambda=13$, $g=8$, $K_S\cdot H_S=1$, $\chi(\O_S)=2$, $K_S^2=-1$;
 \item\label{case: T2} $\lambda=14$, $g=8$, $K_S\cdot H_S=0$, $\chi(\O_S)=-1$, $K_S^2=-23$;
 \item\label{case: T3} $\lambda=14$, $g=9$, $K_S\cdot H_S=2$, $\chi(\O_S)=1$, $K_S^2=-11$;
 \item\label{case: T4} $\lambda=14$, $g=10$, $K_S\cdot H_S=4$, $\chi(\O_S)=3$, $K_S^2=1$.
\end{enumerate}
Case (\ref{case: T1}) coincides with case (\ref{a=0,lambda=13}). 
Case (\ref{case: T2}) is excluded by Claim \ref{claim: KsHs<0}.
In the circumstances of case (\ref{case: T3}), we have $h^1(S,\O_S)=h^2(S,\O_S)=h^0(S,K_S)$.
If $h^1(S,\O_S)>0$, since 
$(K_{\B}+4H_{\B})\cdot K_S=K_S^2+3 K_S\cdot H_S=-5<0$,
we see that $K_{\B}+4H_{\B}$ is not nef and then 
we obtain a contradiction by \cite{ionescu-adjunction}.
If $h^1(S,\O_S)=0$, then we also have $h^1(\B,\O_{\B})=h^2(\B,\O_{\B})=0$ and hence 
$\chi(\O_{\B})=1-h^3(\B,\O_{\B})\leq 1$, against the fact that $\chi(\O_{\B})=2\lambda-g-17=2$. 
Thus case (\ref{case: T3}) does not occur.
Finally,  in the circumstances of case (\ref{case: T4}), note 
that $h^0(S,K_S)=2+h^1(S,\O_S)\geq2$ and we  write 
$|K_S|=|M|+F$, where $|M|$ is the mobile  part of the linear system $|K_S|$ 
and $F$ is the fixed part. If $M_1=M$ 
 is a general member of $|M|$, there exists $M_2\in|M|$
having no common irreducible components with $M_1$ and 
so $M^2=M_1\cdot M_2=\sum_{p}\left(M_1\cdot M_2\right)_{p}\geq0$;
furthermore, by using Bertini Theorem, we see that
 $\sing(M_1)$ consists of points $p$   such that 
the intersection multiplicity $\left(M_1\cdot M_2\right)_{p}$ of $M_1$ and $M_2$ in $p$ is at least $2$.
By definition, we also have $M\cdot F\geq0$ and so we deduce 
$2p_a(M)-2=M\cdot (M+K_S)= 2 M^2+ M\cdot F\geq 0$, from which 
$p_a(M)\geq 1$ and $p_a(M)=2$ if $F=0$.
On the other hand,  we have $M\cdot H_S\leq K_S\cdot H_S=4$ and, 
since $S$ is cut out by quadrics,
 $M$ does not contain  planar curves of degree $\geq3$.
If  $M\cdot H_S=4$, then  $F=0$,
 $M^2=1$ and $M$ is  a (possibly disconnected) smooth curve;
since $p_a(M)=2$, $M$ is actually disconnected 
 and so it is a disjoint union  of twisted cubics, conics and lines.  
 But then we obtain  the contradiction that
$p_a(M)=1-\#\{\mbox{connected components of }M\}<0$.
If $M\cdot H_S\leq3$,  then $M$ must be either 
a twisted cubic or a union of conics and lines.  In all these cases
we again obtain the contradiction that 
$p_a(M)=1-\#\{\mbox{connected components of }M\}\leq 0$.
Thus case (\ref{case: T4}) does not occur.
\end{case}
\begin{case}[$a=1$] By  Proposition \ref{prop: 3-fold in P8 - S nonlinear}  
we fall into case (\ref{a=1,lambda=12}) or (\ref{a geq2}).
 \end{case}
\begin{case}[$a\geq2$] By (\ref{hilbert-function}) it follows that
$\lambda\leq 10$ and by  Castelnuovo's bound  it follows that  
$K_S\cdot H_S\leq -4<0$. Thus, by Claim \ref{claim: KsHs<0}
we fall into case (\ref{a geq2}).
\end{case}
\end{proof}
Now we apply the double point formula 
(Proposition \ref{prop: double point formula}) 
in order to obtain additional numerical restrictions 
under the hypothesis of Lemma \ref{lemma: r=3 B nondegenerate}. 
\begin{lemma}\label{lemma: double point formula}
If $r=3$, $n=8$ and $\langle \B\rangle=\PP^8$, then
$$
K_{\B}^3=\lambda^2+23\lambda-24g-(7d+1)\Delta-4d+36a-226 .
$$
\end{lemma}
\begin{proof}
We have (see \cite[App. A, Exercise~6.7]{hartshorne-ag}):
\begin{eqnarray*}
 s_1(\T_{\B})\cdot H_{\B}^2 &=& -c_1(\B)\cdot H_{\B}^2=K_{\B}\cdot H_{\B}^2 , \\
 s_2(\T_{\B})\cdot H_{\B} &=& c_1(\B)^2\cdot H_{\B}-c_2(\B)\cdot H_{\B} 
 = K_{\B}^2\cdot H_{\B}-c_2(\B)\cdot H_{\B} \\
 &=& 3K_{\B}\cdot H_{\B}^2-2H_{\B}^3-2c_2(\B)\cdot H_{\B}+12\left(\chi(\O_{\B}(H_{\B}))-\chi(\O_{\B})\right), \\
 s_3(\T_{\B}) &=& -c_1(\B)^3+2c_1(\B)\cdot c_2(\B)-c_3(\B)
=K_{\B}^3+48\chi(\O_{\B})-c_3(\B).
\end{eqnarray*}
Hence, applying the double point formula and using 
 the relations 
$\chi(\O_{\B})=2\lambda-g+a-17$, $\chi(\O_{\B}(H_{\B}))=9$, we obtain:
\begin{eqnarray*}
4d-2 &=& 2\,\deg(\Sec(\B))\\
&=& \deg(\B)^2-s_3(\T_{\B})-7\,s_2(\T_{\B})\cdot H_{\B}-21\,s_1(\T_{\B})\cdot H_{\B}^2-35\,H_{\B}^3 \\
 &=& \deg(\B)^2-21\,\deg(\B)-42\,K_{\B}\cdot H_{\B}^2+14\,c_2(\B)\cdot H_{\B}-K_{\B}^3 \\
 && +c_3(\B)-84\,\chi(\O_{\B}(H_{\B}))+36\,\chi(\O_{\B}) \\
 &=& -K_{\B}^3+\lambda^2+23\lambda-24g-(7d+1)\Delta+36a-228.
\end{eqnarray*} 
\end{proof}
\begin{lemma}\label{lemma: quadric fibration}
If $r=3$, $n=8$, $\langle \B\rangle=\PP^8$ and
 $\B$ is a quadric fibration over a curve,
then  one of the following cases holds: 
\begin{itemize}
\item $a=3$, $\lambda=9$, $g=3$, $d=3$, $\Delta=5$; 
\item $a=4$, $\lambda=8$, $g=2$, $d=2$, $\Delta=10$. 
\end{itemize}
\end{lemma}
\begin{proof}
Denote by $\beta:(\B,H_{\B})\rightarrow (Y,H_Y)$ 
the projection over the curve $Y$
such that $\beta^{\ast}(H_Y)=K_{\B}+2H_{\B}$.
We have 
\begin{eqnarray*}
 0&=&\beta^{\ast}(H_Y)^2\cdot H_{\B} 
  = K_{\B}^2\cdot H_{\B}+4K_{\B}\cdot H_{\B}^2+4H_{\B}^3, \\
 0&=& \beta^{\ast}(H_Y)^3= K_{\B}^3+6K_{\B}^2\cdot H_{\B}+12 K_{\B}\cdot H_{\B}^2+8 H_{\B}^3, \\ 
\chi(\O_{\B}(H_{\B})) &=& \frac{1}{12} K_{\B}^2\cdot H_{\B}-\frac{1}{4}K_{\B}\cdot H_{\B}^2+\frac{1}{6}H_{\B}^3+\frac{1}{12}c_2(\B)\cdot H_{\B}+\chi(\O_{\B}), \\ 
\end{eqnarray*}
from which it follows that 
\begin{eqnarray}
K_{\B}^3 &=& -8\lambda+24g-24, \\
c_2(\B)\cdot H_{\B} 
 &=& -36\lambda+26g-12a+298.
\end{eqnarray}
Hence, by Lemma \ref{lemma: double point formula} and 
Proposition \ref{prop: segre and chern classes}, 
 we obtain
\begin{eqnarray}
 d\Delta&=& 23\lambda-16g+12a-180 , \\ 
\Delta+4d&=&\lambda^2-130\lambda+64g-48a+1058 .
\end{eqnarray}
Now the conclusion follows from Lemma 
\ref{lemma: r=3 B nondegenerate},
by observing that 
the case $a=6$ cannot occur. In fact,
if $a=6$, by \cite{ionescu-smallinvariants} 
it follows that $\B$ is a rational normal scroll 
and by a direct calculation (or by Lemma \ref{lemma: scroll over curve}) we see that 
$d=2$ and $\Delta=14$.
\end{proof}
\begin{lemma}\label{lemma: scroll over surface}
If $r=3$, $n=8$, $\langle \B\rangle=\PP^8$ and
 $\B$ is a scroll over a smooth surface $Y$,
then we have:
\begin{eqnarray*}
c_2\left(Y\right) &=& \left(\left(7d-1\right)\lambda^2+\left(177-679d\right)\lambda+\left(292d-92\right)g-28d^2 \right. \\ 
                  && \left. +\left(5554-252a\right)d+36a-1474\right)/\left(2d+2\right), \\ 
\Delta &=& \left(\lambda^2-107\lambda+48g-4d-36a+878\right)/\left(d+1\right) .
\end{eqnarray*}
\end{lemma}
\begin{proof}
Similarly to Lemma \ref{lemma: quadric fibration}, 
denote by $\beta:(\B,H_{\B})\rightarrow (Y,H_Y)$ 
the projection over the surface  $Y$ such that
 $\beta^{\ast}(H_Y)=K_{\B}+2H_{\B}$. Since $\beta^{\ast}(H_Y)^3=0$ we obtain
\begin{eqnarray*}
K_{\B}^3 &=&-8H_{\B}^3-12K_{\B}\cdot H_{\B}^2-6K_{\B}^2\cdot H_{\B} \\
 &=& -30K_{\B}\cdot H_{\B}^2+4H_{\B}^3+6c_2(\B)\cdot H_{\B}-72\chi(\O_{\B}(H_{\B}))+72\chi(\O_{\B}) \\
&=& 130\lambda-72g-6d\Delta+72a-1104.
\end{eqnarray*}
Now  we conclude comparing the last formula
with Lemma \ref{lemma: double point formula} and
using the relation
\begin{equation}
70\lambda-44g+(7d-1)\Delta-596=c_3(\B)=c_1(\PP^1)c_2(Y)=2c_2(Y).
\end{equation}
\end{proof}
\begin{lemma}\label{lemma: scroll over curve}
If $r=3$, $n=8$, $\langle \B\rangle=\PP^8$ and
 $\B$ is a scroll over a smooth curve,
then we have:
 $a=6$, $\lambda=6$, $g=0$, $d=2$, $\Delta=14$.
\end{lemma}
\begin{proof}
We have a projection $\beta:(\B,H_{\B})\rightarrow (Y,H_Y)$ 
 over a curve  $Y$ such that
 $\beta^{\ast}(H_Y)=K_{\B}+3H_{\B}$. By expanding the expressions 
$\beta^{\ast}(H_Y)^2\cdot H_{\B}=0$ and  $\beta^{\ast}(H_Y)^3=0$ we obtain 
$K_{\B}^2\cdot H_{\B}=3\lambda-12g+12$ and 
$K_{\B}^3=54(g-1)$,
and hence by Lemma \ref{lemma: double point formula} we get
\begin{equation}\label{eq: relation scroll}
 \lambda^2+23\lambda-78g-(7d+1)\Delta-4d+36a-172 = 0.
\end{equation}
Also, by expanding the expression $\chi(\O_{\B}(H_{\B}))=9$ we obtain 
$c_2=-35\lambda+30g-12a+294 $
and hence by Proposition \ref{prop: segre and chern classes} we get
\begin{equation}\label{eq: second relation scroll}
 22\lambda-20g-d\Delta+12a-176 = 0.
\end{equation}
Now the conclusion follows from Lemma \ref{lemma: r=3 B nondegenerate}. 
\end{proof}
Finally we conclude our discussion about classification with the following:
\begin{proposition}\label{prop: r=3 B nondegenerate}
If $r=3$, $n=8$ and $\langle \B\rangle=\PP^8$, 
then  one of the following cases holds:
\begin{enumerate}[(I)]
\setcounter{enumi}{18} 
\item $a=0$,  $\lambda=12$,  $g=6$,  $\B$ is a scroll $\PP_{Y}(\E)$ over a rational surface $Y$ with $K_Y^2=5$, $c_2(\E)=8$ and $c_1^2(\E)=20$; 
\item $a=0$,  $\lambda=13$,  $g=8$,  $\B$ is  obtained as the blow-up of a Fano variety $X$ at a point $p\in X$,  $|H_{\B}|=|H_{X}-p|$;  
\item\label{case: cubic hypersurface} $a=1$,  $\lambda=11$,  $g=5$,  $\B$ is the blow-up of $Q^3$ at $5$ points $p_1,\ldots,p_5$, $|H_{\B}|=|2H_{Q^3}-p_1-\cdots-p_5|$;
\item $a=1$,  $\lambda=11$,  $g=5$,  $\B$ is a scroll over $\PP_{\PP^1}(\O\oplus\O(-1))$; 
\item $a=1$,  $\lambda=12$,  $g=7$,  $\B$ is a linear section of $S^{10}\subset\PP^{15}$;
\item\label{case: scroll over Q2 or quadric fibration} $a=2$,  $\lambda=10$,  $g=4$,  $\B$ is a scroll over $Q^2$;
\item $a=3$,  $\lambda=9$,   $g=3$,  $\B$ is a scroll over $\PP^2$ or a quadric fibration over $\PP^1$; 
\item $a=4$,  $\lambda=8$,   $g=2$,  $\B$ is a hyperplane section of $\PP^1\times Q^3$; 
\item $a=6$,  $\lambda=6$,   $g=0$,  $\B$ is a rational normal scroll.
\end{enumerate}
\end{proposition}
\begin{proof}
For $a=6$ the statement follows from \cite{ionescu-smallinvariants}.
The case with $a=5$ is excluded by  \cite{ionescu-smallinvariants} and Example
\ref{example: a=5}.
For $a=4$ the statement follows from \cite{ionescu-smallinvariantsIII}.  
 For $a\in\{2,3\}$,
by \cite{fania-livorni-nine}, 
   \cite{fania-livorni-ten} and
   \cite{ionescu-smallinvariantsII} 
it follows that the abstract structure of
$\B$  is as asserted, or $a=2$ and $\B$ is a quadric 
fibration over $\PP^1$; the last case is 
excluded by  
Lemma \ref{lemma: quadric fibration}.
 For $a=1$  the statement is just Proposition \ref{prop: 3-fold in P8}. 
Now we treat the cases with $a=0$.
\begin{case}[$a=0, \lambda=12$]
 Since $\deg(\B)\leq 2\mathrm{codim}_{\PP^8}(\B)+2$, it follows that 
$(\B,H_{\B})$  must be as in one of the cases
 (a),\ldots,(h) of \cite[Theorem~1]{ionescu-degsmallrespectcodim}.
Cases (a), (d), (e), (g), (h) are of course impossible 
and case (c) is excluded by Lemma \ref{lemma: quadric fibration}. 
If $\B$ is as in case (b), 
by Lemma \ref{lemma: scroll over curve}
we obtain that
$\B$ is a scroll over a birationally ruled surface 
(hence over a rational surface since $q=0$). 
Now suppose that   $(\B,H_{\B})$ is as in case (f). 
Thus there is a reduction $(X,H_X)$ as in one of the cases:
\begin{enumerate}[(f1)]
 \item\label{case1} $X=\PP^3$, $H_X\in|\O(3)|$;
 \item\label{case2} $X=Q^3$, $H_X\in|\O(2)|$;
 \item\label{case3} $X$ is a $\PP^2$-bundle over a smooth curve such that 
                    $\O_X(H_X)$ induces $\O(2)$ on each fiber.
\end{enumerate}
By definition of reduction we have $X\subset\PP^{N}$,
where $N=8+s$, 
$\deg(X)=\lambda+s=12+s$ and
 $s$ is the number of points blown up on $X$ to get $\B$. 
Case (f\ref{case1}) and (f\ref{case2}) 
are impossible because
they force $\lambda$ to be
respectively $16$ and $11$.
In case (f\ref{case3}), 
we have a projection $\beta:(X,H_{X})\rightarrow (Y,H_Y)$ 
 over a curve  $Y$ such that
 $\beta^{\ast}(H_Y)=2K_{X}+3H_{X}$. 
Hence we get 
\begin{displaymath}
 K_X H_X^2= (2K_X+3H_X)^2\cdot H_X/12 -K_X^2\cdot H_X/3 - 3H_X^3/4 = -K_X^2\cdot H_X/3 - 3H_X^3/4 ,
\end{displaymath}
from which we deduce that
\begin{eqnarray*}
 0&=&(2K_X+3H_X)^3 
   = 8K_X^3+36K_X^2\cdot H_X+54K_X\cdot H_X^2+27H_X^3 \\
  &=& 8K_X^3+18K_X^2\cdot H_X-27 H_X^3/2 \\
 &=& 8( K_{\B}^3 - 8s )+18K_X^2\cdot H_X-27 (\deg(\B)+s)/2 \\ 
 &=& 18 K_X^2\cdot H_X-155s/2-210.
\end{eqnarray*} 
Since $s\leq 12$ (see \cite[Lemma~8.1]{besana-biancofiore-numerical}),  
we conclude that case (f)
does not occur. 
Thus,  $\B=\PP_{Y}(\E)$ is a scroll over a surface $Y$; 
moreover,
by Lemma \ref{lemma: scroll over surface} and 
\cite[Theorem~11.1.2]{beltrametti-sommese}, 
we obtain $K_Y^2=5$, $c_2(\E)=K_Y^2-K_S^2=8$ and $c_1^2(\E)=\lambda+c_2(\E)=20$.
\end{case}
\begin{case}[$a=0, \lambda=13$]
The proof is located in  \cite[page~16]{mella-russo-baselocusleq3},
but we sketch it for the reader's convenience.
By Lemma \ref{lemma: r=3 B nondegenerate} we know that
$\chi(\O_S)=2$ and  $K_S$ is an exceptional curve of the first kind.
Thus, if we blow-down the divisor $K_S$,  we obtain 
a $K3$-surface. 
By using adjunction theory 
(see for instance \cite{beltrametti-sommese} or  Ionescu's papers cited in the 
bibliography)  
and  by Lemmas \ref{lemma: quadric fibration}, 
\ref{lemma: scroll over surface} and \ref{lemma: scroll over curve} it follows that 
the adjunction map $\phi_{|K_{\B}+2H_{\B}|}$ is a generically finite morphism;
moreover, since $(K_{\B}+2H_{\B})\cdot K_S=0$, we see that
$\phi_{|K_{\B}+2H_{\B}|}$ is not a finite morphism.
So, we deduce that 
there is a $(\PP^2,\O_{\PP^2}(-1))$ inside $\B$  
and, after the blow-down of this divisor, we get a smooth Fano $3$-fold 
$X\subset\PP^9$  of sectional genus $8$ and degree $14$. 
\end{case}
\end{proof}
\section{Examples}\label{sec: examples}
As in \S \ref{sec: examples hypersurface}, 
in order to verify the calculations in the following examples,
we suggest the use of \cite{macaulay2} or \cite{sagemath}.
\begin{example}[$r=1,2,3; n=3,4,5; a=1; d=1$]\label{example: 1}
As already said in \S \ref{sec: type 2-1},
 if $Q\subset\PP^{n-1}\subset\PP^n$ is a smooth quadric, then
the linear system $|\I_{Q,\PP^n}(2)|$ defines  a birational transformation
$\psi:\PP^n\dashrightarrow\sS\subset\PP^{n+1}$ of type $(2,1)$ whose image
 is a smooth quadric.
\end{example}
\begin{example}[$r=1; n=4; a=0; d=3$]\label{example: 2}
See also \cite{crauder-katz-1989}. If $X\subset\PP^4$ 
is a nondegenerate curve of genus $1$ and degree $5$, then $X$ is 
the scheme-theoretic intersection of the quadrics (of rank $3$) containing $X$ and 
$|\I_{X,\PP^4}(2)|$ defines a Cremona transformation 
$\PP^4\dashrightarrow\PP^4$ of type $(2,3)$.
\end{example}
\begin{example}[$r=1,2,3; n=4,5,7; a=1,0,1; d=2$]\label{example: 3}
As already said in Example \ref{example: d=2 Delta=2},
if $X\subset\PP^n$
is either $\nu_2(\PP^2)\subset\PP^5$ or $\PP^2\times\PP^2\subset\PP^8$,
%
 then $|\I_{X,\PP^n}(2)|$ defines
a birational transformation
 $\psi:\PP^n\dashrightarrow\PP^n$ of type $(2,2)$. 
The restriction of  $\psi$
to a general hyperplane is a birational transformation 
 $\PP^{n-1}\dashrightarrow\sS\subset\PP^n$ of type $(2,2)$ and
 $\sS$ is a smooth quadric.
\end{example}
\begin{example}[$r=1; n=4; a=2; d=1$ - not satisfying \ref{assumption: ipotesi}]\label{example: B2=singSred}
This essentially gives an example for case (\ref{case: escluso 1-fold}) of Lemma \ref{lemma: numerical 1-fold}. 
We have a special birational transformation
$\psi:\PP^4\dashrightarrow\sS\subset\PP^6$ of type $(2,1)$
with base locus $X$, image  $\sS$
and base locus of the inverse $Y$, as follows:
\begin{eqnarray*}
X &=& V(x_0x_1-x_2^2-x_3^2,-x_0^2-x_1^2+x_2x_3,x_4), \\
\sS &= & V(y_2y_3-y_4^2-y_5^2-y_0y_6,y_2^2+y_3^2-y_4y_5+y_1y_6), \\
 P_{\sS}(t) &=& (4t^4+24t^3+56t^2+60t+24)/4!, \\ 
\sing(\sS) &=& V(y_6,y_5^2,y_4y_5,y_3y_5,y_2y_5,y_4^2,y_3y_4,y_2y_4,2y_1y_4+y_0y_5, \\
           && y_0y_4+2y_1y_5,y_3^2,y_2y_3,y_2^2,y_1y_2+2y_0y_3,2y_0y_2+y_1y_3), \\
P_{\sing(\sS)}(t) &=& t + 5, \\
(\sing(\sS))_{\mathrm{red}} &=& V(y_6,y_5,y_4,y_3,y_2), \\
Y=(Y)_{\mathrm{red}}&=&(\sing(\sS))_{\mathrm{red}}= V(y_6,y_5,y_4,y_3,y_2).
\end{eqnarray*}
\end{example}
\begin{example}[$r=1,2,3; n=4,5,6; a=3; d=1$]\label{example: 5}
See also \cite{russo-simis} and \cite{semple}.
If $X=\PP^1\times\PP^2\subset\PP^5\subset\PP^{6}$, 
then $|\I_{X,\PP^6}(2)|$ defines 
a birational transformation 
$\psi:\PP^{6}\dashrightarrow \sS\subset\PP^{9}$ of type $(2,1)$
whose base locus is $X$ and whose image is $\sS=\GG(1,4)$.
Restricting $\psi$ to a general $\PP^5\subset\PP^{6}$ 
(resp. $\PP^4\subset\PP^{6}$)
we obtain a birational transformation
 $\PP^5\dashrightarrow\sS\subset\PP^{8}$
(resp. $\PP^4\dashrightarrow\sS\subset\PP^{7}$) 
whose  image is a smooth linear section of $\GG(1,4)\subset\PP^{9}$.
\end{example}
\begin{example}[$r=2; n=6; a=0; d=4$]\label{example: 6}
See also \cite{crauder-katz-1989} and \cite{hulek-katz-schreyer}. 
Let $Z=\{p_1,\ldots,p_8\}$ be 
a set of $8$ points in $\PP^2$
such that 
no $4$ of the $p_i$ are collinear and no $7$ of the $p_i$ lie on a conic and
consider the blow-up $X=\Bl_Z(\PP^2)$ 
embedded in $\PP^6$ by $|4H_{\PP^2}-p_1-\cdots-p_8|$. 
Then the homogeneous ideal of $X$ is
generated by quadrics and $|\I_{X,\PP^6}(2)|$ 
defines a Cremona transformation $\PP^6\dashrightarrow\PP^6$ of type $(2,4)$.
The same happens when
 $X\subset\PP^6$ is a septic elliptic scroll with $e=-1$.
\end{example}
\begin{example}[$r=2; n=6; a=1; d=3$]\label{example: 7}
As already said in Examples \ref{example: d=3 Delta=2} and \ref{example: d=3 Delta=2 continuing}, 
if $X\subset\PP^6$ is a general hyperplane section of an
Edge variety  of dimension $3$ and degree $7$ in $\PP^7$,
then $|\I_{X,\PP^6}(2)|$ defines a birational transformation
$\psi:\PP^6\dashrightarrow\sS\subset\PP^7$ of type $(2,3)$ whose
image is a rank $6$ quadric.
\end{example}
\begin{example}[$r=2; n=6;a=2;d=2$]\label{example: 8}
If $X\subset\PP^6$ 
is the blow-up of $\PP^2$ at $3$ general points $p_1,p_2,p_3$ with
 $|H_{X}|=|3H_{\PP^2}-p_1-p_2-p_3|$, then 
$\Sec(X)$ is a cubic hypersurface. 
By Fact \ref{fact: K2 property} and \ref{fact: test K2} we deduce  
that $|\I_{X,\PP^6}(2)|$ defines a birational transformation 
$\psi:\PP^6\dashrightarrow\sS\subset\PP^8$ and its type is $(2,2)$.
The image $\sS$ is a complete intersection of two quadrics, $\dim(\sing(\sS))=1$ and 
the base locus of the inverse is $\PP^2\times\PP^2\subset\PP^8$.
Alternatively, we can obtain the transformation 
$\psi:\PP^6\dashrightarrow\sS\subset\PP^8$
by restriction to a general $\PP^6\subset\PP^8$
of the special Cremona transformation $\PP^8\dashrightarrow\PP^8$ of type $(2,2)$.
\end{example}
\begin{example}[$r=2; n=6; a=3; d=2$]\label{example: 9}
See also \cite{russo-simis} and \cite{semple}.
If $X=\PP_{\PP^1}(\O(1)\oplus\O(4))$ or $X=\PP_{\PP^1}(\O(2)\oplus\O(3))$, then
$|\I_{X,\PP^6}(2)|$ defines a birational
 transformations 
 $\psi:\PP^6\dashrightarrow\sS\subset\PP^9$
of type $(2,2)$ whose base locus is $X$ and whose image is
$\sS=\GG(1,4)$.
\end{example}
\begin{example}[$r=2,3;n=6,7;a=5; d=1$]\label{example: 10}
See also \cite[\Rmnum{3} Theorem~3.8]{zak-tangent}.
If $X=\GG(1,4)\subset\PP^9\subset\PP^{10}$, then 
$|\I_{X,\PP^{10}}(2)|$ defines a birational transformation 
$\psi:\PP^{10}\dashrightarrow \sS\subset\PP^{15}$ of type $(2,1)$
whose base locus is $X$ and whose image is the spinorial variety $\sS=S^{10}\subset\PP^{15}$.
Restricting $\psi$
  to a general 
 $\PP^7\subset\PP^{10}$ 
 (resp. $\PP^6\subset\PP^{10}$) 
we obtain a special birational transformation
 $\PP^7\dashrightarrow\sS\subset\PP^{12}$
(resp. $\PP^6\dashrightarrow\sS\subset\PP^{11}$) 
whose dimension of the base locus is
$r=3$ (resp. $r=2$) 
and whose image is a linear section of $S^{10}\subset\PP^{15}$.
In the first case $\sS=\overline{\psi(\PP^7)}$ is smooth while in the second case 
the singular locus of $\sS=\overline{\psi(\PP^6)}$ consists of 
$5$ lines, image of the $5$ Segre $3$-folds 
containing del Pezzo surface of degree $5$ 
and spanned by its pencils of conics.
\end{example}
\begin{example}[$r=2,3; n=6,7; a=6; d=1$]\label{example: 11}
See also \cite{russo-simis}, \cite{semple} and \cite[\Rmnum{3} Theorem~3.8]{zak-tangent}.
We have a birational transformation
$\psi:\PP^{8}\dashrightarrow\GG(1,5)\subset\PP^{14}$ of type $(2,1)$
whose base locus 
is $\PP^1\times\PP^3\subset\PP^7\subset\PP^{8}$ and whose image is $\GG(1,5)$.
Restricting $\psi$
  to a general  $\PP^7\subset\PP^{8}$ 
 we obtain a birational transformation 
 $\PP^7\dashrightarrow\sS\subset\PP^{13}$
whose base locus $X$ is
a rational normal scroll 
and whose image $\sS$ is  a smooth linear section of
 $\GG(1,5)\subset\PP^{14}$. 
Restricting $\psi$
  to a general  $\PP^6\subset\PP^{8}$ 
 we obtain a birational transformation 
 $\psi=\psi|_{\PP^6}:\PP^6\dashrightarrow\sS\subset\PP^{12}$
whose base locus $X$ is
a rational normal scroll 
(hence either $X=\PP_{\PP^1}(\O(1)\oplus\O(3))$ 
or  $X=\PP_{\PP^1}(\O(2)\oplus\O(2))$)
and whose image $\sS$ is  a singular linear section of
 $\GG(1,5)\subset\PP^{14}$. In this case,
we denote by $Y\subset\sS$ the base locus of the inverse of $\psi$ and by 
$F=(F_0,\ldots,F_5):\PP^5\dashrightarrow\PP^5$ 
the restriction of $\psi$ to $\PP^5=\Sec(X)$. 
We have
\begin{eqnarray*}
Y&=&\overline{\psi(\PP^5)}=\overline{F(\PP^5)}=\GG(1,3)\subset\PP^5\subset\PP^{12} , \\
J_4&:=&\left\{x=[x_0,...,x_5]\in\PP^5\setminus X: 
\rk\left(\left({\partial F_i}/{\partial x_j}(x)\right)_{i,j}\right)\leq 4   \right\}_{\mathrm{red}}\\
    &=& \left\{x=[x_0,...,x_5]\in\PP^5\setminus X: \dim\left(\overline{F^{-1}\left(F(x)\right)}\right)\geq2   \right\}_{\mathrm{red}}\mbox{ and }\dim\left(J_4\right) = 3,\\
\overline{\psi\left(J_4\right)} &=& \left(\sing\left(\sS\right)\right)_{\mathrm{red}} =\PP_{\PP^1}(\O(2)) \subset Y. \\
\end{eqnarray*}
\end{example}
\begin{example}\label{example: parzialecremona}
We have a good candidate for the restriction to a general $\PP^6\subset\PP^8$ 
of a quadratic Cremona transformation 
of $\PP^8$ whose base locus is a $3$-fold of 
degree $12$ and sectional genus $6$.
In fact, there exists a  smooth nondegenerate curve 
 $C\subset\PP^6$ of degree $12$, genus $6$, 
and having homogeneous ideal generated 
by $9$ quadrics. 
Moreover, if
$\psi:\PP^6\dashrightarrow\mathbf{Z}\subset\PP^{8}$ is 
the rational map 
defined by $|\I_{C,\PP^6}(2)|$, then the image $\mathbf{Z}$ has degree $14$ and 
so $\psi$ is birational. We also have that the homogeneous ideal of $\mathbf{Z}$ 
is generated by quintics and sextics.

The curve $C\subset\PP^6$ 
can be constructed as follows:
let $Y\subset\PP^5$ be a del Pezzo surface of degree $5$. Consider 
the Veronese embedding $\nu_2:\PP^5\rightarrow\PP^{20}$ and let $Y'=\nu_2(Y)$. 
$Y'$ is a surface of degree $20$, sectional genus $6$ and its linear span is a 
$\PP^{15}\subset\PP^{20}$. 
Thus, via a sequence of $8$ inner projections, we obtain 
a surface $S\subset\PP^7$ of degree $12$ and sectional genus $6$. 
Now we take a smooth hyperplane section of $S$.
This is what the following Macaulay2 code does.
{\footnotesize  
\begin{verbatim}
i1 : ringP2=QQ[z_0..z_2];
i2 : ringP5=QQ[t_0..t_5];
i3 : g=map(ringP2,ringP5,gens image basis(3,intersect(ideal(z_0,z_1),ideal(z_1,z_2),
                ideal(z_2,z_0),ideal(z_0-z_1,z_2-z_0))));
i4 : ringP20=QQ[t_0..t_20];
i5 : v=map(ringP5,ringP20,gens (ideal vars ringP5)^2);
i6 : ringP15=ringP20/(ideal image basis(1,saturate kernel(g*v)));
i7 : p=map(ringP20,ringP15,vars ringP20); 
i8 : f=g*v*p;
i9 : L=ideal image basis(1,intersect(preimage(f,ideal(z_0,z_1-z_2)),
                preimage(f,ideal(z_1,z_2-z_0)),preimage(f,ideal(z_2,z_0-z_1)),
                preimage(f,ideal(z_0,z_1+z_2)),preimage(f,ideal(z_1,z_2+z_0)),
                preimage(f,ideal(z_2,z_0+z_1)),preimage(f,ideal(z_0-z_1,z_2+z_0)),
                preimage(f,ideal(z_0+z_1,z_2-z_0))));
i10 : ringP7=QQ[x_0..x_7];
i11 : h=f*map(ringP15,ringP7,gens(L));
i12 : idealS=saturate kernel h;
i13 : H=sub(idealS, {x_7=>x_0+x_1+x_2+x_3+x_4+x_5+x_6});
i14 : ringP6=QQ[x_0..x_6]; 
i15 : idealC=saturate sub(H,ringP6);
i16 : C=Proj(ringP6/idealC);
i17 : dim singularLocus C
o17 = -infinity
i18 : dim C
o18 = 1
i19 : degree C
o19 = 12
i20 : genus C
o20 = 6
i21 : numgens idealC
o21 = 9
i22 : ringP8=QQ[y_0..y_8];
i23 : psi=map(ringP6,ringP8,gens idealC);
i24 : Z=Proj(coimage(psi));
i25 : dim Z
o25 = 6
i26 : degree Z
o26 = 14
i27 : DegreeOfpsi=(-5*degree(C)+2*genus(C)+62)/degree(Z) 
o27 = 1 
\end{verbatim} 
}
\end{example}
\begin{example}[$r=3;n=8;a=0;d=5$]\label{example: 12}
 See also \cite{hulek-katz-schreyer}. If $\mathcal{X}\subset\PP^9$ is a 
general $3$-dimensional 
 linear section of $\GG(1,5)\subset \PP^{14}$,
 $p\in \mathcal{X}$ is a general point and 
 $X\subset\PP^8$ is the image of $\mathcal{X}$ under the projection from $p$,
then the homogeneous ideal of $X$ is generated by quadrics 
and $|\I_{X,\PP^8}(2)|$ defines a Cremona transformation $\PP^8\dashrightarrow\PP^8$
of type $(2,5)$.
\end{example}
\begin{example}[$r=3;n=8;a=1;d=3$]\label{example: 13} 
As already said in Example \ref{example: d=3 Delta=3},
if $X\subset\PP^8$ 
is the blow-up 
of the smooth quadric $Q^3\subset\PP^4$ 
at $5$ general points $p_1,\ldots,p_5$ with
 $|H_{X}|=|2H_{Q^3}-p_1-\cdots-p_5|$, then 
$|\I_{X,\PP^8}(2)|$ defines a birational transformation 
$\psi:\PP^8\dashrightarrow\sS\subset\PP^9$ of type $(2,3)$ whose 
image is a cubic hypersurface with singular locus of dimension $3$.
\end{example}
\begin{example}[$r=3;n=8;a=1;d=4$ - incomplete]\label{example: 14}
By \cite{alzati-fania-ruled} (see also \cite{besana-fania-flamini-f1}) there exists
a smooth irreducible nondegenerate linearly normal $3$-dimensional variety 
 $X\subset\PP^8$  with
 $h^1(X,\O_X)=0$, 
 degree $\lambda=11$, sectional genus $g=5$, having the structure of a 
 scroll $\PP_{\FF^1}(\E)$  with $c_1(\E)=3C_0+5f$ and $c_2(\E)=10$ 
and hence having degrees of the Segre classes 
 $s_1(X)=-85$, $s_2(X)=386$, $s_3(X)=-1330$.
Now, by Fact \ref{fact: test K2}, $X\subset\PP^8$ is 
arithmetically Cohen-Macaulay
and by Riemann-Roch Theorem, denoting with $C$ a general curve section of $X$,
 we obtain
\begin{equation}\label{eq: riemann-rock}
h^0(\PP^8,\I_X(2))=h^0(\PP^6,\I_C(2))
= h^0(\PP^6,\O_{\PP^6}(2))-h^0(C,\O_C(2))
=28-(2\lambda+1-g),
\end{equation}
hence $h^0(\PP^8,\I_X(2))=10$.
If the homogeneous ideal of $X$ is generated 
by quadratic forms\footnote{One says that 
a subvariety $X\subset\PP^n$ is $2$-regular (in the sense of Castelnuovo-Mumford)
if  $h^j(\PP^n,\I_{X}(2-j))=0$ for all $j>0$.
If $X$ is $2$-regular, 
then its homogeneous ideal is generated 
by forms of degrees $\leq2$ (see \cite{mumford-curves}).
Now, if $X\subset\PP^8$ is a scroll over $\FF_1$ as above, we have 
$h^j(\PP^8,\I_{X}(2-j))=0$ for $j>0$ and $j\neq 4$, 
but unfortunately we also have
$h^4(\PP^8,\I_{X}(-2))=h^3(X,\O_{X}(-2))=-P_{X}(-2)=5$.} 
or at least if  
$X=V(H^0(\I_X(2)))$,
the linear system $|\I_X(2)|$ defines a rational map 
$\psi:\PP^8\dashrightarrow\sS=\overline{\psi(\PP^8)}\subset\PP^{9}$ 
whose base locus is
 $X$ and whose image $\sS$ is nondegenerate.
Now, by (\ref{eq: grado mappa razionale}) we deduce   $\deg(\psi)\deg(\sS)=2$, 
from which 
$\deg(\psi)=1$ and $\deg(\sS)=2$.

We have a good candidate for the restriction of $\psi$ to a general $\PP^6\subset\PP^8$.
In fact, there exists a smooth irreducible nondegenerate linearly normal curve 
 $C\subset\PP^6$ of degree $\lambda=11$, genus $g=5$, 
and having homogeneous ideal generated 
by the $10$ quadrics:
\begin{equation}\label{eq: 10 quadrics generating C}
\begin{array}{c}
x_5^2-x_4 x_6 ,  \    
-x_3 x_6+x_2 x_6+x_4 x_5-x_0 x_1 , \    
x_3 x_5-x_2 x_6 , \     
x_2 x_5-x_1 x_6 , \\   
x_3 x_4-x_1 x_6 , \  
x_2 x_4-x_1 x_5 , \  
x_2 x_6-x_1 x_6+x_1 x_5-x_4^2-x_3^2+x_2 x_3+x_0^2 , \\      
x_0 x_6-x_1 x_5+x_1 x_3 , \   
x_0 x_6-x_1 x_5+x_2^2 ,  \   
x_0 x_5-x_1 x_4+x_1 x_2 .
\end{array}
\end{equation}
The quadrics (\ref{eq: 10 quadrics generating C}) give a rational map 
$\psi':\mathbb{P}^6\dashrightarrow\sS'\subset\mathbb{P}^9$ and since
$\deg(\psi')\deg(\sS')=-5\lambda+2g+62=17$, we have that 
$\psi'$ is birational.  Moreover, 
the homogeneous ideal of $\sS'$ is generated 
by  $6$ quartics and a rank $6$ quadric defined by:
$$
y_3^2-y_3y_4+y_2y_5+y_0y_7-y_0y_8 .
$$
\end{example}
\begin{example}[$r=3;n=8;a=1;d=4$]\label{example: 15} 
As already said in Example \ref{example: d=4 Delta=2},
if $X\subset\PP^8$ is
a general linear $3$-dimensional section 
of the spinorial variety $S^{10}\subset\PP^{15}$, then
$|\I_{X,\PP^8}(2)|$ defines a
 birational transformation 
 $\psi:\PP^8\dashrightarrow\sS\subset\PP^9$
of type $(2,4)$ 
whose image is a smooth quadric.
\end{example}
\begin{example}[$r=3;n=8;a=2;d=3$]\label{example: 16}
By \cite{fania-livorni-ten} (see also \cite{besana-fania-threefolds}) there exists
a smooth irreducible nondegenerate linearly normal $3$-dimensional variety 
 $X\subset\PP^8$  with $h^1(X,\O_X)=0$,  
 degree $\lambda=10$, sectional genus $g=4$, 
having the structure of a  scroll 
$\PP_{Q^2}(\E)$ with $c_1(\E)=\O_Q(3,3)$ and $c_2(\E)=8$ 
and hence having degrees of the Segre classes 
$s_1(X)=-76$, $s_2(X)=340$, $s_3(X)=-1156$.
By Fact \ref{fact: test K2}, $X\subset\PP^8$ is
arithmetically Cohen-Macaulay and its  homogeneous ideal is generated by quadratic forms. 
So by (\ref{eq: riemann-rock}) we have $h^0(\PP^8,\I_X(2))=11$ and 
the linear system $|\I_X(2)|$ defines a rational map 
$\psi:\PP^8\dashrightarrow\sS\subset\PP^{10}$ 
whose base locus is
 $X$ and whose image $\sS$ is nondegenerate.
By (\ref{eq: grado mappa razionale}) it follows that $\deg(\psi)\deg(\sS)=4$ 
and hence $\deg(\psi)=1$ and $\deg(\sS)=4$.
\end{example}
\begin{example}[$r=3;n=8;a=3;d=2$]\label{example: 17}
By \cite{fania-livorni-nine} (see also \cite{besana-fania-threefolds}) there exists
a smooth irreducible nondegenerate linearly normal $3$-dimensional variety 
 $X\subset\PP^8$  with $h^1(X,\O_X)=0$,  
 degree $\lambda=9$, sectional genus $g=3$, 
having the structure of a  scroll $\PP_{\PP^2}(\E)$ with $c_1(\E)=4$ and $c_2(\E)=7$ 
and hence having degrees of the Segre classes 
$s_1(X)=-67$, $s_2(X)=294$, $s_3(X)=-984$. 
By Fact \ref{fact: test K2}, $X\subset\PP^8$ is
arithmetically Cohen-Macaulay and its homogeneous ideal is generated by quadratic forms. 
So by (\ref{eq: riemann-rock}) we have $h^0(\PP^8,\I_X(2))=12$ and 
the linear system $|\I_X(2)|$ defines a rational map 
$\psi:\PP^8\dashrightarrow\sS\subset\PP^{11}$ 
whose base locus is
 $X$  and whose image $\sS$ is nondegenerate.
By (\ref{eq: grado mappa razionale}) it follows that 
$\deg(\psi)\deg(\sS)=8$ 
and  in particular 
$\deg(\psi)\neq 0$ i.e. 
$\psi:\PP^8\dashrightarrow\sS$ is 
generically quasi-finite.
Again by Fact \ref{fact: test K2} and Fact \ref{fact: K2 property} it follows
that $\psi$ is birational and hence $\deg(\sS)=8$.  

Now we find the equations for such an $X$.
 Consider the set of $7$ points in $\PP^2$,  
$$T:=\{[1,0,0],[0,1,0],[0,0,1],[1,1,0],[1,0,1],[0,1,1],[1,1,1]\}.$$
The homogeneous ideal of $T$ is generated by $3$ cubics; we have
$\dim |\I_{T,\PP^2}(4)|=7$ and let  $v=v_{|\I_{T,\PP^2}(4)|}:\PP^2\dashrightarrow\PP^7$. 
The image $S=\overline{v(\PP^2)}\subset\PP^7$ is a smooth nondegenerate surface, 
of degree  $9$, 
sectional genus $3$ and with homogeneous ideal generated by $12$ quadrics.
These $12$ quadrics define a 
special quadratic birational map $\psi':\PP^7\dashrightarrow\PP^{11}$ 
whose image is a complete intersection of $4$ quadrics and whose inverse can be defined 
by quadrics.
We then consider the $8$ quadrics defining 
$\psi'^{-1}$ and the $4$ quadrics defining $\overline{\psi'(\PP^7)}$.
These $12$ quadrics give a Cremona transformation of $\PP^{11}$.
Explicitly, there is a 
Cremona transformation $\phi:\PP^{11}\dashrightarrow\PP^{11}$ defined  
by:
\begin{equation}
\begin{array}{c}
x_6x_{10}-x_5x_{11} , \  
 x_1x_{10}-x_4x_{10}+x_3x_{11} , \  
 x_6x_8-x_2x_{11} , \  
 x_5x_8-x_2x_{10} , \\
 x_3x_8-x_0x_{10} , \  
 x_1x_8-x_4x_8+x_0x_{11} , \  
 x_6x_7-x_1x_9+x_4x_9-x_4x_{11} , \  
 x_5x_7+x_3x_9-x_4x_{10} , \\
 x_2x_7-x_4x_8+x_0x_9 , \  
 x_1x_5-x_4x_5+x_3x_6 , \  
 x_2x_3-x_0x_5 , \  
 x_1x_2-x_2x_4+x_0x_6 .
\end{array}
\end{equation}
The inverse of $\phi$ is defined by:
\begin{equation}
\begin{array}{c}
-y_5y_{10}+y_4y_{11} , \  
 y_5y_9-y_8y_9+y_6y_{10}-y_1y_{11}+y_7y_{11} , \  
 y_2y_{10}+y_3y_{11} , \  
 y_4y_9-y_1y_{10} , \\
 -y_8y_9+y_6y_{10}+y_7y_{11} , \  
 y_3y_9+y_0y_{10} , \  
 y_2y_9-y_0y_{11} , \  
 y_4y_6+y_5y_7-y_1y_8 , \\
 y_2y_4+y_3y_5 , \  
 -y_3y_6+y_2y_7-y_0y_8 , \  
 y_1y_3+y_0y_4 , \  
 y_1y_2-y_0y_5 .
\end{array}
\end{equation}
The base locus of $\phi$ (resp. $\phi^{-1}$) is a variety 
of dimension $6$, 
degree $9$,
sectional genus $3$,
and the support of the singular locus is a plane $\PP^2\subset\PP^{11}$.
In particular,  by restricting the above Cremona transformation 
to a general $\PP^8\subset\PP^{11}$, we obtain an explicit example 
of special quadratic birational transformation 
$\psi:\PP^8\dashrightarrow\mathbf{S}\subset\PP^{11}$; 
 that is we obtain explicit equations 
for a $3$-fold scroll $X\subset\PP^8$ over $\PP^2$, of degree 
$9$ and sectional genus $3$. (For example, by restricting to the subspace 
$V(x_0+x_1+x_2+x_5+x_8-x_{11}, x_1+x_2+x_3+x_4+x_7-x_{10}, x_0+x_3+x_4+x_5+x_6-x_9)$, 
everything works fine and the image is a complete intersection 
with singular locus of dimension $3$.)  

 The following Macaulay2 code computes the maps $\phi$ and $\phi^{-1}$. 
{\footnotesize  
\begin{verbatim}
i1 : installPackage("AdjointIdeal");
i2 : installPackage("Parametrization");
i3 : ringP2=QQ[z_0..z_2];
i4 : points=intersect(ideal(z_0,z_1),ideal(z_0,z_2),ideal(z_1,z_2),ideal(z_0,z_1-z_2),
                ideal(z_1,z_0-z_2),ideal(z_2,z_0-z_1),ideal(z_0-z_1,z_0-z_2));
i5 : parametr=gens(image(basis(4,points)));
i6 : ringP7=QQ[t_0..t_7];
i7 : idealS=saturate(kernel(map(ringP2,ringP7,parametr)));
i8 : -- You could work with the surface S, 
     -- but we prefer to work with a sectional curve for efficiency reasons.
     ringP6=QQ[t_0..t_6]; 
i9 : idealC=sub(sub(idealS, {t_7=>0}),ringP6);
i10 : ImInv=invertBirationalMap(ideal(ringP6),gens(idealC));
i11 : ringP11=QQ[x_0..x_11];
i12 : idealX=sub(ideal(image(basis(2,saturate(ideal(ImInv#0)+ImInv#1)))),vars(ringP11));
i13 : X=Proj(ringP11/idealX);
i14 : ImInv2=invertBirationalMap(ideal(ringP11),gens(idealX));
i15 : ringP11'=QQ[y_0..y_11];
i16 : Phi=map(ringP11,ringP11',gens(idealX));
i17 : InversePhi=map(ringP11',ringP11,sub(transpose(ImInv2#0),vars(ringP11'))); 
\end{verbatim} 
}

\end{example}
\begin{example}[$r=3;n=8;a=3;d=3$]\label{example: 17nuovo}
By \cite{fania-livorni-nine} (see also \cite{besana-fania-threefolds}) there exists
a smooth irreducible nondegenerate linearly normal $3$-dimensional variety 
 $X\subset\PP^8$  with $h^1(X,\O_X)=0$,  
 degree $\lambda=9$, sectional genus $g=3$, 
having the structure  
of a quadric fibration over $\PP^1$    
and hence having degrees of the Segre classes 
$s_1(X)=-67$, $s_2(X)=295$, $s_3(X)=-997$.  
By Fact \ref{fact: test K2}, $X\subset\PP^8$ is
arithmetically Cohen-Macaulay and its homogeneous ideal is generated by quadratic forms. 
So by (\ref{eq: riemann-rock}) we have $h^0(\PP^8,\I_X(2))=12$ and 
the linear system $|\I_X(2)|$ defines a rational map 
$\psi:\PP^8\dashrightarrow\sS\subset\PP^{11}$ 
whose base locus is
 $X$  and whose image $\sS$ is nondegenerate.
By (\ref{eq: grado mappa razionale}) it follows that   
$\deg(\psi)\deg(\sS)=5$  
and hence $\deg(\psi)=1$ and $\deg(\sS)=5$. 

Now we find the equations for such an $X$. 
 Consider the rational normal scroll 
$S(1,1,1,2)=\PP_{\PP^1}(\O(1)\oplus\O(1)\oplus\O(1)\oplus\O(2))\subset\PP^8$, 
which is defined by the $2\times2$ minors of the matrix:
\begin{displaymath}
 \left(\begin{array}{ccccc} x_0 & x_2 & x_4 & x_6 & x_7 \\ x_1 & x_3 & x_5 & x_7 & x_8   \end{array}\right)
\end{displaymath}
Intersect $S(1,1,1,2)$ with the quadric $Q\subset\PP^8$ defined by:
\begin{displaymath}
 x_3^2+x_3x_4+x_0x_5+x_1x_5+x_2x_5+x_3x_5+x_1x_6+x_1x_7+x_6x_7+x_7^2+x_1x_8+x_7x_8 .
\end{displaymath}
We have $S(1,1,1,2)\cap Q=X\cup P$, where 
$P$ is the linear variety $V(x_8,x_7,x_5,x_3,x_1)$, while $X\subset\PP^8$ 
is a nondegenerate smooth variety of degree $9$, sectional genus $3$  
 and  having homogeneous ideal generated by $12$ quadrics. 
These $12$ quadrics give the birational map 
$\psi:\PP^8\dashrightarrow\sS\subset\PP^{11}$ defined by:
\begin{equation}
 \begin{array}{c}
x_7^2-x_6x_8 , \     
x_5x_7-x_4x_8 , \     
x_3x_7-x_2x_8 , \   
x_1x_7-x_0x_8 , \\
x_5x_6-x_4x_7 , \  
x_3x_6-x_2x_7 , \  
x_1x_6-x_0x_7 , \  
x_3x_4-x_2x_5 , \  
x_1x_4-x_0x_5 , \\
x_3^2+x_0x_5+x_1x_5+2x_2x_5+x_3x_5+x_0x_7+x_6x_7+x_0x_8+x_1x_8+x_6x_8+x_7x_8 , \\
x_2x_3+x_0x_4+2x_2x_4+x_0x_5+x_2x_5+x_0x_6+x_6^2+x_0x_7+x_6x_7+x_0x_8+x_6x_8 , \\
x_1x_2-x_0x_3  .
\end{array}
\end{equation}
 The image $\sS$ 
is defined by the quadrics:
\begin{equation}
\begin{array}{c}
y_6y_7-y_5y_8+y_4y_{11} , \ 
y_3y_7-y_2y_8+y_1y_{11} , \ 
y_3y_5-y_2y_6+y_0y_{11} , \\ 
y_3y_4-y_1y_6+y_0y_8 , \ 
y_2y_4-y_1y_5+y_0y_7 .
\end{array}
\end{equation}  
It coincides with the cone over $\GG(1,4)\subset V(y_9,y_{10})\simeq\PP^9\subset\PP^{11}$.  
\end{example} 
\begin{example}[$r=3;n=8;a=4;d=2$]\label{example: 18}
Consider the composition
$$
f:\PP^1\times\PP^3\longrightarrow\PP^1\times Q^3\subset \PP^1\times\PP^4\longrightarrow\PP^9,
$$
where the first map is the identity of $\PP^1$ 
multiplied by
$[z_0,z_1,z_2,z_3]\mapsto [z_0^2,z_0z_1,z_0z_2,z_0z_3,z_1^2+z_2^2+z_3^2]$,
and the last map is
$([t_0,t_1],[y_0,\ldots,y_4])\mapsto [t_0y_0,\ldots,t_0y_4,t_1y_0,\ldots,t_1y_4]
=[x_0,\ldots,x_9]$.
In the equations defining 
$\overline{f(\PP^1\times\PP^3)}\subset\PP^{9}$, 
by replacing 
 $x_9$ with $x_0$, we obtain the quadrics:
\begin{equation}\label{equazioni-di-B-a=4}
 \begin{array}{c}
-x_0x_3 + x_4x_8, \ 
-x_0x_2 + x_4x_7, \ 
 x_3x_7 - x_2x_8, \ 
-x_0x_5 + x_6^2 + x_7^2 + x_8^2, \ 
-x_0x_1 + x_4x_6, \\ 
x_3x_6 - x_1x_8, \ 
x_2x_6 - x_1x_7, \
-x_0^2 + x_1x_6 + x_2x_7 + x_3x_8, \ 
-x_0^2 + x_4x_5, \ 
x_3x_5 - x_0x_8, \\ 
x_2x_5 - x_0x_7, \ 
x_1x_5 - x_0x_6, \ 
x_1^2 + x_2^2 + x_3^2 - x_0x_4.
\end{array}
\end{equation}
Denoting with $I$ the ideal generated by the
quadrics (\ref{equazioni-di-B-a=4}) and $X=V(I)$, 
we have that $I$ is saturated 
and $X$ is smooth.
The linear system $|\I_{X,\PP^8}(2)|$ defines 
a birational transformation 
 $\psi:\PP^8\dashrightarrow\sS\subset \PP^{12}$ 
whose base locus is $X$ and whose image is the variety $\sS$ 
with homogeneous ideal generated by:
\begin{equation}
\begin{array}{c}
 y_6y_9-y_5y_{10}+y_2y_{11}, \ 
y_6y_8-y_4y_{10}+y_1y_{11}, \ 
y_5y_8-y_4y_9+y_0y_{11}, \ 
y_2y_8-y_1y_9+y_0y_{10}, \\ 
y_2y_4-y_1y_5+y_0y_6, \ 
y_2^2+y_5^2+y_6^2+y_7^2-y_7y_8+y_0y_9+y_1y_{10}+y_4y_{11}-y_3y_{12}.
\end{array}
\end{equation}
We have $\deg(\sS)=10$ and  $\dim(\sing(\sS))=3$. 
The inverse of $\psi:\PP^8\dashrightarrow\sS$ 
is defined by:
\begin{equation}
 \begin{array}{c}
-y_7y_8+y_0y_9+y_1y_{10}+y_4y_{11}, \ 
y_0y_5+y_1y_6-y_4y_7-y_{11}y_{12}, \ 
y_0y_2-y_4y_6-y_1y_7-y_{10}y_{12}, \\
-y_1y_2-y_4y_5-y_0y_7-y_9y_{12}, \ 
-y_0^2-y_1^2-y_4^2-y_8y_{12}, \  
-y_3y_8-y_9^2-y_{10}^2-y_{11}^2, \\
-y_3y_4-y_5y_9-y_6y_{10}-y_7y_{11}, \ 
-y_1y_3-y_2y_9-y_7y_{10}+y_6y_{11}, \ 
-y_0y_3-y_7y_9+y_2y_{10}+y_5y_{11}.
\end{array}
\end{equation}
Note that  $\sS\subset\PP^{12}$ 
is the intersection of a quadric hypersurface in $\PP^{12}$ 
with the cone over $\GG(1,4)\subset\PP^9\subset\PP^{12}$.
\end{example}
\begin{example}[$r=3;n=8;a=5$ - with non liftable inverse]\label{example: a=5}
If $X\subset\PP^8$ is the blow-up  of $\PP^3$  at a point $p$ with
 $|H_{X}|=|2H_{\PP^3}-p|$, then (modulo a change of coordinates)
the homogeneous ideal of $X$ is generated by the quadrics: 
\begin{equation}
 \begin{array}{c}
x_6x_7-x_5x_8,\ 
x_3x_7-x_2x_8,\ 
x_5x_6-x_4x_8,\ 
x_2x_6-x_1x_8,\ 
x_5^2-x_4x_7,\ 
x_3x_5-x_1x_8,\ 
x_2x_5-x_1x_7,\\
x_3x_4-x_1x_6,\ 
x_2x_4-x_1x_5,\ 
x_2x_3-x_0x_8,\ 
x_1x_3-x_0x_6,\ 
x_2^2-x_0x_7,\ 
x_1x_2-x_0x_5,\ 
x_1^2-x_0x_4 .
\end{array}
\end{equation}
The linear system $|\I_{X,\PP^8}(2)|$ defines a birational transformation 
$\psi:\PP^{8}\dashrightarrow\PP^{13}$ whose base locus is $X$ and whose 
image is the variety $\sS$ with homogeneous ideal generated by:
\begin{equation}
 \begin{array}{c}
y_8y_{10}-y_7y_{12}-y_3y_{13}+y_5y_{13},\ 
 y_8y_9+y_6y_{10}-y_7y_{11}-y_3y_{12}+y_1y_{13},\ 
 y_6y_9-y_5y_{11}+y_1y_{12},\\ 
 y_6y_7-y_5y_8-y_4y_{10}+y_2y_{12}-y_0y_{13},\  
 y_3y_6-y_5y_6+y_1y_8+y_4y_9-y_2y_{11}+y_0y_{12},\\ 
 y_3y_4-y_2y_6+y_0y_8,\ 
y_3^2y_5-y_3y_5^2+y_1y_3y_7-y_2y_3y_9+y_2y_5y_9-y_0y_7y_9-y_1y_2y_{10}+y_0y_5y_{10}. 
\end{array}
\end{equation}
We have $\deg(\sS)=19$, $\dim(\sing(\sS))=4$ and
the degrees of Segre classes of $X$ are: $s_1=-49$, $s_2=201$, $s_3=-627$. So,
by (\ref{eq: sollevabile}), we deduce that
 the inverse of 
 $\psi:\PP^8\dashrightarrow\sS$ is not liftable; 
however, a representative of the equivalence class of $\psi^{-1}$ is defined by:
\begin{equation}
 \begin{array}{c}
y_{12}^2-y_{11}y_{13},\  
y_8y_{12}-y_6y_{13},\  
y_8y_{11}-y_6y_{12},\  
-y_6y_{10}+y_7y_{11}+y_3y_{12}-y_5y_{12},\  
y_8^2-y_4y_{13},\\  
y_6y_8-y_4y_{12},\   
y_3y_8-y_2y_{12}+y_0y_{13},\  
y_6^2-y_4y_{11},\  
y_5y_6-y_1y_8-y_4y_9.
\end{array}
\end{equation}
We also point out that
  $\Sec(X)$ has dimension 
 $6$ and degree $6$ (against Proposition \ref{prop: B is QEL}).
\end{example}
\begin{example}[$r=3;n=8;a=6;d=2$]\label{example: 20}
 See also \cite{russo-simis} and \cite{semple}.
If 
$X=\PP_{\PP^1}(\O(1)\oplus\O(1)\oplus\O(4))$ or 
   $X=\PP_{\PP^1}(\O(1)\oplus\O(2)\oplus\O(3))$ or
   $X=\PP_{\PP^1}(\O(2)\oplus\O(2)\oplus\O(2))$, then 
the linear system 
$|\I_{X,\PP^8}(2)|$ defines a birational transformation
$\PP^8\dashrightarrow\sS\subset\PP^{14}$ of type $(2,2)$ 
whose base locus is $X$ and whose image is $\sS=\GG(1,5)$.
\end{example}
\begin{example}[$r=3; n=8; a=7; d=1$]\label{example: oadpDegree8}
See also \cite[Example~2.7]{ciliberto-mella-russo} and \cite{ionescu-smallinvariantsIII}.
Let $Z=\{p_1,\ldots,p_8\}\subset\PP^2$ be such that 
no $4$ of the $p_i$ are collinear and no $7$ of the $p_i$ lie on a conic and
consider the scroll $\PP_{\PP^2}(\mathcal{E})\subset\PP^7$ associated 
to the very ample vector bundle $\mathcal{E}$ of rank $2$,
given as an extension by the following exact sequence
$0\rightarrow\O_{\PP^2}\rightarrow\E\rightarrow 
\I_{Z,\PP^2}(4)\rightarrow0.$
The homogeneous ideal of $X\subset\PP^7$ is
generated by $7$ quadrics and so
the linear system $|\I_{X,\PP^8}(2)|$  defines a birational transformation 
$\psi:\PP^8\dashrightarrow\sS\subset\PP^{15}$ of type $(2,1)$.
Since we
have 
$c_1(X)=12$,
$c_2(X)=15$,
$c_3(X)=6$, we deduce
$s_1(\mathcal{N}_{X,\PP^8})=-60$,
$s_2(\mathcal{N}_{X,\PP^8})=267$,
$s_3(\mathcal{N}_{X,\PP^8})=-909$,
and hence $\deg(\sS)=29$, by (\ref{eq: grado mappa razionale}).
The base locus  of the inverse of 
$\psi$ is $\psi(\PP^7)\simeq \PP^6\subset\sS\subset\PP^{15}$.
We also observe that the restriction of $\psi|_{\PP^7}:\PP^7\dashrightarrow\PP^6$
to a general hyperplane $H\simeq\PP^6\subset\PP^7$
gives rise to a transformation as in Example \ref{example: 6}.
\end{example}
\begin{example}[$r=3; n=8; a=8,9; d=1$]\label{example: edge}
If $X\subset\PP^7\subset\PP^8$ is a $3$-dimensional 
Edge variety of degree $7$ (resp. degree $6$), then 
$|\I_{X,\PP^8}(2)|$ defines a birational transformation
$\PP^8\dashrightarrow\sS\subset\PP^{16}$ 
(resp. $\PP^8\dashrightarrow\sS\subset\PP^{17}$) of type $(2,1)$ 
whose base locus is $X$ and 
whose degree of the image is $\deg(\sS)=33$ (resp. $\deg(\sS)=38$).
For memory overflow problems, 
we were not able to calculate the scheme $\sing(\sS)$;
however, it is easy to obtain that 
$1\leq \dim(\sing(\sS))<\dim(Y)=6$
and $\dim(\sing(Y))=1$, where $Y$ denotes the base locus of the inverse.
\end{example}
\begin{example}[$r=3; n=8; a=10; d=1$]\label{example: oadp10}
See also \cite{russo-simis}, \cite{semple} and \cite[\Rmnum{3} Theorem~3.8]{zak-tangent}.
We have a birational transformation
$\PP^{10}\dashrightarrow\GG(1,6)\subset\PP^{20}$ of type $(2,1)$
whose base locus 
is $\PP^1\times\PP^4\subset\PP^9\subset\PP^{10}$ and whose image is $\GG(1,6)$.
Restricting it
  to a general  $\PP^8\subset\PP^{10}$ 
 we obtain a birational transformation 
 $\psi:\PP^8\dashrightarrow\sS\subset\PP^{18}$
whose base locus $X$ is
a rational normal scroll 
(hence either $X=\PP_{\PP^1}(\O(1)\oplus\O(1)\oplus\O(3))$ 
or  $X=\PP_{\PP^1}(\O(1)\oplus\O(2)\oplus\O(2))$)
and whose image $\sS$ is  a linear section of
 $\GG(1,6)\subset\PP^{20}$. 
We denote by $Y\subset\sS$ the base locus of the inverse of $\psi$ and by  
$F=(F_0,\ldots,F_9):\PP^7\dashrightarrow\PP^9$ 
the restriction of $\psi$ to $\PP^7=\Sec(X)$. 
We have
\begin{eqnarray*}
Y&=&\overline{\psi(\PP^7)}=\overline{F(\PP^7)}=\GG(1,4)\subset\PP^9\subset\PP^{18} , \\
J_6&:=&\left\{x=[x_0,...,x_7]\in\PP^7\setminus X: 
\rk\left(\left({\partial F_i}/{\partial x_j}(x)\right)_{i,j}\right)\leq 6   \right\}_{\mathrm{red}}\\
    &=& \left\{x=[x_0,...,x_7]\in\PP^7\setminus X: \dim\left(\overline{F^{-1}\left(F(x)\right)}\right)\geq2   \right\}_{\mathrm{red}}\mbox{ and }\dim\left(J_6\right) = 5,\\
\overline{\psi\left(J_6\right)} &=& \left(\sing\left(\sS\right)\right)_{\mathrm{red}} \subset Y\mbox{ and }\dim\left(\overline{\psi\left(J_6\right)}\right) = 3. \\
\end{eqnarray*}
\end{example}
\section{Summary results}\label{sec: table}
\begin{theorem}\label{theorem: classification}
Table \ref{tabella: all cases 3-fold}
classifies all special quadratic transformations $\varphi$ 
as in \S \ref{sec: notation} and with $r\leq3$.
\end{theorem}
As a consequence, 
we generalize Corollary \ref{prop: classification type 2-3 into cubic}.
\begin{corollary}\label{corollary: coindex 2}
 Let $\varphi:\PP^n\dashrightarrow\sS\subseteq\PP^{n+a}$ be 
as in \S \ref{sec: notation}. If $\varphi$ is of type $(2,3)$ and 
$\sS$ has coindex $c=2$,
then $n=8$, $r=3$ and one of the following cases holds: 
\begin{itemize}
\item $\Delta=3$, $a=1$, $\lambda=11$, $g=5$, $\B$ is the blow-up of $Q^3$ at $5$ points; 
\item $\Delta=4$, $a=2$, $\lambda=10$, $g=4$, $\B$ is a scroll over $Q^2$;  
\item $\Delta=5$, $a=3$, $\lambda=9$,  $g=3$, $\B$ is a quadric fibration over $\PP^1$. 
\end{itemize}
\end{corollary}
\begin{proof}
 We have that $\B\subset\PP^n$ is a $QEL$-variety of type
$\delta=(r-d-c+2)/d=(r-3)/3$ and $n=((2d-1)r+3d+c-2)/d=(5r+9)/3$.
From Divisibility Theorem (Theorem \ref{prop: divisibility theorem}), we deduce 
$(r,n,\delta)\in\{(3,8,0),(6,13,1),(9,18,2)\}$ and 
from the classification 
of $CC$-manifolds (Theorem \ref{prop: classification CC-varieties}), 
we obtain  $(r,n,\delta)=(3,8,0)$. Now we apply the results in \S \ref{sec: dim 3}.
\end{proof}
In the same fashion, one can prove the following:
\begin{proposition}
 Let $\varphi$ be as in \S \ref{sec: notation} and of type $(2,1)$. If
$c=2$, then $r\geq1$ and $\B$ is $\PP^1\times\PP^2\subset\PP^5$ 
or one of its linear sections.
If  $c=3$, then $r\geq2$ and $\B$ is either 
   $\PP^1\times\PP^3\subset\PP^7$ 
or $\GG(1,4)\subset\PP^9$
or one of their linear sections.
If $c=4$, then $r\geq3$ and $\B$ is either an $OADP$ $3$-fold in $\PP^7$ 
or $\PP^1\times\PP^4\subset\PP^{9}$ 
or one of its hyperplane sections.
\end{proposition}
\begin{remark}
Imitating the proof of  Proposition \ref{prop: invariants d=2 Delta=3}
(resp. Proposition \ref{prop: invariants d=2 Delta=4}),
one can also compute all possible Hilbert polynomials and 
Hilbert schemes of lines through a general point 
of the base locus of a
special quadratic birational 
transformation 
of type $(2,2)$ into 
a complete intersection of two quadrics
(resp. three quadrics).
\end{remark}
In Table \ref{tabella: all cases 3-fold} we use the following shortcuts:
\begin{description}
\item[$\exists^{\ast}$] flags cases for which
is known a transformation $\varphi$ with base locus $\B$ as required,
but we do not know if the image $\sS$
satisfies all the assumptions in \S \ref{sec: notation};
\item[$\exists^{\ast\ast}$] flags cases for which
 is known that there is a smooth irreducible variety $X\subset\PP^n$ 
such that, if $X=V(H^0(\I_X(2)))$, then the linear system $|\I_X(2)|$ defines 
a birational  transformation 
$\varphi:\PP^n\dashrightarrow\sS=\overline{\varphi(\PP^n)}\subset\PP^{n+a}$ 
as stated;
\item[$?$] flags cases for which we do not know if there exists at least
an abstract variety
$\B$ having the structure and the invariants required;
\item[$\exists$] flags cases for which everything works fine.
\end{description}
\begin{table}  
\centering
\tabcolsep=5.4pt 
\begin{tabular}{|c||c|c|c|c|c|c|c|c||ll|}
\hline
$r$ & $n$ & $a$ & $\lambda$ & $g$  & Abstract structure of $\B$ & $d$ & $\Delta$ & $c$ & \multicolumn{2}{|c|}{Existence}   \\  
\hline
\hline
\multirow{4}{*}{$1$} & $3$ & $1$ & $2$ & $0$ &  $\nu_2(\PP^1)\subset\PP^2$ & $1$ & $2$ & $1$ & $\exists$ & Ex. \ref{example: 1} \\
\cline{2-11}
 & $4$ & $0$ & $5$ & $1$ & Elliptic curve & $3$ & $1$ & $0$ & $\exists$ & Ex. \ref{example: 2} \\
\cline{2-11}
 & $4$ & $1$ & $4$ & $0$ & $\nu_4(\PP^1)\subset\PP^4$ & $2$ & $2$ & $1$ & $\exists$ & Ex. \ref{example: 3} \\ 
\cline{2-11}
 & $4$ & $3$ & $3$ & $0$ & $\nu_3(\PP^1)\subset\PP^3$ & $1$ & $5$ & $2$ & $\exists$ & Ex. \ref{example: 5} \\
\hline
\hline
\multirow{14}{*}{$2$} & $4$ & $1$ & $2$ & $0$  &  $\PP^1\times\PP^1\subset\PP^3$ & $1$ & $2$ & $1$ & $\exists$ & Ex. \ref{example: 1} \\ 
\cline{2-11}
 & $5$ & $0$ & $4$ & $0$  &  $\nu_2(\PP^2)\subset\PP^5$ & $2$ & $1$ & $0$ & $\exists$ & Ex. \ref{example: 3} \\ 
\cline{2-11}
 & $5$ & $3$ & $3$ & $0$  & Hyperplane section of $\PP^1\times\PP^2\subset\PP^5$ & $1$ & $5$ & $2$ & $\exists$ & Ex. \ref{example: 5} \\ 
\cline{2-11}
 & $6$ & $0$ & $7$ & $1$  & Elliptic scroll $\PP_{C}(\E)$ with $e(\E)=-1$ & $4$ & $1$ & $0$& $\exists$ & Ex. \ref{example: 6} \\ 
\cline{2-11}
 & $6$ & $0$ & $8$ & $3$  & \begin{tabular}{c} Blow-up of $\PP^2$ at $8$ points $p_1,\ldots,p_8$,\\ $|H_{\B}|=|4H_{\PP^2}-p_1-\cdots-p_8|$ \end{tabular} & $4$ & $1$ & $0$&$\exists$ & Ex. \ref{example: 6} \\ 
\cline{2-11}
 & $6$ & $1$ & $7$ & $2$  & \begin{tabular}{c} Blow-up of $\PP^2$ at $6$ points $p_0,\ldots,p_5$,\\ $|H_{\B}|=|4H_{\PP^2}-2p_0-p_1-\cdots-p_5|$ \end{tabular} & $3$ & $2$ & $1$& $\exists$ & Ex. \ref{example: 7} \\ 
\cline{2-11}
 & $6$ & $2$ & $6$ & $1$  & \begin{tabular}{c} Blow-up of $\PP^2$ at $3$ points $p_1,p_2,p_3$,\\ $|H_{\B}|=|3H_{\PP^2}-p_1-p_2-p_3|$ \end{tabular} & $2$ & $4$ & $2$& $\exists$ & Ex. \ref{example: 8} \\ 
\cline{2-11}
 & $6$ & $3$ & $5$ & $0$  & $\PP_{\PP^1}(\O(1)\oplus\O(4))$ or $\PP_{\PP^1}(\O(2)\oplus\O(3))$   & $2$ & $5$ & $2$ & $\exists$ & Ex. \ref{example: 9} \\ 
\cline{2-11}
 & $6$ & $5$ & $5$ & $1$  & \begin{tabular}{c} Blow-up of $\PP^2$ at $4$ points $p_1\ldots,p_4$,\\ $|H_{\B}|=|3H_{\PP^2}-p_1-\cdots-p_4|$ \end{tabular} & $1$ & $12$ & $3$ & $\exists$ & Ex. \ref{example: 10} \\ 
\cline{2-11}
 & $6$ & $6$ & $4$ & $0$  & $\PP_{\PP^1}(\O(1)\oplus\O(3))$ or $\PP_{\PP^1}(\O(2)\oplus\O(2))$ & $1$ & $14$ & $3$ & $\exists$ & Ex. \ref{example: 11} \\ 
\hline
\hline
\multirow{23}{*}{$3$} & $5$ & $1$ & $2$ & $0$  & $Q^3\subset\PP^4$ & $1$ & $2$ & $1$&$\exists$ & Ex. \ref{example: 1} \\
\cline{2-11}
 & $6$ & $3$ & $3$ & $0$  & $\PP^1\times\PP^2\subset\PP^5$ & $1$ & $5$ & $2$& $\exists$ & Ex. \ref{example: 5} \\
\cline{2-11}
 & $7$ & $1$ & $6$ & $1$  & Hyperplane section of $\PP^2\times\PP^2\subset\PP^8$ & $2$ & $2$ & $1$& $\exists$ & Ex. \ref{example: 3} \\
\cline{2-11}
 & $7$ & $5$ & $5$ & $1$  & Linear section of $\GG(1,4)\subset\PP^9$ & $1$ & $12$ &$3$ & $\exists$ & Ex. \ref{example: 10} \\ 
\cline{2-11}
 & $7$ & $6$ & $4$ & $0$  & $\PP_{\PP^1}(\O(1)\oplus\O(1)\oplus\O(2))$ & $1$ & $14$ &$3$ & $\exists$ & Ex. \ref{example: 11} \\
\cline{2-11}
 & $8$ & $0$ & $12$ & $6$    & \begin{tabular}{c} Scroll $\PP_{Y}(\E)$, $Y$ rational surface, \\ $K_Y^2=5$, $c_2(\E)=8$, $c_1^2(\E)=20$ \end{tabular} & $5$ & $1$ & $0$ & $?$ & Ex. \ref{example: parzialecremona} \\
\cline{2-11}
 & $8$ & $0$ & $13$ & $8$    & \begin{tabular}{c} Variety obtained as the projection \\ of a Fano variety $X$ from a point $p\in X$ \end{tabular} & $5$ & $1$ & $0$ & $\exists$ & Ex. \ref{example: 12} \\
\cline{2-11}
 & $8$ & $1$ &  $11$   &  $5$  & \begin{tabular}{c} Blow-up of $Q^3$ at $5$ points $p_1,\ldots,p_5$, \\ $|H_{\B}|=|2H_{Q^3}-p_1-\cdots-p_5|$  \end{tabular} & $3$ &  $3$ & $2$ & $\exists$ &  Ex. \ref{example: 13} \\
\cline{2-11}
 & $8$ & $1$ &  $11$   &  $5$  & Scroll over $\PP_{\PP^1}(\O\oplus\O(-1))$  & $4$ &  $2$& $1$ & $\exists^{\ast\ast}$ & Ex. \ref{example: 14} \\
\cline{2-11}
 & $8$ & $1$ &  $12$   &  $7$ & Linear section of $S^{10}\subset\PP^{15}$  & $4$ & $2$& $1$ &$\exists$ & Ex. \ref{example: 15} \\
\cline{2-11}
 & $8$ & $2$ & $10$ & $4$ & Scroll over $Q^2$ & $3$ & $4$ & $2$ & $\exists^{\ast}$ & Ex. \ref{example: 16} \\
\cline{2-11}
 & $8$ & $3$ & $9$ & $3$ & Scroll over $\PP^2$ & $2$ & $8$ & $3$ & $\exists$ & Ex. \ref{example: 17} \\
\cline{2-11}
 & $8$ & $3$ & $9$ & $3$ & Quadric fibration over $\PP^1$ & $3$ & $5$ & $2$ & $\exists^{\ast}$ & Ex. \ref{example: 17nuovo} \\
\cline{2-11}
 & $8$ &$4$ & $8$ & $2$ & Hyperplane section of $\PP^1\times Q^3$ & $2$ & $10$ & $3$ & $\exists^{\ast}$ & Ex. \ref{example: 18}  \\
\cline{2-11}
 & $8$ &$6$ & $6$ & $0$ & Rational normal scroll & $2$ & $14$ & $3$ & $\exists$ &  Ex. \ref{example: 20} \\
\cline{2-11}
 & $8$ & $7$ & $8$ & $3$  & \begin{tabular}{c} $\PP_{\PP^2}(\E)$, where  $0\rightarrow\O_{\PP^2}\rightarrow$ \\$\rightarrow\E\rightarrow \I_{\{p_1,\ldots,p_8\},\PP^2}(4)\rightarrow0$ \end{tabular}  & $1$ & $29$ & $4$ & $\exists^{\ast}$ &  Ex. \ref{example: oadpDegree8} \\
\cline{2-11}
 & $8$ & $8$ & $7$ & $2$  & Edge variety & $1$ & $33$ & $4$ & $\exists^{\ast}$ & Ex. \ref{example: edge} \\
\cline{2-11}
 & $8$ & $9$ & $6$ & $1$  & $\PP^1\times\PP^1\times\PP^1\subset\PP^7$ & $1$ & $38$ & $4$ & $\exists^{\ast}$ & Ex. \ref{example: edge} \\
\cline{2-11}
 & $8$ & $10$ & $5$ & $0$  & Rational normal scroll & $1$ & $42$ & $4$ & $\exists$ & Ex. \ref{example: oadp10} \\ 
\hline
\end{tabular}
 \caption{All transformations $\varphi$ as in \S \ref{sec: notation} and with $r\leq3$}
\label{tabella: all cases 3-fold} 
\end{table}

\begin{appendices}
\chapter{Towards the study of special quadratic birational transformations whose base locus has dimension four}\label{app: towards the case of dimension four}
In this 
 appendix we shall keep the notation of 
Chap. \ref{chapter: transformations whose base locus has dimension at most three}. 
We treat the case in which
the dimension of the base locus is
 $r=4$, although,
 when $\delta=0$, 
we are well away from having an exhaustive classification.
\section{Easy cases}
Similarly to Propositions \ref{prop: possibili casi 2-fold} and
\ref{prop: possibili casi C1} we deduce Proposition \ref{prop: delta mag0 4fold}
from the theory of $LQEL$-varieties; specifically, from  
Proposition \ref{prop: LQEL with delta=r and delta=r-1}, 
  Theorems \ref{prop: LQEL of type delta=r/2}
and \ref{prop: classification CC-varieties}.
\begin{proposition}\label{prop: delta mag0 4fold} 
Let Assumptions \ref{assumption: base}, \ref{assumption: liftable} and \ref{assumption: ipotesi} be valid.
If $r=4$, then either $n=10$, $d\geq2$, $\langle \B \rangle = \PP^{10}$, or  
one of the following cases holds:
\begin{enumerate}[(I)]
 \item  $n=6$,  $d=1$, $\delta=4$, $\B=Q^4\subset\PP^5$ is a quadric;
 \item  $n=8$,  $d=1$, $\delta=2$, $\B\subset\PP^7$ is either $\PP^1\times\PP^3\subset\PP^7$ or a linear section of $\GG(1,4)\subset\PP^9$;
 \item  $n=8$,  $d=2$, $\delta=2$, $\B$ is $\PP^2\times\PP^2\subset\PP^8$;
 \item  $n=9$,  $d=1$, $\delta=1$, $\B$ is a hyperplane section of $\PP^1\times\PP^4\subset\PP^9$;
 \item  $n=10$, $d=1$, $\delta=0$, $\B\subset\PP^9$ is an $OADP$-variety. 
\end{enumerate}
\end{proposition}
\begin{remark}
 Note that 
in Proposition \ref{prop: delta mag0 4fold}, all cases  with $\delta>0$ 
really occur (see \S \ref{sec: examples});
when $\delta=0$,  
an example is obtained by taking a
general $4$-dimensional linear section of 
$\PP^1\times\PP^5\subset\PP^{11}\subset\PP^{12}$.
Unfortunately, the classification of $OADP$ $4$-folds in $\PP^{10}$ is not known;
so we cannot be more precise.
\end{remark}

\section{Hard cases}
In this section we keep Assumption \ref{assumption: base}, but  
more generally assume that the image $\sS$ is  
nondegenerate, normal and linearly normal (not necessarily factorial); 
we do not require 
Assumptions 
\ref{assumption: liftable} and \ref{assumption: ipotesi}.
As noted earlier, we have $P_{\B}(1)=11$ and $P_{\B}(2)=55-a$ and hence 
\begin{displaymath}
 P_{\B}(t) = \lambda\begin{pmatrix}t\cr 4\end{pmatrix}+\left( 3\lambda+1-g\right) \begin{pmatrix}t\cr 3\end{pmatrix}+\left( \chi(\O_{\B})-a+33\right) \begin{pmatrix}t\cr 2\end{pmatrix}  
           +\left( 11-\chi(\O_{\B})\right) t+\chi(\O_{\B}) . 
\end{displaymath} 
\begin{proposition}\label{prop: 4foldnondegenerate}
If $r=4$, $n=10$ and $\langle \B\rangle=\PP^{10}$, 
then  one of the following cases holds: 
\begin{enumerate}[(I)]
\setcounter{enumi}{5} 
 \item $a=10$, $\lambda=7$,             $g=0$,      $\chi(\O_{\B})=1$,                 $\B$ is a rational normal scroll;
 \item $a=7$,  $\lambda=10$,            $g=3$,      $\chi(\O_{\B})=1$,                 $\B$  is either
  \begin{itemize}
  \item a hyperplane section of $\PP^1\times Q^4\subset\PP^{11}$ or
  \item $\PP(\T_{\PP^2}\oplus \O_{\PP^2}(1))\subset\PP^{10}$;
  \end{itemize}
 \item $a=6$,  $\lambda=11$,            $g=4$,      $\chi(\O_{\B})=1$,                 $\B$ is a quadric fibration over $\PP^1$;
 \item $a=5$,  $\lambda=12$,            $g=5$,      $\chi(\O_{\B})=1$,                 $\B$ is one of the following:
\begin{itemize}
      \item $\PP^4$ blown up at $4$ points $p_1\ldots,p_4$ embedded by $|2H_{\PP^4}-p_1-\cdots-p_4|$,
      \item a scroll over a ruled surface,
      \item a quadric fibration over $\PP^1$; 
\end{itemize}         
 \item $a=4$,  $\lambda=14$,            $g=8$,      $\chi(\O_{\B})=1$,                 $\B$ is either
 \begin{itemize}
 \item a linear section of $\GG(1,5)\subset\PP^{14}$ or 
 \item the product of $\PP^1$ with a Fano variety of even index;  
 \end{itemize}
 \item $a=4$,  $\lambda=13$,            $g=6$,      $\chi(\O_{\B})=1$,                 $\B$ is either  
\begin{itemize}
      \item a scroll over a birationally ruled surface or
      \item a quadric fibration over $\PP^1$; 
\end{itemize} 
 \item $a=3$,  $14\leq\lambda\leq 16$,  $g\leq 11$, $\chi(\O_{\B})=(-g+2\lambda-18)/3$;
 \item $a=2$,  $15\leq \lambda\leq 18$, $g\leq 14$, $\chi(\O_{\B})=(-g+2\lambda-19)/3$; 
 \item $a=1$, $15\leq\lambda\leq 20$, $g\leq17$, $\chi(\O_{\B})=(-g+2\lambda-20)/3$;
 \item $a=0$, $15\leq\lambda$.
\end{enumerate}
\end{proposition}
\begin{proof}
Denote by $\Lambda\subsetneq C\subsetneq S\subsetneq X\subsetneq \B$ a sequence of 
general linear sections of $\B$ and put  
$h_{\Lambda}(2):=h^0(\PP^6,\O(2))-h^0(\PP^6,\I_{\Lambda}(2))$.
Since $C$ is a nondegenerate curve in $\PP^7$, 
 we have $\lambda\geq 7$. 
By Castelnuovo's argument (Lemma \ref{prop: castelnuovo argument}), 
it follows that
\begin{equation}\label{eq: castelnuovo-argument-4fold}
7\leq \min\{\lambda,13\}\leq h_{\Lambda}(2)\leq 28 - h^0(\PP^{10},\I_{\B}(2))=17-a 
\end{equation}
and in particular we have $a\leq 10$.
Moreover
\begin{itemize}
 \item if $\lambda\geq 13$, then $h_{\Lambda}(2)\geq 13$ and $a\leq4$, by (\ref{eq: castelnuovo-argument-4fold});
 \item if $\lambda\geq 15$, then $h_{\Lambda}(2)\geq 14$ and $a\leq3$, by Proposition \ref{prop: castelnuovo lemma}; 
 \item if $\lambda\geq 17$, then $h_{\Lambda}(2)\geq 15$ and $a\leq2$, by Theorem \ref{prop: extension castelnuovo lemma by fano harris}; 
 \item if $\lambda\geq 19$, then $h_{\Lambda}(2)\geq 16$ and $a\leq1$, by Theorem \ref{prop: extension castelnuovo lemma by ciliberto}; 
 \item if $\lambda\geq 21$, then $h_{\Lambda}(2)\geq 17$ and $a=0$, by Theorem \ref{prop: petrakiev}(\ref{part: petrakiev vartheta=3}). 
\end{itemize}
According to the above statements,
we consider the refinement 
$\theta=\theta(\lambda)$ of Castelnuovo's bound $\rho=\rho(\lambda)$, contained in 
Theorem \ref{prop: refinement castelnuovo bound by ciliberto}.
So, we have
\begin{equation}\label{eq: KBHB3}
K_{\B}\cdot H_{\B}^3=2g-2-3\lambda \leq 2\theta(\lambda)-2-3\lambda\leq 2\rho(\lambda)-2-3\lambda .
\end{equation}
Now, if $t\geq1$, by Kodaira Vanishing Theorem and Serre Duality,
it follows that $P_{\B}(-t)=h^4(\B,\O_{\B}(-t))=h^0(\B,K_{\B}+tH_{\B}) $;
hence, if $P_{\B}(-t)\neq0$, then $K_{\B}+tH_{\B}$ is an effective divisor and 
 we have either $K_{\B}\cdot H_{\B}^3 > -tH_{\B}^4=-t \lambda$ or $K_{\B}\sim -tH_{\B}$.
Thus, by (\ref{eq: KBHB3}) and  
 straightforward calculation, we deduce (see Figure \ref{fig: upperbounds}):
\begin{figure}[htb] 
\centering
\includegraphics[width=0.7\textwidth]{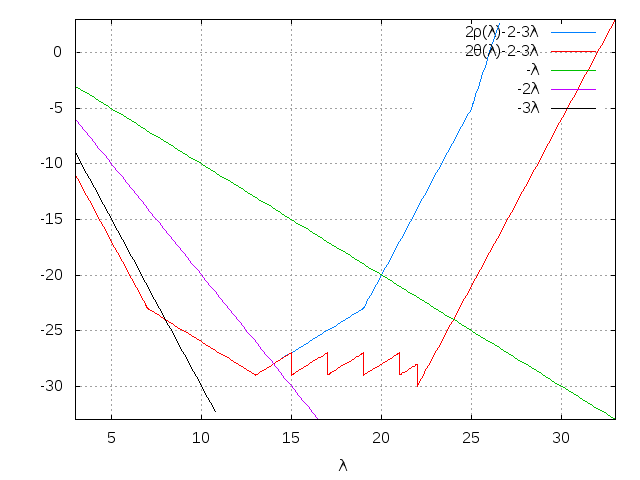} 
\caption{Upper bounds of $K_{\B}\cdot H_{\B}^3$}
\label{fig: upperbounds}
\end{figure}
\begin{enumerate}[(\ref{prop: 4foldnondegenerate}.a)]
 \item\label{case: lambda8}  if $\lambda\leq8$,  then either $P_{\B}(-3)=P_{\B}(-2)=P_{\B}(-1)=0$ or $\lambda=8$ and $K_{\B}\sim -3H_{\B}$;
 \item\label{case: lambda14} if $\lambda\leq14$, then either $P_{\B}(-2)=P_{\B}(-1)=0$ or $\lambda=14$ and $K_{\B}\sim -2H_{\B}$;
 \item\label{case: lambda20} if $\lambda\leq24$, then either $P_{\B}(-1)=0$ or $\lambda=24$ and $K_{\B}\sim -H_{\B}$.
\end{enumerate}
In the same way, one also sees that $h^4(\B,\O_{\B})=0$ 
whenever $\lambda\leq 31$.
Now we discuss the cases according to the value of $a$.
\begin{case}[$9\leq a\leq 10$] We have $\lambda\leq 8$.
From the classification of del Pezzo varieties 
in Theorem \ref{prop: classification del pezzo varieties},
we see that the case $\lambda=8$ with $K_{\B}\sim -3H_{\B}$  is impossible and so
we obtain $\lambda=11-2a/5$, $g=1-a/10$, 
by (\ref{prop: 4foldnondegenerate}.\ref{case: lambda8}).
Hence $a=10$, $\lambda=7$, $g=0$ and $\B$ is a rational normal scroll.
\end{case}
\begin{case}[$5\leq a\leq 8$] We have $\lambda\leq 12$.
  By (\ref{prop: 4foldnondegenerate}.\ref{case: lambda14}) we obtain 
 $g=(3\lambda+a-31)/2$ and $\chi(\O_{\B})=(\lambda+a-11)/6$ and,
since  $\chi(\O_{\B})\in \ZZ$,
we obtain $\lambda=17-a$, $g=10-a$, $\chi(\O_{\B})=1$. So,
we can determine the abstract structure of $\B$ by
\cite{fania-livorni-ten},
\cite{besana-biancofiore-deg11},
\cite[Theorem~2]{ionescu-smallinvariantsII},
\cite[Lemmas~4.1 and 6.1]{besana-biancofiore-numerical}
 and  we also deduce that the case $a=8$ does not occur, by \cite{fania-livorni-nine}.
\end{case}
\begin{case}[$a = 4$] We have $\lambda\leq 14$. 
 Again by (\ref{prop: 4foldnondegenerate}.\ref{case: lambda14}), we deduce that either
 $g=(3\lambda-27)/2$ and $\chi(\O_{\B})=(\lambda-7)/6$ or
 $\B$ is a Mukai variety with 
$\lambda=14$ ($g=8$ and $\chi(\O_{\B})=1$).
In the first case, since
$\chi(\O_{\B})\in \ZZ$ and $g\geq 0$,
we obtain $\lambda=13$, $g=6$, $\chi(\O_{\B})=1$ and 
then we can determine the abstract structure of $\B$ by 
\cite[Theorem~1]{ionescu-degsmallrespectcodim} and
\cite[Lemmas~4.1 and 6.1]{besana-biancofiore-numerical}.
In the second case, if $b_2=b_2(\B)=1$ then $\B$ is a linear section 
of $\GG(1,5)\subset\PP^{14}$, 
otherwise $\B$ is a Fano variety of product type, see
 \cite[Theorems~2 and 7]{mukai-biregularclassification}.
\end{case}
\begin{case}[$a=3$] We have $\lambda\leq 16$ and 
$\chi(\O_{\B})=(-g+2\lambda-18)/3$, 
by (\ref{prop: 4foldnondegenerate}.\ref{case: lambda20}).
Moreover, if $\lambda\leq14$, 
by (\ref{prop: 4foldnondegenerate}.\ref{case: lambda14}) it follows
that $\lambda=14$, $g=7$ and $\chi(\O_{\B})=1$.
\end{case}
\begin{case}[$a=2$] We have $\lambda\leq 18$   
and $\chi(\O_{\B})=(-g+2\lambda-19)/3$, 
by (\ref{prop: 4foldnondegenerate}.\ref{case: lambda20}).
Moreover, by (\ref{prop: 4foldnondegenerate}.\ref{case: lambda14}) it follows that
$\lambda\geq 15$.
\end{case}
\begin{case}[$a=1$] 
We have $\lambda\leq 20$ and 
$\chi(\O_{\B})=(-g+2\lambda-20)/3$, 
by (\ref{prop: 4foldnondegenerate}.\ref{case: lambda20}).
Moreover, if $\lambda\leq14$, 
by (\ref{prop: 4foldnondegenerate}.\ref{case: lambda14}) it follows
that  $\lambda=10$, $g=0$, $\chi(\O_{\B})=0$,
which is of course impossible.
\end{case}
\begin{case}[$a=0$]\label{case: 4-fold cremona}
If $\lambda\leq 14$,   
by (\ref{prop: 4foldnondegenerate}.\ref{case: lambda14})
and (\ref{prop: 4foldnondegenerate}.\ref{case: lambda20}) 
it follows that  $\lambda=11$, $g=1$, $\chi(\O_{\B})=0$. 
Thus, $\B$ must be an elliptic scroll  
and $\varphi$ must be of type $(2,6)$; so, 
by (\ref{eq: c2 4-fold}) we obtain the contradiction
$c_2(\B)\cdot H_{\B}^2=(990+c_4(\B))/37=990/37\notin\ZZ$.
\end{case}
\end{proof}
\begin{remark}
 Reasoning as in Proposition \ref{prop: segre and chern classes}, 
we obtain that
if $\varphi$ is of type $(2,d)$, then
 \begin{eqnarray}
\label{eq: c2 4-fold}  37c_2(\B)\cdot H_{\B}^2-c_4(\B) &=&  -231\lambda+188g+(1-9d)\Delta+3396 ,\\
  37c_3(\B)\cdot H_{\B}+7c_4(\B) &=& 655\lambda-428g+(26d-7)\Delta-5716 .
 \end{eqnarray}
\end{remark}
\begin{remark}
If Conjecture \ref{conjecture: on upper bound of the degree} (with $N=6$ and $m=11$) 
holds, then we have that $\lambda\leq 24$,
even in the case with $a=0$.
If $a=0$ and $\lambda\leq 24$, we have $g\leq \theta(24)=25$ and 
one of the following cases holds:
\begin{itemize}
 \item $\lambda=24$, $g=25$, $\chi(\O_{\B})=1$ and $\B$ is a Fano variety of coindex $4$; 
 \item $g\leq 24$ and $\chi(\O_{\B})=(-g+2\lambda-21)/3$.
\end{itemize}
\end{remark}
\begin{example} 
Below we collect some examples of 
special quadratic birational transformations appearing in Proposition
\ref{prop: 4foldnondegenerate}.
\begin{description}
\item[($a=10$)] If $X\subset\PP^{10}$ is a (smooth) $4$-dimensional rational normal scroll, 
then $|\I_{X,\PP^{10}}(2)|$
defines a birational transformation 
$\psi:\PP^{10}\dashrightarrow\GG(1,6)\subset\PP^{20}$ of type $(2,2)$.
\item[($a=7$)] If $X\subset\PP^{10}$ is a general hyperplane section 
of $\PP^1\times Q^4\subset\PP^{11}$,
then $|\I_{X,\PP^{10}}(2)|$ defines a birational transformation 
$\psi:\PP^{10}\dashrightarrow \overline{\psi(\PP^{10})}\subset\PP^{17}$
of type $(2,2)$ whose image has degree $28$.
\item[($a=7$)] If $X=\PP(\T_{\PP^2}\oplus\O_{\PP^2}(1))\subset\PP^{10}$, since 
$h^1(X,\O_X)=h^1(\PP^2,\O_{\PP^2})=0$, $|\I_{X,\PP^{10}}(2)|$ defines a birational 
transformation 
$\psi:\PP^{10}\dashrightarrow \overline{\psi(\PP^{10})}\subset\PP^{17}$ 
(see Facts \ref{fact: test K2} and \ref{fact: K2 property}).
\item[($a=6$)] There exists a smooth linearly normal $4$-dimensional
variety $X\subset\PP^{10}$ with $h^1(X,\O_X)=0$, degree $11$, sectional genus $4$,
having the structure of a quadric fibration over $\PP^1$ 
(see \cite[Remark~3.2.5]{besana-biancofiore-deg11});
thus $|\I_{X,\PP^{10}}(2)|$ defines a birational 
transformation 
$\psi:\PP^{10}\dashrightarrow \overline{\psi(\PP^{10})}\subset\PP^{16}$
(see Facts \ref{fact: test K2} and \ref{fact: K2 property}).
\item[($a=5$)] If $X\subset\PP^{10}$ is the blow-up of $\PP^4$ at $4$ 
general points $p_1,\ldots,p_4$,
embedded by $|2H_{\PP^4}-p_1-\cdots-p_4|$, then
$|\I_{X,\PP^{10}}(2)|$ defines a birational transformation 
$\psi:\PP^{10}\dashrightarrow \overline{\psi(\PP^{10})}\subset\PP^{15}$
whose image has degree $29$; in this case $\Sec(X)$ is a
complete intersection of two cubics.
\item[($a=4$)] If $X\subset\PP^{10}$ is a general $4$-dimensional linear section of 
$\GG(1,5)\subset\PP^{14}$, then $|\I_{X,\PP^{10}}(2)|$ defines 
a birational transformation 
$\psi:\PP^{10}\dashrightarrow \overline{\psi(\PP^{10})}\subset\PP^{14}$
of type $(2,2)$ whose image is a complete intersection of quadrics.
\end{description}
\end{example}
\begin{remark} Note that in Case \ref{case: 4-fold cremona} and Claim \ref{claim: KsHs<0} 
we excluded the case where the base locus  
of a Cremona transformation is 
an elliptic scroll of dimension $r$ and degree $2r+3$ in $\PP^{2r+2}$.
Actually, in \cite{semple-tyrrell} is stated that 
the quadrics through a generic such a scroll give 
a Cremona transformation;
although, for $r\geq3$ the scroll is not cut out by quadrics and 
so the Cremona transformation is not special (see \cite[\S 4.3]{ein-shepherdbarron}).

Note also that in \cite[\S 5]{hulek-katz-schreyer} a series 
of quadratic Cremona transformations $\psi_r:\PP^{2r+2}\dashrightarrow \PP^{2r+2}$
 having $r$-dimensional base locus is built.
For $r=2$ and $r=3$, $\psi_r$ is respectively as in Examples \ref{example: 6} and \ref{example: 12},
but $\psi_4$ is not special.
\end{remark}

\section
{Numerical invariants of transformations of type \texorpdfstring{$(2,2)$}{(2,2)} into a quintic hypersurface}
Just as in Propositions \ref{prop: invariants d=2 Delta=3} and 
\ref{prop: invariants d=2 Delta=4}
(and also keeping the same notation)
we can determine the possible numerical invariants 
for special birational transformations of type $(2,2)$ into a quintic
hypersurface.

Indeed, for such a transformation,
from Theorem \ref{prop: divisibility theorem} 
and Proposition \ref{prop: cohomology properties},
we obtain either $(n,r,\delta)=(10,4,0)$ or $\B$ is
a Fano variety of the first species 
of  coindex $c(\B)$ and Hilbert polynomial $P$,   as one of the two following cases:
\begin{enumerate}[(i)]
\item $n=16$,   $r=8$,   $\delta=2$,   $c(\B)=4$, 
$P=$ $36$, $216$, $552$, $780$, $661$, $340$, $102$, $16$, $1$;
\item $n=22$,   $r=12$,   $\delta=4$,   $c(\B)=5$,
$P=$ $84$, $798$, $3428$, $8789$, $14946$, $17711$, $14945$, $9009$, $3829$, $1111$, $207$, $22$, $1$. 
\end{enumerate}

In the case in which $\B$ is a $4$-fold in $\PP^{10}$, we apply 
Proposition \ref{prop: 4foldnondegenerate}.
So the Hilbert polynomial of $\B$ 
can be expressed as a function of $\lambda$ and $g$.  
Moreover $g\equiv 1-\lambda\ \mathrm{mod}\ 3$ 
and one of the following holds:
$\lambda=15$ and $g\leq 9$;
$\lambda=16$ and $g\leq 11$;
$\lambda=17$ and $g\leq 12$;
$\lambda=18$ and $g\leq 14$;
$\lambda=19$ and $g\leq 15$;
$\lambda=20$ and $g\leq 17$.

\end{appendices}

\backmatter
\cleardoublepage  \phantomsection 
\addcontentsline{toc}{chapter}{Bibliography} 
\bibliographystyle{amsalpha}
\bibliography{bibliography.bib}

\end{document}